\documentclass{acta_2024}
\usepackage{amssymb}
\usepackage{amsmath}

\usepackage[longnamesfirst]{natbib} \usepackage{har2nat}
\setcitestyle{comma,aysep={}} 
\shortcites{} %for 5+ authors

\usepackage{newtxtext} %Times
\usepackage{newtxmath} %Times math
\DeclareSymbolFont{CMlargesymbols}{OMX}{cmex}{m}{n} %CM big ( )
\DeclareMathDelimiter{(}{\mathopen} {operators}{"28}{CMlargesymbols}{"00}
\DeclareMathDelimiter{)}{\mathclose}{operators}{"29}{CMlargesymbols}{"01}
\DeclareMathAlphabet\mathcal{OMS}{cmsy}{m}{n} %CM mathcal
\SetMathAlphabet\mathcal{bold}{OMS}{cmsy}{b}{n} %CM bold mathcal
\pagerange{\pageref{firstpage}--\pageref{lastpage}}
\doi{XXXXXXXX}  %add last 8 figures from DOI supplied by CUP
\usepackage[twoside,
	paperheight=247mm,
	paperwidth=174mm,
	marginratio={3:5,2:3},
	left=18mm,
	top=20mm]{geometry}
\usepackage{amssymb,amsmath}
\numberwithin{figure}{section}
\numberwithin{table}{section}
\usepackage[cmyk]{xcolor}
\usepackage{graphics,graphicx}
\usepackage{subcaption}
\usepackage[UKenglish]{babel}
\usepackage{enumerate} %or use \usepackage[shortlabels]{enumitem}
\newcommand{\ignore}[1]{}
\allowdisplaybreaks
\newtheorem{theorem}{Theorem}[section]
\newtheorem{lemma}[theorem]{Lemma}
\newtheorem{proposition}[theorem]{Proposition}
\newtheorem{corollary}[theorem]{Corollary}
\newtheorem{example}[theorem]{Example}
\newtheorem{remark}[theorem]{Remark}
\newtheorem{definition}[theorem]{Definition}

\usepackage[ntheorem]{empheq}

\newcommand*\widefbox[1]{\fbox{\hspace{2em}#1\hspace{2em}}}

\usepackage{natbib}
\usepackage{lineno}
\usepackage{algorithm}
\usepackage{algorithmic}
\usepackage[toc,page]{appendix}
\usepackage{subcaption}
\newtheorem{assumptions}[theorem]{Assumptions}

\newcommand*\patchAmsMathEnvironmentForLineno[1]{%
  \expandafter\let\csname old#1\expandafter\endcsname\csname #1\endcsname
  \expandafter\let\csname oldend#1\expandafter\endcsname\csname end#1\endcsname
  \renewenvironment{#1}%
     {\linenomath\csname old#1\endcsname}%
     {\csname oldend#1\endcsname\endlinenomath}}% 
\newcommand*\patchBothAmsMathEnvironmentsForLineno[1]{%
  \patchAmsMathEnvironmentForLineno{#1}%
  \patchAmsMathEnvironmentForLineno{#1*}}%
\AtBeginDocument{%
\patchBothAmsMathEnvironmentsForLineno{equation}%
\patchBothAmsMathEnvironmentsForLineno{align}%
\patchBothAmsMathEnvironmentsForLineno{flalign}%
\patchBothAmsMathEnvironmentsForLineno{alignat}%
\patchBothAmsMathEnvironmentsForLineno{gather}%
\patchBothAmsMathEnvironmentsForLineno{multline}%
}

\usepackage{bbm}
\usepackage{lineno}
\usepackage{url}
\usepackage{listings}
\definecolor{darkgreen}{rgb}{0.1, 0.14, 0.13}

\renewcommand{\nu}{\mathfrak{\pi}}

\newcommand{\ppi}{\pi}
\newcommand{\rv}{\mathsf{rv}}
\newcommand{\hs}{\mathsf{h}}
\newcommand{\Hs}{\mathsf{H}}
\newcommand{\Hess}{\mathsf{S}}
\newcommand{\Gammas}{\mathsf{\Gamma}}
\newcommand{\Sigmas}{\mathsf{\Sigma}}

\newcommand{\ssS}{\Omega}
\newcommand{\varpsi}{\psi}
\newcommand{\law}{\textit{Law}\,} % edo change

\newcommand{\tnu}{\widetilde{\nu}}

\newcommand{\tu}{\tilde{u}}
\newcommand{\tPsi}{\tilde{\Psi}}
\newcommand{\trho}{\tilde{\rho}}

\newcommand{\tK}{\widetilde{K}}
\newcommand{\phhi}{\psi}
\newcommand{\CC}{\mathcal{C}}
\newcommand{\mP}{{\mathfrak P}}
\newcommand{\Tp}{T_\rho\mP_+}
\newcommand{\mE}{\mathcal{E}}
\newcommand{\sM}{\mathcal{M}}

\newcommand{\op}{\mathcal}

\newcommand{\mop}{m_{\rm post}}
\newcommand{\Cp}{C_{\rm post}}
\newcommand{\normZ}{\mathscr{Z}}

\newcommand{\cL}{\mathcal{L}}

\newcommand{\mmu}{\rho}
\newcommand{\sJ}{\mathsf{J}}

\newcommand{\sI}{\mathsf{I}}
\newcommand{\hmu}{\widehat{\mu}}
\newcommand{\hh}{\widehat{h}}
\newcommand{\hw}{\widehat{w}}
\newcommand{\tT}{\widetilde{T}^{S}}
\newcommand{\tTD}{\widetilde{T}^{D}}
\newcommand{\hmug}{\widehat{\mu}^G}

\newcommand{\mean}{m}
\newcommand{\pmean}{\widehat{m}}

\newcommand{\ho}{\widehat{o}}
\newcommand{\hoo}{\widehat{\mathrm{o}}}

\newcommand{\mh}{\pmean}

\newcommand{\tC}{\widetilde{C}}
\newcommand{\ttC}{\check{C}}
\newcommand{\Cov}{C}
\newcommand{\Csov}{\Sigmas}
\newcommand{\pCov}{\widehat{C}}

\newcommand{\hC}{\pCov}
\newcommand{\hCvy}{\pCov^{vy}}
\newcommand{\hCyy}{\pCov^{yy}}
\newcommand{\hCvh}{\pCov^{vh}}

\newcommand{\hChh}{\pCov^{hh}}

\newcommand{\CuG}{\Cov^{uG}}
\newcommand{\CGG}{\Cov^{GG}}
\newcommand{\hCuG}{\widehat{\Cov}^{uG}}
\newcommand{\hCGG}{\widehat{\Cov}^{GG}}
\newcommand{\eps}{\epsilon}
\newcommand{\gP}{\mathfrak{P}}
\newcommand{\mb}{m^{\scriptscriptstyle \rm{BL}}}

\newcommand{\Cblue}{\Cov^{\scriptscriptstyle \rm{BL}}}
\newcommand{\mug}{\mu^G}
\newcommand{\nug}{\nu^G}
\newcommand{\mumf}{\mu^{MF}}

\newcommand{\Td}{T^{D}}
\newcommand{\Ts}{T^{S}}
\newcommand{\TdJ}{T^{D,J}}
\newcommand{\TsJ}{T^{S,J}}
\newcommand{\cB}{\mathcal{B}}
\newcommand{\ccB}{\mathcal{S}}

\newcommand{\hv}{\widehat{v}}
\newcommand{\hy}{\widehat{y}}
\newcommand{\hz}{\widehat{z}}
\newcommand{\hzj}{\widehat{z}^{(j)}}
\newcommand{\hu}{\widehat{u}}

\newcommand{\vd}{{v}^\dag}
\newcommand{\yd}{{y}^\dag}
\newcommand{\wwd}{{w}^\dag}
\newcommand{\ud}{{u}^\dag}
\newcommand{\gammad}{{\gamma}^\dag}

\newcommand{\xid}{{\xi}^\dag}
\newcommand{\etad}{{\eta}^\dag}
\newcommand{\zd}{{z}^\dag}
\newcommand{\Zd}{{Z}^\dag}
\newcommand{\Yd}{{Y}^\dag}
\newcommand{\zdd}{{z}^{\dag,\delta}}
\newcommand{\Wd}{{W}^\dag}
\newcommand{\Bd}{{B}^\dag}
\newcommand{\Wj}{{W}^{(j)}}
\newcommand{\Bj}{{B}^{(j)}}

\newcommand{\hvj}{\widehat{v}^{(j)}}
\newcommand{\hyj}{\widehat{y}^{(j)}}
\newcommand{\hhj}{\widehat{h}^{(j)}}
\newcommand{\vj}{{v}^{(j)}}

\newcommand{\xij}{{\xi}^{(j)}}
\newcommand{\etj}{{\eta}^{(j)}}
\newcommand{\Covpost}{C}

\newcommand{\bbE}{\mathbb{E}}
\newcommand{\E}{\bbE}
\newcommand{\bbR}{\mathbb{R}}
\newcommand{\R}{\bbR}
\newcommand{\bbP}{\mathbb{P}}

\newcommand{\Z}{\mathbb{Z}}
\newcommand{\N}{\mathbb{N}}
\newcommand{\Ng}{\mathsf{N}}
\newcommand{\dt}{\Delta t}
\newcommand{\gG}{\mathfrak{G}}
\newcommand{\dd}{\mathrm{d}}

\usepackage{color,soul}
\usepackage{mathtools}
\definecolor{darkred}{rgb}{.7,0,0}
\definecolor{darkgreen}{rgb}{.3,.7,0}

\usepackage{changes}
\definechangesauthor[name={Reviewer 1}, color = blue]{R1}
\definechangesauthor[name={Reviewer 2}, color = red]{R2}
\definechangesauthor[name={All reviewers}, color = brown]{All}
\definechangesauthor[name={Editor}, color = purple]{Editor}
\definechangesauthor[name={Author}, color = olive]{Author}

\graphicspath{ {./figs/} }

\AtBeginDocument{
  \label{CorrectFirstPageLabel}
  
}

\title[Ensemble Kalman Methods: A Mean Field Perspective]{Ensemble Kalman Methods:\\ A Mean Field Perspective}
\author[E. Calvello, S. Reich, A. M. Stuart]{
Edoardo Calvello \\ California Institute of Technology,\\ Pasadena, CA 91125, USA \\ E-mail: e.calvello@caltech.edu 
\and
Sebastian Reich \\ Institut f\"ur Mathematik, Universit\"at Potsdam,\\ D-14476 Potsdam, Germany \\ E-mail: sebastian.reich@uni-potsdam.de 
\and
Andrew M.~Stuart \\ California Institute of Technology, \\Pasadena, CA 91125, USA \\ E-mail: astuart@caltech.edu}

\begin{document}

\label{firstpage}
\maketitle 

\vspace{0.35in}
\begin{abstract}

Ensemble Kalman methods are widely used for state estimation 
in the geophysical sciences. Their success stems from the fact
that they take an underlying (possibly noisy) dynamical system as a black box to provide a systematic, derivative-free methodology for
incorporating noisy, partial and possibly indirect observations to update estimates of the state; furthermore the ensemble approach allows for sensitivities and uncertainties to be calculated.
The methodology was introduced in 1994 in the context of ocean state estimation. Soon thereafter it was adopted by the numerical weather prediction community and is now a key component of the best weather prediction systems worldwide. Furthermore the methodology is starting to be widely adopted for numerous problems in the geophysical sciences and is being developed as the basis for general purpose derivative-free inversion methods that show great promise.
Despite this empirical success, analysis of the accuracy of
ensemble Kalman methods, in terms of their capabilities as both state
estimators and quantifiers of uncertainty, is lagging. The purpose
of this paper is to provide a unifying mean field based framework 
for the derivation and analysis of ensemble Kalman methods. Both 
state estimation and parameter estimation problems (inverse problems)
are considered, and formulations in both discrete and continuous 
time are employed. For state estimation problems, both the control 
and filtering approaches are considered; analogously for parameter 
estimation problems, the optimization and Bayesian perspectives 
are both studied. The mean field perspective provides an elegant framework, suitable for analysis; furthermore, a variety of methods used in practice can be derived from mean field systems by using interacting particle system approximations.
The approach taken also unifies a wide-ranging literature 
in the field and suggests open problems.

\end{abstract}

\tableofcontents

%\linenumbers

%\as{Need to decide: (i) whether to label all equations, or whether to go through
%and make sure only equations cited are labelled; I prefer the latter to avoid
%proliferation. (ii) A principle for which equations get boxed; currently too many? %(iii) Ensure that we write Subsection, Subsection and not
%subsection, subsection, when we include numbering; but just use
%subsection, subsection when no numbers are used. }

\section{Introduction}
\label{sec:I}

%\as{TODO:
%\begin{itemize}
%\item What to italicize.
%\item Check use $\pi$ for density of posterior.
%\item check $\wwd$ etc have $\dagger$ when needed.
%\item I replaced NU by $\nu$ because NU is indistinguishable from $v$ in AN.
%$\nu$ is hardwired so we can change to another symbol if we wish.
%\end{itemize}
%}

The ensemble Kalman methodology comprises an innovative and flexible
set of tools which can be used for both state estimation in
dynamical systems and parameter estimation for generic inverse
problems. It has primarily been developed by practitioners in
the geophysical sciences, with notable impact on the fields of 
oceanography, oil reservoir simulation and weather forecasting.
Despite its widespread adoption in the geosciences over several decades, firm theoretical foundations are only recently starting to emerge; the methodology is hard to analyze. The purpose of this article is twofold: a) to introduce a mathematical framework for the analysis of
ensemble Kalman methods, describing what is known and highlighting
the many open mathematical challenges in the field; b) to 
provide a literature survey which bridges
the domain-specific development of the methodology with emerging
mathematical analyses. In so doing we will also highlight the
flexibility of the methodology for use in widespread applications, 
beyond its historical development in the geosciences. 

The material is organized around the two separate themes of
state estimation and inverse problems; within each, both discrete time
and continuous time approaches are explained. The novel perspective
which underlies all of this material is the derivation of ensemble
Kalman methods as particle approximations of carefully designed
mean field models. The relationship of these mean field models to
exact transport models, for Gaussian problems, serves to motivate
their form.

In Subsection \ref{ssec:HC} we overview the history of 
ensemble Kalman methods. Subsection \ref{ssec:O} describes
the organization of the paper. In Subsection \ref{ssec:pc}
we make brief remarks about the pseudo-code
that we make available as a supplementary resource for the paper. \
The introduction concludes, in Subsection \ref{ssec:N}, with a summary of the notation
that we adopt throughout.

%%%%%%%%%%%%%%%%%%%%%%%%%%%%%%%%%%%%%%%%%%%%
%
\subsection{Historical Context}
\label{ssec:HC}
%
%%%%%%%%%%%%%%%%%%%%%%%%%%%%%%%%%%%%%%%%%%%%
The Kalman filter (KF) is arguably the first setting in which the
systematic integration of observational data with a dynamical system
was considered, leading to both discrete time
\citep{kalman1960new} and continuous time \citep{kalman1961new}
formulations; see \citet{welch1995introduction} for an overview. The 
Kalman filter applies only in the setting of linear Gaussian dynamics and
observations. In this setting it computes the distribution of the
state of the dynamical system, given observations, exactly; this Bayesian
perspective on the filter was highlighted in \citet{ho1964bayesian} subsequent 
to the original derivation in \citet{kalman1960new} which proceeded by
computing the best linear predictor of the state, given data. The extended 
Kalman filter (historically denoted EKF; however the acronym
ExKF is also used and is useful) was introduced in order to extend Kalman's ideas
to nonlinear problems; see the texts \citet{jazwinski2007stochastic}
and \citet{anderson2012optimal} for overviews. The extended Kalman
approach is based on a linearization approximation; it hence fails to
exactly compute the distribution of the
state of the dynamical system, given observations, in general. Furthermore
it requires propagation of covariance matrices which can be very large
for applications arising in the geosciences \citep{ghil1981applications}.

The ensemble Kalman filter (EnKF) was introduced in the celebrated paper
\citet{evensen1994sequential} which made the consequential observation that,
if an \emph{ensemble} of state estimators is employed then it can also be
used to make an approximation of the covariance. In geosciences applications
this circumvents the computation of large covariances, replacing them
instead with low rank approximations, with rank determined by the number
of ensemble members. The original paper developed the
idea in the context of ocean models, but was rapidly and concurrently
developed in a variety of geoscience application domains
\citep{van1996data,burger1998enkf, houtekamer1998enkf};
the paper \citet{SR-vanLeeuwen2020} provides a historical overview.
These methods are sometimes referred as the \emph{stochastic EnKF}: 
they require simulation of random variables to implement. 
A different class of ensemble methods, known collectively 
as \emph{ensemble square root filters}, was subsequently developed
\citep{anderson2001ensemble,
SR-whitaker2002enkf,
SR-bishop2001,
SR-hunt2007,
SR-TIPPETT03,
SR-sakov2008}; 
these methods are a form of \emph{deterministic EnKF}: 
they do not require simulation of random variables to implement.

Central to our mathematical presentation of Kalman-based methods is the
adoption of mean field and transport perspectives on the subject.
The incorporation of data within filtering constitutes (possibly approximate) application of Bayes' theorem; the papers \citet{daumetal2010}, \citet{reich2011dynamical}, \citet{SR-marzouk2012}, and the survey \citet{SR-CotterReich2013} introduce novel approaches to Bayesian inversion, rooted in transport and mean field models. While \citet{SR-marzouk2012}, \citet{spantini2019coupling} propose a direct numerical approximation of the underlying optimal transport problem, \citet{daumetal2010} and \citet{reich2011dynamical} pursue a homotopy approach. We note that the homotopy approach is closely related to iterative implementations of the EnKF, as first considered by
\citet{li2007iterative}, \citet{gu2007iterative}, and \citet{sakov2012};
these homotopy approaches lead to continuous time formulations of the EnKF in the limit of infinitely many iterations, as first considered by \citet{bergemann2010localization} and \citet{bergemann2010mollified}. Directly starting from the continuous time filtering perspective, mean field models have been introduced independently in \citet{SR-CX10,SR-meyn13}.

The key connection between mean field models and ensemble methods
is that the latter can be derived as particle approximations of
the mean field limit; this viewpoint will play a guiding role
in our presentation of the subject of ensemble methods in this paper.
In this context it is notable that the field of optimization, which is linked to Bayesian sampling through MAP estimation \citep{kaipio2006statistical}, has also seen recent development using mean field models -- see \citet{carrillo2018analytical} for an overview and unifying 
mathematical framework.

The methods covered in this survey  provide only approximate solutions to the underlying filtering, inference and/or optimization problem, in the mean field limit. The approximations
invoked are based on assuming linear Gaussian structure, where the mean field models are exact, but applying
the resulting methodology outside this regime. In the context of the optimization and Bayesian
approaches to inversion, affine invariant algorithms (introduced in \citet{goodman2010ensemble}) play an
important conceptual role in understanding the power of ensemble Kalman methods: affine
invariance can be used to show universal convergence rates  for linear Gaussian problems. 
Empirically the affine invariance confers advantages when 
ensemble Kalman methods are applied beyond the linear Gaussian setting. In practice this benefit must be
weighed against the error resulting from using ensemble Kalman methods outside the linear Gaussian setting.

Alternative methods, such as sequential Monte Carlo, can be designed to be consistent with the underlying nonlinear filtering problem, and do not rely on being exact only for linear Gaussian
problems. The books \citet{doucet2001introduction} and \citet{chopin:20} overview use of sequential Monte Carlo methods for general discrete time filtering and inference problems; the papers \citet{del1997nonlinear} and \citet{del2001stability} prove convergence of sequential Monte Carlo methods, including in some specific cases over long time horizons. 
However, sequential Monte Carlo methods suffer from a curse of dimensionality and are presently not directly applicable to high dimensional problems as arising, for example, from geophysical applications. This issue with the curse of dimensionality provides an important motivation for the ensemble methods covered in this survey. See \citet{snyder2008obstacles}, \citet{bickel2008sharp}, \citet{rebeschini2015can} and \citet{agapiou2017importance} for detailed discussion of these issues. Related issues also arise for Monte Carlo Markov chain (MCMC) when studying Bayesian inverse problems  \citep{kaipio2006statistical}. See \citet{hairer2014spectral} for an analysis of the degeneration of performance of standard MCMC methods in high dimensions, as well as analysis of special MCMC methods tailored to infinite
dimensional problems; the subject is overviewed in \citet{cotter2013mcmc}.

This completes our chronological overview of the historical context for the development of ensemble Kalman methods, and the specific mathematical context which will be our focus. Each of the four subsequent sections concludes with a bibliographic subsection in which a deeper literature review is given.

%%%%%%%%%%%%%%%%%%%%%%%%%%%%%%%%%%%%
%
\subsection{Overview}
\label{ssec:O}
%
%%%%%%%%%%%%%%%%%%%%%%%%%%%%%%%%%%%%%%%%%%%%%%%

Section \ref{sec:SE} is devoted to the problem of
state estimation for discrete time (possibly
stochastic) dynamical systems, given noisy observations.
We formulate the problem from the perspectives of both control
theory and probability, and we provide a unifying approach to
algorithms for these problems; the approach rests on transport of measures and mean field stochastic dynamical systems. Ensemble Kalman methods are 
then derived as particle approximations of the mean field models. 
Section \ref{sec:CT} adopts a perspective that parallels the
previous section, but in the continuous time setting.
Ordinary differential equations (ODEs) and stochastic differential equations (SDEs) are used to describe
the state and its observation process, and mean field SDEs and ODEs, 
and related (stochastic) partial differential equations ((S)PDEs), are
used to provide the underpinnings of algorithms; 
particle approximations of the mean field systems give rise to
interacting systems of SDEs which describe ensemble Kalman
methods. The formulation in continuous time is useful
both because in some applications state estimation problems are most naturally formulated this way, and because they provide insight into discrete time algorithms, giving rise to cleaner
analysis of phenomena present in both discrete and continuous time.

Sections \ref{sec:IPDT} and \ref{sec:CTI} are devoted to the use of ensemble Kalman
methods for inverse problems, including parameter estimation, demonstrating how a useful change
of perspective opens up the use of state estimation in this broader setting.   
Sections \ref{sec:IPDT} and \ref{sec:CTI} consider discrete and continuous time respectively,
and each parallels the ideas developed for state estimation in Sections \ref{sec:SE} and \ref{sec:CT}
respectively.
\begin{figure}[h!]
    \centering
    \includegraphics[width=.8\linewidth]{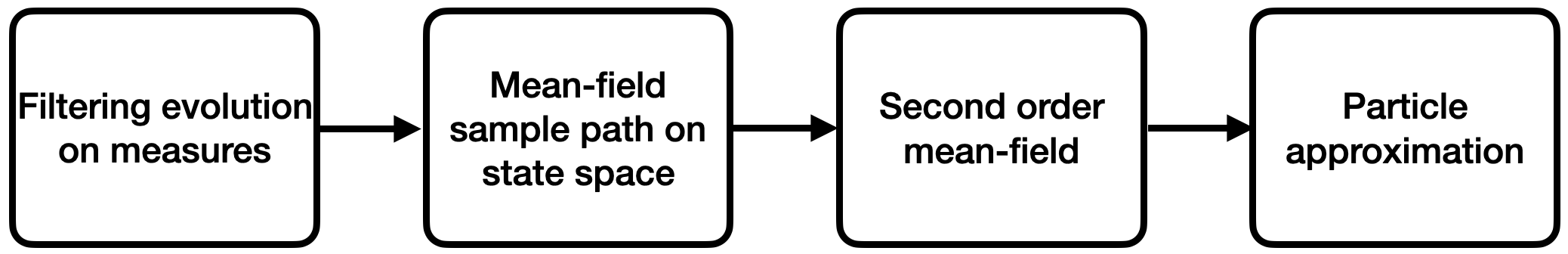}
    \caption{Organizational flow employed in each of Sections \ref{sec:SE} (discrete time) and \ref{sec:CT} (conntinuous time), concerning state estimation. Sections \ref{sec:IPDT} and \ref{sec:CTI} apply this methodology, in discrete and continuous time respectively, to inverse problems, by formulating them as state estimation problems.
    \label{fig:flow}}
\end{figure}
Sections \ref{sec:SE}, \ref{sec:CT}, \ref{sec:IPDT} and  \ref{sec:CTI} are all organized according to the flow of ideas displayed in Figure \ref{fig:flow}. 
Indeed, starting from a probabilistic perspective, that describes the evolution of the filtering distribution, it is possible to formulate mean field maps whose sample paths on state space are equal in law to the filtering evolution. Since it is 
typically not tractable to identify these maps explicitly, it is of interest to determine mean field maps whose sample paths are only \emph{approximately} equal in law to the filtering evolution. This leads to the idea of second order mean field models whose sample paths have law with first and second order moments which match those of a Gaussian approximation
of the Bayesian data incorporation step. Particle approximations of these second order mean field maps are then used to derive implementable numerical algorithms which are the ensemble Kalman methods implemented in practice.

Section \ref{sec:C} concludes, and in particular highlights
open problems in the area, of potential
interest to the mathematical community.
The appendix Section \ref{sec:AA} is devoted to pseudo-code for the 
algorithms introduced in this paper.
The appendix Section \ref{sec:AB} contains background on the underlying 
Lorenz '96 model problems that
we use throughout the paper in a number of illustrative numerical experiments.
The appendices in Sections
\ref{appendix:C} and \ref{sec:AC} provide underpinning material
on mean field maps and on stochastic integration. Section \ref{sec:AGF} contains
some observations linking different flows, in the manifold of Gaussian probability measures, that arise in
the main body of the text.

%%%%%%%%%%%%%%%%%%%%%%%%%%%%%%%%%%%%%%%%%%%%%%%%
%
\subsection{Pseudo-Code}
\label{ssec:pc}
%
%%%%%%%%%%%%%%%%%%%%%%%%%%%%%%%%%%%%%%%%%%%%%%%%
Pseudo-code describing several of the algorithms that we present
and deploy in this paper is given in Appendix \ref{sec:AA} \footnote{The code is available at \href{https://github.com/EdoardoCalvello/EnsembleKalmanMethods/}{https://github.com/EdoardoCalvello/EnsembleKalmanMethods/}.}. The reader is encouraged to consult Algorithms \ref{alg:3DVAR} and \ref{alg:EnKF}, 3DVAR and the ensemble Kalman filter (EnKF) respectively, in the context of the problem of state estimation for discrete time dynamical systems presented in Section \ref{sec:SE}. The scheme 3DVAR is employed in Examples \ref{ex:3dvar}, \ref{ex:s3dvar} and \ref{ex:enkf}. The ensemble Kalman filter is applied in Example \ref{ex:enkf}. Ensemble Kalman methods for inversion, as shown in Algorithms \ref{alg:EKTI}, \ref{alg:EKOI} \ref{alg:EKI_post}, are presented in Section \ref{sec:IPDT} and applied in Examples \ref{ex:EKI_1D} and \ref{ex:EKI}.

%%%%%%%%%%%%%%%%%%%%%%%%%%%%%%%%%%%%%%%%%%%%%%%%%%
%
\subsection{Notation}
\label{ssec:N}
%
%%%%%%%%%%%%%%%%%%%%%%%%%%%%%%%%%%%%%%%%%%%%%%%%%

Throughout we denote the positive integers and non-negative integers
respectively by $\N=\{1,2,\cdots\}$ and $\Z^+=\{0,1,2,\cdots\},$
and the notation $\R=(-\infty,\infty)$ and $\R^+=[0,\infty)$ for the
reals and the non-negative reals. We let
$\langle \cdot, \cdot \rangle, |\cdot|$ denote the
Euclidean inner-product and norm, noting that
$|v|^2=\langle v,v \rangle.$ 
We also use $|\cdot|$ to denote the resulting induced norm on matrices. We use $:$ to denote the Frobenius inner-product between
matrices, and $|\cdot|_{\rm F}$ the induced norm on matrices.
For any function $g:\R^{d_1} \mapsto \R^{d_2}$ we denote by $Dg(v) \in \mathbb{R}^{d_2\times d_1}$ the Jacobian matrix of
first derivatives at $v\in \mathbb{R}^{d_1}$, and by $D^2g(v)[\cdot,\cdot]$ the symmetric bilinear form induced by second derivatives. We also use $\nabla$, $\nabla \cdot$ 
and $\Delta$ to denote the gradient, divergence
and Laplacian operations respectively. 

Throughout the article we will distinguish operators acting on 
infinite-dimensional spaces by using a calligraphic font.  For
example $\mathcal{G}$ will denote the operator acting on probability 
measures effecting a projection onto the nearest Gaussian; in
contrast $G:\R^{d_u}\to\R^{d_w}$ will denote the forward model 
in a generic inverse problem. We will use the $\mathsf{mathsf}$ font to distinguish
matrices and functions between Euclidean spaces that arise in
continuous time from their discrete time counterparts. For example
$\Sigmas$ will denote a covariance arising in continuous time, whilst $\Sigma$
will denote a covariance arising in discrete time; and $\hs$ will denote a
function defining the observation process in continuous time,
whilst $h$ will be the analogous function in discrete time.

We use $\bbE$ and $\bbP$ to denote expectation and probability under
the prevailing probability measure; if we wish to make clear
that measure $\pi$ is the prevailing probability measure then we
write $\bbE^\pi$ and $\bbP^\pi$. For a measure $\pi$ we denote by $T^{\sharp}\pi$ the pushforward measure induced by the map $T$. We use $\law(\rv)$ to denote the
law of random variable $\rv$. We let $\gP=\gP(\R^d)$ denote the set of probability measures on $\R^d$. Under the assumptions made in this article
we are mostly able to work with measures that have density with respect
to Lebesgue measure, and so will blur the distinction between measures and
their densities. However use of Dirac masses will occasionally be useful.

We let $\delta_{u}$ denote the Dirac mass on $\R^d$, centered
at point $u \in \R^d.$
The notation $\Ng(m,C)$ denotes the distribution of a
Gaussian random variable with mean $m$ and covariance $C$. 
We let $\gG=\mathfrak{G}(\R^d)$ denote the set of Gaussian
probability measures on $\R^d$ (including indefinite covariances
and hence all Dirac masses). 

In the following let $A,B,C$ be symmetric matrices.
We write $A \succ B$ when $A-B$ is positive definite and
$A \succeq B$ when $A-B$ is positive semi-definite.
We will also write $A \prec B$ when $B - A \succ 0$
and $A \preceq B$ when $B - A \succeq 0.$
For $C \succ 0$ (and therefore
a covariance matrix) we define 
$\langle \cdot, \cdot \rangle_{C}$ and $ |\cdot|_{C}$, the 
covariance-weighted Euclidean inner-product and norm, by
$\langle u,v \rangle_C= \left\langle u,C^{-1}v \right\rangle$
and $|v|_C^2=\langle v,v \rangle_C.$

%\subsection{Organization}

%\begin{tikzpicture}[node distance=2cm]
%\node (step1) [startstop] {Start}
%\node (step2) [startstop] {Start}
%\node (step3) [startstop] {Start}
%\node (step4) [startstop] {Start}
%\node (dec1) [decision, below of=pro1, yshift=-0.5cm] {Decision 1};
%\node (dec1) [decision, below of=pro1, yshift=-0.5cm] {Decision 1};
%\end{tikzpicture}

%%%%%%%%%%%%%%%%%%%%%%%%%%%%%%%%%%%%%%%%%%%%%%%%%%%%%%%%%%%%%%%%%%%%%%
%
%
%
%
%
%
%
%
%
%
%
%  Section 2: Discrete time
%
%
%
%
%
%
%
%
%
%
%
%%%%%%%%%%%%%%%%%%%%%%%%%%%%%%%%%%%%%%%%%%%%%%%%%%%%%%%%%%%%%%%%%%%%

%%%%%%%%%%%%%%%%%%%%%%%%%%%%%%%%%%%%%%%%
%
\section{State Estimation: Discrete Time}
\label{sec:SE}
%
%%%%%%%%%%%%%%%%%%%%%%%%%%%%%%%%%%%%%%%

In Subsection \ref{ssec:SUSE} we provide the set-up for the
problem of state estimation in discrete time.
Subsection \ref{ssec:CT} introduces
an algorithm for this problem based on a control theoretic
perspective. Subsection \ref{ssec:PP} describes the Bayesian
probabilistic perspective; in this subsection
algorithms are not presented but 
foundations for the creation of algorithms are laid through
the decomposition of the filtering cycle into an iteration
which alternates prediction and the assimilation of data.
In Subsection \ref{ssec:GPFD} we introduce the important idea of
Gaussian projection, and the resulting Gaussian projected filter.
Subsection \ref{ssec:MFM} introduces various mean field dynamical systems 
which approximate the filtering cycle; a unifying transport map perspective is
adopted. This leads,
in Subsection \ref{ssec:EKM}, to
definitions of the ensemble
Kalman filter, the ensemble adjustment filter
and the ensemble transform filter, all derived 
as particle approximations of specific mean field models.
We conclude with bibliographic notes in Subsection \ref{ssec:BSE}
in which we review relevant literature, and include
discussion of a variety of other
algorithms for state estimation, 
and relate them to the perspective we adopt here.

We emphasize that all the mean field methods we derive in this section for the solution of filtering problems, and their continuous time counterparts in Section \ref{sec:CT}, are exact only in the linear Gaussian setting.
They may be implemented beyond this setting, however, and are empirically found to work well for many problems. Theory to justify this observation is very much needed. Our exposition provides a framework for such a theory.

%%%%%%%%%%%%%%%%%%%%%%%%%%%%%%%%%%%%%%%%%%%%%%
%
\subsection{Set-Up}
\label{ssec:SUSE}
%
%%%%%%%%%%%%%%%%%%%%%%%%%%%%%%%%%%%%%%%%%%%%%%%%%%

The objective of {\em sequential data assimilation} 
is to iteratively update the state of a (possibly stochastic) dynamical system based on observations and knowledge of the dynamical and observational
processes; we refer to this as {\em state estimation}. The typical setting
is one in which the initial condition is uncertain, but this uncertainty is
compensated for by (typically noisy) observations of a (possibly
nonlinear and indirect) function of the state. These observations often live in a 
space of lower dimension than the dimension of the state space meaning
that the goal of state estimation goes beyond denoising, and into the
realm of control theoretic considerations such as observability.

A useful starting point is to consider a stochastic
dynamical system in which the evolution of the state, 
and the relationship between the observations 
(which we also refer to as data) and the state,
are defined, respectively, by the equations
\begin{subequations}
\label{eq:sd}
%\begin{align}
\begin{empheq}[box=\widefbox]{align}
v_{n+1} &= \Psi(v_n) + \xi_n, \\
y_{n+1} &= h(v_{n+1}) + \eta_{n+1},
\end{empheq}
%\end{align}
\end{subequations}
taken to hold for all $n \in \Z^+.$
We assume that, for each fixed $n \in \Z^+$, the \emph{state} $v_n \in \R^{d_v}$ and the
\emph{observations} $y_n \in \R^{d_y}.$
The maps $\Psi(\cdot)$ and $h(\cdot)$ describe the systematic,
deterministic components of the dynamics and observation processes, and are assumed to be known measurable functions (with respect to the Borel algebra), bounded on compact sets. The initial condition for $v_0$ is assumed random and the systematic components of the model
are subjected to mean zero noise, $\xi_n$ and $\eta_{n+1}$. 
To be concrete we assume  $v_0,  \{\xi_n\}_{n \in \Z^+}$
and $\{\eta_n\}_{n \in \N}$ are mutually independent
Gaussians defined by
\begin{equation}
\label{eq:ga}
v_0 \sim \Ng(m_0, C_0), \quad 
\xi_n \sim \Ng(0, \Sigma) \,\, \text{i.i.d.}, \quad
\eta_n \sim \Ng(0, \Gamma) \,\, \text{i.i.d.}\,\,.
\end{equation}

In practice we will have available to us the observation
coordinates of a specific true realization
of the random dynamical system \eqref{eq:sd}, 
from which we wish to recover
the specific true 
realization of the state that gave rise to these observations.
We denote the true realizations of the state of the system by sequence $\{\vd_n\}_{n \in \Z^+}$,
and the observed data by $\{\yd_n\}_{n \in \N}.$ These
are generated by $\vd_0, \{\xid_n\}_{n \in \Z^+}$ and
$\{\etad_n\}_{n \in \N}$, specific realizations of the
initial condition and state and observational noise
from the distribution defined by \eqref{eq:ga}.

We let $\Yd_n=\{\yd_\ell\}_{1 \le \ell \le n}.$ 
Our objective is to estimate 
the state $\vd_n$ at time $n$ from data $\Yd_n$. 
Specifically, it is natural to think about this
design objective in two different ways:

\begin{itemize}
\item Objective 1: design an algorithm producing output
$v_n$ from $\Yd_n$ so that
$\{v_n\}_{n \in \Z^+}$ estimates
$\{\vd_n\}_{n \in \Z^+}$, the true state generated 
by (\ref{eq:sd}a).  
\item Objective 2: design an algorithm which estimates
the distribution of random variable $v_n|\Yd_n$.
\end{itemize}

In both cases we are interested in \emph{Markovian}
formulations which update the estimate $v_n$, or the
distribution $v_n|\Yd_n$, sequentially as the data is
acquired. In the next two subsections we describe control theoretic 
and probabilistic approaches to this problem which, respectively,
provide the basis for algorithms addressing Objectives 1 and 2.
We note that fulfilling Objective 2 immediately implies resolution
of Objective 1, for example by taking the mean of $v_n|\Yd_n$ as state estimator; 
but the reverse is not typically true. However, Objective 1 is 
easier to address and is especially relevant when noise levels are
small; furthermore, its failure modes serve to motivate approaches which are used to address Objective 2.

%%%%%%%%%%%%%%%%%%%%%%%%%%%%%%%%%%%%%%%%%%%%
%
\subsection{Control Theory Perspective}
\label{ssec:CT}
%
%%%%%%%%%%%%%%%%%%%%%%%%%%%%%%%%%%%%%%%%

A very natural idea from control theory underlies ensemble Kalman filtering
and is encapsulated in the following algorithmic approach. This 
way of attacking state estimation is most appropriate when $|C_0|,|\Gamma|$
and $|\Sigma|$ are small so that the state and observations
are close to deterministic. To understand this setting we will first study algorithms
in which the covariances $\Gamma$ and $\Sigma$ are set to zero, and then return to the inclusion of noise later.

The algorithmic idea works as follows:
from current state estimate $v_n$, given $\Yd_n$, {\em predict} the
outcome of the model and data, denoted by $(\hv_{n+1},\hh_{n+1})$, 
using the update equations \eqref{eq:sd}, but ignoring the noise; 
then {\em correct} the state estimate by nudging the prediction using
the mismatch between observed and predicted data $(\yd_{n+1},\hh_{n+1})$. This
results in a deterministic map $v_n \mapsto v_{n+1}$,
assumed to hold for all $n \in \Z^+$, of the following form:
\begin{subequations}
\label{eq:sd2}
\begin{empheq}[box=\widefbox]{align}
\hv_{n+1} &= \Psi(v_n),  \\
\hh_{n+1} &= h(\hv_{n+1}), \\
v_{n+1} &= \hv_{n+1}+K(\yd_{n+1}-\hh_{n+1}). \label{eq:sd2c}
\end{empheq}
\end{subequations}
Equations (\ref{eq:sd2}a, \ref{eq:sd2}b) create predicted state and data from current estimate $v_n$ of the state. The
difference between the predicted data $\hh_{n+1}$ and the true data $\yd_{n+1}$, the latter found from a fixed realization of \eqref{eq:sd}, is then used to correct the predicted state resulting in (\ref{eq:sd2}c). We can write the algorithm compactly in the form
\begin{empheq}[box=\widefbox]{equation}
\label{eq:sd2_add}
v_{n+1} = \Psi(v_n)+K\Bigl(\yd_{n+1}-h\bigr(\Psi(v_n)\bigr)\Bigr). 
\end{empheq}
Choice of the {\em gain matrix} $K$ 
completes definition of an algorithm which we refer to
as 3DVAR.

\begin{remark}
\label{rem:3dvar}
The nomenclature ``3DVAR'' was coined in the geophysics community, and stands for three dimensional variational data assimilation. This
is natural for algorithms which incorporate spatially distributed data
sequentially in time, in the context where
the state variable $v$ varies across three spatial coordinates. In our
setting the state variable $v$ is not required to have any spatial structure, but the control formulation (\ref{eq:sd2_add}) reproduces the 3DVAR algorithm from the geophysics community when it does. Hence we still refer to it as 3DVAR. We also note that the method is perhaps more properly termed Cycled 3DVAR -- ``3DVAR'' refers to the assimilation of data at each observation time, and ``Cycled'' to doing this sequentially in time as successive observations are acquired. 
Pseudo-code for 3DVAR may be found as 
Algorithm \ref{alg:3DVAR} in Appendix \ref{sec:AA}. 
The variant on 3DVAR which uses data distributed over several times
steps is known as 4DVAR; see bibliography Subsection \ref{ssec:BSE}.

We note that the difference between the observed value $\yd_{n+1}$ and its estimator $\hh_{n+1}=h\bigr(\Psi(v_n)\bigr)$ is often referred to as the {\em innovation}, and for this we introduce the notation
\begin{equation}
\label{eq:innovation_c}
\mathfrak{I}_n = \yd_{n+1}-\hh_{n+1}.
\end{equation}
$\blacksquare$
\end{remark}

To illustrate 3DVAR we consider the linear setting:

\begin{example}
\label{ex:control}
Assume that, for matrices $M,H$ of appropriate dimensions,
\begin{equation}
\label{eq:linearsd}
\Psi(\cdot):=M\cdot, \quad h(\cdot)=H\cdot
\end{equation}
and consider the setting in which there is no noise present.
The 3DVAR algorithm \eqref{eq:sd2_add} is appropriate in
this setting and reduces to 
\begin{align}
\label{eq:one}
v_{n+1} &= Mv_n+K(\yd_{n+1}-HMv_{n}). 
\end{align}
Since there is no noise present it follows that
$\yd_{n+1}=H\vd_{n+1}=HM\vd_n$ and hence that
\begin{align}
\label{eq:two}
\vd_{n+1} &= M\vd_n+K(\yd_{n+1}-HM\vd_{n}). 
\end{align}
Subtracting \eqref{eq:two} from \eqref{eq:one} and defining
$e_n=v_n-\vd_n$ we find that
\begin{align*}
e_{n+1} &= (I-KH)Me_n. 
\end{align*}
Since the goal of data assimilation is to recover the true state from partial
observations we wish to drive $e_n$ to zero as $n \to \infty.$ Thus a key
question, for given forward dynamical model $M$ and observation operator $H$,
is whether it is possible to design $K$ to ensure that the spectrum of $(I-KH)M$
is inside the unit circle. Such questions are at the heart of the theory of linear
control. Thus the subject of control theory is fundamentally aligned with
Objective 1.
$\blacksquare$
\end{example}

\vspace{0.1in}

In the following example we illustrate similar ideas to those from the
preceding, linear, example but within the nonlinear context. In subsequent 
subsections we will show how the ideas can be generalized to an adaptive
gain matrix $K_n$, leading us to the ensemble Kalman methodology and
to addressing  both Objectives 1 and 2.

\begin{example}
\label{ex:3dvar}
To illustrate the 3DVAR algorithm \eqref{eq:sd2_add} 
for state estimation we consider the Lorenz '96 (singlescale) 
model from Appendix \ref{sec:AB}.
The unknown
$v \in C(\R^+,\R^L)$ satisfies the equations 
\begin{subequations}
  \label{eq:l96}
\begin{align}
\dot{v}_\ell &=  -v_{\ell-1}( v_{\ell-2} - v_{\ell+1}) - v_{\ell} + F + h_v m\bigl(v_\ell\bigr), \quad
\ell=1 \dots L\,,\\
v_{\ell + L} &= v_\ell, \quad \ell=1 \dots L\,. 
\end{align}
\end{subequations}
Here we set $L = 9, h_v=-0.8$ and $F=10.$ 
Function $m$ is  shown in Figure \ref{fig:multiscale_m}. 
\footnote{We note that setting $h_v = 0$ in (\ref{eq:l96}) leads to the standard singlescale Lorenz'96 model. We have $h_v \ne 0$ leading to a nonstandard
version of the model. However the specific choice of function $m(\cdot)$ 
does not make any material difference to what is presented in this example; 
any function $m(\cdot)$ for which the equation is well-posed will lead to 
similar conclusions.  However the specific choice of $m(\cdot)$ shown in
Figure \ref{fig:multiscale_m} is relevant within Example \ref{ex:3dvar_ms}.} 

\begin{figure}[h!]
\centering
\includegraphics[width=\linewidth]{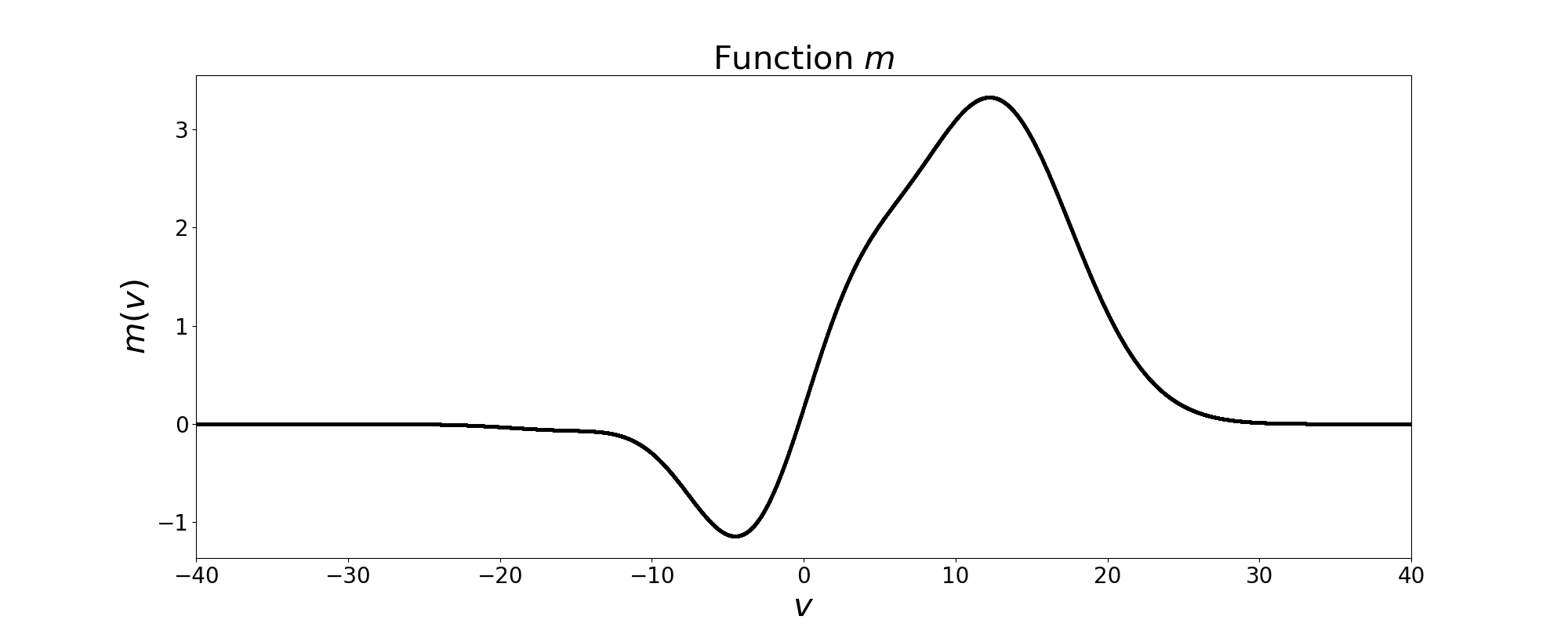}
\caption{Graph of function $m$ appearing in Lorenz '96 model \eqref{eq:l96}.}
\label{fig:multiscale_m}
\end{figure}

We let $\Psi$ denote the solution operator for \eqref{eq:l96} 
over the observation time interval $\tau$, suppresing the
explicit dependence on $\tau$ for notational convenience.
We emphasize that,
at the parameter values we have chosen, the solution to \eqref{eq:l96} is chaotic and exhibits sensitivity to perturbations.
Prediction is thus challenging. But we will show that use of data
enables accurate prediction.

We consider observations $\{\yd_n\}_{n \in \Z^+}$ arising from the model 
\begin{align*}
\label{eq:singlescale_experiment_dynamics}
\vd_{n+1} &= \Psi(\vd_n) + \xid_n, \\
\yd_{n+1} &= h(\vd_{n+1}) + \etad_{n+1},
\end{align*}
where $\{\xid_n\}_{n \in \Z^+}$, $\{\etad_n\}_{n \in \N}$ 
are mutually independent Gaussian sequences defined by
\begin{equation*}
\label{eq:ga_1}
\xid_n \sim \Ng(0, \sigma^2 I) \,\, \text{i.i.d.}, \quad
\etad_n \sim \Ng(0, \gamma^2 I) \,\, \text{i.i.d.}\,\,.
\end{equation*}
Because of the chaotic nature of the dynamical system defined
by iteration of $\Psi$, a key question concerning the problem of determining
the state $\vd_n$ from $\Yd_n$ is whether the observations compensate for
the sensitive dependence of the state evolution, enabling accurate recovery
of the state; and whether there is then a choice of $K$ in 3DVAR which enables
this data to be used to accurately recover the state.

We assume that the observation function is linear: $h(v)=Hv$ for 
matrix $H:\R^9 \to \R^6$ defined by 
\begin{equation}
\label{eq:L96H}
Hv=(v_1,v_2,v_4,v_5,v_7,v_8)^\top.
\end{equation}
The 3DVAR algorithm \eqref{eq:sd2} reduces, 
in the setting of this example, 
to the mapping
\begin{equation}
\label{eq:3DVARL96}
v_{n+1}=(I-KH)\Psi(v_n)+K\yd_{n+1}.
\end{equation}
We define the filter by choosing gain $K:\R^6 \to \R^9$ to be
\begin{equation}
\label{eq:L96K}
Kw=(w_1,w_2,0,w_3,w_4,0,w_5,w_6,0)^\top.
\end{equation}

To interpret the algorithm, and motivate the choice of $K$, given $H$,
notice that
\begin{subequations}
\label{eq:L96KH}
\begin{align}
KHv&=(v_1,v_2,0,v_4,v_5,0,v_7,v_8,0)^\top,\\
(I-KH)v&=(0,0,v_3,0,0,v_6,0,0,v_9)^\top,\\
HK&=I.
\end{align}
\end{subequations}
Applying the observation map $H$ to the recursion \eqref{eq:3DVARL96} 
and using (\ref{eq:L96KH}c) we find that
\begin{equation*}
    \label{eq:test}
Hv_{n+1}= \yd_{n+1}= H\bigl(\Psi(\vd_n)+ \xid_n\bigr)+\etad_{n+1}.
\end{equation*}
Assuming that $\sigma^2$ and $\gamma^2$ are small and neglecting
the noise contributions shows that
\begin{equation*}
    \label{eq:test1}
\yd_{n+1} \approx H\Psi(\vd_n) \approx H \vd_{n+1}.
\end{equation*}
Thus, ignoring small noise perturbations, $\yd_{n+1} \approx H \vd_{n+1}.$  
Then, using (\ref{eq:L96KH}a) and (\ref{eq:L96KH}b) we see that
the algorithm \eqref{eq:3DVARL96} has the very natural (approximate) 
interpretation of iterating using the model $\Psi$ to update
the unobserved components and using the observed true state to update the observed components. This explains why the specific choice of $K$ is reasonable.

Because of this interpretation it is natural to study the 3DVAR algorithm, 
for this example, by displaying the output of 3DVAR on one of the unobserved components, and comparing with the truth; the key question is whether the
observed components induce synchronization of 3DVAR with the truth in the
unobserved components. Thus, in the following numerical experiments,
we display component $v_3.$

\begin{figure}[h!]
\centering
\begin{subfigure}{\textwidth}
  \centering
  \includegraphics[width=1\linewidth]{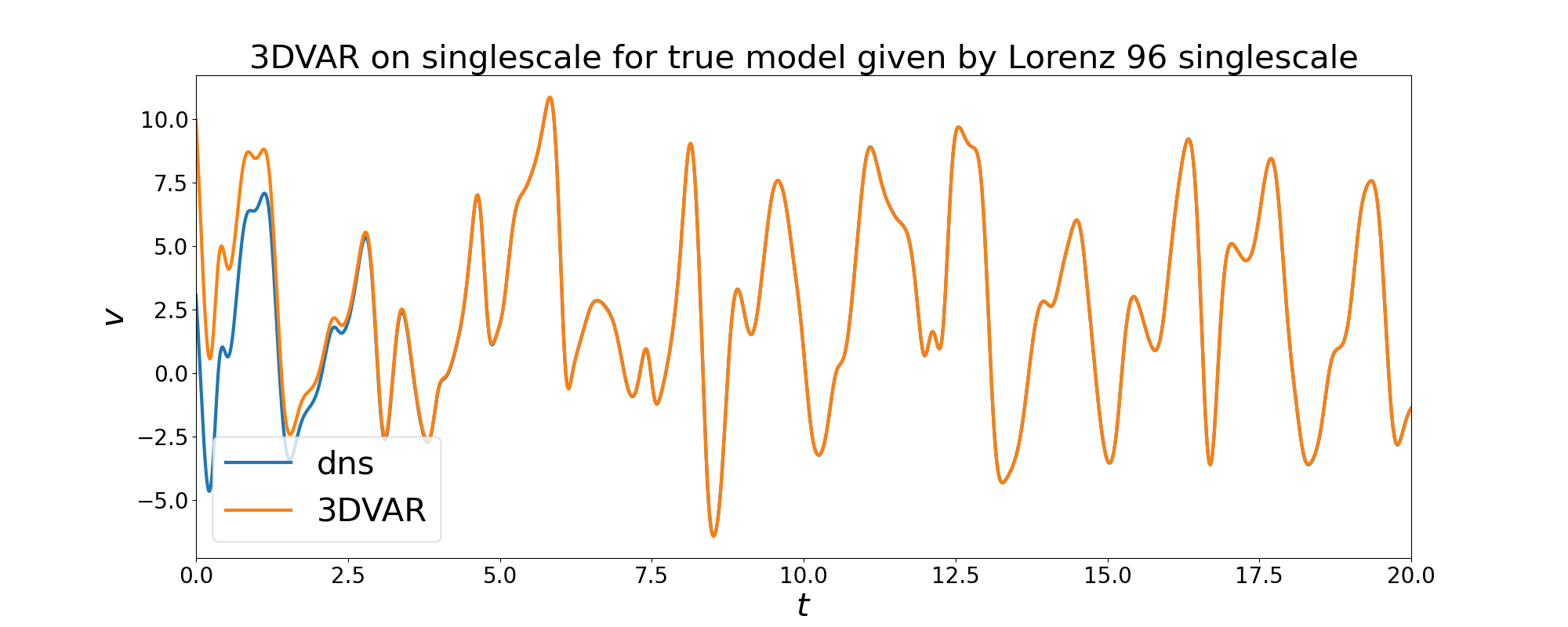}
  \caption{}
  \label{fig:single_dt_s}
\end{subfigure}
\begin{subfigure}{\textwidth}
  \centering
  \includegraphics[width=1\linewidth]{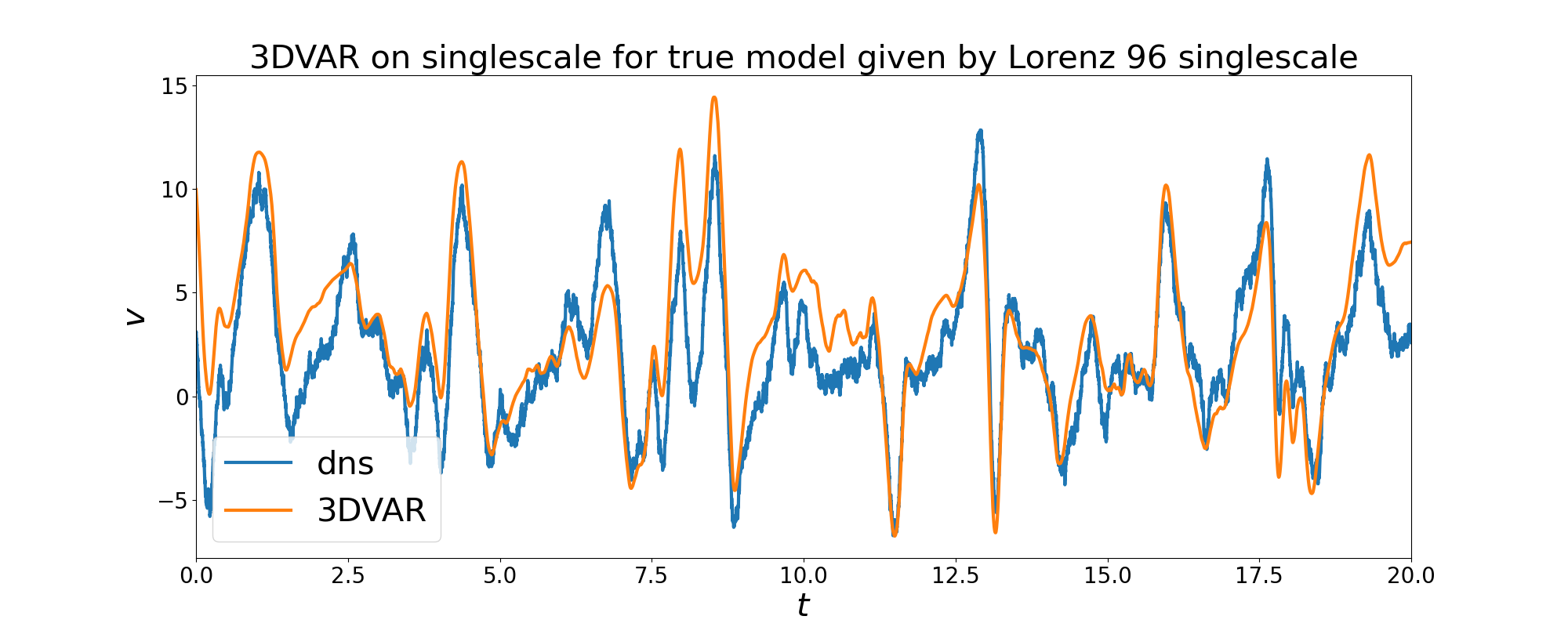}
  \caption{}
  \label{fig:single_dt_b}
\end{subfigure}
\caption{In both (a) and (b) the estimates of $v_3$ in time produced by 3DVAR using observation time interval $\tau =10^{-3}$ are displayed and compared with dns. 
In (a) $\sigma $ and $ \gamma $ are set to $10^{-3}$, while in (b) to $10^{-1}$. The acronym ``dns'' refers to 
direct numerical simulation of a true trajectory of the 
chaotic dynamical system. 
In both cases, it is noteworthy that 3DVAR is able to synchronize with the dns even though it is initialized far from the true initial condition. This is an example of data assimilation overcoming sensitive dependence in a chaotic system.}
\label{fig:single_dt}
\end{figure}

\begin{figure}[h!]
\centering
\begin{subfigure}{\textwidth}
  \centering
  \includegraphics[width=1\linewidth]{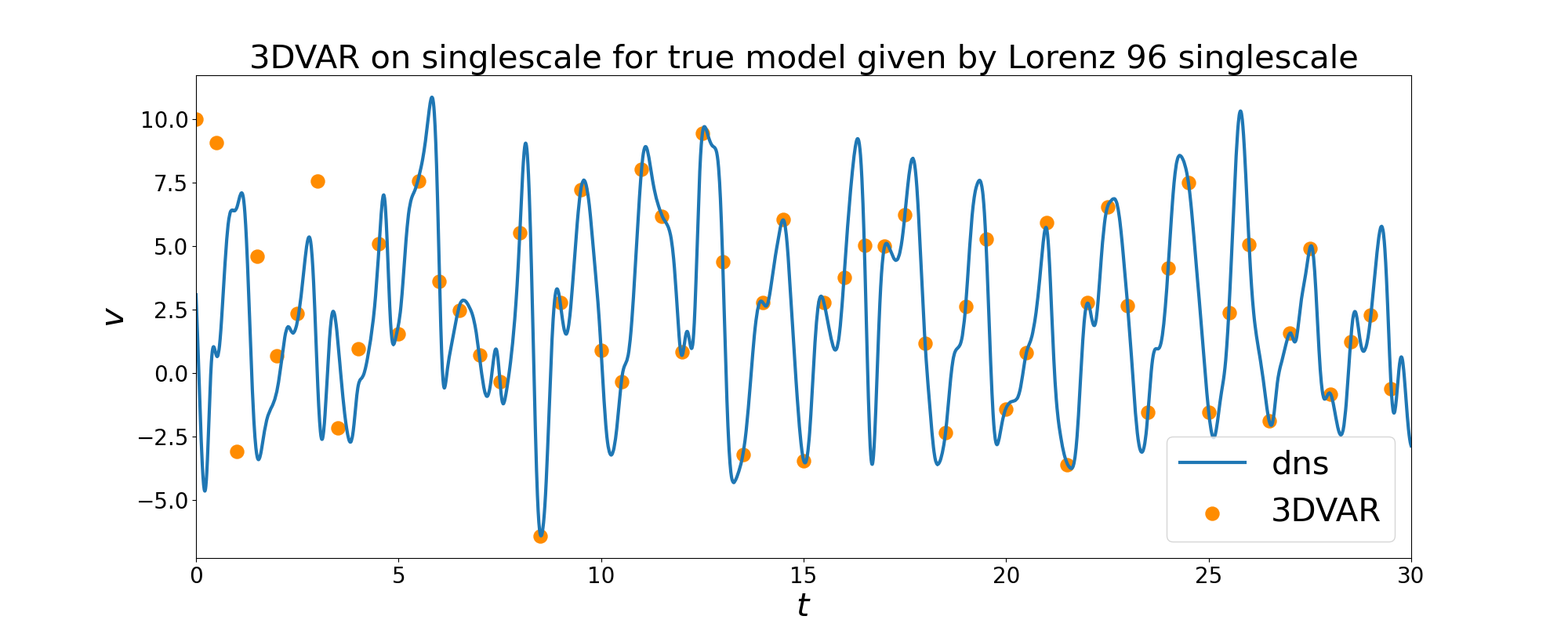}
  \caption{}
  \label{fig:single_dt_smallnoise}
\end{subfigure}
\begin{subfigure}{\textwidth}
  \centering
  \includegraphics[width=1\linewidth]{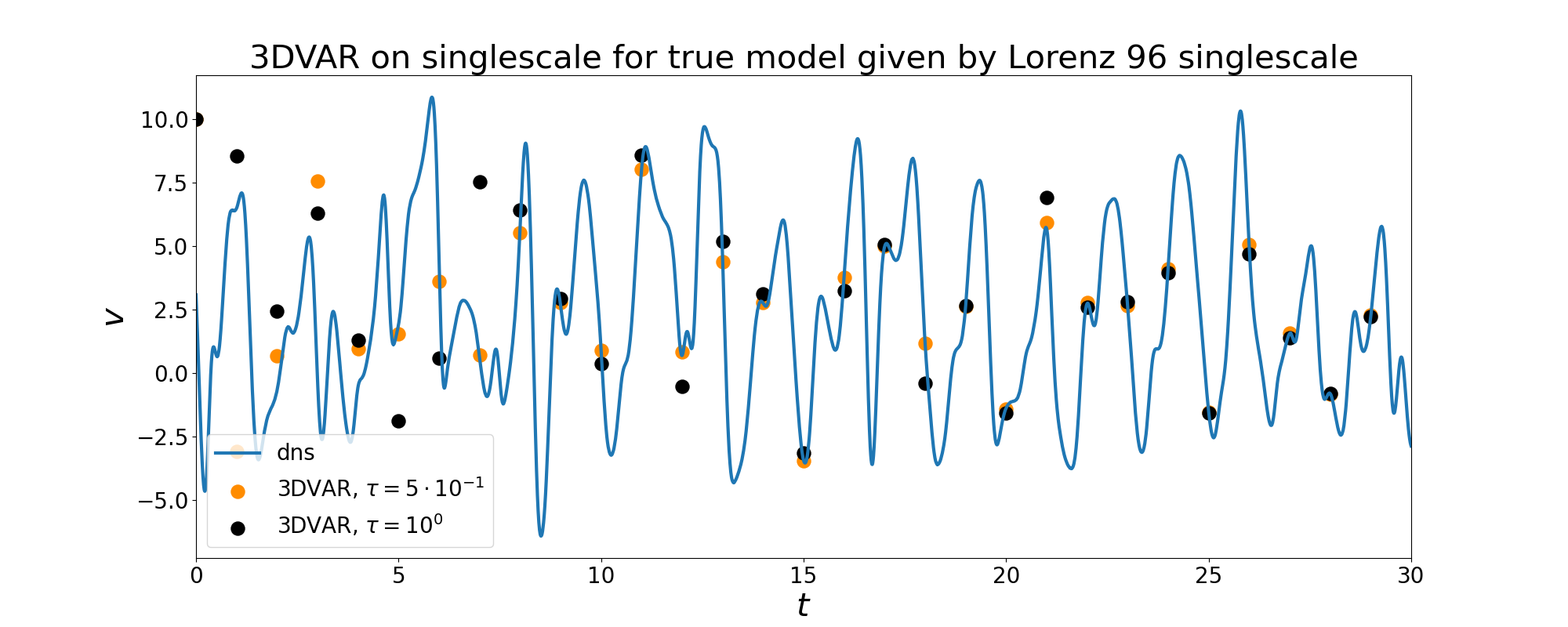}
  \caption{}
  \label{fig:single_dt_bignoise}\end{subfigure}
\caption{In both (a) and (b) the noise standard deviations $\sigma $ and $ \gamma $ are set to $10^{-3}$. In (a) we display the estimates of $v_3$ in time produced by 3DVAR using observation time interval $\tau =5 \cdot 10^{-1}$, compared with dns; (b) displays the estimates obtained at unit time using assimilation at $\tau =5 \cdot 10^{-1}$ and $\tau =10^{0}$. Again the acronym ``dns'' refers to 
direct numerical simulation of the true chaotic dynamics.
 3DVAR successfully synchronizes with the dns 
at the smaller value of $\tau$ but fails to do as well when the observation time interval $\tau$ is larger.}
\label{fig:single_tau}
\end{figure}

Figure \ref{fig:single_dt} illustrates the foregoing intuition about the
behavior of 3DVAR in experiments conducted with the choice $\tau =10^{-3}$. For small noise with standard deviations of size $10^{-3}$ we observe in
Figure \ref{fig:single_dt_s} the phenomenon of
near perfect synchronization of the 3DVAR output with the truth. Although not
shown here, this synchronization occurs in all components of the solution, observed
and unobserved, and thus the entire state of 3DVAR synchronizes with the truth, up to a small error on the scale of the noise. The algorithm thus produces an accurate estimate of the true state. In Figure \ref{fig:single_dt_b}
larger state and observational noise, of standard deviation  $10^{-1}$, is considered.
In this scenario 3DVAR still captures the correct trend of the true dynamics, 
but there are clear overshoots and undershoots in the estimates; this occurs because the
noise is larger than in Figure \ref{fig:single_dt_s} and because noise is not accounted for in the 3DVAR algorithm.

It is intuitive that the synchronization phenomena studied above will depend
not only on the size of the noise, but also on the observation time intervals $\tau.$ Figure \ref{fig:single_tau} illustrates the effect of varying $\tau$, in the setting where the noise standard deviation is $10^{-3}.$ 
The simulations show that, for the larger value of $\tau$, 3DVAR does not estimate the true state as accurately as for the smaller value.
$\blacksquare$ 
\end{example}

\vspace{0.1in}

\begin{remark}
\label{rem:what}
A common source of error in application of data assimilation
algorithms in practice arises from the fact that the data is not produced by the mathematical model used for assimilation. This is known as 
\emph{model misspecification}. For the majority of this paper we will
make the \emph{perfect model} assumption, avoiding the model misspecification
issue. However we do illustrate model misspecification 
in Example \ref{ex:3dvar_ms}. In that example we assimilate data produced
from the Lorenz '96 multiscale model, but we use the Lorenz '96 singlescale 
model \eqref{eq:l96} as the basis for 3DVAR; thus the model used for
assimilation differs from (but is close to) the model generating the data. 
The relationship between the multiscale and singlescale models is detailed in 
Appendix \ref{sec:AB}, where Example \ref{ex:3dvar_ms} may also be found.
$\blacksquare$
\end{remark}

\vspace{0.1in}

The control theoretic approach of 3DVAR addresses
Objective 1. However, when $|C_0|,|\Gamma|$
and $|\Sigma|$ are no longer necessarily small, so that the state and observations are subject to noise, it is natural to try and
generalize the approach to address Objective 2, and include non-zero
covariances; Figure \ref{fig:single_dt_b} from Example \ref{ex:3dvar} 
shows that this may indeed be needed
when the noise is larger. A natural
stochastic generalization of the observer approach is as follows:
from current state estimate $v_n$, given $\Yd_n$, predict the
outcome of the model and data from the update equations \eqref{eq:sd},
which we denote by $(\hv_{n+1},\hy_{n+1})$; 
then correct the state estimate by nudging the mean of the
prediction using
the mismatch between observed and predicted data $(\yd_{n+1},\hy_{n+1})$. Given $v_n$ computed from $\Yd_n$ this
results in state estimate $v_{n+1}$ from $\Yd_{n+1}$ defined 
through the following stochastic dynamical system, assumed to
hold for all $n \in \Z^+:$
\begin{subequations}
\label{eq:sd2n}
\begin{empheq}[box=\widefbox]{align}
\hv_{n+1} &= \Psi(v_n) + \xi_n, \\
\hy_{n+1} &= h(\hv_{n+1}) + \eta_{n+1}, \\
v_{n+1} &= \hv_{n+1}+K_n (\yd_{n+1}-\hy_{n+1}); \label{eq:sd2cn}
\end{empheq}
\end{subequations}
here $v_0, \xi_n$ and $\eta_{n+1}$ are random variables given by the known distributions in \eqref{eq:ga}. Note that the
innovation $\mathfrak{I}_n$ is now modified 
from \eqref{eq:innovation_c} to read 
\begin{equation}
\label{eq:innovation_p}
\mathfrak{I}_n = \yd_{n+1}-\hy_{n+1}.
\end{equation}
The data $\{\yd_{n+1}\}_{n \in \Z^+}$ is a fixed realization
of \eqref{eq:sd}. Given this data,
equations \eqref{eq:sd2n} then define a random map $v_n \mapsto v_{n+1}$
which uses knowledge of the model and the observed data to
update our state estimate. The predicted state and data $(\hv_{n+1},\hy_{n+1})$
are also referred to as the {\em simulated state and data}.
Choice of the {\em gain matrix} $K_n$ 
will complete definition of an algorithm. The key question
we proceed to study in subsequent sections 
concerns the choice of $\{K_n\}_{n \in \Z^+}$ when addressing Objective 2. Before
doing so we build on Example \ref{ex:3dvar}, studying the effect of 
noise on the 3DVAR algorithm, which is aimed at addressing Objective 1, and for which $K_n=K$ is constant. The resulting Example
\ref{ex:s3dvar} demonstrates the need for an adaptive choice of $K_n$ 
when noise is significant; it serves to motivate algorithms which address Objective 2.

\begin{example}
 \label{ex:s3dvar}
We return to the setting of Example \ref{ex:3dvar} and consider the effect of including noise in 3DVAR as in \eqref{eq:sd2n}. 
%\eqref{eq:innovation_p}.
We again study the Lorenz '96 singlescale model \eqref{eq:l96} for unknown
$v \in C(\R^+,\R^L)$, with $L = 9, h_v=-0.8$ and $F=10.$ 
We let $\Psi$ denote the solution operator for \eqref{eq:l96} over the observation time interval $\tau$ and consider observations $\{\yd_n\}_{n \in \Z^+}$ defined by
\begin{align*}
\label{eq:singlescale_experiment_dynamics_noisy1}
\vd_{n+1} &= \Psi(\vd_n) + \xid_n, \\
\yd_{n+1} &= h(\vd_{n+1}) + \etad_{n+1},
\end{align*}
where $\{\xid_n\}_{n \in \Z^+}$, $\{\etad_n\}_{n \in \N}$ are mutually independent Gaussian sequences
\begin{equation*}
\label{eq:ga_33}
\xid_n \sim \Ng(0, \sigma^2 I) \,\, \text{i.i.d.}, \quad
\etad_n \sim \Ng(0, \gamma^2 I) \,\, \text{i.i.d.}\,\,,
\end{equation*}
with $\sigma = 0.1$ and $\gamma = 0.1$. We again assume that the observation function is linear: $h(v)=Hv$ for 
matrix $H:\R^9 \to \R^6$ defined by \eqref{eq:L96H}. As in Example \ref{ex:3dvar} we choose fixed gain $K_n \equiv K$ with $K:\R^6 \to \R^9$  defined by \eqref{eq:L96K}. 

Figure \ref{fig:noisy_3DVAR} illustrates that this
version of noisy 3DVAR produces trajectories that resemble the true signal better
than noise-free 3DVAR \eqref{eq:3DVARL96}. 
However the noisy 3DVAR does not perform better than the noise-free 
3DVAR, in this setting where true state and true 
observation noise levels are high, in a quantitative sense. To demonstrate this, 
we compute the mean squared error between the estimates arising from 
the 3DVAR algorithm \eqref{eq:3DVARL96} and the true states, and the 
mean squared error between the estimates obtained using noisy 3DVAR algorithm
\eqref{eq:sd2n} and the true states. 
Given either method the error $e$ is computed using the following formula, 
in which, recall, $\{\vd_n\}$ is the truth and
$\{v_n\}$ is the output of the 3DVAR or noisy 3DVAR algorithm:
\begin{equation}
\label{eq::MSE_error}
    e = \frac{1}{N\cdot d_v}\sum_{n=1}^N \big|\vd_{n^*+n}-v_{n^*+n}\big|^2.
\end{equation}
Time $t^*=n^* \tau$ is chosen to remove
error from the incorrect initialization, and focus on quantifying
error in statistical steady state, after synchronization.
Here $N$ is such that $(n^*+N)\tau=T$ 
and $T-t^*$ is chosen large enough to allow time-averaging over a long enough window 
to capture the statistical steady state. In this case, we compute 
this average for the estimates obtained after time $t^*=3$, up to $T=20$, 
and find errors $e_{\text{3DVAR}} = 1.65\cdot 10^0$ and 
$e_{\text{noisy3DVAR}} = 3.30\cdot 10^{0}$. 
Clearly use of noisy 3DVAR does not improve the error, 
in comparison with noise-free 3DVAR. 
$\blacksquare$ 
\end{example}

\begin{figure}[h!]
\centering
\includegraphics[width=\linewidth]{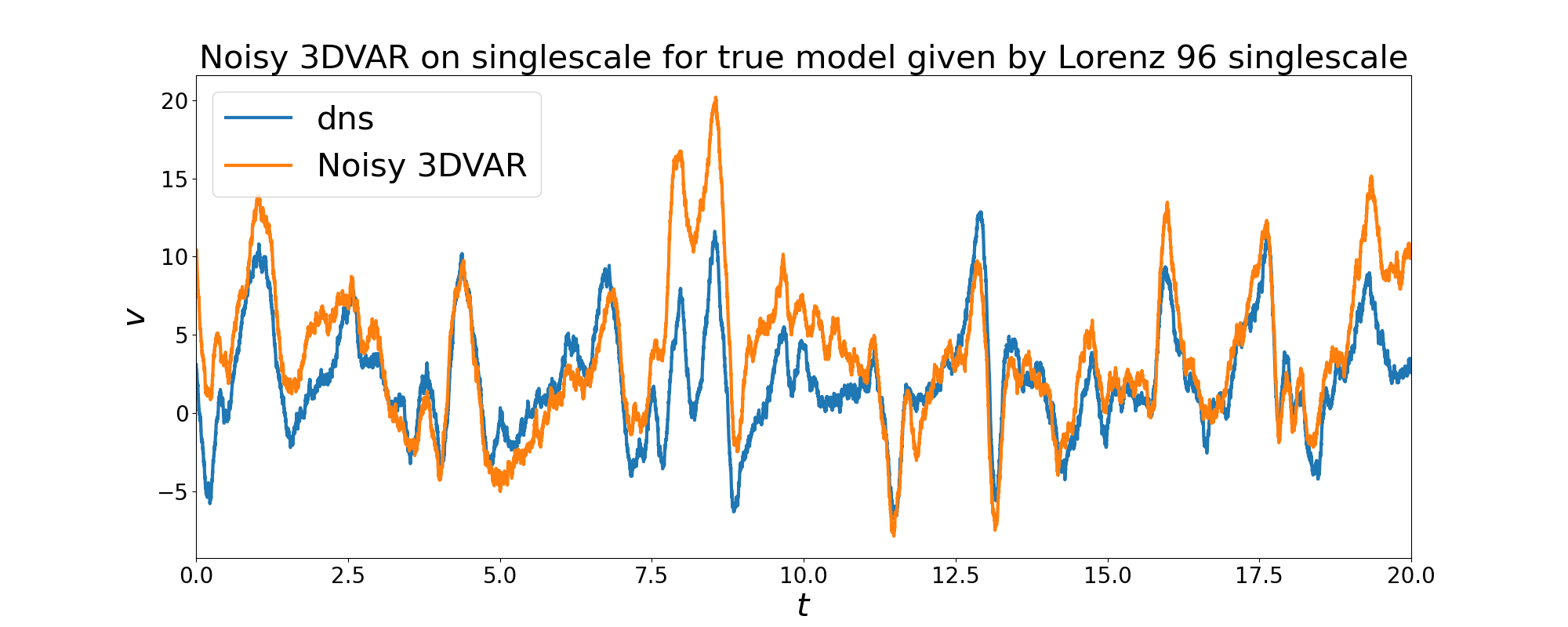}
\caption{In this experiment we set the noise levels $\sigma = 10^{-1}, \gamma = 10^{-1}$. Again the acronym ``dns'' refers to direct numerical simulation. We display the estimates of $v_3$ in time produced by noisy 3DVAR against the true dynamics using observation time interval $\tau = 10^{-3}$. This should be compared with Figure \ref{fig:single_dt}b in which the noise-free 3DVAR is deployed to solve the same problem. Notice that adding noise to 3DVAR has not improved the recovery of the true trajectory. However qualitatively the output of 3DVAR now resembles the true signal more closely.}
\label{fig:noisy_3DVAR}
\end{figure}

\vspace{0.1in}

Example \ref{ex:s3dvar}  highlights the need to quantify uncertainty and pass to a probabilistic interpretation (Objective 2) of the filtering problem;
and, in particular, to make an informed choice of adaptive gain
matrices $K_n.$ We turn to the
probabilistic interpretation in the next subsection; in subsequent
subsections we derive algorithms, leading in particular to a specific choice of adaptive gain matrices. 

%%%%%%%%%%%%%%%%%%%%%%%%%%%%%%%%%%%%%%%%
%
\subsection{Probabilistic Perspective}
\label{ssec:PP}
%
%%%%%%%%%%%%%%%%%%%%%%%%%%%%%%%%%%%%%%%%%

We have shown that the 3DVAR methodology can recover the state of a
(possibly chaotic) dynamical system, even though the initial condition is not known, by exploiting the observations. However 3DVAR does not quantify uncertainty in the state estimate; it is derived from a purely control theoretic perspective. We now introduce
a probabilistic perspective which enables us to address the issue
of uncertainty quantification. Subsection \ref{sssec:ucd} discusses
the unconditioned dynamics from the perspective of evolution of probability densities.
In Subsection \ref{sssec:tfd} we define the filtering distribution and describe this
from the perspective of evolution of probability densities.
Subsection \ref{sssec:SPP} introduces the \emph{sample path perspective}
on algorithms for filtering, a central idea
in this paper. In Subsection \ref{sssec:IN} we establish some  notation, used throughout the
sequel, that is important for the reader to internalize.

\subsubsection{Unconditioned Dynamics}
\label{sssec:ucd}

To open our development of the probabilistic perspective we first consider
the unconditioned dynamics on state $\{v_n\}_{n \in \Z^+}$ defined by (\ref{eq:sd}a).
We let $r_n$ denote the probability density of random variable $v_n$
and derive an evolution equation for $r_n$. Irrespective of whether 
the evolution of the state $\{v_n\}_{n \in \Z^+}$
defined by (\ref{eq:sd}a) is linear, the evolution of $\{r_n\}_{n \in \Z^+}$ is linear. Furthermore the evolution of the state and its probability density are uncoupled from one another. The
evolution of the probability density implied by  (\ref{eq:sd}a) is given by
\begin{subequations}
\label{eq:ops0}
\begin{align}
r_{n+1}&=\op Pr_n,\\
(\op Pr)(\dd v)&=\Bigl(\int_{u \in \R^{d_v}} p(u,v)r(\dd u)\Bigr)\dd v,\\
p(u,v)&=\frac{1}{(2\pi)^{d_v/2} \sqrt{{\rm det}(\Sigma)}} \exp \Bigl(-\frac12|v-\Psi(u)|_{\Sigma}^2\Bigr).
\end{align}
\end{subequations}
Thus, in particular, $r_n$ evolves in time $n$ through application of a linear integral
operator. We refer to $\{v_n\}_{n \in \Z^+}$ defined by (\ref{eq:sd}a) as a \emph{Markov process.}
The situation when we condition the state on observations is different, leading
to nonlinear evolution of densities.

\subsubsection{The Filtering Distribution}
\label{sssec:tfd}

Here we introduce the {\em filtering distribution} with density $\mu_n$: 
the distribution of the conditioned random variable $v_n|\Yd_n$. This 
captures the knowledge of the state of the system, 
and uncertainties in the state, given the observations. 
To understand how uncertainty in estimates of the state evolves
it is thus important to understand how $\mu_n$ evolves with $n$.
Unlike the unconditioned dynamics, this conditioned dynamics has a \emph{nonlinear} structure.
Nonlinear evolution equations arise in filtering through the incorporation of data. This
nonlinearity renders filtering a challenging mathematical and computational problem.
To define this evolution we first define the linear operator
\begin{subequations}
\label{eq:ops2}
\begin{align}
(\op Q\mu)(\dd v,\dd y)&=q(v,y)\mu(\dd v)\dd y,\\
q(v,y)&=\frac{1}{(2\pi)^{d_y/2} \sqrt{{\rm det}(\Gamma)} }
\exp \Bigl(-\frac12|y-h(v)|_{\Gamma}^2\Bigr).
\end{align}
\end{subequations}
Then we define the $n-$dependent family of two nonlinear  operators
\begin{align}
\label{eq:ops}
(\op B_n \nu)(\dd v) &=\int_{y \in \R^{d_y}} \delta_{\yd_{n+1}}(y)\nu(\dd v,\dd y)\Big/\Bigl(\int_{(v,y) \in \R^{d_v} \times \R^{d_y}} \delta_{\yd_{n+1}}(y)\nu(\dd v,\dd y)\Bigr),\\
\op L_n(\mu)(\dd v)&=q(v,\yd_{n+1})\mu(\dd v)\Big/\Bigl(\int_{v \in \R^{d_v}}q(v,\yd_{n+1})\mu(\dd v)\Bigr).
\end{align}

The nonlinear map $\mu_n \mapsto \mu_{n+1}$ is most easily described by first introducing 
$\hmu_{n+1}$, the distribution of $v_{n+1}|\Yd_n$ and 
$\nu_{n+1}$, the distribution of $(v_{n+1},y_{n+1})|\Yd_n.$ 
The map from $\mu_n$ to $\hmu_{n+1}$ is determined by equation 
(\ref{eq:sd}a)  and is {\em linear} as a map from
the space of probability measures defined on $\R^{d_v}$ into 
itself as discussed in the preceding subsection; the map from $\hmu_{n+1}$ to $\nu_{n+1}$
is defined by (\ref{eq:sd}b), and is also {\em linear}, now as a map from
the space of probability measures defined on $\R^{d_v}$ into 
the space of probability measures defined on $\R^{d_v+d_y}$;
the map from $\nu_{n+1}$ to $\mu_{n+1}$ is defined by conditioning
$\nu_{n+1}$ on $\yd_{n+1}$ and is {\em nonlinear} as a map
from the space of probability measures defined on $\R^{d_v+d_y}$ 
into the space of probability measures defined on $\R^{d_v}.$ 
Using the preceding definitions of linear and nonlinear operators we have
\begin{subequations} 
\label{eq:DA}
\begin{empheq}[box=\widefbox]{align}
\hmu_{n+1}&=\op  P\mu_n,\\
\nu_{n+1}&=\op Q\hmu_{n+1},\\
\mu_{n+1}&=\op B_n(\nu_{n+1}). \label{eq:DAc}
\end{empheq}
\end{subequations}
Concatenating we find that
\begin{equation}
\label{eq:fm}
\mu_{n+1}=\op B_n(\op Q\op P\mu_n),\quad \mu_0=\Ng(m_0,C_0).
\end{equation}
This map defines an inhomogeneous nonlinear map on the space of
probability measures on $\R^{d_v}.$ The map
$\op B_n(\op Q\op P\cdot)$ may be decomposed into two maps:
(i) the {\em prediction} $\op P$, which represents application
of the dynamical model (\ref{eq:sd}a); and (ii) application
of {\em Bayes Theorem} \footnote{Often referred to as the {\em analysis} 
step in the geophysical data assimilation community. Bayes Theorem is discussed
in more detail in Section \ref{sec:IPDT}.}
through operator $\op L_n\cdot :=\op B_n(\op Q\cdot)$, which corresponds to use of likelihood defined by the observation model (\ref{eq:sd}b).
With this notation we thus obtain
\begin{subequations} 
\label{eq:DAsimple}
\begin{empheq}[box=\widefbox]{align}
\hmu_{n+1}&=\op P\mu_n,\\
\mu_{n+1}&=\op L_n(\hmu_{n+1}).
\end{empheq}
\end{subequations}
We refer to iteration of \eqref{eq:DAsimple}
as the {\em filtering cycle}. 
The cycle involves iterative interleaving of prediction, using the
dynamical model, a linear operation on measures,
and Bayes Theorem, using the observation model, a nonlinear
operation on measures.

It is important to appreciate that there is, in general,
no closed form expression for $\mu_n$ defined by the 
iteration \eqref{eq:DAsimple}; thus \eqref{eq:DAsimple} 
does not constitute an algorithm.
However if $\Psi$ and $h$ are linear then, 
since $\mu_0=\Ng(m_0,C_0)$ is Gaussian,
it follows that $\hmu_{n+1},\nu_{n+1},\mu_{n+1}$ are all Gaussian for all $n \in \Z^+$,
and closed form expressions, based on dynamical updates of means and
covariances, are available. This linear Gaussian setting is discussed
in the following example and linear Gaussian examples will be used
throughout the paper. However the main thrust of the paper concerns
nonlinear and non-Gaussian problems; for these problems further ideas,
which we will explain in subsequent subsections, are required to make actionable algorithms from the 
iteration \eqref{eq:DAsimple}.

\begin{example}
\label{ex:sssec:2}
To define the {\em Kalman filter}
we consider the setting \eqref{eq:linearsd}, where $\Psi(\cdot)$ and $h(\cdot)$  
are both linear. Then \eqref{eq:sd} becomes
\begin{subequations}
\label{eq:sdl}
\begin{align}
v_{n+1} &= Mv_n + \xi_n, \\
y_{n+1} &= Hv_{n+1} + \eta_{n+1}.
\end{align}
\end{subequations}
For this problem the mapping \eqref{eq:DAsimple} may be solved explicitly.
In fact $\mu_n$ and $\hmu_n$ are both Gaussian and we write their mean-covariance
pairs as $(m_n,C_n)$ and $(\pmean_n,\pCov_n)$ respectively. 
Then $\hmu_{n+1}$ is determined from $\mu_n$ by the formulae
\begin{subequations}
\label{eq:KF_pred_mean2b}
\begin{align}
    \pmean_{n+1} &= M \mean_{n},\\
    \pCov_{n+1} &= M \Cov_{n}M^\top+\Sigma,
\end{align}
\end{subequations}
the prediction step. Measure $\mu_{n+1}$ is determined from $\hmu_n$ by the
application of a Bayesian update, solving the
inverse problem defined by (\ref{eq:sdl}b); this inverse problem is
for $v_{n+1}$ given fixed realization of data $y_{n+1}=\yd_{n+1}$ generated by \eqref{eq:sdl}.
Since the prior and posterior are Gaussian we may
complete the square to solve the Bayesian inverse problem to give
the following update formulae for the mean and precision:
\begin{subequations}
\label{eq:KF_pred_mean2b2}
\begin{align}
     \Cov_{n+1}^{-1} \mean_{n+1} &= \pCov_{n+1}^{-1} \pmean_{n}+H^\top\Gamma^{-1}\yd_{n+1},\\
    \Cov_{n+1}^{-1} &= \pCov_{n+1}^{-1}+H^\top\Gamma^{-1}H. 
\end{align}
\end{subequations}
By use of the Woodbury matrix identity
we obtain the following formulae expressed in terms of covariances
rather than precisions:
\begin{subequations}
\label{eq:KF_analysisL}
\begin{align}
        \mean_{n+1} &= \pmean_{n+1} + \pCov_{n+1}H^\top(H\pCov_{n+1}H^\top+\Gamma)^{-1} (\yd_{n+1} - H\pmean_{n+1}),\\
         \Cov_{n+1} &= \pCov_{n+1} - \pCov_{n+1}H^\top(H\pCov_{n+1}H^\top+\Gamma)^{-1}H \pCov_{n+1}.
    \end{align}
\end{subequations}
The update equations \eqref{eq:KF_pred_mean2b} simply represent propagation of
a Gaussian under the linear dynamics defined by (\ref{eq:sdl}a);
\eqref{eq:KF_analysisL} corresponds to application of Bayes theorem, and
in this particular linear setting to conditioning a Gaussian, using the
observations defined by (\ref{eq:sdl}b) with $y_{n+1}=\yd_{n+1}.$
The Kalman filter update equations \eqref{eq:KF_pred_mean2b}, \eqref{eq:KF_analysisL} 
are well-defined if $\Gamma \succ 0.$
$\blacksquare$
\end{example}

The Gaussian setting of the preceding example is very special.
But the idea of making a \emph{Gaussian approximation}
will play a central role in ensemble Kalman methods and as a consequence
the explicit calculations in the example will be generally useful. 
The perspective of invoking Gaussian approximations is introduced in Subsection \ref{ssec:GPFD}; it is subsequently developed to form the backbone of the methodology
highlighted in this paper.

%%%%%%%%%%%%%%%%%%%%%%%%%%%%%%%%%%%%%%%%
%
\subsubsection{The Sample Path Perspective}
\label{sssec:SPP}
%
%%%%%%%%%%%%%%%%%%%%%%%%%%%%%%%%%%%%%%%%%

In Subsection \ref{sssec:ucd} we demonstrated that the (typically
nonlinear) state space evolution of
$\{v_n\}_{n \in \Z^+}$ defined by (\ref{eq:sd}a) may be alternatively viewed in terms of 
the linear evolution of probability density functions $r_n$ 
defined by \eqref{eq:ops0}. In Subsection \ref{sssec:tfd} we showed that, when
conditioned on observations, the evolution of the probability density function
becomes nonlinear and is defined by \eqref{eq:DA} or \eqref{eq:DAsimple}.
In this section we address the question of finding an evolution in state space
that is consistent with this nonlinear evolution of densities
defined by \eqref{eq:DA} or \eqref{eq:DAsimple}. Our aim is to
generalize the control theoretic data assimilation algorithm given by \eqref{eq:sd2n} in order to find such an evolution.

Throughout this subsection let $v_n$ be a random variable distributed according to
$\law(v_n)$, let $\hv_{n+1}$ be random variable with
$\law(\hv_{n+1})=\op P\,\law(v_n)$
and let $(\hv_{n+1},\hy_{n+1})$ be random variable with law
$\op Q\,\law(\hv_{n+1}).$ Using $\law$ avoids a proliferation of
notation for the different measures arising from the use of 
various different algorithms to approximate the filtering cycle.

All of the algorithms that we will introduce in what follows are based on a prediction of state, and possibly observation,
from $\law(v_n)$. To this end recall \eqref{eq:sd2n} and for all $n \in \Z^+$,
\begin{subequations}
\label{eq:recall}
\begin{align}
\hv_{n+1} &= \Psi(v_n) + \xi_n, \\
\hy_{n+1} &= h(\hv_{n+1}) + \eta_{n+1}.
\end{align}
\end{subequations}
The distributions of $\hv_{n+1}$ and $(\hv_{n+1},\hy_{n+1})$ given by these equations define $\op P\,\law(v_n)$ and $\op Q\,\law(\hv_{n+1})$, respectively. A common theme in this paper is to augment these equations
for the predicted state and data with a final map, to generalize (\ref{eq:sd2n}c), of the
form
\begin{align}
\label{eq:0recall}
(\hv_{n+1},\hy_{n+1}) \mapsto v_{n+1},
\end{align}
or
\begin{align}
\label{eq:0recall2}
\hv_{n+1} \mapsto v_{n+1}.
\end{align}
These maps are chosen, respectively, to mimic \eqref{eq:DA} or \eqref{eq:DAsimple}.
Then, if $v_n \sim \mu_n$ the maps are designed so that $v_{n+1} \sim \mu_{n+1}$ where $\mu_n$ and $\mu_{n+1}$ are related by the filtering update \eqref{eq:DA} or, equivalently, \eqref{eq:DAsimple}. A key observation is that these maps will necessarily depend on the distributions of $\hv_{n+1}$ and $\hy_{n+1}$ rendering the associated Markov processes of mean field type. Thus, in contrast to the unconditioned dynamics,
the state space evolution does not decouple from the evolution of the associated probability
density function; rather it depends on it. The resulting state space evolution is said to define
a \emph{nonlinear Markov process.}

Note that, once \eqref{eq:recall} is augmented with either \eqref{eq:0recall} or \eqref{eq:0recall2}, we have a state space
evolution that describes, through the probability distribution of $v_n$, the filtering process; the state space
evolution can then be used as the basis for algorithms. Thus
the key question for such a program is the identification of
either \eqref{eq:0recall} or \eqref{eq:0recall2}.
The existence of \emph{transport maps} or, more generally, \emph{couplings}, which define the 
steps \eqref{eq:0recall} or \eqref{eq:0recall2}, follows under quite general conditions. But
finding them explicitly is generally difficult. Furthermore, the maps \eqref{eq:0recall} and \eqref{eq:0recall2} 
are not uniquely defined, in general. As a consequence of the difficulty in identifying mean field transport maps, the algorithms we study will be based on identifying maps \eqref{eq:0recall} or \eqref{eq:0recall2} which only \emph{approximately} achieve the filtering update \eqref{eq:DA}. However, all formulations considered in this survey will be of mean field type.

In summary, we refer to equations (\ref{eq:recall},\ref{eq:0recall}) or 
(\ref{eq:recall}a,\ref{eq:0recall2}) as providing a {\em sample path perspective}. The resulting algorithms update
state $v_n \mapsto v_{n+1}$ and are designed to exactly (in theory), and approximately (in practice), reproduce the {\em probabilistic perspective} encapsulated in the three steps of \eqref{eq:DA},
or the two steps of \eqref{eq:DAsimple}. These algorithms
result in a sample path $\{v_\ell\}_{\ell=0}^n$ with property that (possibly only approximately) $\law(v_\ell) = \mu_\ell.$ This idea is central to the algorithmic developments in the paper. 

%%%%%%%%%%%%%%%%%%%%%%%%%%%%%%%%%%%%%%%%
%
\subsubsection{Important Repeatedly Used Notation}
\label{sssec:IN}
%
%%%%%%%%%%%%%%%%%%%%%%%%%%%%%%%%%%%%%%%%%

The approximations we develop will be based on matching first and second
order moments. In service of designing these approximations,
it is useful to define various first and second order statistics
computed under the law of $(\hv_{n+1},\hy_{n+1})$. First define mean and
covariance of $\hv_{n+1}:$
\begin{subequations}
\label{eq:KF_pred_mean}
\begin{empheq}[box=\widefbox]{align}
    \pmean_{n+1} &= \E \hv_{n+1},\\
    \pCov_{n+1} &= \E \Bigl(\bigl(\hv_{n+1}-\pmean_{n+1})\otimes
\bigl(\hv_{n+1}-\pmean_{n+1}\bigr)\Bigr). 
\end{empheq}
\end{subequations}
Then define mean of the predicted data, 
cross-covariance from predicted data to state and covariance of the data:
\begin{subequations}
\label{eq:KF_pred_mean2}
\begin{empheq}[box=\widefbox]{align}
\ho_{n+1} &=\E \hy_{n+1},\\
\pCov_{n+1}^{vy}      &=  \E\Bigl(\bigl(\hv_{n+1}- \pmean_{n+1}\bigr)\otimes
\bigl(\hy_{n+1}- \ho_{n+1}\bigr)\Bigr),\\
    \pCov_{n+1}^{yy} &=  \E\Bigl(\bigl(\hy_{n+1}- \ho_{n+1}\bigr)\otimes
\bigl(\hy_{n+1}-\ho_{n+1}\bigr)\Bigr).
\end{empheq}
\end{subequations}
From these covariances we define the matrix
\begin{align} \label{eq:Kalman_gain}
K_{n} = \hCvy_{n+1}(\hCyy_{n+1})^{-1}. 
\end{align}
This particular choice of $K_n$, known as the \emph{Kalman gain},
plays a central role in the mean field maps
which underpin ensemble Kalman methods through their particle approximations.
Note that, if $\hCyy_{n+1}$ is not invertible, then its action
may still be defined through a pseudo-inverse.

It is sometimes useful to express the Kalman gain $K_n$ in terms of the variable $\hh_{n+1}=h(\hv_{n+1})$ and without
reference to predicted data $\hy_{n+1}.$
For this purpose we define the following correlation matrices,
computed under the law of $\hv_{n+1}$:
\begin{subequations}
\label{eq:KF_joint2b}
\begin{empheq}[box=\widefbox]{align}
\ho_{n+1} &=\E h(\hv_{n+1}),\\
     \pCov_{n+1}^{vh} &=       \E\Bigl(\bigl(\hv_{n+1}- \pmean_{n+1}\bigr)\otimes
\bigl(\hh_{n+1}-\ho_{n+1}\bigr)\Bigr),\\
    \pCov_{n+1}^{hh} &=   \E\Bigl(\bigl(\hh_{n+1}- \ho_{n+1}\bigr)\otimes
\bigl(\hh_{n+1}-\ho_{n+1}\bigr)\Bigr),
\end{empheq}
\end{subequations}
Then, in place of 
\eqref{eq:KF_pred_mean2} and \eqref{eq:Kalman_gain}, we have
\begin{subequations}
\label{eq:KF_joint2c}
\begin{empheq}[box=\widefbox]{align}
%\ho_{n+1} &=\E h(\hv_{n+1}),\\
\pCov_{n+1}^{vy} &= \pCov_{n+1}^{vh},\quad
\pCov_{n+1}^{yy} = \pCov_{n+1}^{hh} + \Gamma,\\
K_n &= \pCov_{n+1}^{vh}\left( \pCov_{n+1}^{hh} + \Gamma\right)^{-1}.
\end{empheq}
\end{subequations}
Note that, if $\Gamma \succ 0$,  $K_n$ is well-defined without
recourse to the use of pseudo-inverse.

\begin{example}
In the setting of the linear and Gaussian Example \ref{ex:sssec:2} we have
\begin{align*}
\hCvy_{n+1}&=\pCov_{n+1}H^\top,\\
\hCyy_{n+1}&=H\pCov_{n+1}H^\top+\Gamma.
\end{align*} 
and the mean update (\ref{eq:KF_analysisL}a) may be written
$$\mean_{n+1} = \pmean_{n+1} + K_n(\yd_{n+1} - H\pmean_{n+1}).$$
In particular only the single covariance $\pCov_{n+1}$ needs to be computed.
Note the similarity of the resulting algorithm with the 3DVAR algorithm \eqref{eq:two}, in
the linear Gaussian setting,
since in this linear case $\pmean_{n+1}=M\mean_n.$ It differs only through having a 
time-varying gain matrix $K_n.$
$\blacksquare$
\end{example}

This linear Gaussian setting provides some motivation for the
Kalman gain.  The origin of this key concept in the more general
nonlinear and non-Gaussian setting  will be described in the next subsection.

%%%%%%%%%%%%%%%%%%%%%%%%%%%%%%%%%%%%%%%%%%%%%%%%%%%%%%%%%%
%
\subsection{Gaussian Projected Filtering Distribution}
\label{ssec:GPFD}
%
%%%%%%%%%%%%%%%%%%%%%%%%%%%%%%%%%%%%%%%%%%%%%%%%%%%%%%%%

The \emph{Gaussian projected filter} gives a Gaussian
approximation of the true filtering distribution.
It is defined by the following three steps: 
\begin{itemize}
\item (i) taking input Gaussian at time $n$ as $\law(v_n)$ and pushing 
this measure forward under \eqref{eq:recall} to find (typically non-Gaussian)
measure $\law(\hv_{n+1},\hy_{n+1})$; 
\item (ii) projecting this joint measure onto the nearest
Gaussian (in a sense that we will make precise); 
\item (iii)
conditioning this Gaussian on the data $\yd_{n+1}$ to find output
Gaussian at time $n+1.$ 
\end{itemize}
We note that conditioning a Gaussian random
variable on linear functionals of the random variable returns another
Gaussian. Thus the algorithm maps Gaussians to Gaussians.
The resulting approximation of the filtering distribution
plays an important role in motivating the mean field maps,
introduced in Subection \ref{ssec:MFM}, that underly ensemble Kalman methods. It is also of interest as a method in its own right. 

In what follows in this subsection we introduce  
the Gaussian projected approximation to the evolution \eqref{eq:fm}. In the case where $\Psi(\cdot)$ and $h(\cdot)$ are linear the resulting formulae deliver exact solutions of
the filtering cycle \eqref{eq:fm}, leading to the Kalman filter
as given in Example \ref{ex:sssec:2}. The Gaussian projected filter also
leads to a derivation of the Kalman gain \eqref{eq:Kalman_gain}, beyond
the linear Gaussian setting; an alternative derivation, using 
the minimum variance approach, may be found in Appendix C, Subsection \ref{ssec:TM_MVA}.

To describe the Gaussian projected approximation to the
filtering distribution, we define the map $\op G$,
definition of which uses $\gP=\gP(\R^d)$ and
$\gG=\mathfrak{G}(\R^d)$ defined in Subsection \ref{ssec:N}. 
\begin{definition}
Define $\op G: \gP \mapsto \gG$ by
\begin{equation*} \label{eq:Gaus_projection}
\op G\mu=\Ng(m^\mu,C^\mu), \qquad m^\mu=\E^\mu u,\quad C^\mu=\E^\mu \bigl((u-\E^\mu u)\otimes(u-\E^\mu u)\bigr),
\end{equation*}
where $u \sim \mu$.
$\blacksquare$ \end{definition}
Thus the map $\op G$ applied to measure $\mu$ simply computes 
the Gaussian with mean and covariance calculated with respect to the, typically non-Gaussian, measure $\mu.$
Notice that $\op G$ is the identity on Gaussians. Furthermore $\op G \circ \op G=\op G.$
We refer to this as a {\em projection} onto Gaussians because it
corresponds to finding the closest point to given measure $\mu$,
with respect to a Kullback-Leibler divergence\footnote{For details see the bibliography Subsection \ref{ssec:BSE}.}:
\begin{equation}
\label{eq:KLG}
\op G\mu={\rm argmin}_{\pi \in \mathfrak{G}} d_{\rm {KL}} (\mu||\pi).
\end{equation}

We now use mapping $\op G$ to find an approximation to the evolution \eqref{eq:DA} 
which generates measures remaining Gaussian; intuitively this will
be a good approximation whilst the measures $\{\mu_n\}$ evolving under \eqref{eq:fm} remain close to Gaussian. To this end we consider random variable 
$v_n \sim \mug_n$ where the
probability measure $\mug_n$ evolves according to 
\begin{equation}
\label{eq:fmg}
\mug_{n+1}=\op B_n(\op G\op Q\op P\mug_n), \quad \mug_0=\Ng(m_0,C_0).
\end{equation}
This may be decomposed as follows:
\begin{subequations}
\label{eq:gprop}
\begin{empheq}[box=\widefbox]{align}
\hmug_{n+1}&=\op P\mug_n,\\
\nug_{n+1}&=\op Q\hmug_{n+1},\\
\mug_{n+1}&=\op B_n(\op G\nug_{n+1}).
\end{empheq}
\end{subequations}
Map \eqref{eq:fmg} defines a nonlinear Markov process,
similarly to \eqref{eq:fm}. It also maps Gaussians into Gaussians.
This fact follows from the fact that the nonlinear map $\op B_n(\cdot)$,
which represents conditioning, maps Gaussians into Gaussians.
The map $\mug_n \mapsto \mug_{n+1}$ hence defines a deterministic 
mapping from the mean $m_n$ and covariance $C_n$ of $\mug_n$ 
into the mean $m_{n+1}$ and covariance $C_{n+1}$ of $\mug_{n+1}.$
We now identify this map explicitly.

For $v_n \sim \mug_n$
we introduce the random variables $\hv_{n+1}, \hy_{n+1}$ defined by
\eqref{eq:recall}.
It then follows that $\hv_{n+1} \sim \hmug_{n+1}=\op P\mug_n$ and that
$(\hv_{n+1}, \hy_{n+1}) \sim \nug_{n+1}=\op Q\op P\mug_n.$ 
Note also that $\hmug_{n+1}$ and $\nug_{n+1}$ are not Gaussian,  
but are defined by the  
Gaussian $\mug_n$ and hence completely determined by $m_n$ and $C_n.$
Thus the mean and covariance under $\hmug_{n+1}$ are 
$(\pmean_{n+1},\pCov_{n+1})$ given by \eqref{eq:KF_pred_mean}.
Furthermore, $\op G\nug_{n+1}$ is defined by
\begin{equation}
\label{eq:KF_jointy}
    \op G\nug_{n+1}=\Ng\Bigl(
    \begin{bmatrix}
    \pmean_{n+1}\\
    \ho_{n+1}
    \end{bmatrix}, 
    \begin{bmatrix}
   \pCov_{n+1} & \pCov_{n+1}^{vy}\\
    \bigl(\pCov_{n+1}^{vy}\bigr)^\top & \pCov_{n+1}^{yy}
    \end{bmatrix}
    \Bigr),
\end{equation}
where all relevant quantities are defined in Subsection \ref{sssec:IN}. 
%\eqref{eq:KF_pred_mean}, \eqref{eq:KF_pred_mean2} and \eqref{eq:underlie}.

We now condition the Gaussian $\op G\nug_{n+1}$,
on the second component of the vector
taking value $\yd_{n+1}$. From this we find the
Gaussian measure $\mug_{n+1}=\op B_n(\op G\nug_{n+1})$ characterized by
mean $\mean_{n+1}$ and covariance $\Cov_{n+1}$ given by the following lemma.

\begin{lemma}
\label{lemma:KF_analysis}
Assume that $\Gamma \succ 0.$ Let $m_n$ and $C_n$ denote the
mean and covariance under the Gaussian projected filter. 
Consider equations \eqref{eq:recall} initialized at $v_n \sim \Ng(m_n,C_n)$, and then  $(\mh_{n+1}, \pCov_{n+1})$ defined by \eqref{eq:KF_pred_mean}; furthermore, define the mean of the observed data and covariances given by \eqref{eq:KF_pred_mean2}. Then $\pCov_{n+1}^{yy} \succ 0$ for all $n \in \Z^+$ and
\begin{subequations}
\label{eq:KF_analysis}
\begin{empheq}[box=\widefbox]{align}
        \mean_{n+1} &= \pmean_{n+1} + \pCov_{n+1}^{vy} (\pCov_{n+1}^{yy})^{-1} \bigl(\yd_{n+1} -  \ho_{n+1}\bigr),\\
         \Cov_{n+1} &= \pCov_{n+1} - \pCov_{n+1}^{vy}(\pCov_{n+1}^{yy})^{-1} \bigl(\pCov_{n+1}^{vy}\bigr)^\top,
    \end{empheq}
\end{subequations}
where $\{\yd_{n}\}$ arises from a fixed realization of \eqref{eq:sd}.
$\Diamond$ \end{lemma}
\begin{proof}
We first note that $\pCov_{n+1}^{yy} \succ 0$. Indeed, since by assumption $\Gamma \succ 0$ and by definition $\pCov_{n+1}^{hh}\succeq 0$, then $\pCov_{n+1}^{hh} + \Gamma \succ 0$. Hence by (\ref{eq:KF_joint2c}) we have that $\pCov_{n+1}^{yy} \succ 0$. Now consider the distribution of the Gaussian $\op G\nug_{n+1}$ 
given by \eqref{eq:KF_jointy}. Conditioning the resulting joint random variable on
$(v,y) \in \R^{d_v} \times \R^{d_y}$ on $y=\yd_{n+1}$, it is possible to conclude that from standard formulae for conditioned Gaussians that
$m_{n+1}$ and $C_{n+1}$ are given by the expressions in \eqref{eq:KF_analysis}. \end{proof}

Equations \eqref{eq:recall}, \eqref{eq:KF_pred_mean}, \eqref{eq:KF_pred_mean2} and \eqref{eq:KF_analysis}
define a mapping from $\mug_n$, characterized by $(m_n,\Cov_n)$, into
$\mug_{n+1}$, characterized by $(m_{n+1},\Cov_{n+1})$. They comprise an
explicit set of formulae for the mapping \eqref{eq:fmg}:
since Gaussians are determined by mean and covariance,  the map
on measures is completely determined by the map from $(m_n,C_{n})$ to $(m_{n+1},C_{n+1})$.

We note that \eqref{eq:KF_analysis} can also be rewritten as
\begin{subequations}
\label{eq:KF_analysis_add}
\begin{empheq}[box=\widefbox]{align}
        \mean_{n+1} &= \pmean_{n+1} + \pCov_{n+1}^{vh} 
(\pCov_{n+1}^{hh}+\Gamma)^{-1} 
%\bigl(\yd_{n+1} - \E h(\hv_{n+1})\bigr),\\
\bigl(\yd_{n+1} - \ho_{n+1}\bigr),\\
\Cov_{n+1} &= \pCov_{n+1} - \pCov_{n+1}^{vh}(\pCov_{n+1}^{hh}+\Gamma)^{-1} \bigl(\pCov_{n+1}^{vh}\bigr)^\top.
\end{empheq}
\end{subequations} 
Equations \eqref{eq:recall}, \eqref{eq:KF_pred_mean}, \eqref{eq:KF_joint2b} 
and \eqref{eq:KF_analysis_add} then also
define the mapping from $\mug_n$ into
$\mug_{n+1}$ and also comprise an
explicit set of formulae for the updates of the mean and
covariance which characterize mapping \eqref{eq:fmg}.

Finally we note that, using the definition \eqref{eq:Kalman_gain} of Kalman gain, we can rewrite \eqref{eq:KF_analysis} as
\begin{align*}
\label{eq:KF_analysis99}
        \mean_{n+1} &= \pmean_{n+1} + K_n \bigl(\yd_{n+1} -  \ho_{n+1}\bigr),\\
         \Cov_{n+1} &= \pCov_{n+1} - K_n \bigl(\pCov_{n+1}^{vy}\bigr)^\top.
\end{align*}
Thus we see how the Kalman gain arises naturally within the context of this Gaussian
projected filter.

\begin{remark}
\label{rem:rem1}
The preceding explicit formulae  for \eqref{eq:fmg}
involve the computation of expectations under the non-Gaussian measure $\law(\hv_{n+1},\hy_{n+1})$. For this reason they do not constitute an algorithm. A possible approach to algorithmic implementations involves quadrature to approximate the expectations, leading, for example, to the \emph{unscented
Kalman filter} approach; details may be found in the bibliography
Subsection \ref{ssec:BSE}. However the formulation of explicit maps on Gaussians plays another, important, role in this paper: we use it as
a way of explaining the sense in which the distribution of our mean field models approximate evolution of measures under the true filtering distribution; see Subsection \ref{sssec:sots}. The explicit map on means and covariances 
can be used to derive the Kalman filter, which applies in the linear Gaussian setting, and is presented in Example \ref{ex:sssec:2}.
$\blacksquare$
\end{remark}

%%%%%%%%%%%%%%%%%%%%%%%%%%%%%%%%%%%%%%%%%%%%%%%%%%%%%%%%%%%
%
\subsection{Mean Field Maps}
\label{ssec:MFM}
%
%%%%%%%%%%%%%%%%%%%%%%%%%%%%%%%%%%%%%%%%%%%%%%%%%%%%%%%%%%%

In the previous section we did not adopt the sample path perspective, but rather chose to represent the evolution of the filtering distribution, approximately, as the evolution of Gaussians. In this subsection we introduce our first instance of the sample path perspective, finding an evolution in state space which approximately
captures the evolution of the filtering distribution. Like the Gaussian projected filter it uses a Gaussian ansatz, but in a different way, leading to a state space evolution which is not Gaussian.

Note that elements of $\gP(\R^d)$ are \emph{infinite dimensional} objects, for any $d$. Thus the filtering
distribution defines an evolution in an infinite dimensional space. This fact goes to the
heart of the computational challenges faced when solving the filtering problem. These computational
challenges are further exacerbated when $d \gg 1.$
The manifold of Gaussians $\gG(\R^d)$ is finite dimensional, because it is parameterized
by the mean and covariance and hence has dimension $\frac12 d(d+3).$
The preceding subsection provides explicit {\em finite dimensional maps} for the mean
and covariance which characterize the Gaussian projected filter $\mug_n \mapsto \mug_{n+1},$ an
approximation which is (intuitively) accurate when the true filter is close to Gaussian.
Nonetheless if $d \gg 1$ this method can still be prohibitive because 
the algorithm acts on a space of dimension
that grows quadratically in $d$.

To address the issue that the Gaussian projected filter may not be efficient
if $d \gg 1$ we introduce, in this section, a more ambitious aim: 
to find maps on finite dimensional spaces of dimension
$d$ with the property that (possibly only approximately) their output is equal in law to the map on measures $\mu_n \mapsto \mu_{n+1}$ given by
the filtering cycle. We achieve this by studying transport maps that achieve \eqref{eq:0recall} or \eqref{eq:0recall2}; we then weaken this requirement
and ask only for transport maps that approximately achieve \eqref{eq:0recall} 
or \eqref{eq:0recall2} in a manner that we will make precise.
When combined with \eqref{eq:recall}, either \eqref{eq:0recall} 
or \eqref{eq:0recall2} gives a sample path evolution which can be used
as the basis of algorithms to solve the filtering problem.
This transport 
map viewpoint leads us to the subject of mean field maps, namely random maps 
that depend on the law of the state being mapped. When approximated by 
particle methods these maps lead to methods that scale linearly with $d$, 
in contrast to the Gaussian projected filter which scales quadratically.

This section is organized as follows. We start in Subsection
 \ref{ssec:TM} with an introduction to transport maps. Subsection 
\ref{sssec:pt} describes two distinct transport approaches which effect exact filtering: one based on the conditioning component of the overall Bayesian inference step, a transport between probability measures on different spaces; and the other based on the prior to posterior map that constitutes the Bayesian inference step of filtering, a transport between probability measures on the same space. We refer to these maps which effect exact filtering as {\em perfect transport.}\footnote{Perfect transport should not be confused with \emph{optimal transport} which identifies among all (perfect) transport maps the one minimizing a certain cost functional such as that leading to the Wasserstein distance. See the bibliography Subsection \ref{ssec:BSE}, and Theorem \ref{t:OTT}, for more details. We use \emph{perfect} here to distinguish from the \emph{approximate} transport maps, based only on matching first and second moment; these approximate transport maps underpin ensemble Kalman methods.} Subsections \ref{sssec:ats} and \ref{sssec:atd} are concerned with approximations of these two perfect transports, and are motivated
in Subsection \ref{sssec:soti} with an explicit example. In these approximations the pushforward under the transport map is designed to match only the first and second order moments of the target measure. Hence these approximations are closely related to, but they are different from, the previously defined Gaussian projected filter; we elaborate on this connection in the summary Subsection \ref{sssec:sots}. That subsection also includes Example \ref{ex:mfk} in which we identify mean field formulations of the Kalman filter; recall that this filter
applies only to linear Gaussian systems, and is defined 
in Example \ref{ex:sssec:2}.

\subsubsection{Transport Maps}
\label{ssec:TM}

We start by setting-up notation used throughout. Consider probability measures $\nu$ and $\nu'$ on 
$\R^{d}$ and $\R^{d'}$ respectively,
and recall the definition of \emph{pushforward} of a measure under
under $T: \R^{d} \to\R^{d'}$: the statement $\nu'= T^\sharp \nu$ is a succinct way of stating that, if $\law(v)=\nu$ and $v'=T(v)$, then $\law(v')=\nu'.$
Given probability measures $\pi$ and $\pi'$ on $\R^{d}$ and $\R^{d'}$
respectively, a {\em transport} $T: \R^{d} \to \R^{d'}$ from $\pi$ to $\pi'$  is a map with
property that the pushforward of probability measure $\pi$ under $T$, 
$T^\sharp \pi$, is equal to probability measure $\pi'.$
In the following we refer to $\pi$ as the {\em source measure} and 
$\pi'$ as the {\em target measure} defining the transport.
In our setting $\pi'$ will be uniquely determined by $\pi$ and
an observed piece of finite dimensional data. Thus $T$ depends on $\pi$ and 
so we may view $T$ as a mapping $\R^{d} \times \gP \to \R^{d'},$
suppressing, for the moment, explicit dependence on the observed data. We can compute the pushforward under $T$ on any measure in $\gP$; but when we compute the pushforward on $\pi$ we obtain $\pi'.$ 
We emphasize that transport maps are not uniquely defined by their source and target measures and require certain conditions for their existence, which we assume here to be satisfied. The underlying mathematical concept is that of \emph{coupling of measures.} See  bibliography Subsection \ref{ssec:BSE} for discussion of transport, optimal transport and coupling.

We will also consider classes of \emph{approximate transport maps} which do
not achieve transport from $\pi$ to $\pi'$, but instead match first and second order moment information (we will be precise below). Such maps will also depend on $\pi.$ 

We now clarify an important notational issue. The dependence of (possibly approximate)
transport $T$ on a measure in $\gP$ does not affect the definition of pushforward; we employ the following general definition of pushforward for measure-dependent maps, taken to hold for all $\pi_1, \pi_2$
regardless of any assumed relationship between them:
\begin{equation}
\label{eq:overload}
T(\cdot;\pi_1)^\sharp \pi_2 =   T(\cdot;\widetilde{\pi})^\sharp \pi_2 \Bigr|_{\widetilde{\pi}=\pi_1};\\ 
\end{equation}
in particular,
\begin{subequations}
\label{eq:overload2}
\begin{align}
T(\cdot;\pi)^\sharp \pi &=   T(\cdot;\widetilde{\pi})^\sharp \pi \Bigr|_{\widetilde{\pi}=\pi},\\
T(\cdot;\pi)^\sharp (\op G\pi) &=   T(\cdot;\widetilde{\pi})^\sharp (\op G\pi) \Bigr|_{\widetilde{\pi}=\pi}.
\end{align}
\end{subequations}
In the preceding, pushforward under $\widetilde{\pi}-$dependent map $T(\cdot,\widetilde{\pi})$ denotes
regular pushforward with no relationship assumed between $\widetilde{\pi}$
and measure being pushed forward.
Note that we may define $\op T: \gP(\R^d) \to \gP(\R^d)$ by
$\op T(\pi)=T(\cdot;\pi)^\sharp \pi.$ 
The (approximate) transport maps just identified can be recast as
mean field maps when used in the context where source $\pi$ is the
distribution of the input to $T$. 

%%%%%%%%%%%%%%%%%%%%%%%%%%%%%%%%%%%%%%%%%%%%%%%%
%
\subsubsection{Perfect Transport}
\label{sssec:pt}
%
%%%%%%%%%%%%%%%%%%%%%%%%%%%%%%%%%%%%%%%%%%%%%%%%

Consider the idea of finding a transport map that acts
on the joint space of state and data, to effect conditioning with
respect to the observed data. To this end we consider the dynamical
system, assumed to hold for all $n \in \Z^+$: 
\begin{subequations}
\label{eq:sd2nn}
\begin{empheq}[box=\widefbox]{align}
\hv_{n+1} &= \Psi(v_n) + \xi_n, \\
\hy_{n+1} &= h(\hv_{n+1}) + \eta_{n+1}, \\
v_{n+1} &= \Ts(\hv_{n+1},\hy_{n+1};\nu_{n+1},\yd_{n+1}),
\end{empheq}
\end{subequations}
where $\{\yd_{n}\}$ arises from a fixed realization of \eqref{eq:sd}. Recall that
$\mu_n=\law(v_n)$, $\hmu_{n+1}=\law(\hv_{n+1})$ and $\nu_{n+1}=\law(\hv_{n+1},\hy_{n+1}).$ 

This is an example of the sample path perspective, and
(\ref{eq:recall}, \ref{eq:0recall}) in particular. The first two equations, which coincide with \eqref{eq:recall},
effect the mappings from $\mu_n$ to $\hmu_{n+1}$ and from
$\hmu_{n+1}$ to $\nu_{n+1}.$
Map\footnote{It is convenient to use both the notation $\Ts(\cdot,\cdot;\nu_{n+1},\yd_{n+1})$, to be explicit about important dependencies in $\Ts$, and to use the notation $\Ts_n$, for succinct statement of certain formulae when dropping explicit dependence of $\Ts$ on $\pi_{n+1}$ and $\yd_{n+1}$. We will use analogous notation for other mean field maps in what follows.}
$\Ts_n(\cdot,\cdot):=\Ts(\cdot,\cdot;\nu_{n+1},\yd_{n+1})$
is then an explicit example of \eqref{eq:0recall} defined to effect the desired conditioning of $\nu_{n+1}$ on $\yd_{n+1}$ in order to obtain $\mu_{n+1}.$
The letter $S$ in $\Ts$ connotes the dependence of the map on the {\em stochastic} data $\hy_{n+1}$. Thus equations \eqref{eq:sd2nn} define a mean field stochastic dynamical system mapping $v_n$ to $v_{n+1}$: 
stochastic because of the noise in (\ref{eq:sd2nn}a, \ref{eq:sd2nn}b), mean field because the map $\Ts$ in (\ref{eq:sd2nn}c) depends on the law $\nu_{n+1}$ of $(\hv_{n+1},\hy_{n+1})$, and hence on $\mu_n$.
The three update steps in this mean field stochastic dynamical system
lead to the following maps on measures:
\begin{subequations} 
\label{eq:DAmf}
\begin{empheq}[box=\widefbox]{align}
\hmu_{n+1}&=\op P\mu_n,\\
\nu_{n+1}&=\op Q\hmu_{n+1},\\
\mu_{n+1}&=(\Ts_n)^\sharp\nu_{n+1}. 
\end{empheq}
\end{subequations}
This is simply a restatement of \eqref{eq:DA}, noting
that $(\Ts_n)^\sharp$ has been chosen so that pushforward corresponds
to conditioning $\nu_{n+1}$ on data $\yd_{n+1}$ to obtain $\mu_{n+1}.$
In particular $\op T_n^S(\nu_{n+1}):=(\Ts_n)^\sharp\nu_{n+1}$
has property $\op T_n^S(\nu_{n+1})=\op B_n(\nu_{n+1}).$
Note that the implied map from $\mu_n$ to $\mu_{n+1}$ is a nonlinear Markov process, because of the dependence of $\Ts_n$ on $\nu_{n+1}$ and hence on $\mu_n.$ Furthermore, we have that  $\law(v_\ell)=\mu_\ell$ for all $\ell \in \Z^+$. The important
takeaway is that \eqref{eq:sd2nn} defines a sample-path picture of the evolution
of the filtering distribution: it provides a map in state space with law governed by
the filter. We reemphasize that such a sample-path representation is not uniquely defined.

Consider now a different approach to transport for filtering: we seek a transport map that acts on the state space only to effect Bayes Theorem, i.e.~mapping prior $\hmu_{n+1}$ to posterior $\mu_{n+1}$. To this end we consider the dynamical system
\begin{subequations}
\label{eq:mf0}
\begin{empheq}[box=\widefbox]{align}
\hv_{n+1} &= \Psi(v_n) + \xi_n, \\%\: n \in \Z^+,\\
v_{n+1}&=\Td(\hv_{n+1};\hmu_{n+1},\yd_{n+1}), %\: n \in \Z^+.
\end{empheq}
\end{subequations}
again assumed to hold for all $n \in \Z^+$. The first equation maps $v_n \sim \mu_n$ to $\hv_{n+1} \sim \hmu_{n+1}$, thus giving a sample path realization
of (\ref{eq:DAsimple}a). In the second equation, the map $\Td_n(\cdot):=\Td(\cdot;\hmu_{n+1},\yd_{n+1})$ is chosen so that, if
$\hv_{n+1} \sim \hmu_{n+1}$ then
$v_{n+1} \sim \mu_{n+1}$, thus giving a sample path realization
of (\ref{eq:DAsimple}b). 
Thus we have another instance of the sample path perspective, and
(\ref{eq:recall}, \ref{eq:0recall2}) in particular.

Equation \eqref{eq:mf0} constitutes another mean field stochastic dynamical system: stochastic because of the noise in (\ref{eq:mf0}a); mean field because the map $\Td$ in (\ref{eq:mf0}b) depends on the law of $\hv_{n+1}$ itself,
and hence on $\mu_n.$ The symbol $D$ distinguishes map $\Td$ from map $\Ts$: map $\Td$ is {\em deterministic} in the sense that it does not require stochastic data $\hy_{n+1}$, in contrast to $\Ts.$ Therefore, we again have that  $\law(v_\ell)=\mu_\ell$ for all $\ell \in \Z^+$, where the
probability measure $\mu_n$ evolves according to 
\begin{subequations}
\label{eq:pmfd}
\begin{empheq}[box=\widefbox]{align}
\hmu_{n+1}&=\op P\mu_n,\\
\mu_{n+1}&=(\Td_n)^\sharp\hmu_{n+1}.
\end{empheq}
\end{subequations}
This is simply a restatement of \eqref{eq:DAsimple}, noting
that $(\Td_n)^\sharp$ has been chosen so that the pushforward corresponds
to the application of Bayes Theorem to incorporate data
$\yd_{n+1}.$ 
In particular $\op T_n^D(\hmu_{n+1}):=(\Td_n)^\sharp\hmu_{n+1}$
has property $\op T_n^D(\hmu_{n+1})=\op L_n(\hmu_{n+1}).$
The evolution \eqref{eq:pmfd} is another nonlinear Markov process, now because of the dependence of $\Td_n$ on $\hmu_{n+1}.$
Again, the underlying sample-path representation \eqref{eq:mf0} is not uniquely defined.

The two transport maps $\Ts$ and $\Td$ introduce an important conceptual
approach to algorithms for filtering, but determining the maps can be
as hard, or harder, than solving the filtering problem itself. Thus,
in the next two subsections, we turn to relaxations of the perfect
transport effected by $\Ts$ and $\Td$. We instead seek  
mean field maps which match only first and second order moment information; this relaxation allows for approximate transport maps with simple affine (in 
senses to be made precise) forms. The perspective of matching first and second order moments naturally suggests working with Gaussians and hence we also relate the approximate transport to the Gaussian projected filter.

%%%%%%%%%%%%%%%%%%%%%%%%%%%%%%%%%%%%%%%%%%%%%%%%%%%%%%%%%%%
%
\subsubsection{Second Order Transport -- Motivation}
\label{sssec:soti}
%
%%%%%%%%%%%%%%%%%%%%%%%%%%%%%%%%%%%%%%%%%%%%%%%%%%%%%%%%

To motivate the more general ideas behind second order transport,
we first study an explicit example. It is well known how to transform samples from a unit centered Gaussian random variable on $\R$ into samples from a Gaussian random variable with mean $m\not=0$ and variance $\sigma \not=1$ by a simple scaling and shifting operation. An appropriate generalization of such a procedure suggests consideration of the map
\begin{subequations}
\label{eq:sd2nn_add_again}
\begin{empheq}[box=\widefbox]{align}
\hv_{n+1} &= \Psi(v_n) + \xi_n, \\
\hy_{n+1} &= h(\hv_{n+1}) + \eta_{n+1}, \\
v_{n+1} &= m_{n+1}+C_{n+1}^{\frac12}\pCov_{n+1}^{-\frac12}(\hv_{n+1}-\E \hv_{n+1}),
\end{empheq}
\end{subequations}
where $m_{n+1}$, $\pCov_{n+1}$, and $C_{n+1}$ are determined by
\eqref{eq:KF_pred_mean}, \eqref{eq:KF_pred_mean2} and
\eqref{eq:KF_analysis}, using (\ref{eq:sd2nn_add_again}a) and
(\ref{eq:sd2nn_add_again}b). Here $\{\yd_{n}\}$ arises again from a fixed realization of \eqref{eq:sd}.
This is a specific instance of the sample path perspective, and
(\ref{eq:recall}, \ref{eq:0recall2}) in particular. In this case the
sample path evolution of $v_n$ has law which
only approximates the true filtering law.

The map $v_n \mapsto v_{n+1}$ defined by \eqref{eq:sd2nn_add_again}
is a mean field map because of the dependence of $m_{n+1}, C_{n+1}$ and
$\pCov_{n+1}$ on $\op Q\,\law(v_n)$. It may be viewed as an approximation to
\eqref{eq:sd2nn} which is exact when $\op Q\,\law(v_n)$ is Gaussian. To demonstrate
exactness on Gaussians it suffices to show that we obtain the
desired mean and covariance after application of the map. It is clear from (\ref{eq:sd2nn_add_again}c) 
that $\E v_{n+1}=m_{n+1}$ and also that
\begin{align*}
&\E\bigl( (v_{n+1}-m_{n+1}) \otimes (v_{n+1}-m_{n+1}) \bigr)\\
&\quad\quad\quad\quad\quad =C_{n+1}^{\frac12}\pCov_{n+1}^{-\frac12}
\E\bigl( (v_{n+1}-m_{n+1}) \otimes (v_{n+1}-m_{n+1})\bigr)\pCov_{n+1}^{-\frac12} C_{n+1}^{\frac12}\\
&\quad\quad\quad\quad\quad =C_{n+1}^{\frac12}\pCov_{n+1}^{-\frac12}\pCov_{n+1}\pCov_{n+1}^{-\frac12} C_{n+1}^{\frac12}\\
&\quad\quad\quad\quad\quad =C_{n+1}.
\end{align*}
It is important to recognize that, in general, $v_{n+1}$ defined by (\ref{eq:sd2nn_add_again}c) will not be Gaussian distributed since $\hv_{n+1}$, defined by (\ref{eq:sd2nn_add_again}a), will not be Gaussian either. However, although \eqref{eq:sd2nn_add_again} does not provide a closed iteration on Gaussians, the map from $\hv_{n+1}$ to $v_{n+1}$ agrees with the same step in the Gaussian projected filter, at the level of first and second order moments. But, because it is not
a closed iteration on $\mathfrak{G}(\R^{d_v}),$ it is clearly not the same as the Gaussian projected filter.

Whilst the mean field map from \eqref{eq:sd2nn_add_again} is a relatively transparent way to achieve the goal of matching first and second order moments in the transport
step there is an uncountable set of ways of achieving this objective; the next two subsections demonstrate this, identifying all mean field maps effecting approximate transport from within two specific classes of affine transformations. We will then  highlight a small subset that have been used in practice, each of which is useful in certain specific contexts.

%%%%%%%%%%%%%%%%%%%%%%%%%%%%%%%%%%%%%%%%%%%%%%%%%%%%%%
%
\subsubsection{Second Order Transport -- Stochastic Case}
\label{sssec:ats}
%
%%%%%%%%%%%%%%%%%%%%%%%%%%%%%%%%%%%%%%%%%%

The first class of approximate filters determined by mean field maps
have the sample path form
\begin{subequations}
\label{eq:sd2nn-seek}
\begin{empheq}[box=\widefbox]{align}
\hv_{n+1} &= \Psi(v_n) + \xi_n, \\
\hy_{n+1} &= h(\hv_{n+1}) + \eta_{n+1}, \\
v_{n+1} &= \tT(\hv_{n+1},\hy_{n+1};\nu_{n+1},\yd_{n+1}),
\end{empheq}
\end{subequations}
where $\{\yd_{n}\}$ arises from a fixed realization of \eqref{eq:sd}. We identify
$\tT_n: \R^{d_v} \times \R^{d_y} \to \R^{d_v}$, where
$\tT_n(\cdot) := \tT(\cdot\,;\nu_{n+1},\yd_{n+1})$ approximates
an exact transport map $\Ts_n(\cdot)=\Ts(\cdot\,;\nu_{n+1},\yd_{n+1})$, defined
previously, by matching the first and second order moments. 

We next introduce\footnote{We temporarily drop explicit notational dependence on $n+1$ in $\nu$ and in $\yd$. This should not cause confusion as the approximate map we derive is concerned simply with finding a push forward which approximates conditioning of $\law(\hv_{n+1},\hy_{n+1})$  on $\hy_{n+1}=\yd$.} $\nu=\law(\hv_{n+1},\hy_{n+1})$ 
and assume that the exact and approximate transport maps satisfy, respectively,
\begin{subequations}
\label{eq:compare}
\begin{align}
(\Ts_n)^\sharp \nu&=\op B_n(\nu),\\
\op G\bigl((\tT_n)^\sharp \nu\bigr)&=\op B_n(\op G\nu),
\end{align}
\end{subequations}
for all measures $\nu$ on the product space $\R^{d_v}\times \R^{d_y}.$ Of course $\Ts_n$ and $\tT_n$ will depend on $\nu$ and then pushforward is to be interpreted as in \eqref{eq:overload}, \eqref{eq:overload2}. It is intuitive that (\ref{eq:compare}b) enforces map $\tT_n$ to satisfy (\ref{eq:compare}a) when $\nu$ is Gaussian; we prove this in Lemma \ref{lem:mot1} below.

Perfect transport corresponds to asking that, for all measures $\nu$ on the space $\R^{d_v} \times \R^{d_y}$, (\ref{eq:compare}a) holds; second order transport relaxes this and asks only that
(\ref{eq:compare}b) holds. Whilst achieving (\ref{eq:compare}a) may be harder than solving the filtering problem directly, we will show that achieving (\ref{eq:compare}b) is straightforward and leads to computationally tractable methods. This gain in tractability comes at the price of only achieving (\ref{eq:compare}b) in place of (\ref{eq:compare}a);
however it is intuitive that this price will not be high for settings
in which the filtering distribution, and the predictive
distribution on state and data, is not too far from Gaussian; we flesh out this idea in Subsection \ref{sssec:sots} below. 

Working to satisfy (\ref{eq:compare}b) allows us to find tractable approximate second order transport maps by seeking $\tT$ in the form
\begin{equation}
\tT (\hv_{n+1},\hy_{n+1};\nu,\yd)   :=A\hv_{n+1}+B\hy_{n+1}+a. 
\label{eq:sd2cnadd2}
\end{equation}
We allow the matrices/vectors $A,B,a$ to depend on $(\nu,\yd)$; however, they are assumed to be independent of $(\hv_{n+1},\hy_{n+1})$. Making this assumption ensures that the transport map is affine with respect to the realization of $(\hv_{n+1},\hy_{n+1})$ (but not their law). This in turn leads to tractable computations to determine $A,B,a$ on the basis of matching second order moments of perfect transport. In addition to computational tractability, the affine form of the transport map $\tT$ is motivated by the following which shows that the approximate transport is perfect when applied to Gaussian source:

\begin{lemma}
\label{lem:mot1}
Consider approximate transport map $\tT_n=\tT (\hv_{n+1},\hy_{n+1};\nu,\yd)$ with the form \eqref{eq:sd2cnadd2},
assumed to satisfy (\ref{eq:compare}b). Then $\tT$ depends on $\nu$ only through $\op G\nu$. Furthermore, 
$$(\tT_n)^\sharp (\op G\nu)=\op B_n(\op G\nu);$$ 
thus, if  $\nu$ is Gaussian, equation (\ref{eq:compare}b) 
implies (\ref{eq:compare}a).
$\Diamond$ \end{lemma}

\begin{proof}
We first note that (\ref{eq:compare}b) is equivalent to insisting that
\begin{equation}
\label{eq:Gthis}
\op G\bigl((\tT_n)^\sharp \op G\nu\bigr)=\op B_n(\op G\nu)
\end{equation}
for all measures $\nu;$ this follows because, noting the
definition \eqref{eq:overload} and consequence \eqref{eq:overload2},
the first and second moments of $(\tT_n)^\sharp \op G\nu$ and $(\tT_n)^\sharp \nu$
agree, because of the affine form \eqref{eq:sd2cnadd2}
assumed for $\tT_n$. Recall that $(\tT_n)$ depends on $(\pi,\yd)=\bigl(\law(\hv_{n+1},\hy_{n+1}),\yd_{n+1}\bigr)$.
From the identity \eqref{eq:Gthis},
it is clear that $\tT_n$ only depends on $\nu$
through $\op G\nu$ because changing $\nu \to \op G\nu$ leaves the identity invariant, 
as $\op G \circ \op G=\op G.$ 
%$$\tT (\hv_{n+1},\hy_{n+1};\nu,\yd)=\tT (\hv_{n+1},\hy_{n+1};G\nu,\yd).$$
%Thus
%$$\bigl(\tT (\hv_{n+1},\hy_{n+1};\nu,\yd)\bigr)^\sharp (G\nu)=
%\bigl(\tT (\hv_{n+1},\hy_{n+1};G\nu,\yd)\bigr)^\sharp (G\nu).$$
Now note that, because Gaussians are preserved under affine transformations,
$$\bigl(\tT (\hv_{n+1},\hy_{n+1};\nu,\yd)\bigr)^\sharp (\op G\nu)=
\op G\Bigl(\bigl(\tT (\hv_{n+1},\hy_{n+1};\nu,\yd)\bigr)^\sharp \nu\Bigr),$$
or, in compact notational form,
\begin{equation}
    \label{eq:cnf}
    (\tT_n)^\sharp (\op G\nu)=\op G\bigl((\tT_n)^\sharp \nu\bigr).
\end{equation}
The desired display in the lemma is then immediate from (\ref{eq:compare}b).
\end{proof}

An affine transport map of the form (\ref{eq:sd2cnadd2}), when combined with particle approximations, leads to practical implementable algorithms and achieves (\ref{eq:compare}b) by ensuring that $(\tT_n)^\sharp \nu$ has first and second moments which agree with those of the Gaussian projected filter; these are given by equations \eqref{eq:KF_pred_mean}, \eqref{eq:KF_pred_mean2} and \eqref{eq:KF_analysis} when $\nu$ is the law of $(\hv_{n+1},\hy_{n+1})$.

In Appendix C, Subsection \ref{ssec:mfmsd}, we identify the (uncountable) set of all possible $A,B,a$ which achieve the desired matching of first and
second order moments. Here we focus on the two specific choices given in
Example \ref{ex:twoe} from that Appendix. The first that we highlight corresponds to the choice
\begin{equation*}
\tT (\hv_{n+1},\hy_{n+1};\tnu_{n+1},\yd_{n+1})   := \hv_{n+1}+K_n(\yd_{n+1}-\hy_{n+1}), \label{eq:sd2cnadd}
\end{equation*}
with $K_n= K(\nu_{n+1})$ given by \eqref{eq:Kalman_gain}.
Thus we obtain the following mean field dynamical system,
which corresponds to \eqref{eq:sd2n} in the setting where the Kalman gain $K_n$ 
is defined by \eqref{eq:Kalman_gain}:
\begin{subequations}
\label{eq:sd2nn_add}
\begin{empheq}[box=\widefbox]{align}
\hv_{n+1} &= \Psi(v_n) + \xi_n, \\
\hy_{n+1} &= h(\hv_{n+1}) + \eta_{n+1}, \\
v_{n+1} &= \hv_{n+1}+\hCvy_{n+1}(\hCyy_{n+1})^{-1}(\yd_{n+1}-\hy_{n+1}), 
\end{empheq}
\end{subequations}
where $\{\yd_{n}\}$ arises from a fixed realization of \eqref{eq:sd}
and equations \eqref{eq:KF_pred_mean}, \eqref{eq:KF_pred_mean2} define
the Kalman gain $K_n = \hCvy_{n+1}(\hCyy_{n+1})^{-1}$. We refer to this as {\em Kalman transport}, noting that it serves as a derivation of the Kalman gain, beyond the linear Gaussian setting.
This is a specific instance of the sample path perspective, and
(\ref{eq:recall}, \ref{eq:0recall}) in particular. Again, this 
is a case in which the
sample path evolution for $v_n$ has law which
only approximates the true filtering law.

The second transport map from Example \ref{ex:twoe} corresponds to the choice
\begin{equation*}
\tT (\hv_{n+1},\hy_{n+1};\nu_{n+1},\yd_{n+1})   := m_{n+1}+C_{n+1}^{\frac12}\pCov_{n+1}^{-\frac12}(\hv_{n+1}-\E \hv_{n+1}), \label{eq:sd2cnadd_again}
\end{equation*}
leading to the mean field map \eqref{eq:sd2nn_add_again}, recalling
that $m_{n+1}$, $\pCov_{n+1}$, and $C_{n+1}$ are determined by
\eqref{eq:KF_pred_mean}, \eqref{eq:KF_pred_mean2} and
\eqref{eq:KF_analysis}, using (\ref{eq:sd2nn_add_again}a) and
(\ref{eq:sd2nn_add_again}b).

\begin{remark}
One important difference between the mean field 
models \eqref{eq:sd2nn_add} and \eqref{eq:sd2nn_add_again} 
is that the former involves inversion of matrices
in data space, and the latter in state space. The relative dimensions of the two spaces plays a role in determining which mean field model is more appropriate as the basis of algorithms. A second noteable difference is that the mean field model 
\eqref{eq:sd2nn_add_again} does not require generation 
of the stochastic data $\hy_{n+1}$. This is because 
we may employ the identity $\E \hy_{n+1}= \E h(\hv_{n+1})$ and
use \eqref{eq:KF_joint2b}, \eqref{eq:KF_joint2c} to compute 
$m_{n+1}, C_{n+1}$ and $\pCov_{n+1}$. 
Motivated by this observation, the next subsection studies a wide class 
of approximate transport maps with the property that they do not require
generation of stochastic data.
$\blacksquare$
\end{remark}

%%%%%%%%%%%%%%%%%%%%%%%%%%%%%%%%%%%%%%%%%%
%
\subsubsection{Second Order Transport -- Deterministic Case}
\label{sssec:atd}
%
%%%%%%%%%%%%%%%%%%%%%%%%%%%%%%%%%%%%%%%%%

We now turn our attention to approximate filters defined by deterministic mean field maps. We seek to approximate the exact transport \eqref{eq:mf0} by  mean field maps with the sample path form
\begin{subequations}
\label{eq:mf0-seek}
\begin{empheq}[box=\widefbox]{align}
\hv_{n+1} &= \Psi(v_n) + \xi_n, \\%\: n \in \Z^+,\\
v_{n+1}&=\tTD(\hv_{n+1};\hmu_{n+1},\yd_{n+1}), %\: n \in \Z^+.
\end{empheq}
\end{subequations}
where $\{\yd_{n}\}$ arises from a fixed realization of \eqref{eq:sd}.
As in the previous subsection we drop 
explicit $n$-dependence on the measure $\hmu_{n+1}$ and on the data $\yd_{n+1}$
when no confusion arises from doing so. To this end
we define, for $\hv_{n+1}$ given by (\ref{eq:recall}b), 
$\hmu=\law(\hv_{n+1})$ and $\yd=\yd_{n+1}.$ 
In the following
$\Td_n: \R^{d_v}  \to \R^{d_v}$ is defined by $\Td_n(\cdot)=\Td(\cdot;\hmu,\yd)$
and $\tTD_n: \R^{d_v}  \to \R^{d_v}$ is defined by $\tTD_n(\cdot)=\tTD(\cdot;\hmu,\yd)$; this is
a useful notational convention for the reasons explained in the
stochastic transport setting.

Analogously to the identities \eqref{eq:compare} in the previous subsection, we seek an approximation $\tTD$
which, in comparison with the true transport map $\Td$, satisfies
\begin{subequations}
\label{eq:compare2}
\begin{align}
(\Td_n)^\sharp \hmu&=\op B_n(\op Q\hmu),\\
\op G\bigl((\tTD_n)^\sharp \hmu\bigr)&=\op B_n(\op G\op Q\hmu),
\end{align}
\end{subequations}
for all measures $\hmu$ on the state space $\R^{d_v}.$
As in the previous subsection, where we studied approximate stochastic 
transport, we again seek maps with a specific affine form. Concretely, the maps are assumed to be affine in the pair $\bigl(\hv_{n+1},\hh_{n+1}\bigr)$, with $\hh_{n+1} = h(\hv_{n+1})$, leading to the assumed form
$\tTD_n(\cdot) = \tTD(\cdot;\mu,\yd)$ with
\begin{equation}
\tTD (\hv_{n+1},\hh_{n+1};\mu,\yd)   :=R\hv_{n+1}+S\hh_{n+1}+r, \label{eq:sd2cnadd2d}
\end{equation}
for $(\hmu,\yd)-$dependent matrices/vectors $R,S,r$ of appropriate dimensions. Note, however, that $R,S,r$ are assumed to be independent of the realization
$(\hv_{n+1},\hh_{n+1}$), depending only on its law, so that the transport map is affine in $(\hv_{n+1},\hh_{n+1})$. With this restriction, which will lead to practical implementable algorithms, 
we simply ask that (\ref{eq:compare2}b) holds:
the first and second moments of the output map agree with those of the
Gaussian projected filter, given by equations
\eqref{eq:KF_pred_mean}, \eqref{eq:KF_joint2b} and \eqref{eq:KF_analysis_add}. 

As in the previous subsection, there are uncountably many choices of $R,S,r$ which we identify in Appendix C, Subsection \ref{ssec:mfmnsd}; Example \ref{ex:twoee} highlights two important cases. 
The first coincides with \eqref{eq:sd2nn_add_again} since $S=0$,
but the second leads to a new mean field map. To formulate this
new map we first define
$\tK_{n}=\tK(\hmu)$ by
\begin{equation}
    \label{eq:tidleK}
    \tK_{n}= \pCov^{vh}_{n+1}
    \Bigl( (\pCov^{hh}_{n+1}+\Gamma) + \Gamma^{1/2} (\pCov^{hh}_{n+1}+\Gamma)^{1/2}
    \Bigr)^{-1}.
\end{equation}
We then make the choice
\begin{equation*}
\label{eq:sd2cnadd_again2}
\tTD (\hv_{n+1},\hh_{n+1};\hmu,\yd)   := \hv_{n+1}-\tK_{n}(\hh_{n+1}-\bbE \hh_{n+1})+K_n(\yd-\bbE \hh_{n+1}), 
\end{equation*}
with $K_n$ given by \eqref{eq:KF_joint2c} and repeated here for convenience:
$$K_n = \pCov_{n+1}^{vh}\left( \pCov_{n+1}^{hh} + \Gamma\right)^{-1}.$$
The second mean field map identified in Example \ref{ex:twoee} is 
\begin{subequations}
\label{eq:sd2nn_add_again_add}
\begin{empheq}[box=\widefbox]{align}
\hv_{n+1} &= \Psi(v_n) + \xi_n, \\
\hh_{n+1} &= h(\hv_{n+1}), \\
v_{n+1} &= \hv_{n+1}-\tK_{n}(\hh_{n+1}-\bbE \hh_{n+1})+K_n(\yd_{n+1}-\bbE \hh_{n+1}),
\end{empheq}
\end{subequations}
where $K_n$ and $\tK_n$ are computed under $\law(\hv_{n+1})$.
This is another instance of the sample path perspective, and
(\ref{eq:recall}, \ref{eq:0recall2}) in particular. Again this
sample path evolution for $v_n$ has law which
only approximates the true filtering law.

\begin{remark}
\label{rem:denkf}
If the ensemble spread is such that
the size of $\pCov_{n+1}^{hh}$ is much smaller
than the size of the observational covariance
$\Gamma$ then we may invoke the approximation 
$\pCov_{n+1}^{hh}+ \Gamma \approx \Gamma$. 
With this approximation it follows that  $\widetilde K_{n} \approx \frac{1}{2} K_n$ in
(\ref{eq:tidleK}). Some
deterministic ensemble Kalman filters are derived from
mean field dynamics which exploit this approximation by setting $\widetilde K_n = \frac{1}{2}K_n$ in (\ref{eq:sd2nn_add_again_add}). 
We then replace (\ref{eq:sd2nn_add_again_add}c) by the compact update step
\begin{align} \label{eq:DEnKF}
    v_{n+1} = \hv_{n+1} + K_{n}\left(\yd_{n+1}-\frac{1}{2} \bigl(\mathbb{E}\hh_{n+1} + \hh_{n+1})\bigr)\right).
\end{align}
Such a formulation corresponds to the control-theoretic 
perspective of (\ref{eq:sd2n}), with 
$K_n$ given by \eqref{eq:KF_joint2c} and the innovation (\ref{eq:innovation_p}) replaced by 
\begin{equation}
\label{eq:innovation_d}
\mathfrak{I}_n = \yd_{n+1}-\frac{1}{2}\bigl(\mathbb{E}\hh_{n+1} + \hh_{n+1}\bigr).
\end{equation}
Filters based on this mean field dynamics thus invoke an additional approximation
of perfect transport, over and above that stemming
from matching only first and second moments: they assume 
further that the observational noise dominates ensemble variation.
However we will see that, in the continuous time limit 
described in Section \ref{sec:CT},
this form of the innovation arises naturally and does not
constitute an additional approximation. $\hfill \blacksquare$ 
\end{remark}

%%%%%%%%%%%%%%%%%%%%%%%%%%%%%%%%%%%
%
\subsubsection{Second Order Transport -- Summary}
\label{sssec:sots}
%
%%%%%%%%%%%%%%%%%%%%%%%%%%%%%%%%%%%%%

It is helpful at this point to take stock of two approximations
to filtering that we have introduced, Gaussian projected filtering and
approximate transport, and discuss their inter-relations. For simplicity
we do this in the context of mean field stochastic maps, but similar considerations extend to mean field deterministic maps. In this subsection we also include Example \ref{ex:mfk} demonstrating the existence of mean field maps for the 
Kalman filter which, recall, applies only in the linear Gaussian setting. 

Recall that 
\begin{subequations}
\label{eq:stock1}
\begin{align}
\mu_{n+1}&=\op B_n(\op Q \op P\mu_n),\quad \mu_0=\Ng(m_0,C_0),\\
\mug_{n+1}&=\op B_n(\op G \op Q \op P\mug_n), \quad \mug_0=\Ng(m_0,C_0),
\end{align}
\end{subequations}
With the goal of discussing the inter-relations between
Gaussian projected filtering and approximate transport methods,
we let $\mumf$ denote the measure associated with using the mean field 
map $\tT_n$ to approximate the conditioning step in \eqref{eq:DA}.
Using (\ref{eq:compare}b) to rewrite the Gaussian projected filter, and using
the construction of the stochastic mean field model \eqref{eq:sd2nn-seek},
we obtain
\begin{subequations}
\label{eq:stock2}
\begin{align}
\mug_{n+1}&=\op G\bigl((\tT_n)^\sharp (\op Q\op P\mug_n)\bigr), \quad \mug_0=\Ng(m_0,C_0),\\
\mumf_{n+1}&=(\tT_n)^\sharp (\op Q\op P\mumf_n), \quad \mumf_0=\Ng(m_0,C_0).
\end{align}
\end{subequations}

\begin{remark}
\label{rem:show}
Equations \eqref{eq:stock2}
show that $\{\mumf_n\}$ is close to $\{\mug_n\}$, if the Gaussian projection in
(\ref{eq:stock2}a) is close to the identity where it acts on the output
of onestep. 
Equations \eqref{eq:stock1} show that $\{\mug_n\}$ is close
to $\{\mu_n\}$ if the Gaussian projection in
(\ref{eq:stock1}b) is close to the identity where it acts on the joint
space of state and observation. Together these two facts suggest that
$\{\mumf_n\}$, $\{\mug_n\}$  and  $\{\mu_n\}$ are all close
to one another if the two Gaussian projections can be viewed as being close
to the identity map, where they appear in \eqref{eq:stock1} and in 
\eqref{eq:stock2}. This provides a potential path for analysis of the mean 
field model, away from the linear setting where it is exact. Note also that (\ref{eq:stock2}a) shows that the Gaussian
projected filter evolves within the manifold of Gaussian probability measures; the mean field model (\ref{eq:stock2}b) does not.
\end{remark}

\begin{example}
\label{ex:mfk}
Assume that $v_0 \sim \Ng(m_0,C_0),$  that $\Gamma \succ 0$ and
consider the Kalman filter setting of  Example \ref{ex:sssec:2}; in particular \eqref{eq:linearsd} prevails rendering $\Psi$ and $h$ linear. The mean field stochastic
dynamical system \eqref{eq:sd2nn_add} then takes the form
\begin{subequations}
\label{eq:sd2nn_sub}
\begin{align}
\hv_{n+1} &= Mv_n + \xi_n, \\
\hy_{n+1} &= H\hv_{n+1} + \eta_{n+1}, \\
v_{n+1} &= \hv_{n+1}+\pCov_{n+1}H^\top(H\pCov_{n+1}H^\top+\Gamma)^{-1}(\yd_{n+1}-\hy_{n+1}),
\end{align}
\end{subequations}
where $\{\yd_{n}\}$ arises from a fixed realization of \eqref{eq:sdl} and
where $C_n$ is the covariance of $v_n$ and $\pCov_{n+1} = M \Cov_{n}M^\top+\Sigma$ is the covariance of $\hv_{n+1}.$
The resulting dynamics give a sample path representation of the Kalman filter in that $v_n \sim \Ng(m_n,C_n)$ where $m_n, C_n$ are as given in Example \ref{ex:sssec:2}.
This follows because the map defined by \eqref{eq:sd2nn_sub} is well-defined,
since $\Gamma \succ 0.$ Lemma \ref{lem:mot1} shows that the approximate 
transport is exact in this Gaussian setting.

Similar ideas can be applied to \eqref{eq:sd2nn_add_again}
and \eqref{eq:sd2nn_add_again_add} to determine other mean field models with law equal to that of the Kalman filter. Furthermore we observe that the formulation  based on \eqref{eq:sd2nn_add_again} can be symmetrized to obtain, in the linear Gaussian setting
of \eqref{eq:linearsd},
\begin{subequations}
\label{eq:ESRFcite}
\begin{align}
\hv_{n+1} &= Mv_n + \xi_n, \: n \in \Z^+,  \\
v_{n+1} &= m_{n+1} + A_n(\hv_{n+1}-\pmean_{n+1}),\\
    A_{n} &= (\Cov_{n+1})^{1/2}\left[(\Cov_{n+1})^{1/2} \pCov_{n+1}
    (\Cov_{n+1})^{1/2}\right]^{-1/2}(\Cov_{n+1})^{1/2},
\end{align}
\end{subequations}
with 
$(\pmean_{n+1},\pCov_{n+1}, \pCov_{n+1}^{vh}, \pCov_{n+1}^{hh})$ 
defined by \eqref{eq:KF_pred_mean}, \eqref{eq:KF_joint2b} 
and $(m_{n+1},\Cov_{n+1})$ defined by \eqref{eq:KF_analysis}.
We note that the second component of the map may be written in gradient
form and corresponds to an optimal transport from $\hmu_{n+1}$ into $\mu_{n+1}$ in the sense of the Euclidean Wasserstein distance of optimal transportation (see Subsection \ref{ssec:BSE} for details); indeed this is true of for an entire family of
weighted Wasserstein distances -- see Example \ref{ex:cite}.
To recognize the gradient structure define
$$\Phi_n(v):= \langle m_{n+1}, v \rangle+\frac12 \langle A_n(v-\pmean_{n+1}), v-\pmean_{n+1}\rangle$$
and note that then
\begin{subequations}
\label{eq:needt}
\begin{align}
\hv_{n+1} &= Mv_n + \xi_n, \: n \in \Z^+,  \\
v_{n+1}&=\nabla \Phi_n(\hv_{n+1}).
\end{align}
\end{subequations}
$\blacksquare$
\end{example}

%%%%%%%%%%%%%%%%%%%%%%%%%%%%%%%%%%%%%%%%%%%%%%%%%%%%%%%%%%%
%
\subsection{Ensemble Kalman Methods}
\label{ssec:EKM}
%
%%%%%%%%%%%%%%%%%%%%%%%%%%%%%%%%%%%%%%%%%%%%%%%%%%%%%%%%%%%

The mean field formulations from Subsection \ref{ssec:MFM} provide clear insights into many of the design choices and mechanisms that underlie ensemble Kalman methods.
In this subsection we take the mean field models 
and use particle approximations to derive implementable numerical algorithms. When approximated by interacting particle systems, the mean field formulations of ensemble Kalman methods
lead to actionable algorithms.

We start, in Subsection \ref{sssec:ppf}, with the setting in which transport
is perfect; these algorithms are not, in general, implementable since determining perfect transport is itself a difficult computational task and the subject of ongoing research; see the bibliography Subsection \ref{ssec:BSE}. Thus we turn to particle approximations of the transports designed to match first and second order statistics. This leads to the (stochastic) {\em ensemble Kalman filter} in Subsection \ref{sssec:senkf}, and to (deterministic) {\em ensemble square root filters} in
Subsection \ref{sssec:esrf}. The methods derived in this subsection involve approximating the Kalman gain by computing covariances under the empirical measure defined by the ensemble of particles. To avoid overloading notation, in the rest of this subsection $C_{n+1}$, $\pCov_{n+1}$, $\pCov_{n+1}^{vy}$ and $\pCov_{n+1}^{yy}$ will denote covariances computed with expectation under the empirical measure. With this notation in place, in the rest of this specific subsection, $K_n$ will directly refer to the particle approximation of the Kalman gain, computed using the covariances with respect to the empirical measure, without further specification needed. Throughout this section $\{\yd_{n}\}$ arises from a fixed realization of \eqref{eq:sd}.

%%%%%%%%%%%%%%%%%%%%%%%%%%%%%%%%%%%%%%%%%%%%%%%%%%%%%%
%
\subsubsection{Perfect Particle Filters}
\label{sssec:ppf}
%
%%%%%%%%%%%%%%%%%%%%%%%%%%%%%%%%%%%%%%%%%%%%%%%%%%%%%

The mean field equations  \eqref{eq:sd2nn} could, in principle, be
approximated through a particle approximation of the mean field
leading to the following conceptual (because map $\Ts_n$ is not known explicitly) algorithm: let $\sJ=\{1,\cdots, J\}$ and consider,  
for $(n,j) \in \Z^+ \times \sJ$, the 
interacting particle dynamical system
\begin{subequations}
\label{eq:mf0nz}
\begin{empheq}[box=\widefbox]{align}
\hv_{n+1}^{(j)} &= \Psi(v_n^{(j}) + \xi_n^{(j)}, \\
\hy_{n+1}^{(j)} &= h(\hv_{n+1}^{(j}) + \eta_{n+1}^{(j)}, \\
v_{n+1}^{(j)}&=\Ts(\hv_{n+1}^{(j)},\hy_{n+1}^{(j)};\nu_{n+1}^\sJ,\yd_{n+1}),\\
\nu_{n+1}^\sJ& =\frac{1}{J}\sum_{j=1}^J \delta_{(\hv_{n+1}^{(j)},\hy_{n+1}^{(j)})}.
\end{empheq}
\end{subequations}
This evolves the particles $\{\vj_{n}\}_{j \in \sJ}$ into 
$\{\vj_{n+1}\}_{j \in \sJ}$.  Here the $\{\xij_n\}$  are, for each $j$,
random variables given by the known distribution  of $\xi_n$
specified in \eqref{eq:ga} and, furthermore, are drawn independently 
with respect to each $(n,j) \in \Z^+ \times \sJ$. 
Similar considerations apply to the $\{\etj_n\}$ which, additionally,
are independent of the $\{\xij_n\}$.
It is intuitive that the large $J$ limit of this
system recovers the mean field dynamics \eqref{eq:sd2nn} and, in particular, 
 \begin{equation}
    \label{eq:mua}
\mu_{n}^\sJ =\frac{1}{J}\sum_{j=1}^J \delta_{v_{n}^{(j)}} \approx \mu_n.
\end{equation}

Applying a similar idea to
\eqref{eq:mf0} leads to the following conceptual (because map $\Td_n$ is not known explicitly) algorithm. 
Consider, for $(n,j) \in \Z^+ \times \sJ$,
the interacting particle dynamical system
\begin{subequations}
\label{eq:mf0n}
\begin{empheq}[box=\widefbox]{align}
\hv_{n+1}^{(j)} &= \Psi(v_n^{(j}) + \xi_n^{(j)}, \\
v_{n+1}^{(j)}&=\Td(\hv_{n+1}^{(j)};\hmu_{n+1}^\sJ,\yd_{n+1}),\\
\hmu_{n+1}^\sJ& =\frac{1}{J}\sum_{j=1}^J \delta_{\hv_{n+1}^{(j)}}.
\end{empheq}
\end{subequations}
This evolves the particles $\{\vj_{n}\}_{j \in \sJ}$ into 
$\{\vj_{n+1}\}_{j \in \sJ}$. 
The same assumptions are made about the $\{\xij_n\}$  
as for the preceding interacting particle dynamical system.
It is again intuitive that the large $J$ limit of this
system recovers the mean field dynamics \eqref{eq:mf0}, and the evolution
\eqref{eq:pmfd}. In particular it is intuitive that
\eqref{eq:mua} holds for this particle approximation too. 
We reiterate that in practice these algorithms are, in general, not easy 
to use. More specifically, finding particle based approximations $\TsJ$ and $\TdJ$ to the desired transport maps such that
\begin{equation*} \label{eq:mean field limit}
    \lim_{J \to \infty} \TsJ = \Ts, \qquad 
    \lim_{J \to \infty} \TdJ = \Td
\end{equation*}
in an appropriate sense is a computationally challenging task and the subject of ongoing research. This leads to the next two subsections in
which we replace $\Ts$ and $\Td$, in the interacting particles systems \eqref{eq:mf0nz} and \eqref{eq:mf0n}, by the previously introduced affine approximate transports $\tT$ and $\tTD,$ respectively. 

%%%%%%%%%%%%%%%%%%%%%%%%%%%%%%%%%%%%%%%%%%%%%%%%%%%%%%%%
%
\subsubsection{Stochastic Ensemble Kalman Filters}
\label{sssec:senkf}
%
%%%%%%%%%%%%%%%%%%%%%%%%%%%%%%%%%%%%%%%%%%%%%%%%%%%%%%%

Particle approximation of the mean field dynamical system \eqref{eq:sd2nn_add},
effecting Kalman transport, bring us to the stochastic EnKF
({\em ensemble Kalman filter}). This method may be
derived by writing down a particle approximation of
the mean field stochastic dynamics defined by 
\eqref{eq:sd2nn_add}.  
We evolve the particles $\{\vj_{n}\}_{j \in \sJ}$ into 
$\{\vj_{n+1}\}_{j \in \sJ}$ according to the following
stochastic interacting particle system, holding 
for $(n,j) \in \Z^+ \times \sJ$:
\begin{subequations}
\label{eq:sd4e}
\begin{empheq}[box=\widefbox]{align}
\hvj_{n+1} &= \Psi(\vj_n) + \xij_n, \: n \in \Z^+,  \\
\hyj_{n+1} &= h(\hvj_{n+1}) + \etj_{n+1} , \: n \in \Z^+,\\
\vj_{n+1} &= \hvj_{n+1}+K_n\bigl(\yd_{n+1}-\hyj_{n+1}\bigr),\\
\nu_{n+1}^\sJ& =\frac{1}{J}\sum_{j=1}^J \delta_{(\hv_{n+1}^{(j)},\hy_{n+1}^{(j)})}.
\end{empheq}
\end{subequations}
Here the Kalman gain from \eqref{eq:Kalman_gain} is approximated using the empirical measure $\nu_{n+1}^\sJ$, but is still denoted by $K_n$ to avoid proliferation of notation; 
details follow below. The same assumptions regarding 
$\{\xij_n\}$ and $\{\etj_{n+1}\}$ 
are made as for \eqref{eq:mf0nz}. We let $\E^\sJ_n$ denote expectation under $\nu_n^\sJ.$
For the basic implementation of EnKF \eqref{eq:sd4e} the desired covariance matrices, and
Kalman gain \eqref{eq:Kalman_gain}, are then approximated by expectation under 
$\nu_{n+1}^\sJ$, so that
\footnote{The empirical covariance computations are often
modified to accommodate the widely adopted convention
of scaling by $1/(J-1)$, instead of $1/J$, in view of the matrix being computed from $J-1$
independent increments about the mean.}
\begin{align*}
%\label{eq:KF_joint33}
     \pCov_{n+1}^{vy} &= \, \E^\sJ_{n+1}\Bigl(\bigl(\hv_{n+1}-\E^\sJ_{n+1} \hv_{n+1}\bigr)\otimes
\bigl(\hy_{n+1}-\E^\sJ_{n+1} \hy_{n+1}\bigr)\Bigr),\\
    \pCov_{n+1}^{yy} &= \, \E^\sJ_{n+1}\Bigl(\bigl(\hy_{n+1}-\E^\sJ_{n+1} \hy_{n+1}\bigr)\otimes
\bigl(\hy_{n+1}-\E^\sJ_{n+1} \hy_{n+1}\bigr)\Bigr),\\
K_n &= \pCov_{n+1}^{vy} \bigl(\pCov_{n+1}^{yy}\bigr)^{-1}.
\end{align*}
Note that a pseudo-inverse may be required to define $K_n$. An alternative, avoiding pseudo-inverse, is to use a particle approximation in formula
(\ref{eq:KF_joint2c}b), leading to the use of
\begin{subequations}
  \label{eq:KF_joint33c}  
\begin{align}
     \pCov_{n+1}^{vh} &=    \E_{n+1}^\sJ\Bigl(\bigl(\hv_{n+1}-\E^\sJ_{n+1} \hv_{n+1}\bigr)\otimes
\bigl(h(\hv_{n+1})-\E_{n+1}^\sJ h(\hv_{n+1})\bigr)\Bigr),\\
    \pCov_{n+1}^{hh} &=  \E_{n+1}^\sJ\Bigl(\bigl(h(\hv_{n+1})-\E^\sJ_{n+1} h(\hv_{n+1})\bigr)\otimes
\bigl(h(\hv_{n+1})-\E_{n+1}^\sJ h(\hv_{n+1})\bigr)\Bigr),\\
K_n &= \pCov_{n+1}^{vh} \bigl(\pCov_{n+1}^{hh}+\Gamma\bigr)^{-1}.
\end{align}
\end{subequations}
to compute the gain $K_n.$
The advantage of this latter formulation is that it ensures positivity, and hence invertibility, of the covariance in data space, if $\Gamma$ is assumed positive-definite. It is hence typically preferred.

Pseudo-code for the stochastic EnKF may be found as 
Algorithm \ref{alg:EnKF} in Appendix \ref{sec:AA}.

\begin{example}
\label{ex:enkf}
We return to the set-up of Example \ref{ex:3dvar}, and now
demonstrate performance of the stochastic EnKF on the same Lorenz '96 model. Indeed, we again study the Lorenz '96 (singlescale) model  for unknown
$v \in C(\R^+,\R^L)$ satisfying the equations \eqref{eq:l96} with $L = 9, h_v=-0.8$ and $F=10$
and function $m$ as shown in Figure \ref{fig:multiscale_m}. We consider observations $\{\yd_n\}_{n \in \Z^+}$ arising from the model 
\begin{align*}
\label{eq:singlescale_experiment_dynamics_noisy}
\vd_{n+1} &= \Psi
(\vd_n) + \xid_n, \\
\yd_{n+1} &= h(\vd_{n+1}) + \etad_{n+1},
\end{align*}
where $\Psi$ is the solution operator for \eqref{eq:l96} over the observation time interval $\tau$, and $\{\xid_n\}_{n \in \Z^+}$, $\{\etad_n\}_{n \in \N}$ are mutually independent Gaussian sequences defined by 
\begin{equation*}
\label{eq:ga_3}
\xid_n \sim \Ng(0, \sigma^2 I) \,\, \text{i.i.d.}\,, \quad
\etad_n \sim \Ng(0, \gamma^2 I) \,\, \text{i.i.d.}\,,
\end{equation*}
with $\sigma = 0.1$ and $\gamma = 0.1$. We again assume that the observation function is linear: $h(v)=Hv$ for matrix $H:\R^9 \to \R^6$ defined by \eqref{eq:L96H}. 

Figures \ref{fig:enkf_100} and \ref{fig:enkf_1000} demonstrate the performance of stochastic EnKF in this experimental setting with $\tau = 10^{-3}$ and using $J=10^2$ and $J=5\cdot 10^2$, respectively, against the performance of 3DVAR with no noise; note that the EnKF uses a time-varying estimate of the gain $K_n$, whilst 3DVAR uses the fixed $K$ given in Example \ref{ex:3dvar}. These experiments illustrate that using sufficiently large ensembles, the ensemble Kalman filter outperforms 3DVAR on such a nonlinear filtering problem where the true state and observational noise levels are high. 
Here \emph{outperforms} refers to mean square error in recovery
of the state. To quantitatively demonstrate this improvement, we compute: (a) the mean squared error between the estimates yielded by 3DVAR and the true states; and (b) the mean squared error between the ensemble mean of stochastic EnKF and the true states. In particular we report  time-averaged  mean squared errors obtained from both 3DVAR and stochastic EnKF given by use of formula \eqref{eq::MSE_error} from Example \ref{ex:s3dvar} using $t^*=3$ and $T=10$. An ensemble size of $J=10^2$ yields $e_{\text{EnKF}} = 1.05\cdot 10^{0}$, while for $J=5\cdot 10^2$ we obtain $e_{\text{EnKF}} = 5.24\cdot 10^{-1}$. For comparison, the error obtained using 3DVAR is $e_{\text{3DVAR}} = 1.85\cdot 10^{0}$.
$\blacksquare$
\end{example}

\begin{figure}[h!]
\centering
\begin{subfigure}{\textwidth}
  \centering
  \includegraphics[width=1\linewidth]{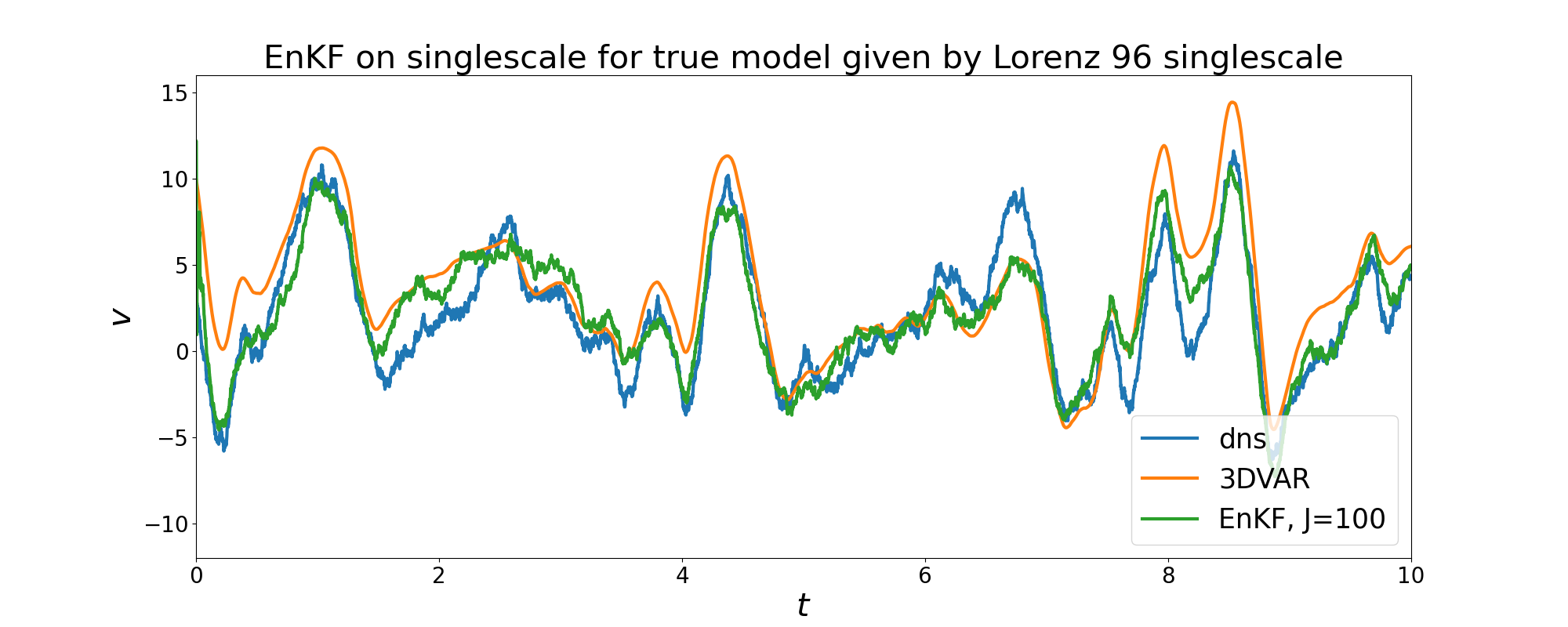}
  \caption{EnKF with ensemble size $J=10^2$.}
  \label{fig:enkf_100}
\end{subfigure}
\begin{subfigure}{\textwidth}
  \centering
  \includegraphics[width=1\linewidth]{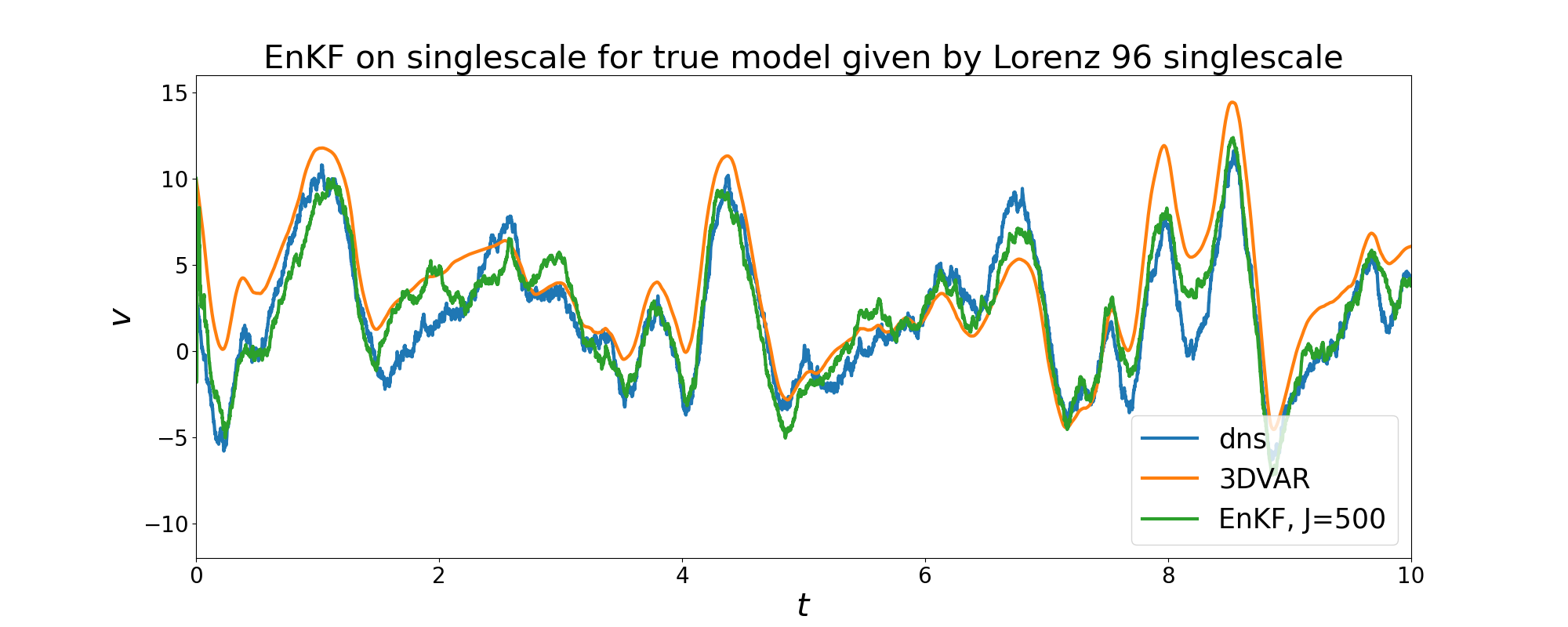}
  \caption{EnKF with ensemble size $J=5\cdot 10^2$.}
  \label{fig:enkf_1000}
\end{subfigure}
\caption{In this experiment we set the noise levels $\sigma = 10^{-1}, \gamma = 10^{-1}$. We display the estimates of $v_3$ in time produced by EnKF (using ensemble average) and 3DVAR against the true dynamics using observation time interval $\tau = 10^{-3}$. Again ``dns'' refers to direct numerical simulation. The results show that the EnKF provides a more accurate estimate of the trajectory, albeit at higher cost in terms of number of model evaluations.}
\label{fig:enkf}
\end{figure}

%%%%%%%%%%%%%%%%%%%%%%%%%%%%%%%%%%%%%%%%%%%%%%%
%
\subsubsection{Ensemble Square Root Filters}
\label{sssec:esrf}
%
%%%%%%%%%%%%%%%%%%%%%%%%%%%%%%%%%%%%%%%%%%%%

The variants on the EnKF described in this subsection
are known as ensemble square root filters; they are based on
mean field maps \eqref{eq:sd2nn_add_again} and \eqref{eq:sd2nn_add_again_add},
approximated by interacting particle systems.

\begin{remark}
\label{rem:imp}
Square root filters are sometimes referred to as deterministic ensemble Kalman filters, to distinguish them from the ensemble Kalman filters described in the preceding subsection (see discussion of this,
and bibliographic information, in Subsection \ref{ssec:HC}.)
However we have already used the terminology
``stochastic'' and ``deterministic'' to distinguish between
different variants on the mean field models that we
describe in Subsections \ref{sssec:ats} and \ref{sssec:atd} 
respectively. With the exception of the discussion in Subsection \ref{ssec:HC}, we simply refer to square root filters
for the methods introduced in this subsection. We note however that
they are deterministic in the sense that they do not require generation of random variables.

We introduce two families of square root filters: of adjustment and transform type.
We emphasize again that the choice of which method to use in practice
is determined by implementation details such as the number of particles $J$, the dimension of state space $d_v$, and the dimension of the observation space $d_y$. These implementation details; although important, are not the focus of this paper.
$\blacksquare$
\end{remark}

%%%%%%%%%%%%%%%%%%%%%%%%%%%%%%%%%%%%%%%%%%%%%%%%%%%%%%
\vspace{0.1in}
\noindent\paragraph{Ensemble Adjustment Kalman Filters}
\vspace{0.1in}
%%%%%%%%%%%%%%%%%%%%%%%%%%%%%%%%%%%%%%%%%%%%%%%%%%%%

We introduce two different particle-based approximations
of mean field models. Both methods are examples of a general class of algorithms known as {\em ensemble adjustment Kalman filters}: EAKF.
The starting point for the first of these EAKF methods is the mean field map
(\ref{eq:sd2nn_add_again}a),  (\ref{eq:sd2nn_add_again}c), 
repeated
here for convenience:
\begin{subequations}
\label{eq:sd2nn_add_again2}
\begin{empheq}[box=\widefbox]{align}
\hv_{n+1} &= \Psi(v_n) + \xi_n, \\
v_{n+1} &= m_{n+1}+C_{n+1}^{\frac12}\pCov_{n+1}^{-\frac12}(\hv_{n+1}-\E \hv_{n+1}),
\end{empheq}
\end{subequations}
where $m_{n+1}$, $\pCov_{n+1}$, and $C_{n+1}$ are determined by
\eqref{eq:KF_pred_mean}, \eqref{eq:KF_joint2b}, \eqref{eq:KF_joint2c} and \eqref{eq:KF_analysis}.\footnote{Although the identity
\eqref{eq:KF_analysis} is derived in a subsection concerning the
Gaussian projected filter, it is contained in Lemma \ref{lemma:KF_analysis} which simply concerns conditioning of Gaussians.}

As in the previous subsection the methods evolve
particle ensemble $\{\vj_{n}\}_{j \in \sJ}$ into 
$\{\vj_{n+1}\}_{j \in \sJ}$, via the predictive ensemble
$\{\hvj_{n+1}\}_{j \in \sJ}.$ However they do not employ simulated data
$\{\hyj_{n+1}\}_{j \in \sJ}$, rather they make use of
$\{\hhj_{n+1}\}_{j \in \sJ}$, where $\hhj_n=h(\hvj_n).$
To define the methods it helps to introduce new
notation. Slightly modifying the notation in the preceding subsection,
we now let $\E^\sJ_n$ denote expectation with respect to the empirical measure
\begin{equation*}
\label{eq:empm}
\widehat \mu_{n}^\sJ = \frac{1}{J}\sum_{j=1}^J \delta_{\hvj_{n}}.
\end{equation*}
We let $\hv_{n+1}$ denote the random variable with this distribution and,
as before, $\hh_{n+1}=h(\hv_{n+1}).$

Next we define matrix $\widehat V_{n}$ comprising 
scaled ensemble deviations in state space:
\begin{equation*}
\widehat V_{n} = \frac{1}{\sqrt{J}}\left( \hv_n^{(1)}-\E^\sJ_n \hv_n,
\hv_{n}^{(2)}-\E^\sJ_n \hv_{n}, \cdots,
\hv_{n}^{(J)}-\E^\sJ_n \hv_{n} \right) \in \mathbb{R}^{d_v \times J},
\end{equation*}
we then define the analogous matrix $\widehat H_{n}$ in observation space
\begin{equation*}
\widehat H_{n} = \frac{1}{\sqrt{J}}\left( \hh_n^{(1)}-\E^\sJ_n \hh_n,
\hh_{n}^{(2)}-\E^\sJ_n \hh_{n}, \cdots,
\hh_{n}^{(J)}-\E^\sJ_n \hh_{n} \right) \in \mathbb{R}^{d_y \times J}.
\end{equation*}

With these notations in hand,
we have, with expectations computed under $\bbE^\sJ_n$,
\begin{subequations}
\label{eq:similar}
\begin{align}
\pCov^{vh}_n &= \widehat V_n \widehat H_n^\top,\\
\pCov^{hh}_n &= \widehat H_n \widehat H_n^\top,
\end{align}
\end{subequations}
and the ensemble-based approximation of the Kalman gain matrix is
\footnote{Here too, the empirical covariance computations are often
modified to accommodate the widely adopted convention
of scaling by $1/(J-1)$, instead of $1/J$.} 
%\footnote{Here we have used the widely adopted convention
%of computing empirical covariance matrices and then rescaling
%by $J/(J-1)$, in view of the matrix being computed from $J-1$
%independent increments about the mean. If this convention is not
%adopted then the scaling of $\widehat V_n$, $\widehat \Yd_n$ under
%the square root is by
%$J$ rather than $J-1$.}
\begin{equation} \label{eq:KJ}
K_n = \widehat V_{n+1} \widehat H_{n+1}^\top(\widehat H_{n+1} \widehat H_{n+1}^\top + \Gamma)^{-1}.
\end{equation}
This is a linear algebraic reformulation of the Kalman gain approximation resulting from
\eqref{eq:KF_joint33c}. 
By making a particle approximation of \eqref{eq:sd2nn_add_again2},
using \eqref{eq:KF_pred_mean}, \eqref{eq:KF_joint2b}, \eqref{eq:KF_joint2c} and \eqref{eq:KF_analysis}, we obtain
\begin{subequations}
\label{eq:left_transform}
\begin{empheq}[box=\widefbox]{align}
\hvj_{n+1} &= \Psi(\vj_n) + \xij_n, \: n \in \Z^+,  \\
\widehat m_{n+1} &= \E^\sJ_{n+1} \hv_{n+1},\\
m_{n+1} &=  \widehat m_{n+1}+ K_n (\yd_{n+1}-\E^\sJ_{n+1} \hh_{n+1}),\\
v_{n+1}^{(j)} &= m_{n+1} + C_{n+1}^{\frac12}\pCov_{n+1}^{-\frac12}\left(\hv_{n+1}^{(j)}- \widehat m_{n+1}\right),
\end{empheq}
\end{subequations}
where $K_n$ is given by \eqref{eq:KJ}; furthermore,
$\pCov_{n+1}$ and $C_{n+1}$ are computed empirically using
\begin{equation}
\label{eq:emprical_post_covariance}   
\pCov_{n+1} = \widehat V_{n+1} \widehat V_{n+1}^\top, \quad \Cov_{n+1}=\widehat V_{n+1} \left( I + \widehat H_{n+1}^\top \Gamma^{-1} \widehat H_{n+1} \right)^{-1} \widehat V_{n+1}^\top.
\end{equation}
The first of these two formulae follows similarly to \eqref{eq:similar}; 
the second of these two formulae is derived as follows:
\footnote{Using, in the last line, the identity $I-W^\top(WW^\top+I)^{-1}W=(I+W^\top W)^{-1}$ which holds for all (not necessarily square) matrices $W$.}
\begin{subequations}
\label{eq:left_transform2}
\begin{align}
\Cov_{n+1} &= \pCov_{n+1}-\pCov^{vh}_{n+1}(\pCov^{hh}_{n+1}+\Gamma)^{-1}(\pCov^{vh}_{n+1})^\top\\
&=\widehat V_{n+1} \widehat V_{n+1}^\top - \widehat V_{n+1} \widehat H_{n+1}^\top(\widehat H_{n+1} \widehat H_{n+1}^\top + \Gamma )^{-1} \widehat H_{n+1} \widehat V_{n+1}^\top\\
&= \widehat V_{n+1} \left( I - \widehat H_{n+1}^\top(\widehat H_{n+1} \widehat H_{n+1}^\top + \Gamma )^{-1} \widehat H_{n+1} \right) \widehat V_{n+1}^\top\\
&= \widehat V_{n+1} \left( I + \widehat H_{n+1}^\top \Gamma^{-1} \widehat H_{n+1} \right)^{-1} \widehat V_{n+1}^\top.
\end{align}
\end{subequations}

The  EAKF \eqref{eq:left_transform} takes as starting point \eqref{eq:sd2nn_add_again}. If instead we apply a particle approximation to the mean field dynamical
system \eqref{eq:sd2nn_add_again_add} we obtain a second version of the EAKF:
\begin{subequations}
\label{eq:left_transform_add}
\begin{empheq}[box=\widefbox]{align}
\hvj_{n+1} &= \Psi(\vj_n) + \xij_n, \: n \in \Z^+,  \\
\hh_{n+1}^{(j)} &= h(\hv_{n+1}^{(j)}),\\
\widehat m_{n+1} &= \E^\sJ_{n+1} \hv_{n+1},\\
m_{n+1} &=  \widehat m_{n+1}+ K_n (\yd_{n+1}-\E^\sJ_{n+1} \hh_{n+1}),\\
v_{n+1}^{(j)} &= m_{n+1} + (\hv_{n+1}^{(j)}-\widehat m_{n+1}) -
\widetilde K_n (\hh_{n+1}^{(j)}-\E^\sJ_{n+1} \hh_{n+1}).
%v_{n+1}^{(j)} &= m_{n+1} + \left(\hv_{n+1}^{(j)}- \widehat m_{n+1}\right)
%+ S_n^{\rm L} \left(\hy_{n+1}^{(j)}- \E^\sJ \hh_{n+1} \right).
\end{empheq}
\end{subequations}
By making an empirical approximation of the formula \eqref{eq:tidleK}
the matrix $\widetilde K_n$ is defined using the identification
\begin{equation*}
\widetilde K_n = \widehat V_{n+1} \widehat H_{n+1}^\top \left[
(\widehat H_{n+1} \widehat H_{n+1}^\top + \Gamma) + (\widehat H_{n+1} \widehat H_{n+1}^\top + \Gamma )^{1/2} \Gamma^{1/2} \right]^{-1}.
\end{equation*}

\begin{remark}
The key difference between \eqref{eq:left_transform} and \eqref{eq:left_transform_add}  is that the former involves inversion in state space, and the latter in data space. The relative size of the two dimensions dictates which is preferable. $\blacksquare$
\end{remark}

%%%%%%%%%%%%%%%%%%%%%%%%%%%%%%%%%%%%%%%%%%%%%%%%%%%%%%%%%
\vspace{0.1in}
\noindent\paragraph{Ensemble Transform Kalman Filters}
\vspace{0.1in}
%%%%%%%%%%%%%%%%%%%%%%%%%%%%%%%%%%%%%%%%%%%%%%%%%%%%%

The two EAKFs just defined both involve application, and inversion, 
of matrices which are applied on the left and act on state space. 
A different class of
algorithms, known as {\em ensemble transform Kalman filters} (ETKF), involve matrix multiplication
from the right and consequently inversions take place in the ensemble space of dimension $J$. In many applications this is far smaller than the dimension of the state or data spaces, and then use of this version of the methodology is preferred. The aim is to determine matrix $Z_n \in \mathbb{R}^{J\times J}$, and to derive Kalman gain $K_n$ from $Z_n$, so that the following interacting particle system produces an ensemble of particles $\{\vj_{n+1}\}_{j \in \sJ}$ with empirical covariance
$C_{n+1}$ defined by the second item in display \eqref{eq:emprical_post_covariance}:
\begin{subequations}
\label{eq:right_transform}
\begin{empheq}[box=\widefbox]{align}
\hvj_{n+1} &= \Psi(\vj_n) + \xij_n, \: n \in \Z^+,  \\
\widehat m_{n+1} &= \E^\sJ_{n+1} \hv_{n+1},\\
m_{n+1} &=  \widehat m_{n+1}+ K_n (\yd_{n+1}-\E^\sJ_{n+1} \hh_{n+1}),\\
v_{n+1}^{(j)} &= m_{n+1} + \sum_{i=1}^J \left(\hv_{n+1}^{(i)}-\widehat m_{n+1}\right) (Z_{n})_{ij}.
%v_{n+1}^{(j)} &= m_{n+1} + C_{n+1}^{\frac12}\pCov_{n+1}^{-\frac12}\left(\hv_{n+1}^{(j)}- \widehat m_{n+1}\right),
\end{empheq}
\end{subequations}

To this end we define the matrix of ensemble deviations
\begin{equation*}
V_{n} = \frac{1}{\sqrt{J}}\left( v_n^{(1)}-\E_{*,n}^\sJ v_n,
v_{n}^{(2)}-\E_{*,n}^\sJ v_{n}, \cdots,
v_{n}^{(J)}-\E_{*,n}^\sJ v_{n} \right) \in \mathbb{R}^{d_v \times J},
\end{equation*}
where, here, expectation $\E_{*,n}^\sJ$ is with respect to the empirical measure (\ref{eq:mua})
and $v_n$ is a random variable with this distribution. It then follows from (\ref{eq:right_transform}d) that
\begin{equation} \label{eq:transform ensemble deviations}
V_{n+1} = \widehat{V}_{n+1} Z_n.
\end{equation}
If we define
\begin{equation} \label{eq:ETKF_matrix}
Z_n = \left( I + \widehat H_{n+1}^\top \Gamma^{-1} \widehat H_{n+1} \right)^{-1/2} \in \bbR^{J \times J}.
\end{equation}
then, as desired, $V_{n+1}V_{n+1}^\top = \Cov_{n+1}$ as defined by (\ref{eq:emprical_post_covariance}), by virtue of \eqref{eq:left_transform2}.
The calculations in \eqref{eq:left_transform2} can also be utilized to verify that
the empirical Kalman gain matrix defined by (\ref{eq:KJ}) satisfies
\begin{equation*}
K_n=\widehat V_{n+1} Z_n^2 \widehat H_{n+1}^\top \Gamma^{-1}.
\end{equation*}
Using this formula for $K_n$ in \eqref{eq:left_transform} leads to an algorithm which matches first and second order statistics of the Gaussian projected filter, at $n+1$, requiring only matrix inversions in space of dimension defined by the number of particles $J$.

We finally note that (\ref{eq:right_transform}c) and
(\ref{eq:right_transform}d) can be combined into a single transformation step of the form
\begin{equation} \label{eq:linear transform filter}
    v_{n+1}^{(j)} = \sum_{i=1}^J \hv_{n+1}^{(i)} (S_n)_{ij}\,,
\end{equation}
where $S_n \in \mathbb{R}^{J\times J}$ replaces the matrix $Z_n$
in (\ref{eq:right_transform}d) such that
\begin{equation*}
    m_{n+1} = \sum_{j=1}^J v_{n+1}^{(j)} = 
    \sum_{i,j=1}^J \hv_{n+1}^{(i)} (S_n)_{ij}
\end{equation*}
holds in addition to (\ref{eq:transform ensemble deviations}) with $Z_n$ replaced by $S_n$. 

\begin{remark}
Formulation (\ref{eq:linear transform filter}) has a number of attractive features. First, it clearly reveals that the analysis $\{\vj_{n+1}\}_{j\in \sJ}$ lies in the span of the space spanned by the predictions $\{\hvj_{n+1}\}_{j\in \sJ}$, which is relevant whenever $J < d_v$. Second, all particle implementations of the ensemble Kalman filter and many of its extensions can be put into the framework (\ref{eq:linear transform filter}) with the (possibly random) 
matrix $S_n$ chosen appropriately. Third, it encodes a coupling between the prediction $\{\hvj_{n+1}\}_{j\in \sJ}$ and the analysis
$\{\vj_{n+1}\}_{j\in \sJ}$ at the level of their associated empirical measures $\widehat \mu_n^\sJ$ and $\mu_n^\sJ$, respectively.
See the following bibliographic Section \ref{ssec:BSE} for more details.
$\blacksquare$ \end{remark}

%%%%%%%%%%%%%%%%%%%%%%%%%%%%%%%%%%%%%%%%%%%%%%%%%
%
\subsection{Bibliographical Notes}
\label{ssec:BSE}
%
%%%%%%%%%%%%%%%%%%%%%%%%%%%%%%%%%%%%%%%%%%%%%%%%%

Ideas from feedback control underlie the material in
Subsection \ref{ssec:CT}, addressing Objective 1. 
Control theory is an enormous subject in its
own right and we cannot do justice to it in this paper. For study of linear control theory, as illustrated in Example \ref{ex:control}, see \citet{aastrom2021feedback,sontag2013mathematical} for engineering and mathematical treatments respectively. 
For the control theoretic approach to the state estimation problem see \citet{luenberger1964observing} and \citet{luenberger1971introduction}.

Our study of control-theoretic methods has focused on
3DVAR. Recent analysis of the 3DVAR method rests heavily on ideas 
arising from determining modes for dissipative evolution equations, an idea with roots in the paper \citet{foias1967comportement} and
unified in the book \citet{temam2012infinite}. 
The use of these ideas in data
assimilation was introduced in \citet{olson2003determining} 
and developed further in \citet{hayden2011discrete} and \citet{foias2016discrete}; the papers \citet{law2012evaluating}, \citet{law2012analysis}, \citet{law2016filter} and \citet{sanz2015long}
essentially establish the stability of these deterministic results to small noise perturbations; see also \cite{moodey2013nonlinear} for related analysis. In the context of using observations to control the instability of chaotic systems, all of this work may be seen as building on the study of synchronization \citep{pecora1990synchronization}, reviewed in \citet{ashwin2003synchronization}.

In Subsection \ref{ssec:PP} we introduce the probabilistic approach to filtering, addressing Objective 2. In low dimensional
systems particle filters provide a flexible and efficient tool
for attacking probabilistic filtering; see \cite{doucet2001introduction}. However in this paper our focus
is on high dimensional problems and Kalman-based methods
specifically. The books
\citet{reich2015probabilistic},
\citet{asch2016data},
\citet{law2015data},
\citet{harlim2010filtering}, 
\citet{abarbanel2013predicting}, and
\citet{Evensenetal2022} provide overviews of a variety of
filtering methods, and ensemble Kalman methods 
from Subsection \ref{ssec:EKM} in particular.

The Kalman filter \citep{kalman1960new} from Example \ref{ex:sssec:2} led to arguably the first systematic analysis of an algorithm for incorporation of discrete time data into estimation of a discrete time stochastic dynamical system; it applies
only to linear Gaussian systems. 
The monograph \citet{jazwinski2007stochastic} provides an introduction to nonlinear filtering both in discrete and continuous time; in particular it discusses the extended Kalman filter, found by applying the Kalman filter to a linearization of the state and data dynamics. However this method does not work well in high dimensions
\cite{ghil1981applications}, motivating the use of mean field
maps, as introduced in Subsection \ref{ssec:MFM}, and the
ensemble-based methods from Subsection \ref{ssec:EKM} which 
approximate them. The approximation of
mean field maps by interacting particle systems is overviewed in \citet{sznitman1991topics}.

We refer the reader to the excellent monographs by \citet{asch2016data} and \citet{Evensenetal2022} for texts with emphasis on important implementation details, not covered in this paper, relating to these ensemble Kalman methods; these include techniques such as
inflation and localization that are central to the success of
these methods in high dimensions. We also refer to the paper by \citet{vetracarvalho2018stateoftheart} for a comprehensive review of the algorithmic details of ensemble Kalman methods. Here we point to two particular implementation details that are
of particular practical importance. The first concerns the use of ensemble square-root filters from Subsection \ref{sssec:esrf}. The matrix $Z_n$ in (\ref{eq:transform ensemble deviations}) is not uniquely defined by the requirement $V_{n+1} V_{n+1}^\top = C_{n+1}$. Formula \eqref{eq:ETKF_matrix} constitutes one possible choice, which leads to a symmetric $Z_n$. See \citet{sr:nichols08} for more details.
Second, finite particle implementations of the stochastic EnKF
from Subsection \ref{sssec:senkf}
entail that the random realizations $\hyj_{n+1}$ appear both in
the Kalman gain $K_n$ as well as in the innovation term in (\ref{eq:sd4e}c). As first observed by \citet{sr:houtekamer05}, this leads to a systematic underestimation of the ensemble spread, which vanishes in the $J\to \infty$ limit; but can affect the performance of the EnKF for small particle sizes. 
In Subsection \ref{ssec:IPBIBC} we will 
highlight the same effect when discussing finite particle 
implementations \citep{nusken2019note,garbuno2020affine} 
of the ensemble Kalman sampler for Bayesian inversion, 
based on the mean field model proposed in \citet{garbuno2020interacting}.

Particle-based extensions of the classical Kalman filter to nonlinear filtering problems include the unscented Kalman filter and the ensemble Kalman filter. While this paper focuses primarily on ensemble Kalman filter techniques, the unscented Kalman filter is an approach based on application of quadrature to the Gaussian projected filter from
Subsection \ref{ssec:GPFD}; see \citet{julier2000new} as well as \citet{sr:sarkka}. A discussion and evaluation in the context of ensemble square root filters 
may be found in \citet{SR-wang2004}. 

Much of the development of ensemble Kalman methods reflects
the historical roots of the subject in the geophysical sciences, the atmosphere-ocean sciences in particular, including Lagrangian data assimilation, and in
the modeling of subsurface flow
\citep{burger1998enkf,
houtekamer1998enkf,
anderson2001ensemble,
SR-bishop2001,
SR-whitaker2002enkf,
SR-TIPPETT03,
SR-hunt2007,
li2007iterative,
sakov2012,
bocquet2014,
evensen2019accounting,
bocquet2012combining,
bocquet2017degenerate,
gurumoorthy2017rank,
sampson2021ensemble,
kuznetsov2003method,salman2006method}. We also mention the randomized maximum likelihood (RML) approach to Bayesian inference which is closely related to the analysis step of a stochastic EnKF and has also been developed primarily through application in the geophysical sciences \citep{kitanidis:95,oliver:96d,oliver:08}. 

There is also a body of literature concerning the analysis and development of ensemble methods with an emphasis on applications in
complex and turbulent flows:
\citet{grooms2014ensemble},
\citet{grooms2015ensemble},
\citet{robinson2018improving},
\citet{lee2017preventing},
\citet{gottwald2013mechanism},
\citet{kelly2015concrete},
\citet{tong2016nonlinear},
\citet{tong2015nonlinear},
\citet{kelly2015concrete},
\citet{harlim2014ensemble},
\citet{majda2018performance},
\citet{fertig2007comparative},
\citet{harlim2007non},
\citet{harlim2007four}.
The conceptual fluid dynamics models of Lorenz \citep{lorenz1996predictability} (often referred to, collectively, as Lorenz '96 models) have been particularly influential in germinating this body of work and we use them exclusively in our illustrative Examples \ref{ex:3dvar}, \ref{ex:enkf}, \ref{ex:s3dvar}, \ref{ex:EKI} and \ref{ex:3dvar_ms}. Furthermore we will make use of the relationship between the multiscale and singlescale version of the model as developed in \citet{fatkullin2004computational}.

Recently ensemble Kalman methods have been developed for potential use in machine learning
\citep{haber2018never,
kovachki2019ensemble,
guth2020ensemble,
grooms2021analog,
gottwald2021supervised,yang2021machine,
SR-PR21};
see also \citet{bocquet2017degenerate}, \citet{chen2022autodifferentiable}
for research at the intersection of machine learning with ensemble Kalman methodology.
%in the continuous time-setting long-time error estimates, exploiting ergodicity of the Kalman-Bucy filter itself and propagation of chaos ideas to extend to the
%ensemble Kalman approximations, are developed in \cite{delMoral18}; see
%Subsection \ref{ssec:BCT} in the continuous time-setting long-time error estimates, exploiting ergodicity of the Kalman-Bucy filter itself and propagation of chaos ideas to extend to the
%ensemble Kalman approximations, are developed in \cite{delMoral18}; see
%Subsection \ref{ssec:BCT} 

In this survey we have started from mean field equations, in
Subsection \ref{ssec:MFM}, and then discretized the mean field limit using $J$ particles in subsequent subsections. It is of interest to demonstrate that the discrete formulations which arise actually converge to the mean field equations in the $J\to \infty$ limit. This has indeed been established for the ensemble Kalman filter when applied in the
linear Gaussian setting in which the mean field limit exactly recovers the filtering distribution \citep{gland:11,mandel:11,mandel:15} and indeed some
results also apply in the nonlinear setting. 
In the continuous time-setting long-time error estimates, exploiting ergodicity of the Kalman-Bucy filter itself and propagation of chaos ideas to extend to the
ensemble Kalman approximations, are developed in \citet{delMoral18}; see
Subsection \ref{ssec:BCT} for further details concerning continuous time. The paper \citet{law2016deterministic} studies related work concerning the mean field limit of ensemble Kalman
methods in the context of non-Gaussian problems. The papers \citet{ding2020ensemble}, 
\citet{ding2021ensemble,ding2021ensembleb} study particle approximation of mean field
limits beyond the Gaussian setting, primarily in the context of the solution of inverse problems; see the discussion in Subsection \ref{ssec:IPBIBC}. The papers \citet{hoel2016multilevel,chernov2021multilevel} study the use of multilevel approximation of the mean field limit, coupling ensemble approximations at different levels of space or time discretization.

In Subsection \ref{sssec:soti} we introduce the idea of
second-order transport: approximations of the perfect transport maps
that effect filtering. 
The non-uniqueness of second order transport maps is studied
in continuous time, for linear Gaussian stochastic differential
equations, in \citet{taghvaei2020optimal}; this work is closely
related to our analysis in the first two subsections of Section \ref{appendix:C}. Non-Gaussian extensions are discussed in \citet{taghvaei2022optimal}.
We also  highlight that it is possible to construct second
order transport maps $\widetilde T$, which satisfy conditions different from
(\ref{eq:compare}b) and (\ref{eq:compare2}b), respectively. For example,
one could request that 
\begin{equation*}
\op G\bigl((\tT_n)^\sharp \nu\bigr)=\op G \op B_n(\nu).
\end{equation*}
This approximation has been utilized by \citet{SR-bickel11}. The analogous
deterministic approach has been put forward by \citet{SR-Toedter15}
and has been explored further, for example, in \citet{SR-AdWR17}. Alternatively, one can also replace the definition (\ref{eq:KLG}) of the Gaussian projection operator $G$. To this end, recall that the Kullback--Leibler divergence
between probability measures $\pi_1$ and $\pi_2$ on $\R^d$ is defined as
\begin{equation*} \label{eq:KL}
 d_{\rm {KL}} (\pi_1||\pi_2) = \int \pi_1(du) \log \frac{\dd\pi_1}{\dd\pi_2}(u);
\end{equation*}
in particular, it is not symmetric in its two arguments. Gaussian variational inference
\citep{bishop} is for example based of the definition
\begin{equation}
\label{eq:KLG2}
    \op G\mu={\rm argmin}_{\pi \in \mathfrak{G}} d_{\rm {KL}} (\pi||\mu)
\end{equation}
in place of \eqref{eq:KLG}; note, however, that the minimizer of 
\eqref{eq:KLG2} may not be unique, whilst the minimizer of \eqref{eq:KLG}
is always unique.

While we follow the moment matching perspective on the derivation of ensemble Kalman filter methods in this survey, we mention in passing that there is an alternative perspective based on linear minimum variance estimators. See \citet{SR-vanLeeuwen2020} in the context of the stochastic ensemble Kalman filter and the Appendix \ref{ssec:TM_MVA} as well as \citet{SR-bickel11} for nonlinear extensions. The Bayes linear methodology is also an alternative
approach of potential interest \cite{goldstein2007bayes}.

Even in the mean field limit $J\to \infty$, the ensemble Kalman filter provides approximations only to approximate transport based filters. They are only exact in the linear Gaussian setting \citep{gland:11}; recent works develop new tools of analysis
to extend this to nonlinear filtering problems that are close to
Gaussian \citep{carrillo2022ensemble,calvello2024accuracy}. As mentioned in Subsection \ref{ssec:HC}, sequential Monte Carlo methods can be designed to be consistent with the underlying nonlinear filtering problem as defined, for example, by perfect transport based filters. Foundational analysis of these particle methods is undertaken in \citet{crisan1998discrete} and \citet{delmoral:04}; but we reiterate
that, in contrast to ensemble Kalman based methods,
they do not scale well to high dimensions.
We also point to \citet{del2006sequential} for an application of the sequential Monte Carlo method to Bayesian inference problems; the approach therein is closely related to iterative implementations of the EnKF that were
subsequently developed in the papers \citet{li2007iterative}, \citet{gu2007iterative}, \citet{sakov2012}.

Despite only providing approximations to the exact filtering distribution, i.e.~from the perspective of Objective 2, accuracy and stability results for the ensemble Kalman filter, viewed as a state estimator and hence from the 
perspective of Objective 1, have been derived. See, for example, 
\citet{gonzalez2013ensemble}, \citet{kelly2014well},
\citet{tong2015nonlinear}, \citet{tong2016nonlinear}
and \citet{del2021theoretical}. Mechanisms for finite time filter divergence have also been identified \citep{gottwald2013mechanism,kelly2015concrete}.

Extending the ensemble Kalman filter to strongly nonlinear and high dimensional state estimation problems constitute an area of active ongoing research. The current state of the art has been summarized in \citet{SR-LKNPR19} in the context of high dimensional geophysical applications. Extensions of the transport framework (\ref{eq:mf0}), which build on approximating the perfect transport maps $\Td$ in (\ref{eq:mf0}b) in an asymptotically consistent manner, include the work by \citet{SR-R13}, \citet{SR-cheng14}, \citet{spantini2019coupling}, and \citet{zech:22}. In an alternative line of research there have been several proposals to construct hybrid methods, which aim to adaptively bridge between the ensemble Kalman and particle filters; including the work by \citet{SR-stordal11}, \citet{frei2013bridging}, \citet{SR-CRR15}, and \citet{nerger2022data}.

In this context, the transformation formula (\ref{eq:linear transform filter}) proves to be rather useful since most existing particle based methods can be covered by appropriate choices of $S_n$, where $S_n$ is typically the realization of a random matrix. See \citet{reich2015probabilistic} for more details. For example, a resampling step in a sequential Monte Carlo method gives rise to a matrix $S_n$ with a single non-zero entry equal to one in each of its columns. More generally it holds that, for all $j \in \sJ$,
\begin{equation*}
    \sum_{i \in \sJ} (S_n)_{ij} = 1.
\end{equation*}
In particular, the matrix $S_n$ can be chosen to correspond to an optimal coupling between two discrete random variables \citep{SR-R13}, which builds a link between filtering and optimal transport also explored in \citet{corenflos2021differentiable}. The subject of optimal transport is given a comprehensive treatment in \citet{villani2008optimal}; see also \citet{villani2021topics}. Computational aspects of the subject including entropy-regularized optimal transport are discussed in \citet{cuturi2013sinkhorn} and \citet{peyre2019computational}. %Ideas from the field are impacting filtering and Bayesian inference \citep{corenflos2021differentiable}, building on the computational tractability of entropy-regularized optimal transport \citep{cuturi2013sinkhorn}. 
Entropy-regularization is linked to the Schr\"odinger bridge problem, and connections with data assimilation are developed in \citet{reich2019data}. See also Subsection \ref{ssec:HC} for discussion of transport-based methodologies within the context of ensemble Kalman methods. See Example \ref{ex:cite} and \citet{reich2015probabilistic} for the connection between the map \eqref{eq:needt} and optimal transport. For a derivation of the standard formulae for mean and covariance of conditioned Gaussians used in the proof of Lemma \ref{lemma:KF_analysis} see \citet{eaton2007}.

Finally we observe that we do not discuss, in this paper, the \emph{smoothing} approach to state estimation from data in model \eqref{eq:sd}. This approach aims at finding
the entire sequence $\{v_\ell\}_{\ell=0}^n$ from the data $\Yd_n$. Thus state estimates depend on data in their future. To read about smoothing see \citet{Evensenetal2022,sanz2023inverse}. We note here that there is a smoothing counterpart
of the 3DVAR algorithm (see Remark \ref{rem:3dvar}) known as 4DVAR because, for physical systems, it uses data distributed
in the three space and one time dimensions.

%%%%%%%%%%%%%%%%%%%%%%%%%%%%%%%%%%%%%%%%%%%%%%%%%%%%%%%%%%%%%%%%%%%%%%
%
%
%
%
%
%
%
%
%
%
%
%  Section 3: continuous time
%
%
%
%
%
%
%
%
%
%
%
%%%%%%%%%%%%%%%%%%%%%%%%%%%%%%%%%%%%%%%%%%%%%%%%%%%%%%%%%%%%%%%%%%%%%

%%%%%%%%%%%%%%%%%%%%%%%%%%%%%%%%%%%%%%%%%%%%%%%%%%%%%%%%%%%%%%%%%%%%%%%
%
\section{State Estimation: Continuous Time}
\label{sec:CT}
%
%%%%%%%%%%%%%%%%%%%%%%%%%%%%%%%%%%%%%%%%%%%%%%%%%%%%%%%%%%%%%%%%%%%%%%

This section is devoted to deriving, and studying properties of, continuous
time analogs of concepts introduced in the previous Section
\ref{sec:SE}. We start in Subsection \ref{ssec:CTSE} 
defining the set-up.
Thereafter the subsections mirror those from the preceding
Section \ref{sec:SE}, describing the relevant continuous time analogs;
in particular we conclude with Subsection \ref{ssec:BCT} 
containing bibliographic notes.

All problems arising in practice are implemented as
algorithms in discrete time, so it is important to establish
motivation for the continuous time formulations.
There are two primary reasons for introducing them. 
The first is that continuous time limits of the
discrete algorithms provide a way to understand and interpret
the behavior of the discrete algorithms; results about
accuracy, stability and uncertainty quantification, which shed
light on the relative merits of different algorithmic approaches,
are often cleanest in the continuous time setting. The
second is that many problems arising in practice involve
physical processes which evolve in continuous time; the
data informing these models is typically discrete in time, but when
the observations take place at very high frequency, it is insightful to consider the idealization of continuous time data. Both of these motivations underlie the developments in this section.

%%%%%%%%%%%%%%%%%%%%%%%%%%%%%%%%%%%%%%%%%%%%%%%%%%%%%%
\subsection{Set-Up}
\label{ssec:CTSE}
%%%%%%%%%%%%%%%%%%%%%%%%%%%%%%%%%%%%%%%%%%%%%%%%%%%%%%%

We start by deriving the continuous-time analog of the discrete-time set-up \eqref{eq:sd} for state-observation coevolution: for all $n \in \Z^+$ we have
\begin{align*}
\label{eq:sd3}
v_{n+1} &= \Psi(v_n) + \xi_n, \\
y_{n+1} &= h(v_{n+1}) + \eta_{n+1};
\end{align*}
recall that we assume that $v_0,  \{\xi_n\}_{n \in \Z^+}$
and $\{\eta_n\}_{n \in \N}$ are mutually independent
Gaussians defined by
\begin{equation*}
\label{eq:ga3}
v_0 \sim \Ng(m_0, C_0), \quad 
\xi_n \sim \Ng(0, \Sigma) \,\, \text{i.i.d.}, \quad
\eta_n \sim \Ng(0, \Gamma) \,\, \text{i.i.d.}\,\,.
\end{equation*}
We introduce a small increment in time, denoted by $\Delta t$.
From the map $\Psi(\cdot)$ defining the systematic component of the state dynamics, we now define an infinitesimal analog $f(\cdot)$;  we also introduce the rescaled observation operators $\hs(\cdot)$ from the original nonlinear  observation operator $h(\cdot)$; and we introduce state/observational covariances $(\Gammas,\Sigmas)$
by rescaling $(\Gamma,\Sigma)$:
\begin{subequations}
\label{eq:rescalings}
\begin{align}
\Psi(v) &= v +\dt f(v),\quad
h(v) = \dt \hs(v),\\
\Sigma &= \dt \Sigmas,\quad \Gamma = \dt \Gammas.
\end{align}
\end{subequations}
By virtue of our assumptions on $\Psi$ and $h$, functions
$f$ and $\hs$ are assumed to be known measurable functions (with respect to the Borel algebra), bounded on compact sets. 
In the linear setting $\Psi(\cdot)=M\cdot$, $ h(\cdot)=H\cdot$ we will also introduce 
an infinitesimal vector field $f(\cdot)=F\cdot$ for matrix $F$, and 
rescaled linear observation operator $\Hs$
\footnote{Note the difference, conceptual and notational, between the discrete time objects
$(h(\cdot),H,\Gamma,\Sigma)$ and the related continuous time objects $(\hs(\cdot),\Hs,\Gammas,\Sigmas)$.}
\begin{align}
\label{eq:rescalings2}
   M= \text{Id} + \dt F,   \quad H=\dt \Hs.
\end{align}
The observation $\{y_n\}$ is
best thought of, in the scalings we introduce, as capturing increments of a process
$\{z_n\}$. To capture this, and extend it to the specific realization of the data
appearing in the algorithms, and the artificial data used in some algorithms, 
we introduce the variables $z_n$, $\zd_n$, $\hz_n$ by assuming that 
\begin{subequations}
\label{eq:reparam}
\begin{align}
y_{n+1} &:= z_{n+1} - z_n = \Delta z_{n+1}, \\
\yd_{n+1} &:= \zd_{n+1} - \zd_n = \Delta \zd_{n+1}, \\
\hy_{n+1} &:= \hz_{n+1} - \hz_n = \Delta \hz_{n+1}.
\end{align}
\end{subequations}
Note that $z_n$, $\zd_n$, $\hz_n$ have dimension $d_z = d_y$.
We assume that $z_0=\zd_0=\hz_0=0$.  Then $z_n$, $\zd_n$, $\hz_n$ 
are uniquely defined from $y_n$, $\yd_n$, $\hy_n$, respectively.

In the following we define $t_n=n\Delta t.$ With the scalings above in hand, we may view the state $v_n$ and observation $y_n$ as
relating to approximations of continuous time processes $v(\cdot)$ and $z(\cdot)$:
$v_n \approx v(t_n)$, $z_n \approx z(t_n)$. 
We also introduce continuous time process $\hv(\cdot)$, which will be used
in the prediction steps of algorithms, and
$\bigl(\zd(\cdot), \hz(\cdot)\bigr)$, which denotes the continuous time observed
data which we are conditioning on and predicted data, respectively. We assume that $z(0)=\zd(0)=\hz(0)=0.$
Under the rescalings above, and in the limit $\Delta t \to 0$,
the data assimilation problem may be reformulated in
terms of SDEs. Furthermore, the related mappings on measures, and discrete-time algorithms that stem from them, may be reformulated in terms of SPDEs and SDEs respectively; we now go on to identify these continuous time stochastic processes.

Applying the rescalings in \eqref{eq:rescalings} and the reparametrization of $y_{n+1}$ in (\ref{eq:reparam}a), we obtain the system 
\begin{subequations}
\label{eq:sd3res}
\begin{align}
v_{n+1} &= v_n +\dt f(v_n) + \xi_n, \\
z_{n+1} &= z_n+ \dt\hs(v_{n+1}) + \eta_{n+1},
\end{align}
\end{subequations}
for all $n \in \Z^+$, where we assume  $v_0,  \{\xi_n\}_{n \in \Z^+}$
and $\{\eta_n\}_{n \in \N}$ are mutually independent
Gaussians defined by
\begin{equation*}
\label{eq:ga3res}
v_0 \sim \Ng(m_0, C_0), \quad 
\xi_n \sim \Ng(0, \Delta t\Sigmas) \,\, \text{i.i.d.}, \quad
\eta_n \sim \Ng(0, \Delta t\Gammas) \,\, \text{i.i.d.}\,\,.
\end{equation*}
Note that \eqref{eq:sd3res} is a variant on the Euler-Maruyama discretization of a vector-valued SDE. Indeed, by taking the $\dt \rightarrow 0$ limit it is clear that the natural continuous time analog of equations \eqref{eq:sd} is the SDE
\begin{subequations}
\label{eq:sdco}
\begin{empheq}[box=\widefbox]{align}
\dd v &= f(v)\dd t + \sqrt{\Sigmas}\dd W, \quad v_0 \sim \Ng(m_0, C_0), \\
\dd z &= \hs(v)\dd t + \sqrt{\Gammas}\dd B, \quad z(0)=0,
\end{empheq}
\end{subequations}
taken to hold for all $t \in \R^+.$
The vector fields $f(\cdot)$ and $\hs(\cdot)$ describe the systematic,
deterministic components of the dynamics and observation processes
and are assumed known. The systematic 
components of the model
are subjected to white noise defined through the 
independent unit Brownian motions $W$ and $B$, in $\R^{d_v}$
and $\R^{d_z}$ respectively, and correlated across the
state and data spaces via the covariances $\Sigmas,\Gammas.$ The initial condition for $v$
is Gaussian and independent of $W$ and $B$. Analogous to the discrete time setting, we assume that
\begin{equation}
C_0 \succeq 0, \quad  \Sigmas \succeq 0, \quad \Gammas \succ 0.
\label{eq:gaco}
\end{equation}
Note that, for each fixed $t \in \R^+$, the state $v(t) \in \R^{d_v}$ and the
observations $z(t) \in \R^{d_z}.$

Throughout we use $\dagger$ again
to denote a specific realization of a process,
as in the discrete time setting. We assume that we
have available to us a sample path $\{\zd(t)\}_{t \in \R^+}$
of the observation coordinates of a realization
of the SDE \eqref{eq:sdco}. From this sample path we wish to recover the true realization of the state  $\{\vd(t)\}_{t \in \R^+}$
which gave rise to it. These observation and state sample paths
are generated by $\vd_0, \{\Wd\}_{t \in \R^+}$ and
$\{\Bd\}_{t \in \R^+}$, specific realizations of the
initial condition and the Brownian motions
driving the state and observation components of \eqref{eq:sdco}.
We also introduce $Z^\dagger(t)=\{\zd(s)\}_{0 \le s \le t}.$ 

Analogously to the discrete time setting, it is natural to establish two distinct objectives, both related to recovery of the state from the observation:
\begin{itemize}
\item Objective 1: design an algorithm producing output
$v(t)$ from $Z^\dagger(t)$ so that
$\{v(t)\}_{t \in \R^+}$ estimates
$\{\vd(t)\}_{t \in \R^+}$, the true signal generated 
by (\ref{eq:sdco}a);  
\item Objective 2: design an algorithm which estimates
the distribution of random variable $v(t)|Z^\dagger(t)$.
\end{itemize}
As in discrete time we are interested in Markovian
formulations which update the estimate $v(t)$, or the
distribution $v(t)|Z^\dagger(t)$, sequentially as the data is
acquired. All of the algorithms we describe depend only on 
the increments of the process $\zd(t)$, hence the fixing of $\zd(0)=0$ is immaterial. In the next two subsections we describe control theoretic and probabilistic approaches to this problem which, respectively, provide the basis for algorithms addressing Objectives 1 and 2. Following the road-map
from the previous section in the discrete time setting, 
we then proceed to study exact transport leading to mean field equations 
related to the Objective 2; we then study second order approximations
of exact transport, and finally reach ensemble Kalman methods through particle approximations.

%%%%%%%%%%%%%%%%%%%%%%%%%%%%%%%%%%%%%%%%%%%%%%%%%%%
\subsection{Control Theory Perspective}
\label{ssec:CTCT}
%%%%%%%%%%%%%%%%%%%%%%%%%%%%%%%%%%%%%%%%%%%%%%%%%%%

As for the time-discrete problem, we again start with the control theoretical approach based on the small uncertainty  assumption. Specifically we assume that the three covariances appearing in \eqref{eq:gaco} 
are small so that the states and observations can be well approximated as deterministic. Thus
we initially set $\Sigmas$ and $\Gammas$ to zero. In this setting
we derive a continuous time analog of the 3DVAR methodology.

Applying the rescalings \eqref{eq:rescalings} to (\ref{eq:sd2}) we obtain, in the deterministic setting,
\begin{subequations}
\label{eq:sdc2}
\begin{align}
\hv_{n+1} &= v_n + \dt f(v_n),  \\
\hz_{n+1} &= \hz_n + \dt \hs(\hv_{n+1}), \\
v_{n+1} &= \hv_{n+1}+K(\Delta \zd_{n+1}-\Delta \hz_{n+1})
\end{align}
\end{subequations}
with the observed increments from (\ref{eq:reparam}b) so that
\begin{align*}%\label{eq:numerical_data_increments}
\Delta \zd_{n+1} &:= \zd(t_{n+1})-\zd(t_n),
\end{align*}
derived from a specific fixed realization of \eqref{eq:sdco} and $t_n = n\dt$.
For fixed observed increments $\Delta \zd_{n+1}$, equations \eqref{eq:sdc2} define a deterministic
map $v_n \mapsto v_{n+1}$. Taking the continuous time limit,
and eliminating $\hz$, we obtain the following estimator $v$ for $\vd$ given $Z^\dagger:$
\begin{empheq}[box=\widefbox]{equation}
\label{eq:3DVARc}
{\dd v} = f(v)\dd t +  K\bigl({\dd \zd} - \hs(v)\dd t\bigr),
\end{empheq}
where $\zd$  (the data) is obtained 
from a specific fixed realization of \eqref{eq:sdco}:
\begin{subequations}
\label{eq:data_ct}
\begin{align}
\dd\vd &= f(\vd)\dd t + \sqrt{\Sigmas}\dd W^\dagger, \quad \vd(0) \sim \Ng(m_0, C_0), \\
\dd\zd &= \hs(\vd)\dd t + \sqrt{\Gammas}\dd B^\dagger, \quad \zd(0) =0.
\end{align}
\end{subequations}
Equation \eqref{eq:3DVARc} defines a continuous time analog of the 3DVAR algorithm \eqref{eq:sd2_add}
and \emph{gain matrix} $K$ should be viewed as a parameter to be chosen. The equation
has the form of a controlled 
ordinary differential equation (ODE); typically it is initialized
with $v(0) \sim \Ng (m_0,C_0)$.

We now include the effect of uncertainty, allowing for non-zero covariances in \eqref{eq:gaco}.
Accounting for noise in the expressions for $\hv_{n+1}$ and $\hz_{n+1}$ in \eqref{eq:sdc2} we obtain the following
rescaling of (\ref{eq:sd2n}):
\begin{align*}
\label{eq:sdc2noise}
\hv_{n+1} &= v_n + \dt f(v_n) +\xi_n,  \\
\hz_{n+1} &= \hz_n + \dt \hs(\hv_{n+1}) +\eta_{n+1}, \\
v_{n+1} &= \hv_{n+1}+K_n(\Delta \zd_{n+1}-\Delta \hz_{n+1}),
\end{align*}
for all $n \in \Z^+$, where we assume  $v_0,  \{\xi_n\}_{n \in \Z^+}$
and $\{\eta_n\}_{n \in \N}$ are mutually independent
Gaussians defined by \eqref{eq:ga3res}.
%\begin{equation}
%\label{eq:ga3res2}
%v_0 \sim \Ng(m_0, C_0), \quad 
%\xi_n \sim \Ng(0, \Sigmas) \,\, \text{i.i.d.}, \quad
%\eta_n \sim \Ng(0, \Gammas) \,\, \text{i.i.d.}\,\,.
%\end{equation}
%%%%%%%%%%%%%%%%%%%%%%%%%%%%%%%%%%%
%%%%%%%%%%%%%%%%%%%%%%%%%%%%%%%%%%%
We may now formally take the $\dt \to 0$ limit and obtain 
the following continuous time analog of
(\ref{eq:sd2n}), namely the controlled SDE formulation:
\begin{subequations}
\label{eq:sd2nc}
\begin{empheq}[box=\widefbox]{align}
\dd v &= f(v)\dd t + \sqrt{\Sigmas} \dd W  + K(\dd\zd - \dd\hz),\\
\dd\hz &= \hs(v)\dd t + \sqrt{\Gammas} \dd B.
\end{empheq}
\end{subequations}
where $\zd$ is given by \eqref{eq:data_ct}. The unit Brownian motions $W, W^\dagger, B$ and $B^\dagger$, in $\R^{d_v}$,  $\R^{d_v}$, $\R^{d_z}$ and $\R^{d_z}$ respectively, are all independent of one another.
As in the discrete time analog, encapsulated in (\ref{eq:sd2n}), 
the choice of a (now \emph{time-dependent}) gain matrix $K$ remains to be determined and is crucial for the success of such a methodology;
and as in discrete time, a time-evolving gain matrix is often
desirable. To determine $K$ we adopt a mean field perspective, as we did in discrete time. To this end we now discuss the evolution of probability measures describing the conditional distribution of $v(t)|Z^\dagger(t)$.

%%%%%%%%%%%%%%%%%%%%%%%%%%%%%%%%%%%%%%%%%%%%%%%%%%%
%
\subsection{Probabilistic Perspective}
\label{ssec:CTPP}
%
%%%%%%%%%%%%%%%%%%%%%%%%%%%%%%%%%%%%%%%%%%%%%%%%%%

We start, in Subsection \ref{sssec:cucd}, by discussing the unconditioned dynamics
and introducing the Fokker-Planck equation associated with the state space evolution.
In Subsection \ref{sssec:ctfd} we take the formulation of 
the filtering iteration  in discrete time, from Subsection \ref{ssec:PP},  
and take a continuous time limit to derive the Kushner-Stratonovich
equation; we study the linear Gaussian setting, and the Kalman-Bucy
filter, as a special case. In Subsection \ref{sssec:CSPP} we
introduce the \emph{sample path perspective}, central to the algorithmic
perspective developed in this paper. Subsection \ref{ssec:irun} defines
notation that will be useful throughout the remainder of
this section on continuous time data assimilation.

The derivation of continuous time limits in the previous two subsections is relatively
straightforward. However, there is an important practical and theoretical issue which we need to address. As mentioned
at the start of Section \ref{sec:CT} continuous time observations $\zd(t)$ are typically an idealization of discrete time data collected at instances $\tau_k = k\delta$, $\delta >0$, $k \in \N$ only.\footnote{Non-equally spaced data is also of relevance in this context; but we do not consider it here.} In order to make use of continuous time
algorithms and theory it is then useful to construct a continuous time 
approximation $\zdd$; to be concrete we will use piecewise linear interpolation. With this set-up we need to deal with two small parameters; the time-step $\dt$ used in \eqref{eq:rescalings} to obtain
a continuous time limit, and the data sampling interval $\delta$. 
The following remark addresses choices that we make in these
notes about the manner in which we take the limit
$(\dt, \delta) \to 0$.

\begin{remark}
There are results for continuous time filtering which imply that the desired limiting equation can be found in either It\^o or Stratonovich forms by considering different orders of the limits $\dt \to 0$ and $\delta \to 0$; see the bibliographic remarks in Subsection \ref{ssec:BCT}. 
Many theoretical results are derived by first taking $\delta \to 0$ and then $\dt \to 0$. From a practical and theoretical perspective, however, it is sometimes more convenient to first consider the limit $\dt \to 0$ followed by the limit $\delta \to 0$. We will utilize the latter sequence of limits in the following subsection in order to derive a set of evolution equations for the conditional probability measure $\mu (v,t)$ solving Objective 2. These equations will in turn guide our choice of the gain matrix $K$ in (\ref{eq:sd2nc}a). We emphasize that the choice about the order in which to take the limits $(\dt,\delta) \to 0$ is problem dependent and should be
considered carefully whenever continuous time modeling is employed.

When the data $\zd$ arises itself from numerical simulations of a continuous problem, then it is most convenient to set $\delta = \dt$, and hence $\tau_k=t_k$; this is implicitly used in the derivation of the continuous time sample path equations in the two preceding subsections. 
$\blacksquare$
\end{remark}

%%%%%%%%%%%%%%%%%%%%%%%%%%%%%%%%%%%%%%%%%%%%%%%%%%%%%%%%%
%
\subsubsection{Unconditioned Dynamics}
\label{sssec:cucd}
%
%%%%%%%%%%%%%%%%%%%%%%%%%%%%%%%%%%%%%%%%%%%%%%%%%%%%%%%%%%%

First consider the continuous time limit of the evolution
associated with $\op P$ from (\ref{eq:ops0}b),  (\ref{eq:ops0}c) which, 
with the scaling adopted in this section, is defined by
\begin{subequations}
\label{eq:ops_c99}
\begin{align}
(\op P\mu)(\dd v)&=\Bigl(\int_{u \in \R^{d_v}} p(u,v)\mu(\dd u)\Bigr)\dd v,\\
p(u,v)&=\frac{1}{(2\pi \dt)^{d_v/2} \sqrt{{\rm det}(\Sigmas)}}\exp \Bigl(-\frac{1}{2\dt}|v-u-\dt f(u)|_{\Sigmas}^2\Bigr).
\end{align}
\end{subequations}
Now view $r_n$ given by (\ref{eq:ops0}a) as approximating $r(n\dt).$
Recall that the underlying continuous time limit of the
sample path evolution is given by the SDE (\ref{eq:sdco}a). Thus
the time evolution of the probability density $r(\cdot,t)$ is given by the 
Fokker--Planck equation
\footnote{Here, and in what follows, we use the standard 
notation from continuum mechanics for the divergence of vector
and second order tensor fields, and for the gradient of scalar and vector fields;
see the bibliography Subsection \ref{ssec:BCT} for references.}
\begin{align} \label{eq:FPE}
    \partial_t r = -\nabla \cdot ( r f)  + \frac{1}{2} \nabla \cdot (\nabla \cdot( r \Sigmas)).
\end{align}
Note that, as in discrete time, this evolution is linear and decoupled from the state
space evolution (\ref{eq:sdco}a). We refer to the latter as a \emph{continuous time Markov process.}

%%%%%%%%%%%%%%%%%%%%%%%%%%%%%%%%%%%%%%%%%%%%%%%%%%%%%%%%%
%
\subsubsection{The Filtering Distribution}
\label{sssec:ctfd}
%
%%%%%%%%%%%%%%%%%%%%%%%%%%%%%%%%%%%%%%%%%%%%%%%%%%%%%%%%%%%

As in discrete time, we now consider the evolution equation for the state conditioned on
observations. Our starting point here is the iteration on measures,
the filtering cycle, defined by \eqref{eq:DAsimple}, 
under the scalings \eqref{eq:rescalings}. 
We assume that $\delta$ is an integer multiple of $\dt$ so that the $\{\tau_k\}$ are a subset of
the $\{t_n\}.$ In what follows we will first fix $\delta$ and let $\dt \to 0$; in order to obtain the integer multiple property we thus consider $\dt \to 0$ along a subsequence. We replace the true observation path $\zd(t)$ by
its piecewise linear approximation $\zdd(t)$ based on linear interpolation of values
$\{\zd(\tau_k)\}$. To be precise we assume that the derivative is
\emph{cadlag}.\footnote{Continuous from the right, limits exist from
the left.} We then have that the implied observation increments $\Delta \zd_{n+1}$ are constant over the time intervals $[t_n,t_{n+1})$ and are given by
\begin{align}\label{eq:data_interpolated}
    \Delta \zd_{n+1} = \frac{\dd\zdd}{\dd t}(t_n) \dt.
\end{align}
In this setting, the operators $\op P$ and $\op L_n$, defined by (\ref{eq:DAsimple}a) and (\ref{eq:DAsimple}b), respectively, become
\begin{subequations}
\label{eq:ops_c}
\begin{align}
(\op P\mu)(\dd v)&=\Bigl(\int_{u \in \R^{d_v}} p(u,v)\mu(\dd u)\Bigr)\dd v,\\
\op L_n(\mu)(\dd v)&=q(v,\Delta \zd_{n+1})\mu(\dd v)\Big/\Bigl(\int_{\R^{d_v}}q(v,\Delta\zd_{n+1})\mu(\dd v)\Bigr),
\end{align}
\end{subequations}
where, from \eqref{eq:ops2} with the scalings \eqref{eq:rescalings},(\ref{eq:reparam}b),
\begin{subequations}
\label{eq:ops2_c}
\begin{align}
p(u,v)&=\frac{1}{(2\pi \dt)^{d_v/2} \sqrt{{\rm det}(\Sigmas)}}\exp \Bigl(-\frac{1}{2\dt}|v-u-\dt f(u)|_{\Sigmas}^2\Bigr),\\
q(v,\Delta z)&=\frac{1}{(2\pi\dt)^{d_y/2} \sqrt{{\rm det}(\Gammas)}} 
\exp \Bigl(-\frac{1}{2\dt}|\Delta z-\dt \hs(v)|_{\Gammas}^2\Bigr).
\end{align}
\end{subequations}
With these formulae in hand we may now derive the continuous time analog of \eqref{eq:DAsimple}, for $\mu(v,t)$. 

For ease of exposition we
assume that $\mu$ has density  $\rho$ and derive the
equation satisfied by $\rho.$ To do this we employ
the {\em split-step principle}: we find the 
continuous time evolution equation
associated with each of $\op P$ and $\op L_n$ (equations
(\ref{eq:ops_c}a) and (\ref{eq:ops_c}b) respectively)
separately, and then add the right-hand sides of the resulting 
evolution equations to obtain the desired continuous time limit
resulting from the composition of $\op L_n$ and $\op P$. We use $r$ as
a dummy variable to denote
the density being evolved, for both of the split-steps, and in both discrete
($r_n$) and continuous ($r(t)$) time, in what follows.

First recall that the continuous time limit of the evolution
associated with $\op P$, as defined by (\ref{eq:ops_c}a) and (\ref{eq:ops2_c}a).
is given by the Fokker--Planck equation \eqref{eq:FPE}.
Secondly, consider the second component of the split-step argument: we determine a
continuous time limit of the evolution
associated with $\op L_n$ as described by (\ref{eq:ops_c}b) and (\ref{eq:ops2_c}b). The following lemma presents an evolution equation for $r$ associated with $\op L_n$, describing how observation of the piecewise continuous interpolated data $\zdd(t)$ changes the density $r(t,v)$. 
\begin{lemma}
Assume that $\Gammas \succ 0.$ The continuous time limit of the evolution 
associated with $\op L_n$, as described by (\ref{eq:ops_c}b) and (\ref{eq:ops2_c}b), is given by 
\begin{align} \label{eq:KSD_S1}
    \partial_t r = \left\langle \hs- \E \hs, \frac{\dd\zdd}{\dd t} \right\rangle_\Gammas r  - \frac{1}{2} \left\{ \left| \hs \right|^2_\Gammas - \E \left| \hs \right|^2_\Gammas \right\}r.
    %dr = r \langle (h - \E h), \Gamma^{-1}(d\yd - \E h dt)\rangle,
\end{align}
$\Diamond$ \end{lemma} 
\begin{proof}
Consider the discrete-time evolution 
\begin{align} \label{eq:r_update}
    r_{n+1} = \op L_n r_n, 
\end{align}
where $\op L_n$ is defined by (\ref{eq:ops_c}b) and (\ref{eq:ops2_c}b). By Taylor expansion we have
\begin{equation*}
\exp\left(-\frac{1}{2\dt}\left|\Delta\zd_{n+1}-\dt \hs(v)\right|_{\Gammas}^2 \right) =  1 - \frac{\dt}{2}\left|\frac{\Delta\zd_{n+1}}{\dt}-\hs(v)\right|^2_\Gammas + \mathcal{O}(\dt^2).
\end{equation*}
Then, using expressions (\ref{eq:ops_c}b) and (\ref{eq:ops2_c}b), we obtain
\begin{equation}\label{eq:Lnrn}
    (\op L_nr_n)(v) = \frac{1}{C(\dt)}\left(1 - \frac{\dt}{2}\left|\frac{\Delta\zd_{n+1}}{\dt}-\hs(v)\right|^2_\Gammas + \mathcal{O}(\dt^2)\right)r_n(v),
\end{equation}
where
\begin{equation*}
    C(\dt) = \int_{\mathbb{R}^{d_v}}\left(1 - \frac{\dt}{2}\left|\frac{\Delta\zd_{n+1}}{\dt}-\hs(v)\right|^2_\Gammas + \mathcal{O}(\dt^2) \right)r_n(v)\dd v.
\end{equation*}
By integrating and noting that $r_n$ is a density, it follows that 
\begin{equation}
C(\dt) = 1 - \frac{\dt}{2}\mathbb{E}\left| \frac{\Delta\zd_{n+1}}{\dt} - \hs(v) \right|^2_\Gammas + \mathcal{O}(\dt^2), 
\end{equation}
where expectation is with respect to $v$ distributed as random
variable with density $r_n$. Hence, combining \eqref{eq:Lnrn} and (\ref{eq:r_update}) we obtain, to leading order in $\dt$,
\begin{equation*}
r_{n+1} =  r_n \left\{1 + \dt \left\langle \hs-\E \hs,\frac{\dd\zdd}{\dd t}\right\rangle_\Gammas 
-\frac{\Delta t}{2} \left\{ \left| \hs \right|^2_\Gammas - \E \left| \hs \right|^2_\Gammas \right\} +\mathcal{O}(\dt^2) \right\}; 
\end{equation*}
here we have used that the data increments $\Delta \zd_{n+1}$ are given by (\ref{eq:data_interpolated}) for fixed $\delta$. Taking the time continuous limit $\dt \to 0$ with fixed observation interval $\delta>0$ leads to the evolution equation \eqref{eq:KSD_S1}.
\end{proof}
%\begin{align} \label{eq:KSD_S1_.}
%    \partial_t r = \left\langle \hs- \E \hs, \frac{d\zdd}{dt} %\right\rangle_\Gammas r  - \frac{1}{2} \left\{ \left| \hs \right|^2_\Gammas - \E %\left| \hs \right|^2_\Gammas \right\}r.
    %dr = r \langle (h - \E h), \Gamma^{-1}(d\yd - \E h dt)\rangle,
%\end{align}

Now taking the $\delta \to 0$ limit in (\ref{eq:KSD_S1}), we obtain the
following nonlocal nonlinear stochastic evolution equation for density $r(v,t):$
\begin{align} \label{eq:KSD_S2}
    \dd r = \left\langle \hs- \E \hs, \circ \dd\zd \right\rangle_\Gammas r  - \frac{1}{2} \left\{ \left| \hs \right|^2_\Gammas - \E \left| \hs \right|^2_\Gammas\right\}r \dd t.
    %dr = r \langle (h - \E h), \Gamma^{-1}(d\yd - \E h dt)\rangle,
\end{align}
Here $\circ$ denotes Stratonovitch integration; this form of integration arises
in the limit $\delta \to 0$ because the equation
is derived by making a smooth approximation $\zdd$ of $\zd$ and passing to
the limit. Recall that $\zd$ is given by \eqref{eq:data_ct}.
The equation is nonlocal and nonlinear because $\E$ denotes
expectation at time $t$ with respect to density $r(\cdot,t).$

We now wish to invoke the split-step principle and combine the evolutions
\eqref{eq:FPE} and \eqref{eq:KSD_S2}. However before doing this
we proceed to convert the equation \eqref{eq:KSD_S2} into its more common 
It\^o representation. For this, the following lemma is crucial. 
In proving it we will use the concepts of quadratic variation
and covariation. For an introduction to these concepts consult the lecture notes by Eberle, which are referenced in the bibliography Subsection \ref{ssec:BCT}.
We note that quadratic variation and covariation are first defined between scalar-valued process and
can then be lifted to define: (i) the covariation of an 
inner-product $\langle x, y\rangle$ between vector processes $x$
and $y$, which is scalar-valued; (ii) the quadratic variation of vector process $x$, which is matrix-valued; and (iii) the covariation of vector process $x$ with scalar process $z$, which is vector-valued.

Use of quadratic variation and covariation leads to a succinct, streamlined proof. Furthermore, in Appendix \ref{sec:AC}, we provide explicit calculations 
for the reader who is interested in understanding the details of the conversion by means of the definitions of It\^o and Stratonovitch integrals as limits. In particular, the concepts underlying the quadratic variation and covariation calculations in the following are derived from first principles in Lemma \ref{lem:ISadd}.

\begin{lemma}
\label{lem:ISC}
Assume that $\Gammas \succ 0$ and that $\zd$ is given by \eqref{eq:data_ct}. The It\^o and Stratonovich interpretations of the stochastic forcing term in (\ref{eq:KSD_S2}) are related through
\begin{align*}
\dd r &= \left\langle \hs- \E \hs, \circ \dd\zd \right\rangle_\Gammas  r - \frac{1}{2} \left\{ \left| \hs \right|^2_\Gammas - \E \left| \hs \right|^2_\Gammas\right\} r \dd t\\
    &= \left\langle \hs - \E \hs, \dd\zd - \E \hs \dd t\right\rangle_\Gammas r.
\end{align*}
$\Diamond$ \end{lemma}

%\as{Edo: please check following proof as I rearranged things to try and have variable $r$ appear at far right of every term in every identity (except when its under an integral); I didn't do it in the appendix proof for reasons I can explain in person.}

\begin{proof} 
Using the formula for the It\^o--Stratonovich conversion 
between two semimartingales, we  have
\begin{align}\label{eq:conversion_IS}
     \left\langle \hs-\E \hs, \circ \dd\zd\right\rangle_\Gammas r
    =   \left\langle \hs-\E \hs, \dd\zd \right\rangle_\Gammas r + \frac{1}{2}
     \dd \left[\left\langle (\hs-\E \hs)r,\zd \right\rangle_\Gammas \right] , 
\end{align}
where $[\langle \cdot,\cdot \rangle]$ denotes covariation of the
inner-product $\langle \cdot,\cdot \rangle$. 
Note that (\ref{eq:data_ct}b) implies that the quadratic variation of $\zd$ is given by the matrix identity 
$[\zd,\zd] = t\Gammas$. Furthermore, because of (\ref{eq:KSD_S2}), the conversion formula \eqref{eq:conversion_IS} and the fact that $\dd [\zd,\zd] = \Gammas \dd t$, it follows that the covariation between scalar $r$ and  vector $\zd$ satisfies the vector identity 
\begin{equation}
\label{eq:VI}
\dd[r,\zd] =  (\hs-\E \hs)r \dd t.
\end{equation}
Consider $f: \R \to \R^{d_y}$ so that derivative $f'(r)$ may be identified with an element in $\R^{d_y}.$ Then, for any such differentiable $f$,
\begin{equation}
\label{eq:citeinbib}
\dd[\langle f(r),\zd\rangle_\Gammas] = \langle f'(r), \dd[r,\zd] \rangle_\Gammas.
\end{equation}
We now apply this identity in the setting where $f(r) = (\hs-\E \hs)r$, noting that $r$ defines
the expectation in this definition. Thus
\begin{equation*}
    f'(r)\,\delta r = (\hs-\E \hs)\,\delta r - \Bigl( \int \hs \,\delta r \,\dd v\Bigr)r. 
\end{equation*}
Hence the covariation in (\ref{eq:conversion_IS}) satisfies, using \eqref{eq:citeinbib} for the first equality and \eqref{eq:VI} for the second equality,
\begin{align*}
  \dd [\langle (\hs-\E \hs)r,\zd \rangle_\Gammas ] &=
  \langle \hs-\E \hs ,\dd[r,\zd]\rangle_\Gammas - \Bigl(\int \langle \hs,\dd[r,\zd]\rangle_\Gammas  \dd v\Bigr) r\\
  %&= r \left| \hs-\E \hs\right|^2_\Gammas dt - r \int \langle \hs,d[r,\zd]\rangle_\Gammas dv\\
  &=  \left| \hs-\E \hs\right|^2_\Gammas r \dd t - \,\E \bigl(\langle \hs,\hs-\E \hs\rangle_\Gammas\bigr) r \dd t \\
  &= \left\{ \left| \hs-\E \hs\right|^2_\Gammas - \E \left| \hs-\E \hs\right|^2_\Gammas \right\}r  \dd t.
\end{align*}
Using this identity in (\ref{eq:conversion_IS}) and rearranging, we find that
\begin{align*} 
    \left\langle \hs- \E \hs, \circ \dd\zd \right\rangle_\Gammas r - \frac{1}{2} \left\{ \left| \hs \right|^2_\Gammas - \E \left| \hs \right|^2_\Gammas\right\}r \dd t =
    \left\langle \hs- \E \hs, \dd\zd - \E \hs \dd t\right\rangle_\Gammas r  ,
\end{align*}
which in turn leads to the following desired
It\^o representation of (\ref{eq:KSD_S2}):
\begin{align} \label{eq:KSD}
    \dd r =  \left\langle \hs - \E \hs, \dd \zd - \E \hs \dd t\right\rangle_\Gammas r .
\end{align}
\end{proof}

Using the It\^o form of the equation from the preceding lemma, the Fokker-Planck equation (\ref{eq:FPE})
and the split-step principle delivers the following:
\begin{theorem}
\label{thm:KS_equation}
Assume that $\Gammas \succ 0$ 
and that $\zd$ is given by \eqref{eq:data_ct}. The time evolution of the density $\mmu (\cdot,t)$ for the random variable $v(t)|Z^\dagger(t)$ is characterized by the It\^o SPDE
\begin{empheq}[box=\fbox]{equation}\label{eq:KSE}
\dd\mmu = -\nabla \cdot (\mmu f) \dd t + \frac{1}{2}
\nabla \cdot (\nabla \cdot (\mmu \Sigmas))  \dd t +  \left\langle \hs - \E \hs, \dd\zd - \E \hs \dd t\right\rangle_\Gammas \mmu.
\end{empheq}
$\Diamond$ \end{theorem}

Equation \eqref{eq:KSE} is known as the Kushner--Stratonovich equation. The equation \eqref{eq:KSE} is to be interpreted 
in the It\^o sense with respect to the driving noise $\zd (t)$. The equation is nonlinear, unlike the unconditioned dynamics governed
by the Fokker-Planck equation \eqref{eq:FPE}.
The corresponding Stratonovich formulation follows immediately from (\ref{eq:KSD_S2}) in combination with (\ref{eq:FPE}):
\begin{subequations}\label{eq:KSEadded}
\begin{align}
\dd\mmu &= \Bigl(-\nabla \cdot (\mmu f)  + \frac{1}{2}
\nabla \cdot (\nabla \cdot (\mmu \Sigmas))
- \frac{1}{2} \left\{ \left| \hs \right|^2_\Gammas - \E \left| \hs \right|^2_\Gammas\right\}\rho\Bigr)\dd t\\
&\hspace{2.75in}+  \left\langle \hs - \E \hs, \circ \dd\zd \right\rangle_\Gammas \mmu.
\end{align}
\end{subequations}

\begin{example}
\label{ex:KBF}
We consider the setting where
\begin{equation}
\label{ex:linearc}
f(\cdot) = F \cdot, \quad \hs(\cdot) = \Hs\cdot.
\end{equation}
in equations \eqref{eq:sdco} so that
\begin{subequations}
\label{eq:sdcoz}
\begin{empheq}[box=\widefbox]{align}
\dd v &= Fv \dd t + \sqrt{\Sigmas}\dd W, \quad v(0) \sim \Ng(m_0, C_0), \\
\dd z &= \Hs v \dd t + \sqrt{\Gammas}\dd B, \quad z(0)=0.
\end{empheq}
\end{subequations}
Now consider data $\zd(t)$ generated by 
\begin{subequations}
\label{eq:data_ctz}
\begin{align}
\dd \vd &= F\vd \dd t + \sqrt{\Sigmas}\dd W^\dagger, \quad \vd(0) \sim \Ng(m_0, C_0), \\
\dd\zd &= \Hs \vd \dd t + \sqrt{\Gammas}\dd B^\dagger, \quad \zd(0) =0.
\end{align}
\end{subequations}
We are interested in he filtering distribution for $v(t)|\Zd(t)$.

Becase of the linearity and additive Gaussian noise, this
filtering distribution is Gaussian and is
given by the Kalman-Bucy filter, the continuous time 
analog of the Kalman filter described in Example \ref{ex:sssec:2}. 
Indeed if $\Gammas \succ 0$ in \eqref{eq:sdcoz} 
then the density $\mmu (\cdot,t)$ associated with the random variable 
$v(t)|Z^\dagger(t)$, with $\zd$ given by \eqref{eq:data_ctz},  
is Gaussian with mean $m(\cdot)$
and covariance $C(\cdot)$ given by 
\begin{subequations}
\label{eq:KB_cov_evol}
\begin{empheq}[box=\widefbox]{align}
\dd m &= Fm \dd t +\Cov \Hs^\top\Gammas^{-1}\left( \dd z^{\dag}-\Hs m \dd t\right),\quad m(0)=m_0,\\
\dd\Cov &= F\Cov \dd t +\Cov F^\top \dd t + \Sigmas \dd t - \Cov \Hs^\top\Gammas^{-1}\Hs\Cov \dd t, \Cov (0)=C_0.
\end{empheq}
\end{subequations}
These equations for the mean and covariance of
 the Kalman-Bucy filter may be obtained by taking the continuous time limit of the Kalman filter from Example \ref{ex:sssec:2} under the scalings (\ref{eq:rescalings}b), \eqref{eq:rescalings2} and (\ref{eq:reparam}b).

The Kalman-Bucy filter also yields an exact solution of the
Kushner-Stratonovich equation \eqref{eq:KSE} under \eqref{ex:linearc}.
Indeed, if $\rho(\cdot,0)$ is initialized at the Gaussian
with mean $m_0$ and covariance $C_0$ then $\rho(\cdot,t)$ solving the Kushner-Stratonovich equation \eqref{eq:KSE} has solution given by the Gaussian $\Ng\bigl(m(t),C(t)\bigr).$
$\blacksquare$
\end{example}

%%%%%%%%%%%%%%%%%%%%%%%%%%%%%%%%%%%%%%%%
%
\subsubsection{The Sample Path Perspective}
\label{sssec:CSPP}
%
%%%%%%%%%%%%%%%%%%%%%%%%%%%%%%%%%%%%%%%%%

Similarly to the discrete time setting, a key idea in developing algorithms
for filtering in continuous time is the \emph{sample path perspective.}
We will seek to identify (for implementable algorithms only approximately) mean field SDEs
with the property that their solution $v(t)$, has $\law\bigl(v(t)\bigr)$ equal to that given by the density $r$ governed by the Kushner-Stratonovich equation \eqref{eq:KSE}. Analogously to discrete time the mean field SDEs will couple to the solution of the equation for evolution of the density; the resulting processes are termed \emph{nonlinear Markov processes.}
We will first introduce Gaussian projected filters and then discuss mean field models
derived through the sample path perspective. To these ends we now introduce some
important notational conventions.

%%%%%%%%%%%%%%%%%%%%%%%%%%%%%%%%%%%%%%%%%%%%
%
\subsubsection{Important Repeatedly Used Notation}
\label{ssec:irun}
%
%%%%%%%%%%%%%%%%%%%%%%%%%%%%%%%%%%%%%%%%%%%%%%

The notation we define here is used primarily to discuss
exact and approximate mean field models, here evolving in
continuous time, to (approximately) solve the filtering problem.
In discrete time expectations may be taken under the law of
the (possibly approximate) discretely evolving mean field model, 
or under the predictive distribution found from pushing 
this law forward under the model.
This distinction disappears in continuous time and we simply
need to compute expectations under the (possibly approximate) 
continuously evolving mean field models that we will introduce later. To this end we define, with
$\mean:=\mathbb{E}v$,
\begin{subequations}
\label{eq:KF_pred_mean2c}
\begin{empheq}[box=\widefbox]{align}
\Cov &:=\mathbb{E}\bigl((v-\mean)\otimes(v-\mean)\bigr),\\
\Cov^{vf} &:=\mathbb{E}\Bigl(\bigl(v-\mean\bigr)\otimes\bigl(f(v)-\mathbb{E}f(v)\bigr)\Bigr),\\
\Cov^{v\hs} &:=\mathbb{E}\Bigl(\bigl(v-\mean\bigr)\otimes\bigl(\hs(v)-\mathbb{E}\hs(v)\bigr)\Bigr).
\end{empheq}
\end{subequations}
All expectations are under the mean field model for $v$.
The covariances should be viewed as functions of time $t$. 
In deriving continuous time models
from discrete time models, we will also use discrete
time analogs, computed under the law of random variable $v_n$ and denoted
$\Cov_n, \Cov_n^{vf}$ and $\Cov_n^{v\hs}.$

%%%%%%%%%%%%%%%%%%%%%%%%%%%%%%%%%%%%%%%%%%%%%%%%%%%%%%%%
%
\subsection{Gaussian Projected Filtering Distribution}
\label{ssec:GPFDc}
%
%%%%%%%%%%%%%%%%%%%%%%%%%%%%%%%%%%%%%%%%%%%%%%%%%%%

As in discrete time, the Gaussian projected filtering distribution
plays an important conceptual role in understanding later filtering
algorithms. The evolution equations for the mean $m$ and the covariance matrix $C$ follow naturally from a continuous time limit of the associated discrete time filtering formulations. We summarize the resulting equations in the following:
\begin{theorem}
\label{t:34}
Assume that $\Gammas \succ 0.$ and that $\zd$ is given by \eqref{eq:data_ct}. Consider the discrete time Gaussian projected filter, namely map from $(m_n,C_{n})$ to $(m_{n+1},C_{n+1})$ defined by choosing $v_n \sim \Ng(m_n,C_{n})$ and then using \eqref{eq:recall}, \eqref{eq:KF_pred_mean}, \eqref{eq:KF_joint2b} and
\eqref{eq:KF_analysis_add}. Under the rescalings \eqref{eq:rescalings}, and in the limit $\dt \to 0$, we obtain the following continuous-time limit of this map:
\begin{subequations}
\label{eq:meanGPFcont}
\begin{empheq}[box=\widefbox]{align}
\dd m &= \mathbb{E}f(v)\dd t + \Cov^{v\hs}\Gammas^{-1}\bigl(\dd\zd - \E \hs(v)\dd t\bigr),\\
\dd\Cov &= \Cov^{vf}\dd t + (\Cov^{vf})^\top \dd t +\Sigmas \dd t - \Cov^{v\hs}\Gammas^{-1}(\Cov^{v\hs})^\top \dd t,
\end{empheq}
\end{subequations}
where the expectations are computed under $\Ng\bigl(m(t),C(t)\bigr)$, using \eqref{eq:KF_pred_mean2c}.
$\Diamond$ \end{theorem}

\begin{remark}
\label{rem:gaincts}
Note that the preceding equation implicitly defines gain matrix
\begin{equation}
\label{eq:gaincts}
K=\Cov^{v\hs}\Gammas^{-1}.
\end{equation}
This specific choice of gain will play a central role in what follows.
$\blacksquare$
\end{remark}

\begin{proof}(Theorem \ref{t:34})
The Gaussian projected filter is defined by
evolution of mean and covariance given by equations
\eqref{eq:KF_pred_mean}, \eqref{eq:KF_joint2b} and 
\eqref{eq:KF_analysis_add}, repeated and reordered here for convenience:
\begin{align*}
    \pmean_{n+1} &= \E \Psi(v_n),\\
\pCov_{n+1} &= \E \Bigl(\bigl(\Psi(v_n)-\pmean_{n+1})\otimes
\bigl(\Psi(v_n)-\pmean_{n+1}\bigr)\Bigr)+\Sigma,\\
\mean_{n+1} &= \pmean_{n+1} + \pCov_{n+1}^{vh}
(\pCov_{n+1}^{hh}+\Gamma)^{-1}
%\bigl(\yd_{n+1} - \E h(\hv_{n+1})\bigr),\\
\bigl(\yd_{n+1} - \E\hh_{n+1}\bigr),\\
\Cov_{n+1} &= \pCov_{n+1} - \pCov_{n+1}^{vh}(\pCov_{n+1}^{hh}+\Gamma)^{-1} \left({\pCov_{n+1}^{vh}}\right)^\top,\\
     \pCov_{n+1}^{vh} &=\E\Bigl(\bigl(\hv_{n+1}-\E \hv_{n+1}\bigr)\otimes
\bigl(\hh_{n+1}-\E \hh_{n+1}\bigr)\Bigr),\\
    \pCov_{n+1}^{hh}& =   \E\Bigl(\bigl(\hh_{n+1}-\E \hh_{n+1}\bigr)\otimes
\bigl(\hh_{n+1}-\E \hh_{n+1}\bigr)\Bigr).
\end{align*}
Recall that $\hh_{n+1}:=h(\hv_{n+1}).$ Expectations in the prediction step
are with respect to the law of $v_n  \sim \Ng(m_n,C_n); $ and expectations in the analysis 
step are with respect to the law of $\hv_{n+1}$ given by (\ref{eq:recall}a), 
assuming that $v_n \sim \Ng(m_n,C_n);$ thus $\law (\hv_{n+1})=\op P\,\law(v_n)$.

We now impose the rescalings \eqref{eq:rescalings} on these equations.
The reader's attention is drawn to the fact that $h=\dt \hs$ when reading the next formula in order to understand the scalings with $\dt;$ similar notational shift from regular to mathsf occurs in other formulae that follow it and are crucial to understanding
orders of magnitude with respect to $\dt.$
Under the rescalings and
using that $\xi_{n+1}=\mathcal{O}(\Delta t^{\frac12})$ has mean zero, 
\begin{align*}\pCov_{n+1}^{vh}&=\dt\E \left(\bigl(v_n-\E v_n+\xi_n+\mathcal{O}(\dt)\bigr) \otimes \bigl(\hs(v_n)-\E \hs(v_n)+D\hs(v_n)\xi_n+\mathcal{O}(\dt)\bigr)\right)\\
&=\dt\E \left(\bigl(v_n-\E v_n\bigr) \otimes \bigl(\hs(v_n)-\E \hs(v_n)\bigr)\right)+\mathcal{O}(\dt^{2})\\
&=\dt \Cov^{v\hs}_n+\mathcal{O}(\dt^{2}). 
\end{align*}
Similarly 
\begin{align*}
    \pCov_{n+1}^{hh}& =   \dt^2 \E \left(\bigl(\hs(v_n)-\E \hs(v_n)+\mathcal{O}(\dt^{\frac12})\bigr)
    \otimes \bigl(\hs(v_n)-\E \hs(v_n)+\mathcal{O}(\dt^{\frac12})\bigr)\right)\\
    &=\mathcal{O}(\dt^{2}).
\end{align*}
Thus, since $\Gamma = \dt \Gammas$,
\begin{subequations}
\label{eq:labelz}
\begin{align}
\pCov_{n+1}^{vh}(\pCov_{n+1}^{hh}+\Gamma)^{-1}&=\Cov^{v\hs}_n\Gammas^{-1}+\mathcal{O}(\dt),\\
\pCov_{n+1}^{vh}(\pCov_{n+1}^{hh}+\Gamma)^{-1}\pCov_{n+1}^{vh}&=\dt \Cov^{v\hs}_n\Gammas^{-1}(\Cov^{v\hs}_n)^\top+\mathcal{O}(\dt^{2}).
\end{align}
\end{subequations}
Furthermore, since $\Psi(v_n) = v_n + \dt f(v_n)$ and $\pmean_{n+1} = \E v_n + \dt \E f(v_n)$,
\begin{align*}
    \E \Bigl(\bigl(\Psi(v_n)&-\pmean_{n+1}\bigr)\otimes
\bigl(\Psi(v_n)-\pmean_{n+1}\bigr)\Bigr)
%&=
%\E \Bigl(\bigl(v_n+\dt f(v_n)-\E v_n-\dt\E f(v_n)\bigr)\otimes
%\Bigl(\bigl(v_n+\dt f(v_n)-\E v_n-\dt\E f(v_n)\bigr)\Bigr)\\
=C_n+\dt \Cov_n^{vf}+ \dt(\Cov_n^{vf})^\top +\mathcal{O}(\dt^{2}).
\end{align*}
Using these approximations, and the fact that
$\Delta \zd_{n+1}=\mathcal{O}(\Delta t^{\frac12})$,
we obtain
\begin{align*}
\pmean_{n+1} &= m_n +\dt \mathbb{E}f(v_n),\\
\pCov_{n+1} &=  \Cov_n +\dt \Cov_n^{vf}+ \dt(\Cov_n^{vf})^\top  +\dt\Sigmas+\mathcal{O}(\dt^2),\\
m_{n+1} &= \pmean_{n+1}  + \Cov^{v\hs}_n\Gammas^{-1}\bigl(\Delta\zd_{n+1} - \dt\E \hs(\hv_{n+1})\bigr)+\mathcal{O}(\dt^{\frac32}),\\
\Cov_{n+1} &= \pCov_{n+1}  - \dt\Cov^{v\hs}_n\Gammas^{-1}\left(\Cov^{v\hs}_n\right)^\top +\mathcal{O}(\dt^2),\\
\Cov^{vf}_n &=\mathbb{E}\bigl((v_n-\mean_n)\otimes(f(v_n)-\mathbb{E}f(v_n))\bigr),\\
\Cov^{v\hs}_n &=\mathbb{E}\bigl((v_n-\mean_n)\otimes(\hs(v_n)-\mathbb{E}\hs(v_n))\bigr).
\end{align*}
In the following  $\Cov^{v\hs}, \Cov^{vf}$ are functions of time, defined by \eqref{eq:KF_pred_mean2c}, and $\Cov_n^{v\hs}, \Cov_n^{vf}$ are discrete time
analogs computed under the law of $v_n.$ 

Combining the prediction and analysis steps we find that
\begin{align*}
m_{n+1} &= m_n +\dt \mathbb{E}f(v_n) + \Cov_n^{v\hs}\Gammas^{-1}\bigl(\Delta\zd_{n+1} - \dt\E \hs(\hv_{n+1})\bigr)+\mathcal{O}(\dt^\frac32),\\
\Cov_{n+1} &= \Cov_n + \dt\Cov^{vf}_n + \dt\left(\Cov^{vf}_n\right)^\top +\dt\Sigmas - \dt\Cov^{v\hs}_n\Gammas^{-1}\left(\Cov^{v\hs}_n\right)^\top+\mathcal{O}(\dt^2).
\end{align*}
Taking the $\dt \rightarrow 0 $ limit, we deduce the continuous-time analog of the equations \eqref{eq:KF_pred_mean}, \eqref{eq:KF_joint2b}, namely \eqref{eq:meanGPFcont} as desired.
\end{proof}

\begin{example}
In the linear setting \eqref{ex:linearc}, the Gaussian projected filter 
\eqref{eq:meanGPFcont} reduces to the Kalman-Bucy filter from 
Example \ref{ex:KBF}.
\end{example}

%%%%%%%%%%%%%%%%%%%%%%%%%%%%%%%%%%%%%%%%%%%%%%%%%%%%%%%%
%
\subsection{Mean Field Evolution Equations}
\label{ssec:CTM_MFD}
%
%%%%%%%%%%%%%%%%%%%%%%%%%%%%%%%%%%%%%%%%%%%%%%%%%%%%%%

In the preceding subsection we approximated the evolution
of the filtering distribution by the evolution of a Gaussian.
Here we take a different approach, seeking a sample path
perspective, identifying mean field SDEs which (possibly only
approximately) have solutions with law given by the filtering distribution.
We develop a continuous time analog of the discrete
time mean field approach from Subsection \ref{ssec:MFM}.

In Subsection \ref{sssec:PTc} we describe work concerning the
derivation of explicit mean field SDEs equal in law to the
filtering distribution; this is a departure from our discussion
of this topic in discrete time where no explicit maps were 
identified in the general setting. Subsections \ref{sssec:T_Sc}
and \ref{sssec:T_Dc} describe a variety of explicit approximate
mean field SDEs, based on matching first and second order moment information,
and arising from rescaling of the discrete time setting.

%%%%%%%%%%%%%%%%%%%%%%%%%%%%%%%%%%%%%%%%%%%%%%%
%
\subsubsection{Perfect Transport}
\label{sssec:PTc}
%
%%%%%%%%%%%%%%%%%%%%%%%%%%%%%%%%%%%%%%%%%%%%%%%%
Here we seek a mean field dynamical system with law given
by that of the Kushner-Stratonovich equation. The
analog of the discrete-time transport map $\Ts_n$, is 
to find transport evolution equations for $v$ and $\hz$
in the form a mean field SDE. We will seek to achieve 
this in the specific sample path form
\begin{subequations}
\label{eq:KS_mf}
\begin{align}
\dd v &= f(v)\dd t + \sqrt{\Sigmas} \dd W + a(v;\mmu)\dd t + K(v;\mmu)(\dd\zd-\dd\hz),\\
\dd\hz &= \hs(v)\dd t + \sqrt{\Gammas} \dd B,
\end{align}
\end{subequations}
where $\mmu$ is the time-dependent density of $v$. Note that
(\ref{eq:KS_mf}b) simply generates simulated data from the state $v$.
On the other hand (\ref{eq:KS_mf}a) combines the underlying known
model evolution with a nudged innovation term, based on the
difference between the observed and simulated data, and a correction
to the drift $f$. The nudging and correction terms are defined by
gain $K$ and drift correction $a$.
The goal is to chose drift correction $a$ and gain $K$ such that the induced time evolution of the density $\mmu$ of $v$ agrees with the density
$\mmu$ given by the Kushner-Stratonovich equation (\ref{eq:KSE}). 
We thus find a controlled SDE, similar in form to that proposed in
(\ref{eq:sd2nc}) but with an additional drift term and with mean field dependence.

\begin{remark}
We emphasize that, as in the discrete-time transport formulation,
there is a considerable degree of non-uniqueness in the choice
of transport in general, and here in the choice of $a$ and $K$
specifically. We make specific, simple, choices in the theorem that follows.
Working in continuous time enables very explicit identification of
exact mean field models; this is not possible in discrete time.
$\blacksquare$
\end{remark}

The following theorem identifies a sample path perspective that
exactly captures evolution of the filtering distribution.

\begin{theorem}
\label{thm:MF_KS_equal}
Assume that $\Gammas \succ 0$ and that there exists $K=K(v;\mmu)$ satisfying the identity
\begin{align} \label{eq:Poisson}
    -\nabla \cdot (\mmu K^\top) =  \Gammas^{-1}(\hs-\E \hs)\mmu.
\end{align}
Consider the stochastic mean field dynamics given by 
\begin{subequations}
\label{eq:KS_mf3}
\begin{empheq}[box=\fbox]{align}
\dd v & = f(v)\dd t + \sqrt{\Sigmas} \dd W +\nabla \cdot (K\Gammas K^\top)\dd t - K\Gammas \nabla \cdot K^\top \dd t + K (\dd\zd - \dd\hz),\\
\dd\hz &= \hs(v)\dd t + \sqrt{\Gammas} \dd B,
\end{empheq}
\end{subequations}
where $\zd$ given by \eqref{eq:data_ct}.
Assume that solution $v$ has law with smooth density and that
the Kushner-Stratonovich equation \eqref{eq:KSE} has smooth 
density as solution. Then the law of $v$ has density given 
by the Kushner-Stratonovich equation \eqref{eq:KSE}.
$\Diamond$ \end{theorem}
\begin{proof}
To simplify calculations we will first look at choosing $a$ and $K$
to get agreement, at the level of densities of $v$, 
with (\ref{eq:KSD}). A straightforward modification, using
the split-step principle again, then 
provides the generalization to (\ref{eq:KSE}).
This split-step approach enables us to consider only the case where $f(\cdot) \equiv 0$ and $\Sigmas=0.$

Note that, in \eqref{eq:KS_mf}, $\zd$ is a fixed, given, trajectory
and we are interested in the evolution of the probability density
induced for $v$ by the randomness over the distribution on trajectories $\hz$. Using (\ref{eq:KS_mf}b) 
in (\ref{eq:KS_mf}a), and recalling that it suffices to set $f$ and $\Sigmas$ to zero, we obtain
\begin{align} \label{eq:KS_mfn}
    \dd v &= a(v;\mmu)\dd t + K(v;\mmu)\bigl(\dd\zd-\hs(v)\dd t-\sqrt{\Gammas}\dd B\bigr).
\end{align}
Applying Fokker-Planck analysis, modified to the mean field
setting, shows that the time evolution of the density $\mmu$ of $v$ under (\ref{eq:KS_mf}) is provided by the 
nonlinear (because of dependence of $a,K$ on $\mmu$) SPDE\footnote{Here we use the standard convention from
continuum mechanics that the divergence of a matrix is to be
interpreted via computation of derivatives with respect to the second index; see
Subsection \ref{ssec:BCT} for references to relevant continuum mechanics
textbooks. Strictly speaking equation \eqref{eq:FPE_mf} is an SPDE only if we now view $\zd$ as a random variable rather than a fixed realization.} to be interpreted in the It\^o sense:
\begin{align} \label{eq:FPE_mf}
    \dd\mmu &= -\nabla \cdot (\mmu (a-K\hs))\dd t -
    \langle\nabla \cdot (\mmu K^\top), \dd\zd \rangle + 
    \nabla \cdot (\nabla \cdot (\mmu K\Gammas K^\top))\dd t.
\end{align}
Note that, although $\zd$ is a fixed trajectory, it contributes to the diffusion term in this Fokker-Planck equation, explaining the factor $1$ rather than $1/2.$ This arises as a contribution from
the quadratic variation of the path of $\zd$ to the evolution of $\mmu.$ For further details on the derivation of \eqref{eq:FPE_mf} see Lemma \ref{lm::FP_analysis} in Appendix \ref{sec:AC}. To get agreement with (\ref{eq:KSD}), $a$ and $K$ have to be chosen such that
\begin{align*}
&\mmu  \langle (\hs - \E \hs),\Gammas^{-1}(\dd\zd - \E \hs \dd t)\rangle =\\ &\qquad \qquad 
    -\nabla \cdot (\mmu (a-K\hs))\dd t -
    \langle \nabla \cdot (\mmu K^\top), \dd\zd \rangle +
    \nabla \cdot (\nabla \cdot (\mmu K\Gammas K^\top)) \dd t .
\end{align*}
Equating the two terms involving the data
$\zd$ shows immediately that $K(v;\mmu)$ has to satisfy the (vector-valued)  PDE
\eqref{eq:Poisson}.
From this, equating the terms that do not involve the data $\zd$,
it follows that $a(v;\mmu)$ has to satisfy
\begin{align*} %\label{eq:drift_a}
\nabla \cdot (\mmu a) &= 
\nabla \cdot (\nabla \cdot (\mmu K\Gammas K^\top)) + \nabla \cdot (\mmu K(\hs-\E \hs))\\
& = \nabla \cdot (\nabla \cdot (\mmu K\Gammas K^\top)) - \nabla \cdot (K \Gammas 
\nabla \cdot (\mmu K^\top))\\
&= \nabla \cdot \left( \mmu \left( \nabla \cdot (K\Gammas K^\top) - K\Gammas \nabla \cdot K^\top \right) \right).
\end{align*}
A natural choice for $a$ is provided by asking that the term
on which the divergence acts is zero. This yields
\begin{align}\label{eq:drift_FPF}
    a = \nabla \cdot (K\Gammas K^\top) - K\Gammas \nabla \cdot K^\top,
\end{align}
a solution for $a$ with no explicit dependence on $\mmu;$ note, however, that $a$ does depend on $\mmu$ implicitly through
the dependence of $K$ on $\mmu.$

With these choices of $a,K$, and applying the split-step principle
so that the mean field dynamics are consistent with \eqref{eq:KSE}
rather than \eqref{eq:KSD}, we obtain a version of the
{\em feedback particle filter}; in particular we find
the equation in its stochastic mean field formulation given by
\eqref{eq:KS_mf3} where $K=K(v;\mmu)$ solves \eqref{eq:Poisson} 
and $\mmu$ evolves according to the Kushner-Stratonovich 
equation \eqref{eq:KSE}, an equation which also defines the law of $v.$
\end{proof}

\begin{remark}
\label{rem:diamond}
There is an interesting interpretation of the contribution $a$
to the drift term in \eqref{eq:KS_mf3}: it is simply the It\^o-to-Stratonovich-like correction with regard to the $v$-dependence in $K(v;\mmu)$. However there is a subtlety that the equation for $\mmu$ itself depends on the data $\zd$ and 
the correction does not account for the $\mmu-$dependence of
$K(v;\mmu)$ -- the Stratonovich correction is only with respect
to $v-$dependence of the drift; an It\^o interpretation is retained
with respect to the $\mmu$ dependence. We refer to Appendix \ref{subsec:diamond} for discussion of this unusual form of stochastic integration. There we also derive the full Stratonovich correction (with respect to both $v$ and $\rho$ dependence) of the exact mean field model.
$\blacksquare$
\end{remark}

\begin{remark}
\label{rem:tbd}
The mean field equations (\ref{eq:KS_mf3}) require a gain $K(v;\rho)$ which satisfies (\ref{eq:Poisson}). Let  $\E$ denote expectation with respect to random variable $v$ distributed according to probability measure with density $\rho$. Using appropriately regular test functions $\psi: \R^{d_v} \to \R^{d_v}$, chosen to have mean-zero under $\E$, equation (\ref{eq:Poisson}) can be rephrased in the weak form
\footnote{
Recall that the conventions from continuum mechanics that we use to define the divergence and gradient of vector fields are discussed in the bibliography Subsection \ref{ssec:BCT}.}
\begin{equation} \label{eq:weak Poisson}
    \E \left( K^\top \nabla \psi\right) = \Gammas^{-1} C^{\hs \psi};
\end{equation}
here $C^{\hs \psi}$ denotes the covariance between $\hs$ under $\E$
and, in what follows, we also let $C^{v \hs}$ denote covariance between $v$ and $\hs$. The particular choice $\psi (v) = v-\E v$ leads to
\begin{equation} \label{eq:constant gain approximation}
    \E K = C^{v \hs} \Gammas^{-1}.
\end{equation}
Note, in particular, that this identity can be satisfied by making the \emph{constant gain} ansatz that $K$ is independent of $v$ and is then given by
\begin{subequations} \label{eq:cga}
\begin{align}
    K &= C^{v \hs} \Gammas^{-1},\\
    \Cov^{v\hs} &=\mathbb{E}\Bigl(\bigl(v-\mathbb{E} v\bigr)\otimes
\bigl(\hs (v)-\mathbb{E} \hs(v)\bigr)\Bigr).
\end{align}
\end{subequations}
More generally speaking, we note that (\ref{eq:weak Poisson}) is amenable to numerical approximations. This will be discussed further in Subsection \ref{sssec:PPFc} below.
$\blacksquare$
\end{remark}

We close this subsection with the mean field formulation using deterministic innovations:
\begin{equation}
\label{eq:KS_mf3d}
\dd v  = f(v)\dd t + \sqrt{\Sigmas} \dd W + \frac{1}{2} \left( \nabla \cdot (K\Gammas K^\top) -  K\Gammas \nabla \cdot K^\top 
\right) \dd t + K \left(\dd\zd - \frac{1}{2} (\hs + \E \hs) \dd t \right),
\end{equation}
where $\zd$ is given by \eqref{eq:data_ct}.
The relationship between this equation and \eqref{eq:KS_mf3}, 
is analogous to the relationship between innovation terms \eqref{eq:innovation_d} and \eqref{eq:innovation_p} already encountered in the discrete time setting. See the bibliography
Subsection \ref{ssec:BCT} for discussion on the historical development
of the various exact mean field models presented here.

%%%%%%%%%%%%%%%%%%%%%%%%%%%%%%%%%%%%%%%%%%%%%%%%%%%%%%%%%%%%%%%%%%%%
%
\subsubsection{Second Order Transport – Stochastic Case}
\label{sssec:T_Sc}
%
%%%%%%%%%%%%%%%%%%%%%%%%%%%%%%%%%%%%%%%%%%%%%%%%%%%%%%%%%%%%%%%%%%%%

Recall that, in the discrete-time setting, there are an uncountable set of maps which effect
approximate transport, in the sense of matching the first
and second order statistics of the analysis map, using either
stochastic or deterministic models. These are elucidated in Appendix \ref{appendix:C}. We, however, concentrated in the main text on a handful of examples. In this and the next subsection we study continuous time analogs of some of these examples, starting with the Kalman transport map \eqref{eq:sd2nn_add} recalled here, and reformulated, for convenience:\footnote{The Gaussian notation
used in the first two equations is shorthand for the first two equations appearing in \eqref{eq:sd2n}, together with the 
assumptions detailed following those equations. Thus we use $\Ng(0,\Sigma)$ to denote
an i.i.d. realization from the stated Gaussian distribution, and similarly for other variables.  We use variants of this notation 
in what follows.}
\begin{subequations} \label{eq:discrete time EnKF}
\begin{align}
\hv_{n+1} &= \Psi(v_n) + \Ng(0,\Sigma), \\
\hy_{n+1} &= h(\hv_{n+1}) + \Ng(0,\Gamma), \\
v_{n+1} &= \hv_{n+1}+\hCvh_{n+1}(\hChh_{n+1}+\Gamma)^{-1}(\yd_{n+1}-\hy_{n+1}).
\end{align}
\end{subequations}
We now apply the rescaling \eqref{eq:rescalings} to obtain, using (\ref{eq:labelz}a),
\begin{align*}
\hv_{n+1} &= v_n+\dt f(v_n) + \sqrt{\dt} \Ng(0,\Sigmas), \\
\hz_{n+1} &= \hz_n+\dt \hs(\hv_{n+1}) + \sqrt{\dt} \Ng(0,\Gammas), \\
v_{n+1} &= \hv_{n+1}+C^{v\hs}_{n}\Gammas^{-1}(\Delta \zd_{n+1}-\Delta \hz_{n+1})+\mathcal{O}(\dt^{3/2}).
\end{align*}
The preceding calculation, leading to $\mathcal{O}(\dt^{3/2})$ error,
uses the fact that the noises entering the equations for $\hv_{n+1},$  
$\hz_{n+1}$  and $z^\dagger_{n+1}$ are independent. 
Taking the continuous time limit we obtain
\begin{subequations}
\label{eq:KSA_mf}
\begin{empheq}[box=\widefbox]{align}
\dd v &= f(v)\dd t + \sqrt{\Sigmas} \dd W + C^{v\hs}\Gammas^{-1} (\dd\zd-\dd\hz),\\
\dd\hz &= \hs(v)\dd t + \sqrt{\Gammas} \dd B.
\end{empheq}
\end{subequations}
Again $W,B$ are independent unit Brownian motions of appropriate dimensions
and $\zd$ is given by \eqref{eq:data_ct}. This is an instance of a sample path perspective that leads to approximation of the true
filtering evolution.

\begin{remark}
\label{rem:derive}
Note the recurrence of the continuous time gain $K$ first identified in Remark \ref{rem:gaincts}.
In this subsection we have derived the mean field equations (\ref{eq:KSA_mf}) from the associated stochastic discrete time formulation (\ref{eq:discrete time EnKF}). However there is another way to derive the gain as outlined in Remark \ref{rem:tbd}.
$\blacksquare$ \end{remark}

%%%%%%%%%%%%%%%%%%%%%%%%%%%%%%%%%%%%%%%%%%%%%%%%%%%%%%%%%
%
\subsubsection{Second Order Transport – Deterministic Case}
\label{sssec:T_Dc}
%
%%%%%%%%%%%%%%%%%%%%%%%%%%%%%%%%%%%%%%%%%%%%%%%%%%%%%%

We may apply a similar analysis to that in Subsection
\ref{sssec:T_Sc} but instead working in the setting of
deterministic transport, starting from the discrete time 
transport \eqref{eq:sd2nn_add_again_add}. We may 
rewrite  and reformulate this equation here to obtain
\begin{align*}
\hv_{n+1} &= \Psi(v_n) + \Ng(0,\Sigma), \quad
\hh_{n+1} = h(\hv_{n+1}), \\
v_{n+1} &= \hv_{n+1}-\tK_{n}(\hh_{n+1}-\bbE \hh_{n+1})+K_n(\yd_{n+1}-\bbE \hh_{n+1}),
\end{align*}
where
\begin{align*}
    K_n &= \pCov_{n+1}^{vh}\left( \pCov_{n+1}^{hh} + \Gamma\right)^{-1},\quad
\tK_{n} = \pCov^{vh}_{n+1}
    \Bigl( (\pCov^{hh}_{n+1}+\Gamma) + \Gamma^{1/2} (\pCov^{hh}_{n+1}+\Gamma)^{1/2}
    \Bigr)^{-1}.
\end{align*}

Applying the rescalings \eqref{eq:rescalings} gives
\begin{align*}
\hv_{n+1} &= v_n+\dt f(v_n) + \sqrt{\dt}\Ng(0,\Sigmas),\\ 
v_{n+1} &= \hv_{n+1}+\Cov_{n}^{v\hs}\Gammas^{-1}\Bigl(\Delta \zd_{n+1}-\frac12
(\hs(\hv_{n+1})+\bbE \hs(\hv_{n+1}) \dt\Bigr)+{\mathcal O}(\dt^2),
\end{align*}
where we have used
\begin{align*}
    K_n &= \Cov_{n}^{v\hs}\Gammas^{-1}+\mathcal{O}(\dt),\quad
\tK_{n} = \frac12 \Cov_{n}^{v\hs}\Gammas^{-1}+\mathcal{O}(\dt). 
\end{align*}
Taking the continuous time limit we obtain
\begin{empheq}[box=\widefbox]{equation}
\label{eq:KSA_mf2}
\dd v = f(v)\dd t + \sqrt{\Sigmas} \dd W + C^{v\hs}\Gammas^{-1} \Bigl(\dd\zd-
\frac12\bigl(\hs(v)+\mathbb{E} \hs(v)\bigr)\dd t\Bigr),
\end{empheq}
where, once again, $\zd$ is given by \eqref{eq:data_ct}.

These mean field equations can also be derived directly from the general mean field equation (\ref{eq:KS_mf3d}), assuming a deterministic innovation,
invoking the constant gain approximation $K=K(\rho)$ and using (\ref{eq:constant gain approximation}). The reasoning is similar to that in Remark \ref{rem:derive}.

\begin{example}
\label{ex:PT_KB}
We now consider mean field formulations of the Kalman-Bucy Filter from 
Example \ref{ex:KBF}. First consider the evolution of random variable 
$v$ initialized at a Gaussian and satisfying
\begin{subequations}
\label{eq:now1}
\begin{align}
\dd v &= Fv \dd t + \sqrt{\Sigmas} \dd W + CH^\top\Gammas^{-1} (\dd\zd-\dd\hz),\\
\dd\hz &= Hv \dd t + \sqrt{\Gammas} \dd B,
\end{align}
\end{subequations}
where $C$ is the covariance of $v$ and where 
$\zd$ is given by \eqref{eq:data_ctz}.
Then direct calculation shows that the mean and covariance of $v$ satisfy equations \eqref{eq:KB_cov_evol}. 
Thus we have identified a sample path perspective leading to a
mean field SDE for $v$ with law equal to that of the Kalman-Bucy filter.

Similar considerations, starting 
at \eqref{eq:KSA_mf2}, give the following variant of the 
preceding mean field Kalman-Bucy filter:
\begin{equation}
\label{eq:now2}
\dd v = Fv \dd t + \sqrt{\Sigmas} \dd W + CH^\top\Gammas^{-1} \Bigl(\dd\zd-
\frac12 H\bigl(v+m\bigr)\dd t\Bigr),
\end{equation}
where $(m,C)$ are the mean and covariance of $v$
and where $\zd$ is given by \eqref{eq:data_ctz}.
Again direct computations verify this.
An important observation highlighted by this example is that
mean field models consistent with a given measure evolution
are not typically unique.
$\blacksquare$
\end{example}

%%%%%%%%%%%%%%%%%%%%%%%%%%%%%%%%%%%%%%%%%%%%%%%%%%%%%%%%
%
\subsection{Ensemble Kalman Methods}
\label{ssec:CTM_MFD_PA}
%
%%%%%%%%%%%%%%%%%%%%%%%%%%%%%%%%%%%%%%%%%%%%%%%%%%%%%%

We now discuss several particle approximations of the 
mean field equations derived in the preceding subsections. 
We start with the mean field equations (\ref{eq:KS_mf3}), based on perfect
transport, before considering particle approximations for time continuous 
approximate sample path, and hence transport, formulations.
Throughout this subsection $\zd$ is given by \eqref{eq:data_ctz}.

%%%%%%%%%%%%%%%%%%%%%%%%%%%%%%%%%%%%%%%%%%%%%%%%%%%%
\subsubsection{Perfect Particle Filters}
\label{sssec:PPFc}
%%%%%%%%%%%%%%%%%%%%%%%%%%%%%%%%%%%%%%%%%%%%%%%%%%%%

As in Subsection \ref{sssec:ppf} we define $\sJ:=\{1,\cdots, J\}.$
The desired approximation to the mean field model \eqref{eq:KS_mf3}
evolves particle ensemble $\{\vj\}_{j\in \sJ}$ and its associated 
empirical measure
\begin{equation}
\label{eq:empmc}
\mu^\sJ(t) =\frac{1}{J}\sum_{j=1}^J \delta_{\vj} 
\end{equation}
according to the interacting particle system
\begin{subequations}
\label{eq:KS_mf3j}
\begin{empheq}[box=\widefbox]{align}
\dd\vj &= f(\vj)\dd t + a^{(j)}\dd t + \sqrt{\Sigmas} \dd\Wj + K^{(j)} (\dd\zd-\dd\hzj),\\
\dd\hzj &= \hs(\vj)\dd t + \sqrt{\Gammas} \dd\Bj,
\end{empheq}
\end{subequations}
where $\zd$ is given by \eqref{eq:data_ct} and
the $\{\Wj\}_{j \in \sJ}$ and $\{\Bj\}_{j \in \sJ}$ 
are mutually independent collections of i.i.d. Brownian motions
in $\R^{d_v}$ and $\R^{d_y}$ respectively. The interaction between the particles arises from the gain matrices
\begin{equation} \label{eq:numerical_gain}
    K^{(j)} := K^\sJ (\vj;\mu^\sJ),
\end{equation}
$j\in \sJ$, where the matrix valued function $K^\sJ$ approximates the solution $K$ of (\ref{eq:Poisson}); 
the drift term $a^{(j)}$ is defined by \eqref{eq:drift_FPF} with $K(v)$ replaced
by  $K^{\sJ}(v;\mu^\sJ)$ and then evaluated at $v=v^{(j)}$. The
key analysis question underlying the particle methodology is
to show that the empirical measure (\ref{eq:empmc}) approximates the law of $v$ satisfying \eqref{eq:KS_mf3}; this type of question
is widely studied for numerous problems in the physical, biological
and social sciences -- see bibliography Subsection \ref{ssec:BCT}.

\begin{remark}
In practice this particle approximation of the perfect particle filter
is impractical except in low dimensional systems. This is because
of the challenge of numerically approximating the equation (\ref{eq:Poisson}) for $K$, or its weak formulation \eqref{eq:weak Poisson}. In this remark we discuss various approximation approaches that make this problem tractable.

We start with the numerical approximation of the weak formulation (\ref{eq:weak Poisson}). Let $\mathcal{F}$ denote a space of vector-valued functions mapping $\R^{d_v}$ into $\R^d$, and that are mean zero with respect to expectation $\E$ under density $\rho$.\footnote{We will consider different choices for
$d$ but use the same notation for the space.}
We make the ansatz that $K^{\top}=\nabla \Psi$, for some
$\Psi \in \mathcal{F}.$ Then (\ref{eq:weak Poisson}) can be rewritten as
\begin{equation*} \label{eq:weak Poisson scalar}
    \E (\nabla \Psi \nabla \psi ) =  \Gammas^{-1} C^{\hs\psi},
\end{equation*}
assumed to hold for all $\psi \in \mathcal{F}.$

In order to obtain the desired numerical approximation $K^\sJ$ let $\E^\sJ$ denote expectation with respect to the empirical measure (\ref{eq:empmc}) and consider $(K^\sJ)^\top = \nabla \Psi^\sJ$ and
\begin{equation} \label{eq:weak Poisson scalar empirical}
    \E^\sJ (\nabla \Psi^\sJ \nabla \psi)  
    = \Gammas^{-1}C^{\hs \psi} 
\end{equation}
with the correlation $C^{\hs \psi}$ approximated by
\begin{equation*}
    \Cov^{\hs \psi} =\mathbb{E}^\sJ\Bigl(
\bigl(\hs (v)-\mathbb{E}^\sJ \hs(v)\bigr) \otimes
\bigl(\psi(v)-\mathbb{E}^\sJ \psi (v) \bigr)\Bigr).
\end{equation*}
The final step in the numerical approximation is to choose an appropriate finite dimensional subspace $\mathcal{F}^L \subset \mathcal{F}$ of dimension $L \ll J$ and to determine 
$\Psi^\sJ \in \mathcal{F}^L$ such that (\ref{eq:weak Poisson scalar empirical}) holds for all $\psi \in \mathcal{F}^L$.

Alternatively we may return to the strong formulation (\ref{eq:Poisson}). We again make the ansatz $K^\top = \nabla \Psi$, giving rise to the differential equation
\begin{equation*}
\label{eq:DM approximation}
    \mathcal{L}_\rho \Psi = -\Gammas^{-1}(\hs - \E \hs)
\end{equation*}
with differential operator $\mathcal{L}_\rho$\footnote{Not to be confused
with operator $\op L_n$ defined by viewing Bayes Theorem, within filtering,
as a prior to posterior map at time $n$.}
 defined by
\begin{equation*}
    \mathcal{L}_\rho \Psi = \rho^{-1} \nabla \cdot( \rho \nabla \Psi)
    = \Delta \Psi + \nabla \Psi \nabla \log \rho.
\end{equation*}
Note that in the case of scalar observations, $\mathcal{L}_\rho$ is the infinitesimal generator of a diffusion process with invariant density $\rho$; in the vector case the same statement holds component-by-component. This fact may be used to approximate the action of $\mathcal{L}_\rho$ via its heat semigroup, and this may be used as the basis for numerical approximations. See the bibliography Subsection \ref{ssec:BCT} for references to the relevant literature.
$\blacksquare$
\end{remark}

The two general approaches to 
approximate the gain matrix in (\ref{eq:numerical_gain}), outlined
in the preceding remark, are interesting theoretically and perhaps hold promise in the future as computer power grows, but they lead to very expensive computations. To address this we now return to the constant gain approximation from Remark \ref{rem:tbd}, now in the particle setting. In the current context the identification of $K^\sJ$ arises from (\ref{eq:weak Poisson scalar empirical}), by making the  choice $\Psi(v) = (K^\sJ)^\top (v-\E^\sJ v)$ and $\psi (v) = v-\E^\sJ v$. These choices result in the following 
approximation of \eqref{eq:cga}:
\begin{subequations} \label{eq:numerical_constant_gain}
\begin{align}
    K^\sJ &= C^{v\hs}\Gammas^{-1},\\
    \Cov^{v\hs} &=\mathbb{E}^\sJ\Bigl(\bigl(v-\mathbb{E}^\sJ v\bigr)\otimes
\bigl(\hs (v)-\mathbb{E}^\sJ \hs(v)\bigr)\Bigr).
\end{align}
\end{subequations}
Note that the constant gain approximation \eqref{eq:numerical_constant_gain} also implies that the drift term $a^{(j)}$ in \eqref{eq:KS_mf3j} vanishes. We summarize numerical implementation details in the following two subsections.

%%%%%%%%%%%%%%%%%%%%%%%%%%%%%%%%%%%%%%%%%%%%%%%%%%%%%%%%
\subsubsection{Stochastic Ensemble Kalman Filters}
\label{sssec:SEKFc}
%%%%%%%%%%%%%%%%%%%%%%%%%%%%%%%%%%%%%%%%%%%%%%%%%%%%%%%

We now consider the mean field model \eqref{eq:KSA_mf} and its 
numerical approximation. Since the drift correction $a$ does not
appear here the only significant difference, in comparison
with the interacting particle approximation (\ref{eq:KS_mf3j}), arises from the choice of the interaction term, that is, $K^{(j)}$. This gain is now independent of $j$ and determined by \eqref{eq:numerical_constant_gain}.
In summary, we obtain the following SDEs: for $j \in \sJ:=\{1,\dots, J\}$ we have
\begin{subequations}
\label{eq:KSA_mfj}
\begin{empheq}[box=\widefbox]{align}
\dd\vj &= f(\vj)\dd t + \sqrt{\Sigmas} \dd\Wj + C^{v\hs}\Gammas^{-1} (\dd\zd-\dd\hzj),\\
\dd\hzj &= \hs(\vj)\dd t + \sqrt{\Gammas} \dd\Bj.
\end{empheq}
\end{subequations}
Here, again, $\zd$ is given by \eqref{eq:data_ct} and
the $\{\Wj\}_{j \in \sJ}$ and $\{\Bj\}_{j \in \sJ}$
are mutually independent collections of i.i.d. Brownian motions.
The collection $\{\vj\}_{j=1}^J$ provide a time-evolving
ensemble which approximates the filtering distribution via
\eqref{eq:empmc}; the equations are derived based on
use of second order transport approximation of perfect transport. Each equation for $\vj$ evolves according to the underlying dynamics model, together with a nudging term  based on the difference between simulated data $\hzj$ and observed data $\zd.$ The gain couples the particles together.

%%%%%%%%%%%%%%%%%%%%%%%%%%%%%%%%%%%%%%%%%%%%%%%%%
%
\subsubsection{Deterministic Ensemble Kalman Filters}
\label{sssec:DEKF}
%
%%%%%%%%%%%%%%%%%%%%%%%%%%%%%%%%%%%%%%%%%%%%%%%%%%

Similarly we may make an empirical approximation of the mean field model \eqref{eq:KSA_mf2} as follows:
for $j \in \sJ:=\{1,\dots, J\}$ we consider
\begin{empheq}[box=\fbox]{equation}
\label{eq:KSA_mf2j}
\dd\vj = f(\vj)\dd t + \sqrt{\Sigmas} \dd\Wj + C^{v\hs}\Gammas^{-1} \Bigl(\dd\zd-
\frac12\bigl(\hs(\vj)+\mathbb{E}^\sJ \hs(v)\bigr)\dd t\Bigr).
\end{empheq}
The notation is as in the preceding subsection and, in particular,
the formula \eqref{eq:empmc} again gives the 
particle approximation of the approximate filter; now
the relevant approximate
filter is defined by the distribution of \eqref{eq:KSA_mf2j}.

%%%%%%%%%%%%%%%%%%%%%%%%%%%%%%%%%%%%%%%%%%%%%%%%%
%
\subsection{Bibliographical Notes}
\label{ssec:BCT}
%
%%%%%%%%%%%%%%%%%%%%%%%%%%%%%%%%%%%%%%%%%%%%%%%%%%

The entirety of Section \ref{sec:CT} is  
framed in the language of SDEs; see
\citet{evans2012introduction}, \citet{oksendal2013stochastic} for
background in this area. In passing from discrete to continuous
time we often invoke ideas from the numerical solution of
SDEs; see \citet{kloeden1991numerical} and
\citet{higham2001algorithmic} for introductions to this area.
The conventions from continuum mechanics, that we use to define the divergence and gradient of vector fields, are the same as those adopted, and described in detail, in \citet{gonzalez2008first}, \citet{gurtin1982introduction}.

Next we overview literature in the control theoretic approach introduced in Subsection \ref{ssec:CTCT}. For a general overview of control theory in continuous time see \citet{sontag2013mathematical}. Our presentation of control theoretical
methods for data assimilation focusses on 3DVAR.
Theoretical analysis of the continuous time 3DVAR algorithm
\eqref{eq:3DVARc} may be found in
\citet{law2012analysis}, 
\citet{blomker2013accuracy},
\citet{azouani2014continuous},
\citet{gesho2016computational},
\citet{olson2003determining},
\citet{mondaini2018uniform} and
\citet{larios2017nonlinear}.

We now overview the probabilistic approach introduced in
Subsection \ref{ssec:CTPP}. The Kalman-Bucy filter \citep{kalman1961new} contains
what is perhaps the first systematic derivation and analysis of an 
algorithm for the incorporation of continuous time data into estimation of a sample path of an SDE. Its extension to nonlinear and non-Gaussian distributions is provided by the Kushner-Stratonovich equation (\ref{eq:KSE}). An heuristic derivation of both the Kalman-Bucy filter as well as the Kushner-Stratonovich equation can be found in \citet{jazwinski2007stochastic} while \citet{bain2008fundamentals} covers the field of continuous time filtering in full detail.

The idea of Strang-splitting, which we use to derive the Kushner-Stratonovich equation, originates in \citet{strang1968construction}; for an overview of splitting methods see \citet{mclachlan2002splitting}. Furthermore, we rely on robustness results for continuous time filtering \citep{Clark2005},
which imply that smooth approximations $z^{\dagger,\delta}$ 
to stochastic observations $\zd$ are justified and that
the order of taking limits $\dt \to 0$ and $\delta \to 0$ can be accounted for by appropriate Stratonovich to It\^o correction terms. An introduction to the required covariation formulae used in proof of Lemma \ref{lem:ISC} and identity (\ref{eq:citeinbib}), which follows from covariation of stochastic integrals,  can be found in \citet{Eberle}. We note that robustness results do not carry over to associated mean field equations and filtering problems with correlated noise \citep{CNN2021}.

The numerical approximation of the Kushner-Stratonovich equation \eqref{eq:KSE} has a long history. The paper \citet{crisan1999particle} demonstrated how \eqref{eq:KSE} can be approximated by a particle method; this is a generalization of the bootstrap particle filter to continuous time \citet{crisan1998discrete}. See also \citet{bain2008fundamentals} for a detailed discussion of alternative approximation techniques. The paper
\citet{hu2002approximation} discussed solution of the unnormalized and linear version of the Kushner-Stratonovich equation known as the Zakai equation.

A mean field approach to the Kushner-Stratonovich equation (\ref{eq:KSE}) appeared first in the work of \citet{SR-CX10}, which utilizes robustness results and smoothed data $z^{\dagger,\delta}$ in the $\delta \to 0$ limit. The mean field equations (\ref{eq:KS_mf3d}) 
were proposed in \citet{SR-meyn13} (in the one dimensional setting) while the stochastic counterpart appeared first in \citet{reich2019data}. The mathematical relationship  between the various mean field formulations has been analyzed in \citet{SR-PRS20}. Numerical implementations of the mean field equations (\ref{eq:KS_mf3}) are discussed in \citet{TdWMR17} while \citet{TMM19} provides a detailed analysis of the diffusion map approximation discussed after (\ref{eq:DM approximation}). The constant gain approximation $K = C^{v\hs}\Gammas$, which corresponds to the ensemble Kalman filter, 
arises as a particular scaling limit from the diffusion map approach as discussed, for example, in \citet{TdWMR17}. 
See also the recent survey by \citet{taghvaei2023survey}.

The continuous time ensemble Kalman filter formulations (\ref{eq:KSA_mfj}) and (\ref{eq:KSA_mf2j}) appeared first in \citet{bergemann2012ensemble}. These formulations are based on earlier work on homotopy formulations of the Bayesian inference step by \citet{bergemann2010mollified} and \citet{reich2011dynamical}. See also the subsequent derivations in
\citet{law2015data} which contains a unified derivation of the Kalman-Bucy filter, continuous time 3DVAR and continuous time ensemble Kalman methods, starting from their discrete time counterparts.
 
Rigorous derivation of continuous time ensemble Kalman filter formulations from their discrete time counterparts can be found in \citet{SR-LS19b}, \citet{lange2021continuous}, \citet{blomker2018strongly} and  \citet{BSWW21}.  Derivation and analysis of the properties of continuous time limits 
in the context of solving inverse problems may be found in 
\citet{schillings2017analysis}; see bibliography Subsection
\ref{ssec:IPBIBC}. 
 Numerical time-stepping methods for continuous time ensemble Kalman filter 
formulations are analyzed in \citet{amezcuaensemble}. 

The papers \citet{ding2020ensemble,ding2021ensemble,ding2021ensembleb} undertake a systematic analysis of the link between interacting particle systems and mean field systems in continuous time, mostly focused on the solution of inverse problems; however the methods developed are more widely applicable. Similar to the stochastic ensemble Kalman filter, particle implementation (\ref{eq:KSA_mfj}) leads to undesirable correlations via the synthetic data $\widehat z^{(j)}$, which appear both in the gain $K^\sJ = C^{v\hs}\Gammas^{-1}$ and the innovation term $\dd\mathfrak{I} = \dd\zd - \dd \widehat z^{(j)}$ effectively giving rise to colored noise. These numerically induced correlations vanish in the limit $J\to \infty$.

Well-posedness, stability, and accuracy results for the ensemble Kalman filter first appeared in \citet{kelly2014well}, where the incompressible Navier-Stokes equations were studied, using
the continuous time formulations of stochastic ensemble Kalman methods first identified in \citet{reich2011dynamical}
and reviewed in \citet{law2015data}. Subsequent analyses of related issues for 
variants on the continuous time ensemble Kalman filter may be found in \citet{de2018long} and \citet{SR-dWT19}.
The paper \citet{delMoral18} initiates a line of research
related to the use of particle approximations of mean field models to understand long-time error estimates for particle systems approximating the Kalman-Bucy filter, the setting in which mean field ensemble Kalman filters exactly reproduce the true filtering distribution; see \citet{bishop2018perturbations,bishop2019one,bishop2019stability,bishop2020perturbation}
and, for an overview, \citet{bishop2020mathematical}. The control perspective deployed in \citet{law2012analysis} and \citet{azouani2014continuous} to study filter stability and accuracy has been unified and extended to ensemble Kalman filter formulations 
in \citet{BB24}, utilizing a particular form of covariance localization and additive inflation.

The classical Kalman filter can be viewed from the perspective of minimum variance estimation (discussed in discrete time in Subsection \ref{ssec:TM_MVA}) and optimal control \citep{kalman1961new}. This perspective has recently been extended to nonlinear filtering in \citet{KimMehta2022a,KimMehta2022b}, which opens up new perspectives for developing and analyzing numerical algorithms.

%%%%%%%%%%%%%%%%%%%%%%%%%%%%%%%%%%%%%%%%%%%%%%%%%%%%%%%%%%%%%%%%%%%%%%
%
%
%
%
%
%
%
%
%
%
%
%  Section 4: Inverse problems: discrete time
%
%
%
%
%
%
%
%
%
%
%
%%%%%%%%%%%%%%%%%%%%%%%%%%%%%%%%%%%%%%%%%%%%%%%%%%%%%%%%%%%%%%%%%%%%

%%%%%%%%%%%%%%%%%%%%%%%%%%%%%%%%%%%%%%%%
%
\section{Inverse Problems: Discrete Time}
\label{sec:IPDT}
%
%%%%%%%%%%%%%%%%%%%%%%%%%%%%%%%%%%%%%%%

In this section we adapt the ideas of Section \ref{sec:SE} to solve inverse problems.
We start in Section \ref{ssec:IPDT} with statement of the inverse problem,
followed in Subsections \ref{ssec:IPOP} and  \ref{ssec:IPBP} by discussion of
optimization and Bayesian approaches respectively; these are analogous to the
presentation of control and probabilistic approaches to the data assimilation problem
in Subsections \ref{ssec:CT} and \ref{ssec:PP}. 

The basic methodology we highlight is to formulate filtering problems which
(possibly only approximately) solve the inverse problem.
The Subsection \ref{ssec:IPFT}
is devoted to Bayesian probabilistic filtering methods which solve the inverse problem by morphing the
prior into the posterior in a finite time; Subsection \ref{ssec:IPIFT} discusses filtering methods
which work on infinite time horizons, exhibiting exponential convergence to approximate
solutions of the optimization or Bayesian formulations of the problem from arbitrary starting points. In both of Subsections \ref{ssec:IPFT} and 
\ref{ssec:IPIFT} we demonstrate application of Gaussian projected filtering 
and ensemble Kalman methods to solve the filtering problems arising.
In Section \ref{ssec:EKMI} we present numerical examples illustrating
ensemble Kalman methods for inverse problems.
We conclude in Section \ref{ssec:IPBIB} with bibliographic notes.

\begin{remark}
\label{rem:nopart}
In Subsections \ref{ssec:IPFT} and \ref{ssec:IPIFT}
we present only mean field statements of the Gaussian projected
filter and ensemble Kalman based methods. The reader
can generalize the ideas in Section \ref{sec:SE}, based on
interacting particle system approximations, to derive implementable algorithms from the ensemble based mean field algorithms introduced. Similarly, the unscented Kalman filter can be used to derive implementable algorithms from the Gaussian projected filters introduced here, as discussed in Subsection \ref{ssec:BSE}. Pseudo-code for the schemes implemented in the numerical examples of Subsection \ref{ssec:EKMI} may be found in Appendix \ref{sec:AA}.
$\blacksquare$
\end{remark}

%%%%%%%%%%%%%%%%%%%%%%%%%%%%%%%%%%%%%%%%%%%%%%
%
\subsection{Set-Up}
\label{ssec:IPDT}
%
%%%%%%%%%%%%%%%%%%%%%%%%%%%%%%%%%%%%%%%%%%%%%%%%%%

This and the next Section \ref{sec:CTI}  
are entirely devoted to solution of
the inverse problem of finding 
unknown parameter $u \in \R^{d_u}$ 
from data $w \in \R^{d_w}$, when $w$ is
related to $u$ via the equation
\begin{empheq}[box=\widefbox]{equation}
\label{eq:ip}
w=G(u)+\gamma.
\end{empheq}
Here $G:\R^{d_u} \to \R^{d_w}$ is the \emph{forward model}
and $\gamma$ represents noise polluting the data. 
We assume that $G$ is measurable with respect to the Borel algebra on input and output spaces, and is bounded on compact subsets of $\R^{d_u}.$ 
The basic methodology we highlight is to formulate filtering problems which
(possibly only approximately) solve the inverse problem.

\begin{remark}
The filtering problem from Subsection \ref{ssec:SUSE} requires solution of
an inverse problem, defined by (\ref{eq:sd}b), at each step $n;$ this inverse
problem is a specific instance of \eqref{eq:ip}. Indeed
the map $\op L_n$ in (\ref{eq:DAsimple}b) denotes Bayesian solution of this inverse problem,
a concept we will define, in the more general setting of
this section, in Subsection \ref{ssec:IPBP}.
Furthermore, the smoothing problem, referred to at the very end of the bibliography Subsection \ref{ssec:BSE}, can also be formulated as an instance of the general inverse problem \eqref{eq:ip}. 
$\blacksquare$
\end{remark}

We will work in a setting where we assume a probabilistic model
for the joint random variable $(u,w)$. We then assume that we
have available to us  $\wwd$, the observation
coordinate of a specific realization $(\ud,\wwd)$ under this probabilistic
model. This realization is itself generated by
$\gammad$, a specific realization of the observational noise.
We will consider two approaches to the inverse problem. 

\begin{itemize}
\item Objective 1: design an algorithm producing output 
$u$ from $\wwd$ so that
$u$ estimates
$\ud$, the true state underlying the data; 
\item Objective 2: design an algorithm which estimates
the distribution of random variable $u|\wwd$.
\end{itemize}

In the next subsection we define an optimization approach to
determine an approximation of $\ud$ from $\wwd,$ addressing
Objective 1. In the subsection
following that we define the Bayesian probabilistic formulation,
which also underpins the algorithms derived in the remainder of the section,
addressing Objective 2.

%%%%%%%%%%%%%%%%%%%%%%%%%%%%%%%%%%%%%%%%%%%%
%
\subsection{Optimization Formulation}
\label{ssec:IPOP}
%
%%%%%%%%%%%%%%%%%%%%%%%%%%%%%%%%%%%%%%%%

Given matrices $C_0 \succ 0$, $\Gammas \succ 0$, vector $m_0$ and the specific data
realization. namely $\wwd$,
we may define the nonlinear least squares loss function $\Phi$ and its Tikhonov-regularized counterpart
$\Phi_R$ as follows:
\begin{subequations}
\label{eq:phis}
\begin{empheq}[box=\widefbox]{align}
%\label{eq:phis}
\Phi(u) &= \frac12|\wwd-G(u)|_{\Gammas}^2,\\
\Phi_R(u) &= \Phi(u)+\frac12|u-m_0|_{C_0}^2.
\end{empheq}
\end{subequations}
Minimization of $\Phi_R$ constitutes a solution to the inverse problem. 
The specific weighted norms used in the least squares loss $\Phi$, and the regularization
leading to $\Phi_R$, are best understood from the probabilistic formulation
in the following subsection; however the remainder of this subsection can be understood
without recourse to this probabilistic formulation.

The Tikhonov regularized least squares problem associated with the
inverse problem \eqref{eq:ip} may be viewed as an unregularized
least squares problem  arising from
the modified inverse problem
\begin{equation*}
\label{eq:ip2}
w_R=G_R(u)+\gamma_R,
\end{equation*}
where we write 
\begin{equation}
\label{eq:Rip}
w_R := \begin{pmatrix}
        w\\ m_0
\end{pmatrix},  \quad \quad \quad
G_R(u) := \begin{pmatrix}
        G(u)\\ u
\end{pmatrix},
\quad \quad \quad
\Gammas_R :=
\begin{pmatrix}
        \Gammas & 0 \\ 0 & C_0
\end{pmatrix},
\end{equation}
for $\gamma_R$ being the generalized observation error. In particular, the cost functional (\ref{eq:phis}b) can be rewritten as 
\begin{equation}
\label{eq:phir}
        \Phi_R(u) =  \frac12|\wwd_R-G_R(u)|^2_{\Gammas_R},
\end{equation}
for 
\begin{equation}
\label{eq:wwr}
\wwd_R := \begin{pmatrix}
        \wwd\\ m_0
\end{pmatrix}.
\end{equation}

\begin{remark}
\label{rem:oh1}
A building block in many algorithms for minimization of
$\Psi:\R^{d_u} \to \R^+$ is \emph{gradient descent.} In basic form,
this is an iteration for sequence $\{u_n\}_{n \in { \mathbb Z}^+}$
defined by
\begin{equation}
\label{eq:dgrad}
u_{n+1}=u_n-\alpha \nabla \Psi(u_n).
\end{equation}
To solve the inverse problem, we can use iteration \eqref{eq:dgrad} with $\Psi = \Phi$ or $\Psi = \Phi_R$. In Subsection \ref{sssec:IPO} we develop derivative-free \emph{affine invariant} algorithms\footnote{See
Remark \ref{rem:remcite99} for a discussion of the implications of affine invariance, in the discrete time setting. We will study affine invariance in detail in continuous time in Section \ref{sec:CTI};
see Definition \ref{d:GW} and Remark \ref{rem:remcite99_needed}.} based on Gaussian projected and
ensemble Kalman filters. These algorithms offer an alternative to \eqref{eq:dgrad}; they only
approximately minimize the least squares objective, in general.
However in the quadratic case they reproduce an exact mean-field
gradient descent algorithm, as we will show as this section unfolds.
$\blacksquare$
\end{remark}

%%%%%%%%%%%%%%%%%%%%%%%%%%%%%%%%%%%%%%%%
%
\subsection{Bayesian Formulation}
\label{ssec:IPBP}
%
%%%%%%%%%%%%%%%%%%%%%%%%%%%%%%%%%%%%%%%%%

We now consider the \emph{Bayesian} approach to this inverse problem.
We again assume  $C_0 \succ 0$ and $\Gammas \succ 0$, as in the
previous subsection. To be concrete
we assume \emph{prior} $u \sim \Ng(m_0,C_0)$, that $\gamma \sim \Ng(0,\Gammas)$ and that $u$ and $\gamma$ are independent. It then follows that the likelihood $w|u \sim \Ng\bigl(G(u),\Gammas\bigr)$.  Application of Bayes Theorem shows that the \emph{posterior distribution} on $u|\wwd$ is measure $\mu$ given by
\begin{subequations}
\label{eq:mud}
\begin{empheq}[box=\widefbox]{align}
\mu(\dd u) &= \frac{1}{\normZ} \exp\bigl(-\Phi_R(u)\bigr)\dd u,\\
\normZ &= \int_{\R^{d_u}} \exp\bigl(-\Phi_R(u)\bigr)\dd u.
\end{empheq}
\end{subequations}
We note that the least-squares based optimization approaches to the inverse problem introduced in the preceding subsection can now be explicitly linked to the probabilistic formulation of the
inverse problem. In particular, minimizing $\Phi$ is referred to as the \emph{maximum likelihood approach}, whilst minimizing $\Phi_R$ as the \emph{maximum a posteriori approach}.

\begin{example}
\label{ex:linear}
Consider  the case of linear forward model:
\begin{equation}
\label{eq:linear}
G(\cdot)=L\cdot.
\end{equation}
Thus $\Phi_R$ is quadratic and it is straightforward to show that
the Hessian of $\Phi_R$ is greater than or equal to, in the sense of quadratic
forms, $C_0^{-1}.$ Since $C_0^{-1} \succ 0$  we deduce
that $\Phi_R$ has a unique critical point and this critical point is
a global minimizer. It is then natural to define
\begin{equation}
\label{eq:LR}
L_R := \begin{pmatrix}
        L\\ I
\end{pmatrix}
\end{equation}
and note that, with this definition, $G_R(\cdot)=L_R \cdot.$
We then have
\begin{equation*}
        \Phi_R(u) =  \frac12|\wwd_R-L_R u|^2_{\Gammas_R}.
\end{equation*}
Since $\Phi_R$ is quadratic we deduce that the posterior $\mu$ is Gaussian. We denote the mean by $\mop$ and by $\Cp$ the covariance. Matrix $\Cp$ is readily defined via its precision, the Hessian of $\Phi_R:$ 
\begin{equation}
\label{eq:pcov}
\Cp^{-1}=\Hess=L_R^\top \Gammas_R^{-1}L_R.
\end{equation}
As discussed above $\Cp^{-1} \succ 0$ so that $\Cp \succ 0$; in particular $\Cp$ is hence indeed invertible.
The minimizer of $\Phi_R$ is at the mean $\mop$ of the posterior which  solves the \emph{normal equations} 
\begin{equation}
\label{eq:pmean}
\Cp^{-1} \mop = L_R^\top \Gammas_R^{-1}\wwd_R.
\end{equation}

This representation of the posterior, which is Gaussian, is in terms of the precision matrix and the mean. There is an alternative and useful representation formula for the posterior covariance and mean, derived as follows. Consider
the Gaussian random variable $(u,w)$ defined by choosing $u \sim \Ng(m_0,C_0)$ and $w|u \sim \Ng(Lu,\Gammas).$ Then the solution of the Bayesian inverse problem \eqref{eq:mud} is given in this linear setting by the distribution of $u|\wwd.$ By using standard conditioning formulae for Gaussian random variables we obtain
\begin{subequations}
\label{eq:exactMF99}
\begin{align}
\mop &=m_0 + C_0 L^\top (L C_0 L^\top +  \Gammas)^{-1} (\wwd -Lm_0),\\
\Cp &= C_0 - C_0 L^\top (L C_0 L^\top + \Gammas)^{-1} L C_0.
\end{align}
\end{subequations}
These formulae are equivalent to $(m_1,C_1)$ found from the Kalman filter Bayesian update
step \eqref{eq:KF_analysisL} of Example \ref{ex:sssec:2}, with the choices
$M={\rm Id},\Sigma=0, H=L$ and $\Gamma=\Gammas.$
$\blacksquare$
\end{example}

\begin{remark}
\label{rem:oh2}
A commonly used methodology for sampling from target distribution
$\mu$ on $\R^{d_u}$ is MCMC. At abstract level
this defines a Markov chain for density $\rho_n$ given by transition
kernel $\op K(\alpha)$, where $\alpha$ describes hyper-parameters
that define the specific method used. Thus
\begin{equation}
\label{eq:mcmc}
\rho_{n+1}=\op K(\alpha)\rho_n.
\end{equation}
This is a linear iteration for the density $\rho_n$, designed to converge
to the target density, defined by the posterior, as $n \to \infty.$ In Subsection
\ref{sssec:IPB} we show how mean field ensemble Kalman methods, which induce
a nonlinear iteration on density $\rho_n$, may be used as an alternative to
\eqref{eq:mcmc}, defining an approximate Bayesian
posterior by iterating to infinity. Furthermore, this iteration will
be shown to be exact for Gaussian posteriors, and to benefit from
affine invariance.\footnote{Recall that discussion of the implications of affine invariance may be found in Remark \ref{rem:remcite99}, in the discrete time 
setting; and that  we will study affine invariance in detail in continuous 
time, in Section \ref{sec:CTI}.}
$\blacksquare$   
\end{remark}

%%%%%%%%%%%%%%%%%%%%%%%%%%%%%%%%%%%%%%%%%%%%%%%%%%%%%%%%%%
%
\subsection{Finite-Time Algorithms}
\label{ssec:IPFT}
%
%%%%%%%%%%%%%%%%%%%%%%%%%%%%%%%%%%%%%%%%%%%%%%%%%%%%%%%%

The basic idea used in this subsection, to address the solution of
inverse problems, is rooted in a sequential formulation of Bayesian inference. 
From this sequential formulation we derive a filtering problem whose
solution, at a particular time, gives the desired posterior.
Subsection \ref{sssec:IPFTF} describes the formulation,
Subsection \ref{sssec:gpf_ip} the use of Gaussian projected filters
and Subsection \ref{sssec:kt_sdt} the use of ensemble Kalman methods. 

%%%%%%%%%%%%%%%%%%%%%%%%%%%%%%%%%%%%%%%%%%%%%%%%
%
\subsubsection{Formulation}
\label{sssec:IPFTF}
%
%%%%%%%%%%%%%%%%%%%%%%%%%%%%%%%%%%%%%%%%%%%%%%%%

To understand this sequential approach we define, for integer $N > 1$,
\begin{equation} \label{eq:iterative Bayes}
    \Phi_{R,n}(u) = \frac{n}{N}\Phi(u)+\frac12|u-m_0|_{C_0}^2,
\end{equation}
noting that $\Phi_{R,N}(u)=\Phi_R(u).$
Now consider the sequence $\mu_n$ of probability measures with negative log
density given (up to an additive constant with respect to variation of $u$) 
by $\Phi_{R,n}(u).$ Then the Bayesian inference problem (\ref{eq:mud}) 
can be reformulated as a sequence of $N$ Bayesian inference steps 
where the prior $\mu_n$ is
morphed into posterior $\mu_{n+1}$  using the data likelihood 
$\exp(-\Phi(u)/N)$ in each step:
\begin{equation}
\label{eq:iterative Bayes2}
    \mu_{n+1}(\dd u) \propto \exp\Bigl(-\frac{1}{N}\Phi(u)\Bigr)\mu_n(\dd u).
\end{equation}
Hence
\begin{equation}
\label{eq:iterative Bayes3}
    \mu_{n}(\dd u) \propto \exp\Bigl(-\frac{n}{N}\Phi(u)\Bigr)\mu_0(\dd u).
\end{equation}
Thus
\begin{equation}
\label{eq:iterative Bayes4}
    \mu_{N}(\dd u) \propto \exp\Bigl(-\Phi(u)\Bigr)\mu_0(\dd u).
\end{equation}

The initial prior is set to
$\mu_0=\Ng(m_0,C_0)$ and the $N^{th}$ posterior $\mu_N$ 
delivers the desired Bayesian solution to the
inverse problem, given in \eqref{eq:mud}.

\begin{remark}
\label{rem:chopin}  
Here we have introduced an iteration index $n$ to morph from prior to posterior.
In what follows we will identify $n$ with an \emph{artificial} time and then 
import ideas from filtering to solve the inverse problem.  Notice that, in this
approach, the single inverse problem \eqref{eq:iterative Bayes4} of interest,
is replaced by $N$ inverse problems of the form \eqref{eq:iterative Bayes2}.
This can be beneficial because each of the $N$ inverse problems 
\eqref{eq:iterative Bayes2} is easier to solve than the single inverse
problem \eqref{eq:iterative Bayes4}, because the defining change of measure
is closer to the identity.

A variant on this idea is to morph from prior to posterior
by (possibly artificially) considering the data $w$ as sequentially acquired
and incrementally including components of the data at each step $n$, again leading
to a sequence of measures $\{\mu_n\}_{n=0}^N$, with $\mu_N$ equal to the 
posterior. 
$\blacksquare$
\end{remark}

Given this sequence of measures $\mu_n$, it is possible to identify a stochastic dynamical system with filtering distribution $\mu_n.$ Application of any filtering method to this filtering problem, and ensemble Kalman filters in particular, then
provides a method to approximate the posterior distribution. 
These sequential formulations of Bayesian inversion are well known and have, for example, been exploited in the use of sequential Monte Carlo methods for Bayesian inference; see Subsection \ref{ssec:IPBIB} for details. 

\begin{remark}
\label{rem:refer}
To employ sequential Monte Carlo methods based on \eqref{eq:iterative Bayes2}
it is necessary to invoke some form of approximation. The resulting outcome
of such approximations depends on the choice of positive integer $N$.
Empirically it is found that choosing $N \gg 1$ gives better approximations
of the desired posterior \eqref{eq:iterative Bayes4}; however this must be traded
against the additional cost of taking $N$ steps.

Standard particle filter-based sequential Monte Carlo methods 
do not always scale well to high dimensions, in the same way that the particle filter for state estimation scales poorly. Consequently ensemble Kalman variants of 
sequential Monte Carlo methods  have an important place in the field and we
will deploy these, and variants of them, after introducing the stochastic dynamical system  underlying sequential Monte Carlo.
$\blacksquare$
\end{remark}

Our first step is to show how to realize \eqref{eq:iterative Bayes2} via a filtering problem which we refer to as a \emph{transport problem}:\footnote{Note that we have introduced transport ideas in Section \ref{sec:SE} to underpin algorithms which
approximate the Bayesian inference that defines the analysis
step in filtering.} it transports the prior initial condition into the desired posterior, through a discrete time evolution. See Theorem \ref{t:fstep} below and note that in Subsection
\ref{sssec:IPFTFC} analogous developments are made in continuous time.
The transport problem is exact: it introduces no approximations. To
derive algorithms we proceed to discuss approximations to the transport. We  follow discussion of exact transport with study of the Gaussian projected filter, applied in this inverse problem context, in Subsection \ref{sssec:gpf_ip}, with continuous time analog in Subsection \ref{sssec:smcC}. Study of Gaussian projection is then followed by discussion of the application of mean field Kalman transport algorithms to the transport problem, see Subsection \ref{sssec:kt_sdt}; continuous time analogs are covered in Subsection \ref{sssec:smcC2}. 

In the following development we define $\Delta t$ so that
\begin{empheq}[box=\widefbox]{equation}
\label{eq:ndt}
 N\dt=1.
\end{empheq}
Now consider the combined state-observation system in the form
\begin{subequations}
\label{eq:fstep}
\begin{align}
u_{n+1} &= u_n, \\
w_{n+1} &= G(u_{n+1}) + \frac{1}{\sqrt{\dt}}\gamma_{n+1},
\end{align}
\end{subequations}
for
$n \in \{0,\cdots, N-1\}$, where $\{\gamma_{n+1}\}_{n=0}^{N-1}$ is an i.i.d. sequence with variance $\Ng(0,\Gammas).$
It is intuitive that, since $N\dt=1$, $u_0 \sim \Ng(m_0,C_0)$ and the observed data $w^\dagger_{n+1}=\wwd$ for all
$n \in \{0,\cdots, N-1\}$, then $u_N$ conditioned on $W_N^\dagger:=\{w^\dagger_{n+1}\}_{n=0}^{N-1}$
is distributed as $\mu$, defined as in \eqref{eq:mud}. Indeed, using (\ref{eq:fstep}b) for $n \in \{0,\cdots, N-1\}$ corresponds to making $N$ independent noisy observations of $G(u_0)$, all with noise variance $\dt^{-1} \Gammas;$ this is statistically equivalent to a single noisy observation of $G(u_0)$ with noise variance $\Gammas$. Since $u_0$ is initialized as $\Ng(m_0,C_0)$ (the prior) the problem reduces to the Bayesian inverse problem for $u|\wwd.$

This intuition may be substantiated by using the discussion around equations \eqref{eq:iterative Bayes}, \eqref{eq:iterative Bayes2}. In order to avail ourselves of the results from Section \ref{sec:SE}, we first rescale the observation equation in \eqref{eq:fstep} to obtain
\begin{subequations}
\label{eq:ipsd}
\begin{empheq}[box=\widefbox]{align}
u_{n+1} &= u_n, \\
y_{n+1} &= \dt G(u_{n+1}) + \eta_{n+1},
\end{empheq}
\end{subequations}
with $\eta_{n+1} \sim \Ng(0,\dt\Gammas)$ and $y_{n+1}=\dt w_{n+1}$. Denote the observed data by $\Yd_n=\{\yd_{\ell}\}_{\ell=1}^n$, where $\yd_\ell = \dt\wwd_\ell$. We may then show the following:
\begin{theorem}
\label{t:fstep}
Consider the dynamical system \eqref{eq:ipsd} and 
assume that $C_0 \succ 0$, $\Gammas \succ 0$  and  $N\dt=1$.
Assume also that $u_0 \sim \Ng(m_0,C_0)$ and $\eta_{n+1} \sim \Ng(0,\dt\Gammas);$ furthermore, assume that $\{\eta_n\}_{n=1}^N$ forms an i.i.d. sequence, independent of $u_0.$
Then $\mu_n$, the law of $u_n|\Yd_n$ defined by \eqref{eq:ipsd}, satisfies \eqref{eq:iterative Bayes2},
and in particular $\mu_N$ is equal to the posterior distribution $\mu$, if the data is chosen as $\yd_n=\dt\, \wwd,$ for $n \in \{1,\cdots, N\}$.
$\Diamond$ \end{theorem}

\begin{proof}
Let $\mu_n$ be the law of $u_n|\Yd_n$.
Since (in the notation of Section \ref{sec:SE}) $\hmu_{n+1}=\mu_n$ we see that
the mapping $\mu_n$ to $\mu_{n+1}$ is simply given by Bayes theorem: $\mu_{n+1}=\op L_n(\mu_n).$ 
This observation yields the following 
identity, expressed in terms of
$\rho_n$ the Lebesgue density of measure $\mu_n:$
\begin{align*}
\log \rho_{n+1}-\log \rho_{n}&=-\frac{1}{2\dt}|\yd_n-\dt G(u)|_{\Gammas}^2+{\rm const},\\
&=-\frac{\dt}{2}|w- G(u)|_{\Gammas}^2+{\rm const}.
\end{align*}
Summing over $n \in \{0,\cdots, N-1\}$, using the fact that
\begin{align*}
\log \rho_{0}&=-\frac{1}{2}|u-m_0|_{C_0}^2+{\rm const},
\end{align*}
we deduce that 
\begin{equation*}
    \log \rho_n = -\Phi_{R,n} + {\rm const},
\end{equation*}
with $\Phi_{R,n}$ given by (\ref{eq:iterative Bayes}). Choosing $n=N$
gives the desired result concerning the posterior.
\end{proof}

Thus we may approach the problem of (approximately) sampling from $\mu$
by (approximately) solving the filtering problem defined by 
\eqref{eq:ipsd}, for $\mu_n$, until discrete time $n=N$. In particular, we may use the Gaussian projected filter or ensemble Kalman methods to approximate this filtering problem.
In the next two subsections we consider, respectively, these two approximation methods.

%%%%%%%%%%%%%%%%%%%%%%%%%%%%%%%%%%%%%%%%%%%%%%%%%%%%%%%%%%%%%%
%
\subsubsection{Algorithms: Gaussian Projected Filter}
\label{sssec:gpf_ip}
%
%%%%%%%%%%%%%%%%%%%%%%%%%%%%%%%%%%%%%%%%%%%%%%%%%%%%%%%%%%%%%

We now apply the ideas from Subsection \ref{ssec:GPFD}, which concerns the Gaussian projected filter in the general setting, to the specific setting
of the stochastic dynamical system \eqref{eq:ipsd}.
Let $\E$ denote expectation under $u \sim \Ng(m_n,C_n)$ and define
\begin{subequations}
\label{eq:KF_pred_mean222}
\begin{empheq}[box=\widefbox]{align}
\Cov_{n}^{uG}      &=  \E\Bigl(\bigl(u- \E u\bigr)\otimes
\bigl(G(u)-\E G(u)\bigr)\Bigr),\\
\Cov_{n}^{GG}      &=  \E\Bigl(\bigl(G(u)- \E G(u)\bigr)\otimes
\bigl(G(u)-\E G(u)\bigr)\Bigr).
\end{empheq}
\end{subequations}
Noting that prediction under  (\ref{eq:ipsd}a) is trivial it follows
that the predicted mean and covariance satisfy
$\mh_{n+1}=m_n$ and $\pCov_{n+1}=C_n.$ Hence, using \eqref{eq:KF_analysis_add}
in the specific setting of \eqref{eq:ipsd}, yields
\begin{subequations}
\label{eq:KF_analysis_add_IP_p}
\begin{empheq}[box=\widefbox]{align}
        \mean_{n+1} &= m_{n} + \Delta t \Cov_{n}^{uG}
(\Gammas+\Delta t \Cov_{n}^{GG})^{-1}
\bigl(\wwd - \E G(u)\bigr),\\
\Cov_{n+1} &= \Cov_{n} - \Delta t \Cov_{n}^{uG}(\Gammas+
\Delta t \Cov_{n}^{GG})^{-1} \bigl(\Cov_{n}^{uG}\bigr)^\top.
    \end{empheq}
\end{subequations}
Note that the difference between the data $\wwd$ and the
mean of $G(u)$ under the Gaussian at step $n$ acts as a forcing 
term in the evolution of the mean from $n$ to $n+1$,
promoting a Gaussian which agrees with the data. This forcing
term is weighted by covariance information. The covariance of the
Gaussian projected filter is non-increasing from step to step since
$\langle u, \Cov_{n+1} u \rangle \le \langle u, \Cov_{n} u \rangle$
for all $u \in \R^{d_u};$ this reflects the fact that more information
is received at each step $n \mapsto n+1$ as the unknown
$u$ is repeatedly observed.

\begin{example}
\label{ex:linear1}
In the setting of  the linear inverse
problem \eqref{ex:linear}, where $G(u)=Lu$, the Gaussian projected filter
equations  \eqref{eq:KF_analysis_add_IP_p} reduce to
\begin{subequations}
\label{eq:KF_analysis_add_IPL}
\begin{align}
        \mean_{n+1} &= m_{n} + \Delta t \Cov_{n}L^\top
(\Gammas+\Delta t L \Cov_{n} L^\top)^{-1}
\bigl(\wwd - Lm_n\bigr),\\
\Cov_{n+1} &= \Cov_{n} - \Delta t \Cov_{n}L^\top(\Gammas+
\Delta t L \Cov_{n} L^\top)^{-1} L \Cov_{n}.
    \end{align}
\end{subequations}
These equations may be iterated to map from $(m_0,C_0)$ directly to $(m_n,C_n)$,
obtaining
\begin{subequations}
\label{eq:seblabel}
\begin{align}
        \mean_{n} &= m_0 + n\dt \Cov_{0}L^\top
\left(\Gammas+n\dt L \Cov_{0} L^\top\right)^{-1}
\bigl(\wwd - Lm_0\bigr),\\
\Cov_{n} &= \Cov_{0} - n\dt \Cov_{0}L^\top \left(\Gammas+
n\dt L \Cov_{0} L^\top \right)^{-1} L \Cov_{0},
    \end{align}
\end{subequations}
These formulae may also be obtained by applying Bayes formula, in the linear
setting, to \eqref{eq:iterative Bayes3} and using \eqref{eq:ndt}, namely $N\dt=1.$
The Gaussian posterior measure $\mu=\Ng(\mop,\Cp)$ given
by \eqref{eq:pcov}, \eqref{eq:pmean} may now
be found by choosing mean and covariance $(\mop,\Cp)=(m_N,C_N).$ 
This follows from Section \ref{sec:SE} because the Gaussian projected
filter recovers the Kalman filter, which is exact for linear Gaussian problems.
$\blacksquare$
\end{example}

%%%%%%%%%%%%%%%%%%%%%%%%%%%%%%%%%%%%%%%%%%%%%%%%%%%%%%%%%%%%%%
%
\subsubsection{Algorithms: Ensemble Kalman Filter}
\label{sssec:kt_sdt}
%
%%%%%%%%%%%%%%%%%%%%%%%%%%%%%%%%%%%%%%%%%%%%%%%%%%%%%%%%%%%%%%%

Recall that  mean field models lead to ensemble Kalman methods through particle approximation. In this section we simply highlight use of one of the mean field models, in the context of inverse problems, namely the stochastic Kalman transport approach from Subsection \ref{sssec:ats} and its
deterministic variant from Subsection \ref{sssec:atd}. We leave details of particle
approximations of these mean field models to the reader, and to
Appendix \ref{sec:AA} for related pseudo-code. 

Employing the state-observation model \eqref{eq:ipsd} within the stochastic Kalman transport model \eqref{eq:sd2nn_add} we obtain the mean field dynamical system, for i.i.d. unit Gaussian sequence $\{\xi_n\}$ in $\R^{d_w}$,
\begin{subequations}
\label{eq:ipKT}
\begin{empheq}[box=\fbox]{align}
u_{n+1} &= u_{n}+\dt \CuG_{n}(\dt\CGG_{n}+\Gammas)^{-1}\Bigl(\wwd-G(u_n)-
\sqrt{\frac{\Gammas}{\dt}}\xi_n\Bigr),\\
\CuG_n &= \mathbb{E}\Bigl(\bigl(u_n-\mathbb{E}u_n\bigr)\otimes\bigl(G(u_n)-\mathbb{E}G(u_n)\bigr)\Bigr), \\
\CGG_n &= \mathbb{E}\Bigl(\bigl(G(u_n)-\mathbb{E}G(u_n)\bigr)\otimes\bigl(G(u_n)-\mathbb{E}G(u_n)\bigr)\Bigr).
\end{empheq}
\end{subequations}
Note that, here, expectation $\E$ is computed under the law of $u_n$ itself. As for the Gaussian projected filter \eqref{eq:KF_analysis_add_IP_p} the evolution promotes a distribution which is compatible with the data, here with a forcing term, weighted by covariance information, applied to the evolution of the state $u_n$; but
unlike the mean field evolution equation for the mean, there is additional noise for the state evolution.

Recall Theorem \ref{t:fstep}. Since the ensemble Kalman transport algorithm
used here provides an approximation of the filtering distribution for
the dynamical system \eqref{eq:ipsd}, it follows that the random variable $u_N$ provides an approximation to the posterior distribution $\mu$, 
provided that $u_0 \sim \Ng(m_0,C_0)$,  the prior distribution. 
This statement can be made exact in the linear case as the following example shows.

\begin{example}
\label{ex:linear3}
Assume $G(u)=Lu$ for $L \in \R^{d_w \times d_u}$ so that the posterior
distribution of the Bayesian inverse problem $\mu$ is given in Example \ref{ex:linear}. Then the solution of the mean field model \eqref{eq:ipKT}
satisfies $u_N \sim \mu.$  This is a specific instance of
what we observed in Example \ref{ex:mfk}, namely that the mean field model 
reproduces the Kalman filter on linear Gaussian problems.
We note also that the Gaussian projected filter is identical
to the Kalman filter in this case -- see Example \ref{ex:linear1}.

Although it is implicit in Example \ref{ex:linear}, in this
specific inverse problem context we demonstrate the equivalence with the Kalman filter explicitly. To do this we note that, in the linear setting, equation  \eqref{eq:ipKT} defines a closed evolution in the set of Gaussian probability measures. The updates for the mean $m_n$ and covariance $C_n$ of $u_n$ then coincide with the Kalman filter, and hence the Gaussian
projected filter, in this linear case given by equations
\eqref{eq:KF_analysis_add_IPL}. Indeed, by taking the expectation under the law of $u_n$ of (\ref{eq:ipKT}a), it is readily checked that in the linear setting we obtain (\ref{eq:KF_analysis_add_IPL}a)
for the mean update. To obtain the evolution equation of the covariance, recall that 
\begin{equation}
\label{eq:cov_aux}
C_{n+1} = \mathbb{E}\Bigl((u_{n+1}-m_{n+1})\otimes(u_{n+1}-m_{n+1}) \Bigr).
\end{equation}
Substituting into \eqref{eq:cov_aux} the expression for $u_{n+1}$, given by (\ref{eq:ipKT}a) in the linear setting $G(u)=Lu$, and the expression for $m_{n+1}$, given by (\ref{eq:KF_analysis_add_IPL}a),  and then computing the expectation yields (\ref{eq:KF_analysis_add_IPL}b).
It follows from the calculations of Example \ref{ex:linear1} 
that $u_N \sim \mu.$
$\blacksquare$
\end{example}

We conclude this subsection by stating the corresponding deterministic transport formulation. We employ
the approximation (\ref{eq:DEnKF}). This holds in our case provided $\dt$ is small enough. Choosing $K_n$ as
implicitly defined in (\ref{eq:ipKT}), we obtain,
with expectation $\E$ computed under the law of $u_n$ itself, the following mean field model:
\begin{subequations}
\label{eq:mfKBF0}
\begin{empheq}[box=\fbox]{align}
u_{n+1} &= u_{n}+\dt \CuG_{n}(\dt\CGG_{n}+\Gammas)^{-1}\Bigl(\wwd-\frac{1}{2}\bigl(G(u_n)+
\E G(u_n)\bigr) \Bigr),\\
\CuG_n &= \mathbb{E}\Bigl(\bigl(u_n-\mathbb{E}u_n\bigr)\otimes\bigl(G(u_n)-\mathbb{E}G(u_n)\bigr)\Bigr), \\
\CGG_n &= \mathbb{E}\Bigl(\bigl(G(u_n)-\mathbb{E}G(u_n)\bigr)\otimes\bigl(G(u_n)-\mathbb{E}G(u_n)\bigr)\Bigr).
\end{empheq}
\end{subequations}

%%%%%%%%%%%%%%%%%%%%%%%%%%%%%%%%%%%%%%%%%%%%%%%%%%%%%%%%%%%
%
\subsection{Infinite-Time Algorithms}
\label{ssec:IPIFT}
%
%%%%%%%%%%%%%%%%%%%%%%%%%%%%%%%%%%%%%%%%%%%%%%%%%%%%%%%%%%%

Algorithms which (approximately) transport prior to posterior in finite time, 
as described in the preceding subsection,
are attractive. However they can be quite rigid as they do not benefit from
strong stability to perturbations. An alternative, pursued in this section,
is to seek algorithms which converge to the desired solution on an infinite
time-horizon, from arbitrary starting points, and which exhibit exponential stability.
This is hard to achieve in general, but can be achieved exactly
for Gaussian problems. When applied beyond the Gaussian setting this
hence leads to a methodology consistent with the application of 
ensemble Kalman filter approximations, which themselves invoke a Gaussian 
ansatz and yet are used beyond the Gaussian setting.
Subsection \ref{sssec:IPFTF2} describes the infinite time-horizon
formulation.
In Subsection \ref{sssec:IPO} we consider this infinite time-horizon perspective
for the solution of optimization problems associated with the inverse problem
\eqref{eq:ip}. Subsection \ref{sssec:IPB} considers the same perspective for Bayesian inversion.

%%%%%%%%%%%%%%%%%%%%%%%%%%%%%%%%%%%%%%%%%%%%%%%%
%
\subsubsection{Formulation}
\label{sssec:IPFTF2}
%
%%%%%%%%%%%%%%%%%%%%%%%%%%%%%%%%%%%%%%%%%%%%%%%%

To motivate what follows we consider algorithms that solve the optimization 
problem by extending ideas from the previous section to iterate a filtering
problem over an infinite time-horizon. To explain this idea recall the identity \eqref{eq:iterative Bayes3},
restated  here for convenience: 
\begin{equation*}
%\label{eq:iterative Bayes3}
    \mu_{n}(\dd u) \propto \exp\Bigl(-\frac{n}{N}\Phi(u)\Bigr)\mu_0(\dd u).
\end{equation*}
We note that if we evaluate this identity at $n=N$ then we obtain the Bayesian
posterior distribution; the resulting algorithms are based on solving the 
associated filtering problem on interval $n=0,1, \cdots, N.$  
Now we observe that if, instead,
we iterate $n \to \infty$ for fixed $N$ (and hence fixed $\Delta t$), 
then $\mu_n$ will converge to a sum of Dirac measures supported 
at global minimizers of $\Phi$ that are contained in the support of $\mu_0$. Thus we can iterate algorithms such as the Gaussian projected filter from Subsection \ref{sssec:gpf_ip} or the
ensemble Kalman filter from Subsection \ref{sssec:kt_sdt} to $n=\infty$ in order to obtain an approximate solution of the optimization problem for $\Phi,$ within the support of $\mu_0.$

In the following example we study both the Gaussian projected
filter, and the ensemble Kalman filter, in the linear Gaussian setting.
We consider their application when we iterate $n \to \infty$ for fixed $N.$
The two algorithms coincide in this linear Gaussian setting. Studying
their properties gives insight into the proposed iterative approach
to optimization. In particular it motivates the use of \emph{regularization} and
\emph{variance inflation} as introduced following the example.

\begin{example}
\label{ex:linear1b}
We consider the setting of Example \ref{ex:linear1}, where the linear 
inverse problem with $G(u)=Lu$ is considered, making the additional assumption
that $L C_0 L^\top$ has full rank. The Gaussian projected filter
equations  \eqref{eq:KF_analysis_add_IP_p}, iterated over $n$ steps give, 
the single step update \eqref{eq:seblabel}.
Since $L C_0 L^\top$ has full rank, this delivers the following 
closed form update in the image of $L$:
\begin{subequations}
%\label{eq:KF_analysis_add_IPL}
\begin{align}
        L\mean_{n} &= L m_0 +  \dt L\Cov_{0}L^\top
\left(\frac{1}{n}\Gammas+ \dt L \Cov_{0} L^\top\right)^{-1}
\bigl(\wwd - Lm_0\bigr),\\
L\Cov_{n}L^\top &= L\Cov_{0}L^\top -  \dt L\Cov_{0}L^\top \left(\frac{1}{n}\Gammas+
 \dt L \Cov_{0} L^\top \right)^{-1} L \Cov_{0}L^\top.
    \end{align}
\end{subequations}
If we fix $\dt$ and let $n \to \infty$ then we see that
\begin{align*}
        L\mean_{n} &= \wwd + \mathcal{O}(1/n),\\
L\Cov_{n}L^\top &=  \mathcal{O}(1/n).
    \end{align*}
We notice from the previous example that there are two issues when 
letting $n\to \infty$ for fixed $\dt$. First, the mean converges
to a point $m_\infty$ solving $Lm_\infty=\wwd$, rather than a 
minimizer of the regularized functional $\Phi_R(\cdot)$. Secondly, 
the convergence rate is only of order $1/n$. 
$\blacksquare$
\end{example}

We now seek to address the two problems identified in this
example, to develop improved methodology.

\paragraph{Regularization}
The first problem identified in Example \ref{ex:linear1b}, 
namely that regularization disappears when taking
$n \to \infty$, can be addressed by considering the iteration defined by 
\begin{equation}
\label{eq:ohr}
    \mu_n(\dd u) \propto \exp\bigl(-n\dt \Phi_R(u)\bigr)\,
    \mu_0(\dd u)
\end{equation}
instead of the iteration defined by \eqref{eq:iterative Bayes3}.
Recalling
$G_R, \Gammas_R$ defined by \eqref{eq:Rip} and assuming that
$\Gammas_R \succ 0$ 
then sequence of measures $\mu_n$ given by \eqref{eq:ohr}  
may be generated by the filtering
distribution associated with the following
modification of \eqref{eq:ipsd}: 
\begin{subequations}
\label{eq:ipsd99}
\begin{empheq}[box=\widefbox]{align}
u_{n+1} &= u_n, \\
y_{n+1} &= \dt G_R(u_{n+1}) + \eta_{n+1};
\end{empheq}
\end{subequations}
here $\eta_{n+1} \sim \Ng(0,\dt\Gammas_R)$ and we consider
the setting where the observed data is  $y_{n+1}=\dt \wwd_{R}$.\footnote{Note 
that $\ell$ indexes an iteration in  $\wwd_{\ell}$ but that $\wwd_R$ is the fixed vector defined in \eqref{eq:wwr}.
We may apply sequential filtering techniques, such as the Gaussian projected
filter and the ensemble Kalman filter, to the filtering problem defined
by \eqref{eq:ipsd99}. Since $\mu_n$ converges to 
a Dirac delta distribution centred about the minimizer of $\Phi_R$, within 
the support of $\mu_0,$ this addresses the first problem. However, the rate 
of convergence remains of order $1/n$ so that the second problem is not
addressed; this is verified explicitly for the linear case in the 
forthcoming Example \ref{ex:lin_c1}.}

\paragraph{Variance Inflation}

The second problem identified in Example \ref{ex:linear1b} is algebraic convergence. The root cause of the algebraic convergence 
is the collapse of the covariance $C_n$ to zero. 
Therefore, in order to accelerate the convergence rate, we need 
to modify the sequential update steps to ensure that the covariance 
of (approximate) filters does not collapse to zero; at the same time 
we must ensure that the mean $m_n$ still converges, 
exactly in the linear setting and approximately in the general nonlinear case, 
to the minimizer of $\Phi_R$ as $n\to \infty$.

In order to achieve this non-collapsing covariance we modify \eqref{eq:ipsd99} 
by adding a form of variance inflation to the evolution 
of the parameter $u_n$ and
consider the stochastic dynamical system given by
\begin{subequations}
\label{eq:rescale_beta_dt}
\begin{empheq}[box=\widefbox]{align}
u_{n+1} &= u_{n} + \xi_{n},\\
y_{n+1} &= \dt G_R(u_{n+1}) + \eta_{n+1}.
\end{empheq}
\end{subequations}
Here $\xi_n \sim \Ng(0,\beta\dt\Csov_n)$, $\beta \ge 0$ and $\eta_{n+1} 
\sim \Ng(0,\dt\Gammas_R)$ for covariance inflation matrix $\Csov_n$  
to be defined. Note that if $\beta=0$ we simply recover \eqref{eq:ipsd99}.

We now discuss the choice of $\Csov_n$. Because the true covariance that we
wish to recover is that of filtering distribution it is natural $\Csov_n$
is defined in terms of the covariance of the filter. Defining $\Yd_n=\{\yd_{\ell}\}_{\ell=1}^n$ with  $\yd_{\ell} := \dt \wwd_R$ we may consider the filtering distribution defined by random variable $u_n|\Yd_n.$ We let $\Cov_n$ denote the covariance under this filtered random variable.
We then set $\Csov_n:=\Cov_n$.

\begin{remark}
\label{rem:nonstandard}
With this choice of $\Cov_n$ the equation \eqref{eq:rescale_beta_dt} defines a form of mean field
model for state-observation evolution. Previously in this paper the state-observation models we have considered have not been of mean field type; we only introduced mean field models as the basis of sample path based algorithms to (approximately) solve a filtering problem. In contrast, here, the mean field dependence of the proposed model \eqref{eq:rescale_beta_dt}, with $\Csov_n:=\Cov_n$, is through the filtering distribution associated with $u_n|\Yd_n$. Thus, even before we develop mean field
models to approximate the law of the filtering distribution via sample path based algorithms, the underlying  state-observation model is linked to filtering.
 
Although the filtering distribution depends on the history $\Yd_n$, equation \eqref{eq:rescale_beta_dt} can be rendered Markovian by coupling it to the evolution of the filtering distribution $\mu_n \mapsto \mu_{n+1}$, and noting that $\Cov_n$ is computed under $\mu_n$. 

We note that, in practice, identification of the exact covariance of the filtering distribution is not possible. Thus in the Gaussian projected filter and ensemble Kalman filter that follow we will use approximations of $\Cov_n$; however we will also denote these approximations by $\Cov_n$ to avoid proliferation of notation.
$\blacksquare$
\end{remark}

We now derive implementable algorithms to approximate the filtering
distribution defined by \eqref{eq:rescale_beta_dt}.

\paragraph{Gaussian Projected Filter}
We begin by applying the ideas from Subsection \ref{ssec:GPFD}, which concerns the Gaussian projected filter in the general setting, to the specific setting of the stochastic dynamical system \eqref{eq:rescale_beta_dt}. Using \eqref{eq:KF_analysis_add}
in the specific setting of \eqref{eq:rescale_beta_dt}, yields
\begin{subequations}
\label{eq:KF_analysis_add_IP}
\begin{empheq}[box=\widefbox]{align}
\mh_{n+1} &= \mean_n,\\
\hC_{n+1} &= (1+\beta \Delta t) \Cov_n,\\
\mean_{n+1} &= \mh_{n+1} + \dt \hCuG_{R,n+1}
(\dt \hCGG_{R,n+1}+ \Gammas_R)^{-1}
\bigl(\wwd_R - \hoo_{n+1}\bigr),\\
\Cov_{n+1} &= \hC_{n+1} - \dt \hCuG_{R,n+1}(\dt \hCGG_{R,n+1}+\Gammas_R)^{-1} \bigl(\hCuG_{R,n}\bigr)^\top.
    \end{empheq}
\end{subequations}
Here we define
\begin{subequations}
\label{eq:KF_pred_mean222beta}
\begin{empheq}[box=\fbox]{align}
\hoo_{n+1} &= \mathbb{E} G_R(\widehat{u}_{n+1}),\\
\hCuG_{R,n+1} &= \mathbb{E}\Bigl(\bigl(\widehat{u}_{n+1}-\mathbb{E}\widehat{u}_{n+1}\bigr)\otimes\bigl(G_R(\widehat{u}_{n+1})-
\mathbb{E}G_R(\widehat{u}_{n+1})\bigr)\Bigr),\\
\hCGG_{R,n+1} &= \mathbb{E}\Bigl(\bigl(G_R(\hat{u}_{n+1})-\mathbb{E}G_R(\widehat{u}_{n+1})\bigr)\otimes\bigl(G_R(\widehat{u}_{n+1})-
\mathbb{E}G_R(\widehat{u}_{n+1})\bigr)\Bigr),
\end{empheq}
\end{subequations}
where, in \eqref{eq:KF_analysis_add_IP}, all expectations are with
respect to $\widehat{u}_{n+1} \sim \Ng(\mh_{n+1},\hC_{n+1}).$
Note that we have used the covariance of the Gaussian projected
filter to define the variance inflation required to determine 
(\ref{eq:KF_analysis_add_IP}b), since we do not have the 
covariance under the true filtering distribution. 

\paragraph{Ensemble Kalman Filter}

Instead of the Gaussian projected filter, we may use
the ensemble Kalman filter. 
We use the covariance $C_n$ of the ensemble Kalman filter 
to define the variance inflation since, again, we do not have the 
covariance under the true filtering distribution.
With these considerations in hand, application
of the stochastic Kalman transport mean field model 
\eqref{eq:sd2nn_add} to \eqref{eq:rescale_beta_dt} yields 
\begin{subequations}
\label{eq:notsss1}
\begin{empheq}[box=\widefbox]{align}
\widehat{u}_{n+1} &= u_{n} +   \xi_{n},\\
\widehat{y}_{n+1} &= \dt G_R(\hu_{n+1}) + \eta_{n+1},\\
u_{n+1} &= \widehat{u}_{n+1}+\hCuG_{R,n+1}\bigl(\dt\hCGG_{R,n+1}+\Gammas_R\bigr)^{-1}
\bigl(\dt \wwd_R-\widehat{y}_{n+1}\bigr).
\end{empheq}
\end{subequations}
Here $\xi_n \sim \Ng\bigl(0,\beta \dt C_n\bigr),$
$\eta_{n+1} \sim \Ng\bigl(0,\dt\Gammas_R)$ and  
expecations appearing in \eqref{eq:KF_pred_mean222beta}, to define
$(\hCuG_{R,n+1}, \hCGG_{R,n+1})$, are computed
under the law of $\widehat{u}_{n+1}$.

\begin{remark}
\label{rem:remcite99_notneeded}
Recall the mean field dynamical
system \eqref{eq:rescale_beta_dt} and consider its filtering distribution.
In Subsection \ref{sssec:IPO} we show that in the linear case the mean of
the filtering distribution converges to the posterior mean of the 
underlying Bayesian inverse problem, and hence to a minimizer
of the Tikhonov regularized least squares function $\Phi_R.$
For $\beta=0$ convergence is algebraic,
whilst it is exponential for $\beta>0.$ 
Furthermore, in Subsection \ref{sssec:IPB}, we show that for a particular 
choice of $\beta$, in the linear case the filtering distribution converges to the Bayesian posterior
distribution defined by the inverse problem.

Recall that the Gaussian projected filter \eqref{eq:KF_analysis_add_IP}
and the mean field ensemble Kalman filter \eqref{eq:notsss1}  exactly reproduce the evolution of
the filtering distribution, for linear Gaussian problems. As a consequence
everything stated in this remark for the filtering distribution applies also to the law
defined by \eqref{eq:KF_analysis_add_IP} and by \eqref{eq:notsss1}.
$\blacksquare$
\end{remark}

We also note that it is possible to use corresponding deterministic 
transport formulations in place of \eqref{eq:notsss1}. We employ the approximation \eqref{eq:DEnKF}. 
This holds in our case provided $\dt$ is small enough. Choosing $K_n$ as
implicitly defined in \eqref{eq:notsss1}, we obtain the following mean field model:
\begin{subequations}
\label{eq:mfKBF01}
\begin{empheq}[box=\fbox]{align}
\widehat{u}_{n+1} &= u_{n} +   \xi_{n},\\
u_{n+1} &=\widehat{u}_{n+1}
+\dt K_n
\Bigl(\wwd_R-\frac{1}{2}\bigl(G_R(\hu_{n+1})+
\widehat{o}_{n+1}\bigr) \Bigr),\\
K_n &=\hCuG_{R,n+1}\bigl(\dt\hCGG_{R,n+1}+\Gammas_R\bigr)^{-1}.
\end{empheq}
\end{subequations}
Here $\xi_n \sim \Ng\bigl(0,\beta \dt C_n\bigr).$
All expectations used to define $(\hCuG_{R,n+1}, \hCGG_{R,n+1}, \widehat{o}_{n+1})$
are given by \eqref{eq:KF_pred_mean222beta}, computed under the law of $\widehat{u}_{n+1}$.
This also exactly solves the filtering problem defined by \eqref{eq:rescale_beta_dt},
in the linear Gaussian setting.

\begin{remark} 
\label{rem:det}
We note that it is possible to replace \eqref{eq:notsss1} by the 
mean field model
\begin{subequations}
\label{eq:notsss19}
\begin{empheq}[box=\widefbox]{align}
\widehat{u}_{n+1} &= u_{n} +   \frac{\gamma}{2}(u_{n}-\mathbb{E}u_n),\\
\widehat{y}_{n+1} &= \dt G_R(\hu_{n+1}) + \eta_{n+1},\\
u_{n+1} &= \widehat{u}_{n+1}+\hCuG_{R,n+1}\bigl(\dt\hCGG_{R,n+1}+\Gammas_R\bigr)^{-1}
\bigl(\dt \wwd_R-\widehat{y}_{n+1}\bigr).
\end{empheq}
\end{subequations}
Here the expectation on $u_n$ is with respect to the approximate filtering distribution
generated by this model. Now note that the predictive mean and covariance defined by
\eqref{eq:notsss1} are governed by (\ref{eq:KF_analysis_add_IP}a), (\ref{eq:KF_analysis_add_IP}b).
The same equations govern the evolution of predictive mean and covariance of \eqref{eq:notsss19}
provided that we choose $\gamma$ to be the unique positive solution of the equation
$\gamma+\frac{\gamma^2}{4}=\dt\beta.$ Thus the resulting methodology will coincide with
the Gaussian projected filter and with the ensemble Kalman filter on linear Gaussian problems.
$\blacksquare$
\end{remark}

%%%%%%%%%%%%%%%%%%%%%%%%%%%%%%%%%%%%%%%%%%%%%%%%
%
\subsubsection{Algorithms for Optimization Formulation}
\label{sssec:IPO}
%
%%%%%%%%%%%%%%%%%%%%%%%%%%%%%%%%%%%%%%%%%%%%%%%%

Recall that we have introduced the non-standard mean field dynamical
system \eqref{eq:rescale_beta_dt}.
We have also described how the filtering distribution of the dynamical system
may be approximated by the Gaussian projected filter \eqref{eq:KF_analysis_add_IP}
and the mean field ensemble Kalman filter \eqref{eq:notsss1}.
In this subsection we substantiate the statements made about these algorithms in
Remark \ref{rem:remcite99_notneeded}.
We initially discuss algorithms with algebraic convergence, for $\beta=0$;
and then we introduce
generalization of the analysis to $\beta>0$ which allow us to obtain exponential convergence.

\vspace{0.1in}
\paragraph{Algebraic Convergence}

Here we consider the setting of \eqref{eq:rescale_beta_dt} where $\beta=0$. 
Recall that for this choice of $\beta$, the stochastic dynamical system \eqref{eq:rescale_beta_dt} reduces to \eqref{eq:ipsd99}. The following example illustrates that in the linear case where
$G_R(\cdot)=L_R\cdot$, we recover an algebraic rate of convergence to the 
posterior distribution when fixing $\dt$ and taking $n\to \infty$. 
This is analogous to Example \ref{ex:linear1b} which considers algorithms 
based on $L$, not $L_R$.

\begin{example}
\label{ex:lin_c1}
Assume that $u_0$ is initialized at a Gaussian
$\Ng(m_0,C_0)$ and assume also that $C_0, \Gammas_R \succ 0$.
Consider the setting where
$G_R(\cdot)=L_R \cdot$ for matrix $L_R \in \R^{(d_w+d_u) \times d_u}$. 
Now consider the filtering distribution $u_n|Y^\dagger_n$ given in 
\eqref{eq:ipsd99},
with $Y^\dagger_n$ data defined $\yd_n=\dt \wwd_R,$
with $\wwd_R$ as in \eqref{eq:wwr}.

The desired filtering distribution is Gaussian $\Ng(m_n,C_n)$.
We now show that the iteration $(m_n,C_n) \mapsto (m_{n+1},C_{n+1})$
converges to the posterior distribution, and does so at an algebraic rate.
The reader should compare this with Example \ref{ex:linear1b} which,
using filtering based on $G$ rather than $G_R$, and again in the linear case, 
also results in algebraic convergence; furthermore convergence
is only in the image space under the forward map $L$.

To prove convergence to the posterior distribution we first identify the update equations for $(m_n,C_n)$. Note that the 
predictive mean $\mh_{n+1}$ and covariance $\pCov_{n+1}$
defined by (\ref{eq:ipsd99}a) trivially satisfy
\begin{align*}
\label{eq:notsp1}
\mh_{n+1} &= m_{n},\\
\pCov_{n+1} &= C_n .
\end{align*}
To find $(m_{n+1},C_{n+1})$ it is again convenient 
to derive the formulae using
precision rather than covariance matrices. To this
end we view the Gaussian $\Ng(\mh_{n+1},\pCov_{n+1})$ as prior distribution
for the linear inverse problem defined by (\ref{eq:ipsd99}b) with data realization $\yd_{n+1}=\dt \wwd_R.$ 
Note that the likelihood, since linear and Gaussian, is conjugate to the prior 
so that the posterior on $u_{n+1}|Y^\dagger_{n+1}$ is Gaussian with mean
and covariance $(m_{n+1},C_{n+1})$ which can be found by completing
the square:
\begin{align*}
         \Cov_{n+1}^{-1} &= \pCov_{n+1}^{-1}+\dt L_R^\top \Gammas_R^{-1} L_R,\\ 
\Cov_{n+1}^{-1}m_{n+1} & = \pCov_{n+1}^{-1} \mh_{n+1} + \dt L_R^\top \Gammas_R^{-1} \wwd_R.
\end{align*}
We therefore find that
\begin{equation}
\label{eq:cov_conv}
    \Cov_{n}^{-1} = \Cov_{0}^{-1}+n\dt L_R^\top \Gammas_R^{-1} L_R.
\end{equation}
Recall that the posterior covariance $\Cp=(L_R^\top \Gammas_R^{-1} L_R)^{-1}$
from \eqref{eq:pcov}
is positive-definite. Hence it follows that the covariance converges  to zero algebraically fast: $\Cov_{n}=\mathcal{O}(1/n).$

Now note that
\begin{equation*}
\Cov_{n}^{-1}m_{n}  = \Cov_{0}^{-1}m_0 +n\dt L_R^\top \Gammas_R^{-1} \wwd_R,
\end{equation*}
so that
\begin{equation*}
m_{n}  = \Bigl(\Cov_{0}^{-1}+n\dt L_R^\top \Gammas_R^{-1} L_R \Bigr)^{-1}\Bigl(\Cov_{0}^{-1}m_0 +n\dt L_R^\top \Gammas_R^{-1} \wwd_R\Bigr).
\end{equation*}
We deduce that, since the posterior mean is given by
\eqref{eq:pmean},  
\begin{equation}
\label{eq:mean_conv}
m_{n}  = \Covpost L_R^\top\Gammas_R^{-1}\wwd_R+\mathcal{O}\bigl(1/n\bigr)=\mop+\mathcal{O}\bigl(1/n\bigr),
\end{equation}
again exhibiting algebraic convergence.
Combining \eqref{eq:cov_conv} and \eqref{eq:mean_conv} yields the result. $\blacksquare$
\end{example}

\vspace{0.1in}
\paragraph{Exponential Convergence}

Example \ref{ex:lin_c1} shows that when $\beta=0$ 
filtering based on \eqref{eq:rescale_beta_dt} leads to convergence
to the posterior distribution at algebraic rate, in the linear
setting, when fixing $\dt$ and taking $n\to \infty$.
We now show that when $\beta>0$ an exponential rate of convergence to the posterior distribution is obtained.

\begin{proposition}
\label{lem:lin_c2}
Assume that $u_0$ is initialized at a Gaussian
$\Ng(m_0,C_0)$ and assume also that $C_0, \Gammas_R \succ 0$. Consider the setting where
$G_R(\cdot)=L_R \cdot$ for matrix $L_R \in \R^{(d_w+d_u) \times d_u}$. 
Now consider the filtering distribution $u_n|Y^\dagger_n$ defined by \eqref{eq:rescale_beta_dt} for $\beta >0$, with data $Y^\dagger_n$ defined by $\yd_n=\dt \wwd_R$, where $\wwd_R$ is defined in \eqref{eq:wwr}. Then the filtering distribution is Gaussian $\Ng(m_n,C_n)$ for all $n\ge 1$. For any fixed $\Delta t>0$ the mean and covariance converge at an exponential rate $(1+\dt \beta)^{-n}$, as  $n\to \infty$, to the limits $m_{\infty}=\mop$ and
$C_{\infty}=\frac{\beta}{1+\beta\dt}\Cp$, where $(\mop,\Cp)$ are the posterior mean \eqref{eq:pmean} and covariance \eqref{eq:pcov}. 
$\Diamond$ \end{proposition}

\begin{remark}
Motivated by this proposition we may use the Gaussian Projected filter \eqref{eq:KF_analysis_add_IP} or the stochastic or deterministic Kalman transport algorithms, \eqref{eq:notsss1} and \eqref{eq:mfKBF01} respectively, to approximate the filtering distribution implied by \eqref{eq:rescale_beta_dt}. In so doing we generate approximate 
solutions of the Tikhonov regularized optimization problem  defined by \eqref{eq:ip}. Furthermore
the exact solution is recovered in the linear Gaussian setting.
    $\blacksquare$
\end{remark}

\begin{remark}
\label{rem:remcite99}
    It is a remarkable fact that the rate of convergence is independent of
    the properties of the limiting Gaussian posterior distribution, and in
    particular of the conditioning of the posterior covariance. This desirable
    property is a result of the \emph{affine invariance} of the Gaussian projected filter
    and ensemble Kalman methods that we deploy in this subsection. 
    Affine invariance is a subject
we will study in more detail in the context of continuous time approaches to
inversion, developed in Section \ref{sec:CTI}. In this context we note that
the Definition \ref{d:GW} may be extended to discrete time algorithms.
$\blacksquare$
\end{remark}

\begin{proof}[Proof of Proposition \ref{lem:lin_c2}]
We first identify the update equations for $(m_n,C_n)$. Note that the 
predictive mean $\mh_{n+1}$ and covariance $\pCov_{n+1}$
defined by (\ref{eq:rescale_beta_dt}a) satisfy
\begin{subequations}
\label{eq:notsp12}
\begin{align}
\mh_{n+1} &= m_{n},\\
\pCov_{n+1} &= (1+\beta\dt)C_n .
\end{align}
\end{subequations}
To find $(m_{n+1},C_{n+1})$ it is convenient 
to derive the formulae using
precision rather than covariance matrices. To this
end we view the Gaussian $\Ng(\mh_{n+1},\pCov_{n+1})$ as prior distribution
for the linear inverse problem defined by (\ref{eq:rescale_beta_dt}b) conditioned on specific realization of the data $\wwd_{n+1}=\wwd_R.$ 
Note that the likelihood, since linear and Gaussian, is conjugate to the prior 
so that the posterior on $u_{n+1}|W^\dagger_{n+1}$ is Gaussian with mean
and covariance $(m_{n+1},C_{n+1})$ which can be found by completing
the square:
\begin{subequations}
\label{eq:combining}
\begin{align}
         \Cov_{n+1}^{-1} &= \pCov_{n+1}^{-1}+\dt L_R^\top \Gammas_R^{-1} L_R,\\ 
\Cov_{n+1}^{-1}m_{n+1} & = \pCov_{n+1}^{-1} \mh_{n+1} + \dt L_R^\top \Gammas_R^{-1} \wwd_R.
\end{align}
\end{subequations}

Combining \eqref{eq:notsp12} and \eqref{eq:combining} shows that $(m_{n},C_{n})$ update according to the formulae
\begin{align*}
         \Cov_{n+1}^{-1} &= \Bigl(\frac{1}{1+\beta\dt}\Bigr)\Cov_{n}^{-1}+\dt L_R^\top \Gammas_R^{-1} L_R,\\ 
\Cov_{n+1}^{-1}m_{n+1} & = \Bigl(\frac{1}{1+\beta\dt}\Bigr)\Cov_{n}^{-1} m_{n} + \dt L_R^\top \Gammas_R^{-1} \wwd_R.
\end{align*}
We can therefore write
\begin{equation*}
\label{eq:cov_conv1}
    \Cov_{n}^{-1} = \Bigl(\frac{1}{1+\beta\dt}\Bigr)^n\Cov_{0}^{-1}+\left(\sum_{k=0}^{n-1}\Bigl(\frac{1}{1+\beta\dt} \Bigr)^k  \right) \dt  L_R^\top \Gammas_R^{-1} L_R.
\end{equation*}
and so 
\begin{equation}
\label{eq:cov_conv2}
    \Cov_{n}^{-1} = \Bigl(\frac{1}{1+\beta\dt}\Bigr)^n\Cov_{0}^{-1}+
    \frac{1+\beta \Delta t}{\beta}
    \left( 1- \left(\frac{1}{1+\beta\dt}\right)^{n} \right) L_R^\top \Gammas_R^{-1} L_R.
\end{equation}
Recall the posterior covariance
$\Cp=(L_R^\top \Gammas_R^{-1} L_R)^{-1}$
given in \eqref{eq:pcov}.
It is clear that the precision converges exponentially fast to 
$\Cp^{-1}$, scaled by $\frac{1+\beta \Delta t}{\beta}$, and hence that the covariance converges exponentially fast to the appropriately scaled $\Cp.$

Similarly we may write the expression for the mean as 
\begin{equation*}
%\Cov_{n}^{-1}m_{n}  = \Bigl(\frac{1}{1+\beta\dt}\Bigr)^n\Cov_{0}^{-1}m_0 +\left(\frac{ \bigl(1- (1+\beta\dt)^{-n} \bigr)}{\beta/(1+\beta\dt)}  \right)  L_R^\top \Gammas_R^{-1} \wwd_R,
\Cov_{n}^{-1}m_{n}  = \Bigl(\frac{1}{1+\beta\dt}\Bigr)^n\Cov_{0}^{-1}m_0 + 
\frac{1+\beta \Delta t}{\beta}
    \left( 1- \left(\frac{1}{1+\beta\dt}\right)^{n} \right)L_R^\top \Gammas_R^{-1} \wwd_R,
\end{equation*}
so that the exponential convergence of the mean to the steady state $\mop$ given by \eqref{eq:pmean} may be deduced, using the expression \eqref{eq:cov_conv2}. 
\end{proof}

\begin{remark}
Recall the Hessian $\Hess$ of $\Phi_R$, defined in \eqref{eq:pcov}.
Using the formulae for the predictive mean and covariance it is also easy to deduce that the expression for $m_{n+1}$ is given by 
\[
m_{n+1} = \Bigl(\frac{1}{1+\beta\dt}\Bigr)\cdot m_n + \Bigl(\frac{\beta\dt}{1+\beta\dt}\Bigr)\cdot \Hess^{-1}L^\top_R\Gammas_R^{-1}\wwd_R,
\]
and hence
\begin{equation}
\label{eq:GaussNewton}
    m_{n+1} = m_n + \Bigl(\frac{\beta\dt}{1+\beta\dt}\Bigr)\cdot\Hess^{-1}\bigl(\Hess m_n+ L^\top_R\Gammas_R^{-1}\wwd_R\bigr).
\end{equation}
Since $\nabla \Phi_R(u)=\Hess u+ L^\top_R\Gammas_R^{-1}\wwd_R$ and $D^2\Phi_R(u) = \Hess$, the iteration \eqref{eq:GaussNewton} may be viewed as a Gauss-Newton scheme for minimizing $\Phi_R$.
$\blacksquare$
\end{remark}

%%%%%%%%%%%%%%%%%%%%%%%%%%%%%%%%%%%%%%%%%%%%%%%%%%%%%%%%%%%
%
\subsubsection{Algorithms for Bayesian Formulation}
\label{sssec:IPB}

Again recall that we have introduced the non-standard mean field dynamical
system \eqref{eq:rescale_beta_dt} and shown how
the filtering distribution of the dynamical system
may be approximated by the Gaussian projected filter \eqref{eq:KF_analysis_add_IP}
and the mean field ensemble Kalman filter \eqref{eq:notsss1}.
In this subsection we substantiate the statements made about these algorithms in
Remark \ref{rem:remcite99_notneeded} in relation to Bayesian sampling. In particular
we show that they exactly recover the posterior in the linear Gaussian setting
by choosing
\begin{equation}
\label{eq:betat}
\beta = \frac{1}{1-\dt}.
\end{equation}
The following is a direct corollary of Proposition \ref{lem:lin_c2}. As in 
Remark \ref{rem:remcite99} we note that the rate of convergence, in this case to
the posterior distribution, is universal across all Gaussian posteriors.

\begin{corollary}
\label{cor:lin_c2}
Assume that $u_0$ is initialized at a Gaussian
$\Ng(m_0,C_0)$ and assume also that $C_0, \Gammas_R \succ 0$. Consider the setting where
$G_R(\cdot)=L_R \cdot$ for matrix $L_R \in \R^{(d_w+d_u) \times d_u}$. 
Now consider the filtering distribution $u_n|Y^\dagger_n$ defined by \eqref{eq:rescale_beta_dt} for $\beta$ given by \eqref{eq:betat}, with data $Y^\dagger_n$ defined by $\yd_n=\dt \wwd_R$, where $\wwd_R$ is defined in \eqref{eq:wwr}. Then the filtering distribution is Gaussian $\Ng(m_n,C_n)$ for all $n\ge 1$. For any fixed $\Delta t>0$ the mean and covariance converge at an exponential rate
$(1-\dt)^{n}$, as  $n\to \infty$, to the limits $m_{\infty}=\mop$ and
$C_{\infty}=\Cp$, where $(\mop,\Cp)$ are the posterior mean \eqref{eq:pmean} and covariance \eqref{eq:pcov}. 
$\Diamond$ \end{corollary}

\begin{remark}
Motivated by this corollary, in the context of \eqref{eq:rescale_beta_dt} we may use the Gaussian Projected filter \eqref{eq:KF_analysis_add_IP} or the stochastic or deterministic Kalman transport algorithms, \eqref{eq:notsss1} and \eqref{eq:mfKBF01} respectively, to generate approximate 
solutions of the Bayesian inverse problem 
defined by \eqref{eq:ip}, in the general nonlinear setting. 
Of course, because these algorithms employ Gaussian approximations, this does not produce the exact filtering distribution.
Furthermore, to make resulting algorithms tractable we will predict in (\ref{eq:rescale_beta_dt}a)
using the covariance $C_n$ of the Gaussian projected filter or
the Kalman transport algorithm, rather than the covariance under the true filtering distribution, which is not tractable: see final paragraph in Remark \ref{rem:nonstandard}. We note, however, that in the linear Gaussian case all the algorithms are exact and hence recover the true posterior.
$\blacksquare$
\end{remark}

%
%%%%%%%%%%%%%%%%%%%%%%%%%%%%%%%%%%%%%%%%%%%%%%%%%%%
\subsection{Ensemble Kalman Methods For Inversion: Examples}
\label{ssec:EKMI}
%
%%%%%%%%%%%%%%%%%%%%%%%%%%%%%%%%%%%%%%%%%%%%%%%%%%%

In this Section \ref{sec:IPDT}  we have concentrated, so far, entirely on mean field models, leaving the details of deriving particle approximations to the reader. In this subsection, however, we make a brief foray into finite particle ensemble approximations of the mean field models introduced in our discussion of inverse problems. The algorithms we employ can be found in Appendix \ref{sec:AA} as Algorithms \ref{alg:EKTI}, \ref{alg:EKOI} and \ref{alg:EKI_post}.

 Algorithms \ref{alg:EKTI}, \ref{alg:EKOI} are based on finite particle approximations
 of the mean field model in \eqref{eq:ipKT}. Algorithm \ref{alg:EKTI} is based on iterating
 until $N$ satisfying $N\dt=1$ and, at that time-step, aims to approximate the posterior;
  Algorithm \ref{alg:EKOI} performs the same iteration but to $N_{\infty}$ assumed to satisfy 
  $N_{\infty}\dt \gg 1$ so that it approximately solves an optimization problem -- see the discussion
  at the start of Subsection \ref{sssec:IPFTF2}.
Algorithm \ref{alg:EKI_post} is based on \eqref{eq:notsss1} and aims to approximate the
posterior by iterating to $N_{\infty}: N_{\infty}\dt \gg 1.$

In Example \ref{ex:EKI_1D} we study a one-dimensional nonlinear inverse problem;
working in one dimension enables comparison of the true posterior distribution with approximations 
arising from the various ensemble Kalman inversion schemes described in preceding subsections. We also demonstrate an optimization approach to inversion, in the context of Example \ref{ex:EKI_1D}. Subsequently, in Example \ref{ex:EKI}, we return to the setting of the Lorenz `96 dynamical system,
now to estimate unknown parameters rather than the state; our focus is on studying ensemble Kalman methods from the perspective of the optimization approach to the parameter estimation problem.

\begin{example}
\label{ex:EKI_1D}
Recall Example \ref{ex:linear1} concerning the linear Gaussian inverse
problem; there we show that the exact posterior is obtained either by iterating 
\eqref{eq:KF_analysis_add_IPL} $N$ times or by evaluating \eqref{eq:seblabel} at $n=N.$ Although derived in the context of the Gaussian projected filter 
the example also applies to 
mean field ensemble Kalman methods since they, like the Gaussian projected 
filter, are exact for linear Gaussian problems.
However in the nonlinear case this equivalence does not hold exactly, 
because of approximations that are made by the Gaussian projected and 
ensemble Kalman filtering methods. In this example
we study the effect of these approximations by examining the behaviour of 
particle-based ensemble Kalman methods applied to a nonlinear inverse problem.

We consider the setting of \eqref{eq:ip} with a nonlinear forward map $G:\R \to \R$, given by
\begin{equation}
\label{eq:1D_forward}
G(u) = \frac{7}{12}u^3-\frac{7}{2}u^2+8u.
\end{equation}
The observational noise is assumed to be of the form $\eta \sim \Ng(0,1)$. Assuming observation $\wwd=2$ results in likelihood  
\begin{equation}
\label{eq:1Dlike}
\frac{1}{\sqrt{2}\pi}\exp\Bigl(-\frac12(G(u)-2)^2\Bigr).
\end{equation}
Furthermore, assuming a Gaussian prior of mean $-2$ and variance $1/2$, the posterior on $u|\wwd$ is proportional to 
\begin{equation}
\label{eq:1Dpost}
\frac{1}{\sqrt{2}\pi}\exp\Bigl(-\frac12(G(u)-2)^2-(u+2)^2\Bigr).
\end{equation}
We note that the forward map $G(u)$ is monotonic and the posterior \eqref{eq:1Dpost} is unimodal: in the first panel of Figure \ref{fig:EKI_1D} we display the true posterior distribution, computed via quadrature. We also show the Gaussian prior and the likelihood \eqref{eq:1Dlike}.

The second panel of Figure \ref{fig:EKI_1D} shows the ensemble Kalman inversion iteration \eqref{eq:ipKT} from Subsection \ref{ssec:IPFT}, which is designed to transport prior to posterior in finite time, fixing $N$ iterations and $\dt$ so that $N\dt=1$; see Algorithm \ref{alg:EKTI}. Recall that the iteration exactly recovers the posterior, in the mean field limit, when applied to linear Gaussian inverse problems (Example \ref{ex:linear1}) but that here the inverse problem is nonlinear and non-Gaussian. We employ this algorithm with $J=2 \cdot 10^3$ ensemble members. We run Algorithm \ref{alg:EKTI} with two choices of $N:$
$N=4\cdot 10^3$ and hence $\dt=2.5 \cdot 10^{-4};$ and with $N=1$ and hence $\dt=1.$ The second panel of Figure \ref{fig:EKI_1D} shows that the one-step approach, with $N=1$, leads to a very poor approximation of the posterior. In contrast, the scheme with $N=4\cdot 10^3$ yields reasonable approximation quality of the posterior; see Remark \ref{rem:chopin}. In this panel we also run Algorithm \ref{alg:EKOI} for the finite number of iterations $N_{\infty}=10^6$, with $\dt=2.5 \cdot 10^{-4}$. The result reflects the theoretical interpretation: iterating $n \to \infty$ leads to solution of the optimization formulation of ensemble Kalman inversion and, as discussed in Subsection \ref{ssec:IPIFT}, results in convergence to a Dirac centered at the minimizer of the least squares functional $\Phi,$ the maximum likelihood estimate found by maximizing \eqref{eq:1Dlike};
this simply delivers the point $G^{-1}(2).$

The third panel of  Figure \ref{fig:EKI_1D} shows the ensemble Kalman inversion iteration \eqref{eq:notsss1} from Subsection \ref{ssec:IPIFT}, namely Algorithm \ref{alg:EKI_post}. This is designed to approximate the true posterior, when $1/\beta=1-\dt$: indeed Corollary \ref{cor:lin_c2} shows that in the linear setting the mean field model \eqref{eq:notsss1} converges to the true posterior in limit $n \to \infty$, with this choice of $\beta.$ Our numerical results,
which are conducted with this choice of $\beta$, show that use of the Algorithm \ref{alg:EKI_post}, when applied to the nonlinear and non-Gaussian inverse problem, produces an excellent posterior
approximation; this demonstrates that the linear theory is indicative of 
behaviour of the algorithm beyond
the linear Gaussian setting.
$\blacksquare$

\begin{figure}[h!]
    \centering
\includegraphics[width=0.8\linewidth]{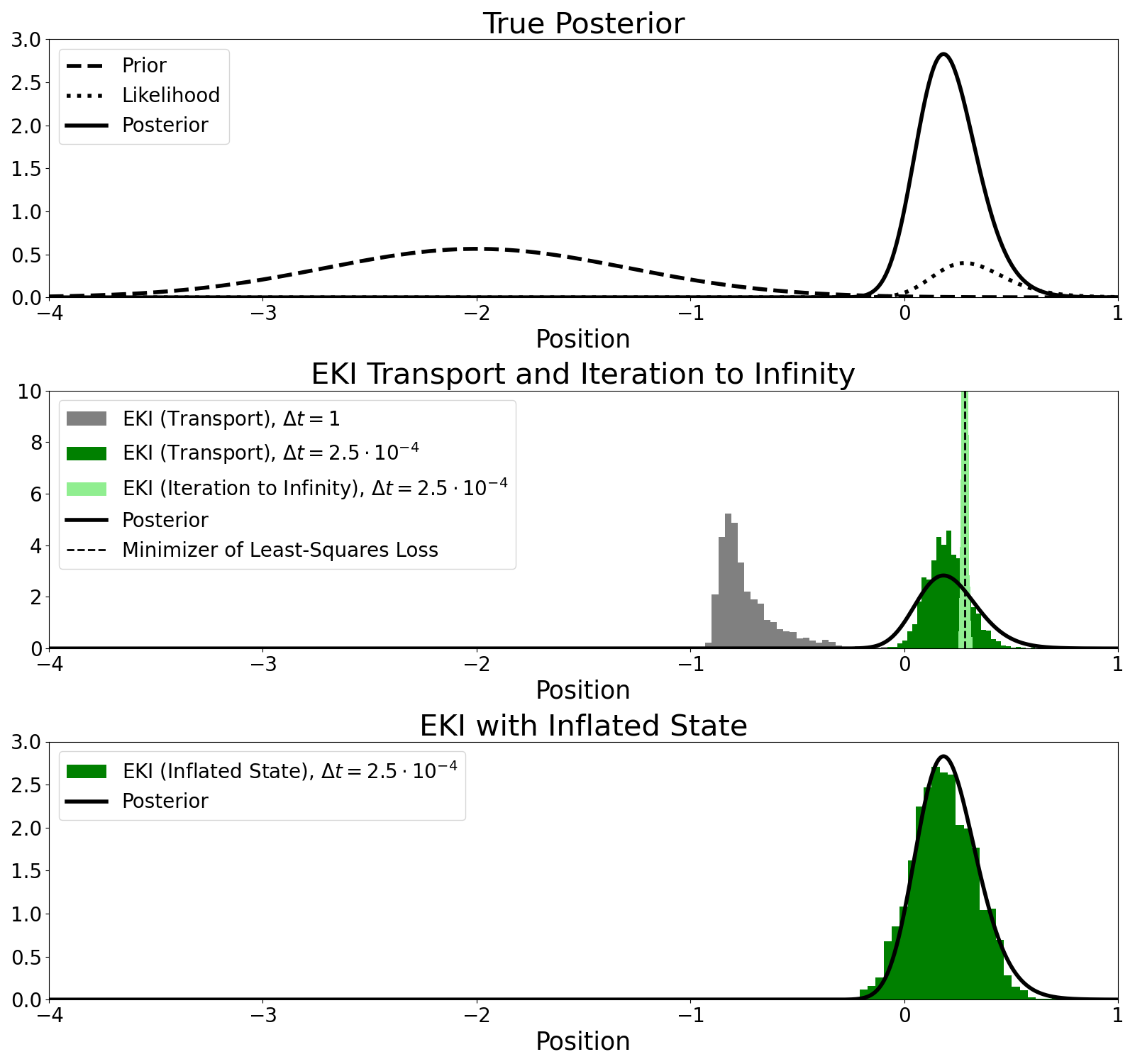}
    \caption{The plots display the results obtained for the inverse problem described by the one-dimensional nonlinear forward map \eqref{eq:1D_forward}. The first panel displays the prior, likelihood and posterior computed via quadrature. In the second panel, we compare the true posterior pdf with the approximation obtained using the finite ensemble Kalman inversion iteration \eqref{eq:ipKT}, as detailed in Algorithm \ref{alg:EKTI}, with $\dt=2.5\cdot 10^{-4}$ and $\dt=1$, 
iterated to time $n=N$ where $N\dt=1.$ The second
    panel also includes results found from applying the optimization Algorithm \ref{alg:EKOI}, iterating over $N_{\infty}=10^6$ steps with $\dt=2.5\cdot 10^{-4}$; in this case, the ensemble approaches the Dirac measure supported on the minimizer of the unregularized least-squares loss $\Phi,$ given by the peak of the likelihood.
    Similarly in the third panel we display the posterior approximation using the ensemble Kalman inversion iteration with covariance inflation as described by the mean field model \eqref{eq:notsss1}, as detailed in Algorithm \ref{alg:EKI_post}, with $\beta=\frac{1}{1-\dt}$ for 
    parameter $\dt=2.5\cdot 10^{-4}$; the algorithm is iterated over $N_{\infty}=10^5$ steps. It is clear that this scheme yields the highest quality posterior approximation.} 
    \label{fig:EKI_1D}
\end{figure}
\end{example}

\begin{example}
\label{ex:EKI}

As in the examples from Section \ref{sec:SE}, we again consider the Lorenz '96 (singlescale) model  for $v \in C(\R^+,\R^L)$ satisfying the equations 
\begin{subequations}
  \label{eq:l96_eki}
\begin{align}
\dot{v}_\ell &=  -v_{\ell-1}( v_{\ell-2} - v_{\ell+1}) - v_{\ell} + u + h_v m\bigl(v_\ell\bigr), \quad
\ell=1 \dots L\,,\\
v_{\ell + L} &= v_\ell, \quad \ell=1 \dots L\,. 
\end{align}
\end{subequations}
As before we set $L = 9, h_v=-0.8$ and $u=10$.
We recall that function $m$ is shown in Figure \ref{fig:multiscale_m}.
In Section \ref{sec:SE} we focused on recovering the state $v$ from partial and
noisy observations. Here we concentrate on recovering the 
parameter $u$.\footnote{We have used the notation $u$ instead of $F$, 
for the forcing parameter, to align with
the notation for the unknown parameter used throughout the section
concerning inverse problems.} 

Our objective is to recover parameter
$u$ from time-averaged data. We assume that the system is ergodic so that infinite time-averages
produce averages over the invariant measure. Furthermore we assume that convergence in time,
of averages, is governed by a central limit theorem. We let 
$G_T:\R \rightarrow \R^2$ denote the mean and variance, 
defined via averaging over time $T$ and
over the $L$ components of $v$, of the state of system \eqref{eq:l96_eki}. 
In principle $G_T$ depends also on intialization, but this
effect is negligible for $T$ large, and zero for $T=\infty$, by ergodicity.
In particular, with system state $\vd$  evolving according to
\begin{equation}
\label{eq:493}
    \vd_{n+1} = \Psi(\vd_{n}),
\end{equation}
where $\Psi$ is the solution operator for \eqref{eq:l96_eki} over the observation time interval $\tau$, with true parameter $u=u^\dagger$, 
the action of forward operator $G_T$ on $u$ is defined as follows:
\begin{equation*}
    G_T(u) = \begin{pmatrix} w_1 \\ w_2\end{pmatrix},
\end{equation*}
with, for $M\tau=T$,
\begin{align*}
    w_1 = \frac{1}{L}\sum_{l=1}^L \Bar{v}^\dag_l, \quad w_2 = \frac{1}{L \cdot {M}}\sum_{n=1}^{M}\sum_{l=1}^L (v^\dag_{n;l} - \Bar{v}^\dag_l)^2,\quad
    \Bar{v}^\dag_l = \frac{1}{M}\sum_{n=1}^{M}v^\dag_{n;l},
\end{align*}
where we have used  $v^\dag_{n;l}$ to denote the $l^{th}$ variable in
vector $\vd_n$.

We consider finding $u$ from an observation $w\in \R^2$ arising from the model
\begin{equation}
\label{eq:492}
    w = G_\infty(u) + \gamma.
\end{equation}
In practice the specific realization $\wwd$, from which we invert to find $u$,
is found by integrating $G_T$ to a finite time $T=100$, not $T=\infty$.
Variable $\gamma \sim \text{N}(0,\Gammas)$ accounts for the 
resulting central limit theorem correction.
To solve the parameter estimation problem for $u$ we use ensemble Kalman methods in Algorithms \ref{alg:EKTI}, \ref{alg:EKOI} and \ref{alg:EKI_post}. We do not have access to $G_\infty$ and so, instead,
the algorithms are implemented by using $G_T$ with $T=10$, initialized after a burn-in time of duration $t^*=10$.
The burn-in phase itself results from an initial condition chosen at random from a Gaussian distribution with mean $0$ and standard deviation $40$. In the experiments shown we take $\Gammas=\sigma^2 I$, with $\sigma = 10^{-1}$. All the ensemble Kalman inversion schemes are initialized 
from a prior Gaussian of mean $0$ and standard deviation $10$, and use an 
ensemble of size $J=30.$ 
 
In Figure \ref{fig:EKI_L96_opt} we display the ensemble approximations obtained via application of the EKI methodology for optimization, namely Algorithm \ref{alg:EKOI}. In practice, the scheme is run for a finite number $N_{\infty}$ of iterations. Indeed, in the first panel of Figure \ref{fig:EKI_L96_opt} the scheme is applied with $\dt =5\cdot 10^{-2}$ and run up to $N_{\infty}=40$ iterations. On the other hand, in the second panel the scheme is run for $20$ iterations with $\dt=1$. In both settings ensemble collapse occurs as the number of iterations grow. The ensemble mean, displayed as the central line in each box plot, converges to a point yielding a qualitatively good estimate of the true forcing parameter, with an error of $\mathcal{O}(10^{-1})$.

\begin{figure}[h!]
    \centering
\includegraphics[width=0.8\linewidth]{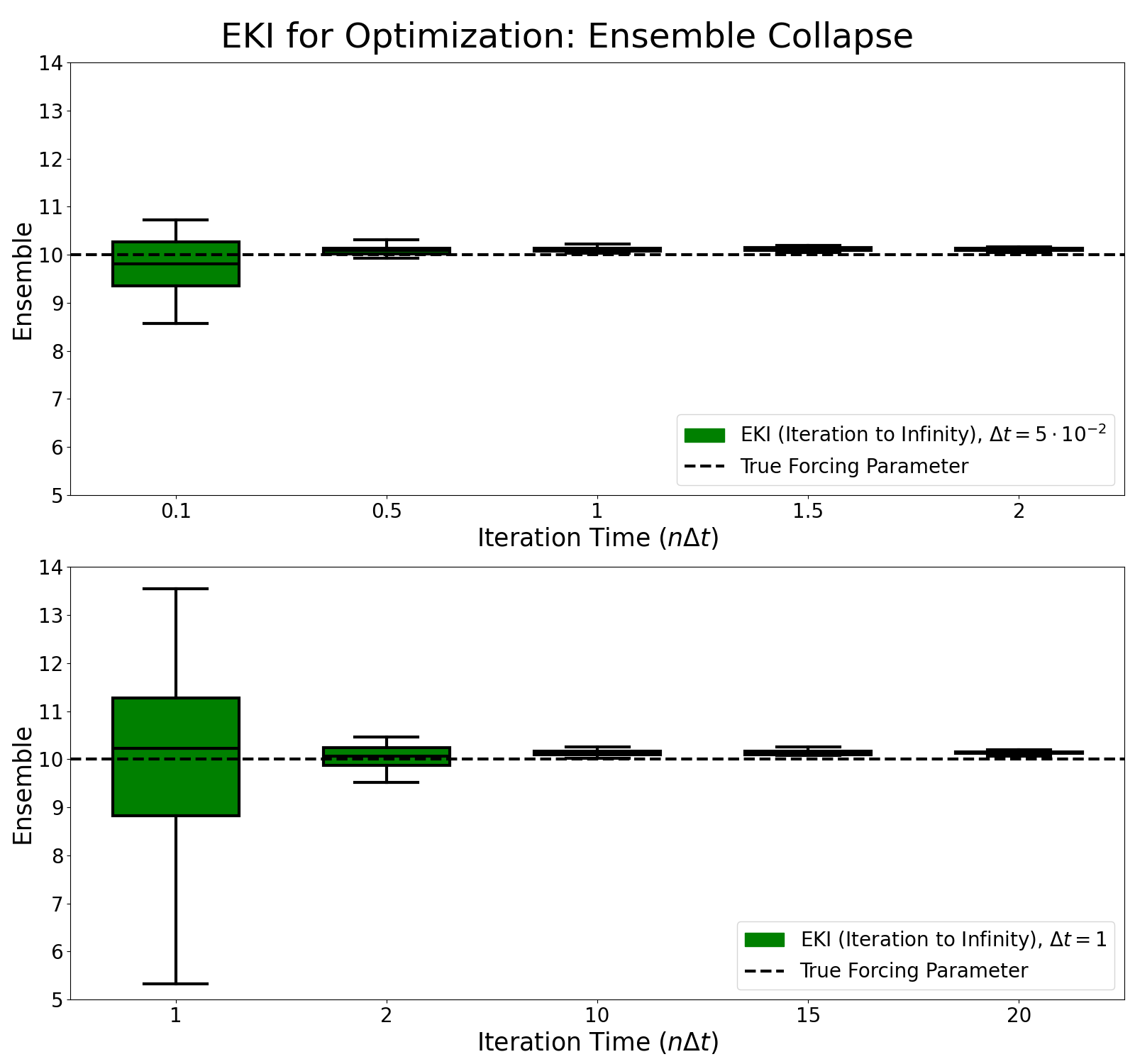}
    \caption{The figure displays box and whisker plots for the ensembles produced by Algorithm \ref{alg:EKOI} with $\dt = 5\cdot 10^{-2}$ (top panel) and $\dt = 1$ (bottom panel). The box and whisker plots represent the ensembles by depicting the ensemble mean, as a line within the shaded region. Furthermore, the edges of the boxes represent the first and third quartiles, i.e. the values of ensemble members corresponding to the median of the first half of the samples, and the median of the second half of the samples, respectively. Finally, the whiskers mark the furthest samples lying within a distance from the box of $1.5$ times the distance between the first and third quartiles (the interquartile range). In both cases we note ensemble collapse onto a value close to the truth underlying the data.}
    \label{fig:EKI_L96_opt}
\end{figure}

In Figure \ref{fig:EKI_L96_bayes} we show application of Algorithm \ref{alg:EKTI} with $\dt$ set to $5\cdot10^{-2}$, running for $N=20$ steps ($N\dt=1$) and of Algorithm \ref{alg:EKI_post} with
 $1/\beta=1-\dt$ and $\dt = 5\cdot10^{-2}$ for $N_{\infty}=40$ steps. 
 For linear inverse problems, the output of both algorithms at these specific steps
delivers the posterior distribution exactly in the mean field limit, by Example \ref{ex:linear1}
and Corollary \ref{cor:lin_c2}. As noted in the previous paragraph, such a posterior approximation should be interpreted with caution for this nonlinear inverse problem. However we show in Figure \ref{fig:EKI_L96_bayes} that the ensemble means accurately predict the true forcing up to an error of $\mathcal{O}(10^{-1})$
and that, furthermore, the two ensembles are similar. However, interpreting the posterior
distributions in this case is harder as we do not have access to the true posterior. We note that in the preceding Example \ref{ex:EKI_1D} we were able to demonstrate that Algorithm \ref{alg:EKI_post} delivered a better posterior approximation than Algorithm \ref{alg:EKTI} and it would be interesting
to determine whether such a conclusion holds more generally.
%Recall from Example \ref{ex:linear1} that, for any $\dt$, the linear inverse problem is solved exactly in the mean field limit by looking at the distribution of Algorithm \ref{alg:EKTI} at time $N=1/\dt.$ However, we caution that in this nonlinear setting we
%do not have the true posterior against which to compare.}
$\blacksquare$
\end{example}

\begin{figure}[h!]
    \centering
\includegraphics[width=0.8\linewidth]{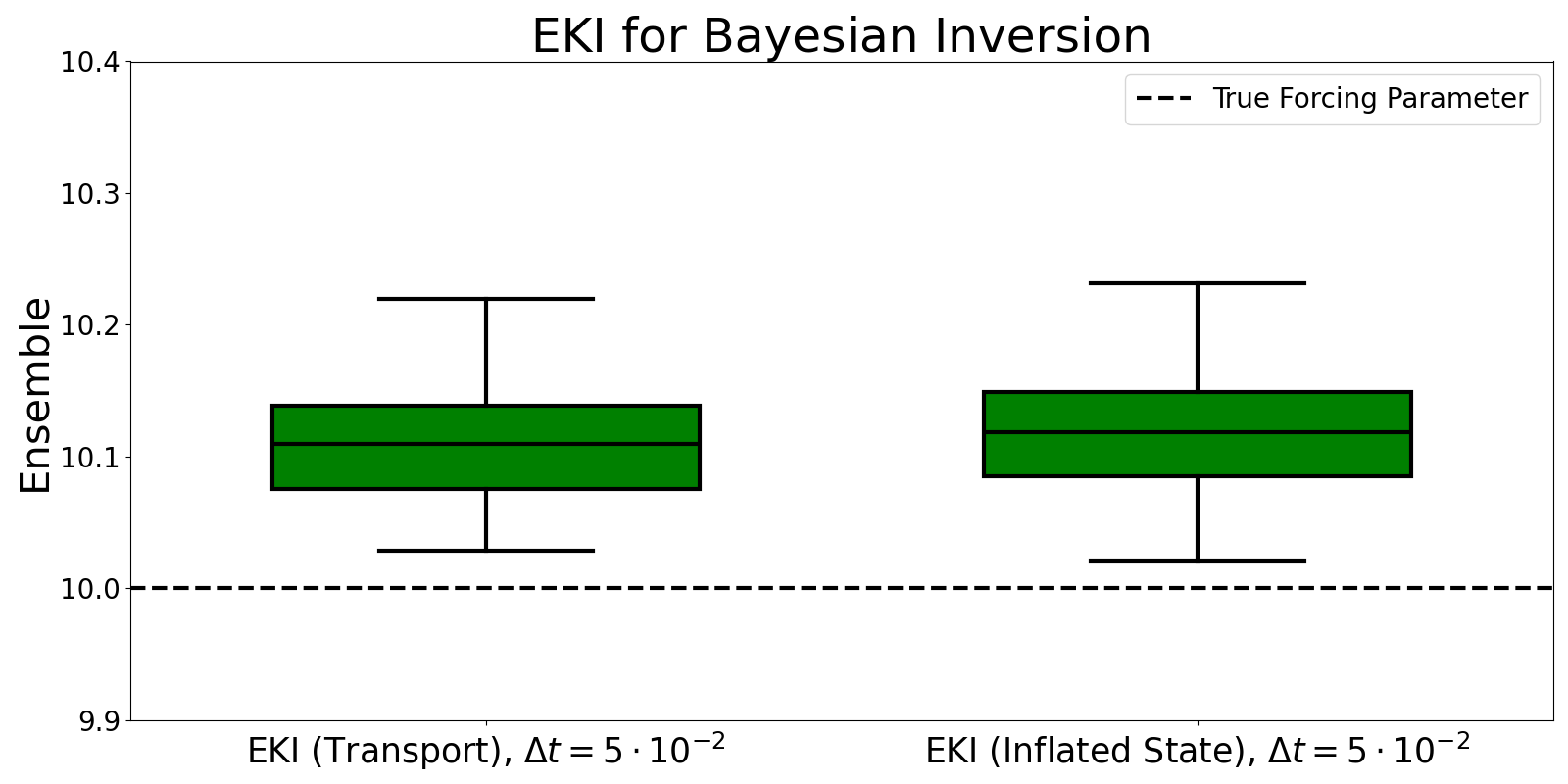}
    \caption{The figure displays box and whisker plots for the Bayesian posterior ensemble approximation produced by Algorithm \ref{alg:EKTI} with $\dt = 5\cdot 10^{-2}$ and Algorithm \ref{alg:EKI_post} with $\dt = 5\cdot 10^{-2}$. The box and whisker plots represent the ensembles by depicting the ensemble mean, as a line within the shaded region. Furthermore, the edges of the boxes represent the first and third quartiles, i.e. the values of ensemble members corresponding to the median of the first half of the samples, and the median of the second half of the samples, respectively. Finally, the whiskers mark the furthest samples lying within a distance from the box of $1.5$ times the distance between the first and third quartiles (the interquartile range). The posterior mean in both cases in close to the true value of the parameter underlying the data; the size of the posterior spread is also similar. Note, however, that in this problem we do not have a true posterior distribution against which to compare.} 
    \label{fig:EKI_L96_bayes}
\end{figure}

%%%%%%%%%%%%%%%%%%%%%%%%%%%%%%%%%%%%%%%%%%%%%%%%%
%
\subsection{Bibliographical Notes}
\label{ssec:IPBIB}
%
%%%%%%%%%%%%%%%%%%%%%%%%%%%%%%%%%%%%%%%%%%%%%%%%%

Sequential Monte Carlo methods may be used to approximately morph one probability distribution (source) into another (target), using empirical approximation and a discrete-time homotopy \citep{del2006sequential, chopin:20}. In general the methodology does not scale well to high dimensional problems \citep{beskos2014stability}. However some success has been achieved in this direction \citep{kantas2014sequential},
and a basic underlying theory is described in \citet{beskos2015sequential}. Our presentation in this paper is confined to the setting of ensemble Kalman methods
because of their empirical success and scalability to high dimensions.

The development of ensemble Kalman methods for inverse
problems was pioneered in the study of reservoir simulation,
in the context of learning subsurface properties from localized
flow measurements \citep{chen2012ensemble, gu2007iterative,li2007iterative, emerick2013ensemble, emerick2013ensembleb,evensen2018analysis}.
Subsequent work has studied parameter estimation in chaotic dynamical systems, such as those arising in weather forecasting
using ensemble methods for joint state and parameter estimation 
\citep{pulido2018,bocquet2020,gottwald2021supervised},  and by matching to
time-averaged statistics \citep{schneider2017earth, CES, dunbar2020calibration}, motivated by climate modeling.
The chaotic dynamics that underlie weather and climate models lead to complicated energy landscapes for minimization and sampling \citep{lea2000sensitivity}. The papers \citet{huang2022iterated,duncan2021} demonstrate the benefits of using ensemble methods, rather than computing exact
derivatives, for such problems: the ensemble approach effectively works in
a smoothened energy landscape.

The idea of transporting prior to posterior, as developed in 
Subsection \ref{sssec:IPFTF}, is widely used in the statistics literature; see \citet{chopin:20,del2006sequential} for
a unified perspective and for citations to earlier works which characterize the deformation of one measure to the other either through incrementally building up the available data, or through a temperature like annealing
parameter in the likelihood. In the context of data assimilation, and ensemble Kalman methods in particular, these ideas were developed by \citet{li2007iterative}, \citet{gu2007iterative}, \citet{daumetal2010}, \citet{reich2011dynamical}, and \citet{sakov2012}. 

Filtering using dynamical systems with equilibrium filtering distribution which solves the inverse problem, the approach adopted in Section \ref{ssec:IPIFT}, 
is an idea developed in \cite{iglesias2013ensemble}. That paper also highlights
an invariant subspace property of finite ensemble Kalman methods for
inverse problems: the iteration remains in the linear span of the initial
ensemble. More recent 
methodology is developed in the context of
 optimization in \citet{huang2022iterated}, although the methodology therein is
not affine invariant; it is also developed
 for sampling in \citet{huang2022efficient}s, resulting in an affine
invariant methodology. See those papers for details concerning uniqueness of, and exponential convergence to, steady state solutions in the linear Gaussian setting. Here we have adopted the approach of \citet{huang2022efficient} to both optimization and sampling, which yields exponential convergence under both settings, as discussed in Proposition \ref{lem:lin_c2}.
A geometric picture of iterative ensemble Kalman methods for inverse problems is developed in \cite{qian24}; the authors consider the fundamental observed
and unobserved subspaces defined by linear inverse problems and their 
interaction with the invariant subspace defined by ensemble Kalman iteration. 

The stochastic perturbations utilized in 
\eqref{eq:rescale_beta_dt} are closely related to multiplicative ensemble inflation methods as widely used in ensemble Kalman filter implementations \citep{asch2016data,Evensenetal2022}. The effects of additive inflation and variable step-size implementations of ensemble Kalman inversion have been studied in \citet{CT22} and in \citet{WCST22}.

The optimization and sampling approaches for inverse problems can in principle be combined with ideas from stochastic annealing \citep{kushner2003stochastic}, and stochastic gradient descent \citep{goodfellow2016deep}.
The papers \citet{haber2018never,kovachki2019ensemble,SR-PR21} demonstrate the use of ensemble Kalman methods for inversion,
when combined with stochastic gradient descent, and mini-batching in particular, as well as the application of ensemble Kalman methods beyond the setting of the $L_2$-loss functions $\Phi$ and $\Phi_R$ considered here.

However, despite the growing use of ensemble Kalman methods to solve inverse problems, it is important to appreciate that all ensemble Kalman-based methods invoke approximations which amount to matching only first and second order statistics, at some point in the algorithmic development. For this reason the methods are intuitively only useful as samplers for problems with posterior distribution close to a Gaussian. This idea is carefully explained in \citet{ernst2015analysis} where the mean field limit of ensemble Kalman methods for inverse problems is compared with the desired posterior distribution;
as is the case for state estimation, analysis is required to justify use of ensemble Kalman methods beyond the linear and Gaussian regime. We also highlight that our analysis in this paper, which focuses on the mean field limit, does not capture important aspects of the performance of ensemble Kalman methods at finite ensemble size, and the important practical
issue of covariance localization; these issues are studied in
\citet{ghattas2022non}. Furthermore, the paper \citet{tong2022localized} studies the relationship between localization, in the solution of inverse problems
using ensemble Kalman methods, and the subspace property explained in \citet{iglesias2013ensemble}[Theorem 2.1]. As an alternative to localization, the concept of dropout combined with variable step-size has been shown to lead to optimal algorithmic performance in \citet{SRT23}. 

The idea of using ensemble methods for performing the optimization step within variational data assimilation was introduced in \citet{Zupanski2005}. The connection between iterative applications of the ensemble Kalman filter and optimization were first investigated in \citet{iglesias2013ensemble} and developed to include constraints in \citet{albers2019ensemble}, \citet{chada2019incorporation} and Tikhonov regularization \citet{chada2020tikhonov}. Recall from Subsection \ref{sssec:IPO} the algebraic rates of convergence arising
in the basic optimization method arising from iterating to infinity. This undesirable feature of optimization methods based on statistical linearization of mean field gradient descent can be ameliorated to some extent by the use of adaptive time-steps, connections to the Levenberg-Marquadt algorithm, and the use of  stopping criteria; see 
\citet{iglesias2015iterative},
\citet{iglesias2016regularizing},
and \citet{iglesias2021adaptive}. 
Recent work of \citet{parzer2022convergence} has developed a systematic theory for early stopping using ensemble Kalman inversion including incorporation of Nystr\"om methodology. Other interacting particle system approaches to optimization have been proposed, including  feedback particle \citep{ZTM17}, unscented Kalman approaches \citep{huang2022iterated,huang2022efficient}, and concensus based optimization \citep{tsianos2012consensus,carrillo2018analytical,fornasier2020consensus,ha2021convergence}. 

Finally we note that Kalman methods have been related to approximate Bayesian computation (ABC) methodologies,  utilizing a linear regression ansatz
\citep{sisson2018handbook, nott2012ensemble}. 
Such methods are in turn closely related to 
Bayes linear and best linear unbiased estimators (BLUE) as discussed in \citet{goldstein2006bayes},
\citet{goldstein2007bayes},
\citet{SR-bickel11},
\citet{nott2012ensemble},
\citet{SR-snyder2014},
\citet{goldstein2014b},
\citet{reich2015probabilistic}, and
\citet{latz2016bayes}. BLUE is discussed in more detail in
Appendix \ref{ssec:TM_MVA}.

%%%%%%%%%%%%%%%%%%%%%%%%%%%%%%%%%%%%%%%%%%%%%%%%%%%%%%%%%%%%%%%%%%%%%%
%
%
%
%
%
%
%
%
%
%
%
%  Section 5: Inverse problems: continuous time
%
%
%
%
%
%
%
%
%
%
%
%%%%%%%%%%%%%%%%%%%%%%%%%%%%%%%%%%%%%%%%%%%%%%%%%%%%%%%%%%%%%%%%%%%%%

%%%%%%%%%%%%%%%%%%%%%%%%%%%%%%%%%%%%%%%%%%%%%%%%%%%%%%%%%%%%%%%%%%%%%%%
%
\section{Inverse Problems: Continuous Time}
\label{sec:CTI}
%
%%%%%%%%%%%%%%%%%%%%%%%%%%%%%%%%%%%%%%%%%%%%%%%%%%%%%%%%%%%%%%%%%%%%%%

In this section we derive continuous time limits of the ideas developed in the preceding
Section \ref{sec:IPDT} for the solution of inverse problems. As a consequence the
ideas may also be viewed as adaptations of Section \ref{sec:CT} to the solution of inverse problems.
We start in Section \ref{ssec:IPCT} by recalling the inverse problem,
followed in Subsections \ref{ssec:IPOPCT} and  \ref{ssec:IPBPCT} by discussion of the
optimization and Bayesian approaches respectively, focusing on gradient flows; this flow
perspective provides a conceptual basis for thinking about the algorithms for inverse problems
that we will subsequently develop. Subsection \ref{ssec:IPFTC}
is devoted to Bayesian probabilistic filtering methods which solve the inverse problem by morphing the
prior into the posterior in a finite time. Subsection \ref{ssec:IPIFTC} discusses filtering methods
which work on infinite time horizons, exhibiting exponential convergence to approximate
solutions of the optimization or Bayesian formulations of the problem, from arbitrary starting points.
Analogously to Section \ref{sec:IPDT} we demonstrate application of Gaussian projected filtering and ensemble Kalman methods to solve the filtering problems defined in Subsections \ref{ssec:IPFTC} and \ref{ssec:IPIFTC}; the comments from Remark \ref{rem:nopart} apply here too.
We conclude in Section \ref{ssec:IPBIBC} with bibliographic notes.

%%%%%%%%%%%%%%%%%%%%%%%%%%%%%%%%%%%%%%%%%%%%%%%%%%%%%%
\subsection{Set-Up}
\label{ssec:IPCT}
%%%%%%%%%%%%%%%%%%%%%%%%%%%%%%%%%%%%%%%%%%%%%%%%%%%%%%%

Recall the inverse problem \eqref{eq:ip} of recovering $u$ from $w$ where
\begin{empheq}[box=\widefbox]{equation*}
w=G(u)+\gamma,
\end{empheq}
introduced in full detail in Subsection \ref{ssec:IPDT}.
Under the assumptions laid out in Subsection \ref{ssec:IPBP}, and in particular
Gaussianity and independence of unknown parameter $u$ and noise $\eta$, and given a specific
realization $\wwd$ of data $w$, we have a posterior
distribution on the random variable $u|\wwd$ which is defined by\footnote{The normalization constant $\normZ$ is the probability of the observed data under the model, sometimes called the \emph{evidence}.}
\begin{subequations}
\label{eq:bayesa}
\begin{empheq}[box=\widefbox]{align}
\mu(\dd u) &= \frac{1}{\normZ} \exp\bigl(-\Phi_R(u)\bigr)\dd u,\\
\normZ &= \int_{\R^{d_u}} \exp\bigl(-\Phi_R(u)\bigr)\dd u.
\end{empheq}
\end{subequations}
Here
\begin{subequations}
\label{eq:phisc}
\begin{empheq}[box=\widefbox]{align}
\Phi(u) &= \frac12|\wwd-G(u)|_{\Gammas}^2,\\
\Phi_R(u) &= \Phi(u)+\frac12|u-m_0|_{C_0}^2,
\end{empheq}
\end{subequations}
for prior mean vector vector $m_0$, prior covariance matrix $C_0 \succ 0$ and noise
covariance matrix $\Gammas \succ 0$. 

Rather than solving the Bayesian inverse problem, which can be prohibitively expensive,
optimization methods may be developed to find a point estimate of $u|\wwd$ as minimizer 
of $\Phi$ over a compact set, or as minimizer of $\Phi_R$ over the whole space $\R^{d_u}.$
The next two sections show, respectively, how we may develop continuous time gradient
flows which minimize $\Phi_R$, or $\Phi$, and gradient flows which find the posterior distribution $\mu$.

%%%%%%%%%%%%%%%%%%%%%%%%%%%%%%%%%%%%%%%%%%%%%%%%%%%
\subsection{Optimization Formulation: Gradient Flows}
\label{ssec:IPOPCT}
%%%%%%%%%%%%%%%%%%%%%%%%%%%%%%%%%%%%%%%%%%%%%%%%%%%

The goal of this subsection is to study gradient flows to minimize
an objective function. For us particular focus is on the choice
of $\Phi$ or $\Phi_R$ as objective, but since some of the considerations are quite general we frame aspects of the discussion
in a general setting.

\paragraph{Deterministic Viewpoint}

Consider the standard gradient descent, applied to an energy function $\Psi: \R^{d_u} \to \R^+$,
namely
\begin{equation} \label{eq:GD1}
    \frac{\dd u}{\dd t} = -\nabla \Psi(u).
\end{equation}
Note that this may be found as the continuous time limit of the
discrete time gradient descent algorithm \eqref{eq:dgrad}, 
choosing $\alpha=\dt$, letting $u(n\dt)=u_n$ and taking the limit $\dt \to 0.$

Along solutions of (\ref{eq:GD1}),
\begin{subequations}
\label{eq:gradientstructure}
\begin{align}
    \frac{\dd}{\dd t}\Psi(u) 
    &= \Big\langle \nabla \Psi(u), \frac{\dd u}{\dd t} \Big\rangle\\
    &= -\Bigl|\frac{\dd u}{\dd t}\Bigr|^2.
   \end{align}
\end{subequations}

Equation (\ref{eq:GD1}) is said to possess a gradient flow structure in parameter space $\mathbb{R}^{d_u}$
because the vector field driving the evolution of $u$ is tangential
to the gradient of the \emph{energy} $\Psi(u)$ in the standard Euclidean 
\emph{metric}; this is the geometric reason for the non-increasing 
property of $\Psi(u)$ along trajectories. 

\paragraph{Geometric Perspective on the Space of Probability Densities}

We now define gradient flow structure from a probabilistic viewpoint, studying evolution of probability densities.
Again we need both an energy and a metric. To this end we introduce some 
notation that will be useful in the probabilistic
formulation of \eqref{eq:GD1}. It will also be used more generally in subsequent discussion of other gradient flows on the space of probability density functions. 

We denote the manifold of all smooth probability
density functions on $\R^{d_u}$ by $\mP_+=\mP_+(\R^{d_u})$. 
We may then define the tangent space $\Tp$ to $\mP_+$, at $\rho\in\mP_+$, by
\begin{align}
\label{eq:citets}
    \Tp  = \left\{ \sigma\in C^{\infty}(\R^{d_u})\,:\, \int_{\R^{d_u}} \sigma(u)\, \dd u=0 \right\}\,.
\end{align}
Given the tangent space 
$\Tp$ we define its dual\footnote{This informal definition of tangent space, and its dual space, requires careful handling for probability measures on non-compact manifolds, such as $\R^{d_u}$; see citations to the literature in the bibliography Subsection \ref{ssec:IPBIBC}. The dual is also known as the cotangent space.
} 
\begin{align}
    \label{eq:citets_dual}
\Tp^\star = \left\{ \psi\in C^{\infty}(\R^{d_u})\,:\, \int_{\R^{d_u}} \psi(u)\, \rho(u)\,\dd u=0 \right\}\,.
\end{align}

In this article, for simplicity of exposition, we will define 
underlying metric structure via the positive operator $\sM(\rho): \Tp \to \Tp^\star$. A precise mathematical treatment requires further assumptions on the considered set $\mP_+$. See the bibliography for relevant literature on this topic.
Operator  $\sM(\rho)$ may be linked to an underlying Riemannian metric tensor $g_{\rho}\,:\, \Tp\times \Tp \to\R$; however, since
this metric tensor plays no role in our presentation, we will
work directly with $\sM(\rho)$, and with its inverse
$\sM(\rho)^{-1}: \Tp^\star \to \Tp$.
Note that $\sM(\rho)^{-1}$ maps into the tangent space $\Tp$, implying that it maps into functions that integrate to zero over $\R^{d_u}$; see \eqref{eq:citets}.

\paragraph{Probabilistic Viewpoint}

We now demonstrate gradient structure  inherent in the probabilistic viewpoint on the ODE \eqref{eq:GD1}, arising from allowing the initial condition $u(0)$ to be random.
We will show that the Liouville equation governing the evolution
of the probability density function associated with the random
variable $u(t)$ also has a gradient structure. In so doing
we must exhibit an appropriate energy and metric.

We assume 
that $u(t)$ has smooth probability density $\rho(u,t)$ for
all $t\ge 0$. 
Then $\rho$ satisfies the Liouville equation
\begin{equation} \label{eq:Liouville}
    \partial_t \rho = \nabla \cdot (\rho \nabla \Psi).
\end{equation}
The energy and metric defining the gradient structure for equation \eqref{eq:Liouville} are
\begin{subequations}
\label{eq:abstractform}
\begin{align}
    \mE(\rho) &\coloneqq \int_{\R^{d_u}} \Psi (u) \rho(u) \dd u,\\
    \sM(\rho)^{-1} \psi &\coloneqq  -\nabla \cdot (\rho \nabla \psi) \in \Tp.
\end{align}
\end{subequations}
Operator $\sM(\rho)$ corresponds to an underlying 
\emph{Wasserstein-2 metric structure.}

The standard $L_2$ variational derivative\footnote{The $L_2$ variational derivative  is identified by writing  $\mE(\rho+\sigma)-
\mE(\rho)$ as a linear operator acting on $\sigma$ (plus higher order terms in $\sigma$ for energies $\mE(\rho)$ which are
not linear in $\rho$).} of $\mE$ is given by
\begin{equation*}
\frac{\delta \mE}{\delta \rho}=\Psi.
\end{equation*}
The restriction of this variational derivative to the dual space $\Tp^\star$ is provided by
\begin{align} \label{eq:VD dual}
\frac{\delta \mE}{\delta \rho}_{\vert \Tp^\ast}=\Psi - \mathbb{E}[\Psi].
\end{align}
In the context of the Wasserstein-2 metric structure as presented
here it is not necessary to distinguish between these two formulations of the variational derivative; however the second formulation allows
for unique solvability of the elliptic equation required to define $\sM(\rho).$ For the Fisher--Rao metric structure considered later in Section \ref{ssec:IPFTC} the second definition will play a more direct role. Hence we can rewrite (\ref{eq:Liouville}) as
\begin{equation} \label{eq:gradient flow1}
    \partial_t \rho = \nabla \cdot \Bigl(\rho  \nabla \frac{\delta \mE}{\delta \rho}\Bigr).
\end{equation}
This may be written abstractly as
\begin{align}
\label{eq:GGS}
   \partial_t \rho  &= -{\sM(\rho)^{-1}}\frac{\delta \mathcal{E}}{\delta {\rho}}(\rho).
\end{align}
From this it follows that
\begin{equation}
\label{eq:GGS2}
\frac{d}{dt}\mathcal{E}(\rho)=
\Bigl\langle \frac{\delta \mathcal{E}}{\delta {\rho}}(\rho), \frac{\partial \rho}{\partial t}\Bigr
\rangle= 
-\Bigl\langle  \sM(\rho) \frac{\partial \rho}{\partial t}, \frac{\partial \rho}{\partial t} \Bigr
\rangle \le 0.
\end{equation}
Hence the energy is decreasing along trajectories and the
gradient structure is apparent.

\begin{remark}
\label{rem:comp}
It is interesting to compare the 
gradient structure \eqref{eq:GGS2}
on $\mP_+$ to the gradient flow structure on $\R^{d_u}$ defined by 
(\ref{eq:gradientstructure}).  The state space
gradient flow on $\R^{d_u}$ ensures decrease of $\Psi\bigl(u(t)\bigr)$
along trajectories whilst the probability space gradient flow on 
$\mP_+$ ensures decrease of the expected value of $\Psi\bigl(u(t)\bigr)$
across a distribution of trajectories found from random initialization of the state space problem. 
$\blacksquare$
\end{remark}

%%%%%%%%%%%%%%%%%%%%%%%%%%%%%%%%%%%%%%%%%%%%%%%%%%%
%
\subsection{Bayesian Formulation: Gradient Flows}
\label{ssec:IPBPCT}
%
%%%%%%%%%%%%%%%%%%%%%%%%%%%%%%%%%%%%%%%%%%%%

In this section we  study  the Langevin SDE
\begin{empheq}[box=\widefbox]{equation}
\label{eq:langevin}
\dd u=-\nabla \Phi_R(u)\dd t+\sqrt{2} \dd W.
\end{empheq}
This is a noisy version of \eqref{eq:GD1} in the case where $\Psi=\Phi_R.$ It may be also found as
a sample path instantiation of the continuous time limit of the MCMC algorithm \eqref{eq:mcmc}, typically arising when $\alpha$
is the standard deviation of the proposal, choosing $\alpha=\dt$, letting $\rho(\cdot,n\dt)=\rho_n(\cdot)$ and sending $\dt \to 0;$
see bibliography Subsection \ref{ssec:IPBIBC} for details.

The probability density function for the SDE \eqref{eq:langevin}
is governed by the Fokker-Planck equation
\begin{subequations} 
\label{eq:FP}
\begin{align}
    \partial_t \rho &= 
    \nabla \cdot (\rho \nabla \Phi_R) + \nabla \cdot (\nabla \rho)\\
    &= \nabla \cdot (\rho \nabla \Phi_R +  \rho \nabla \ln \rho),
\end{align}
\end{subequations}
This equation has the density of the Bayesian posterior distribution 
\eqref{eq:bayesa} as steady state. This can be seen by noting that the right-hand side
is divergence of a quantity which is zero if
\begin{equation*}
    \nabla (\Phi_R +  \ln\rho)=0.
\end{equation*}
This quantity can in turn be made zero by choosing
\begin{equation*}
\rho \propto \exp(-\Phi_R),
\end{equation*}
so that $\rho$ is given by the posterior distribution  \eqref{eq:bayesa}.
As a consequence of the fact that the posterior is a steady
state of the Fokker-Planck equation \eqref{eq:FP}, the Langevin SDE
\eqref{eq:langevin} plays an
important role in understanding algorithms for Bayesian inversion.

The machinery we established in the previous subsection, concerning
gradient flows in the space of probability measures, is very powerful and demonstrates that any evolution equation of type (\ref{eq:gradient flow1})
with appropriate potential $\mE$ induces a gradient flow on $\mP_+$
with respect to the Wasserstein-2 metric.
The Fokker-Planck equation \eqref{eq:FP} associated with 
the Langevin equation \eqref{eq:langevin}
may be cast in this framework by making the choice
\begin{equation} \label{eq:energy0}
    \mE(\rho) = \int \left(\Phi_R+\ln\rho\right)\rho\,\dd u
\end{equation}
with Fréchet derivative given by
\begin{equation*}
    \frac{\delta \mE}{\delta \rho}=\Phi_R+\ln \rho .
\end{equation*}

\begin{remark}
\label{rem:KLE}
We observe that, for $\mathrm{KL}[\cdot \Vert \cdot]$ denoting the 
\emph{Kullback-Leibler divergence},
\begin{align}
\label{eq:energy}
{\rm {KL}}[\rho \Vert  \ppi]  =\int\rho \log\Bigl(\frac{\rho}{\ppi}\Bigr)\,\dd u=\mathcal{E}(\rho)+\log \normZ,
\end{align}
where $\ppi$ is the posterior density associated with posterior measure $\mu$ given by \eqref{eq:mud}\footnote{This notation is used throughout Section \ref{sec:CTI} and is not to be confused with the notation used for the joint law of state and data in previous sections.}
and the normalization constant $\normZ$ is defined in (\ref{eq:bayesa}b). 
It is thus possible to choose the energy to be $\mathrm{KL}[\rho \Vert  \ppi]$,
since shifts by a constant in the energy do not change the evolution equations \eqref{eq:gradient flow1} and \eqref{eq:GGS}.

By the property of a divergence, the global minimizer of $\mathcal{E}(\rho)$ 
is attained at $\rho=\ppi$ and hence solves the Bayesian inverse problem. 
It is thus of considerable value  to have
identified a gradient flow to minimize $\mathcal{E}(\rho)$ since such
minimizers solve the Bayesian inverse problem.
Theory concerning the equation \eqref{eq:FP} as a gradient
flow is contained in the bibliography Subsection \ref{ssec:IPBIBC}.
$\blacksquare$
\end{remark}

In summary, the Fokker-Planck equation may be written in the abstract gradient form \eqref{eq:GGS}. We choose
\begin{subequations}
\label{eq:abstractform2}
\begin{align}
\mathcal{E}(\rho) &\coloneqq 
    \mE(\rho) = \int \left(\Phi_R+\ln\rho\right)\rho\,\dd u,\\
\sM(\rho)^{-1} \psi &\coloneqq  -\nabla \cdot (\rho \nabla \psi) \in \Tp.
\end{align}
\end{subequations}
This should be compared with \eqref{eq:abstractform} with the choice $\Psi=\Phi_R$: the
metric structure defined by $\sM$ is the same, but the energy $\mathcal{E}(\rho)$ has an additional
term accounting for the Brownian motion appearing in \eqref{eq:langevin}.

%%%%%%%%%%%%%%%%%%%%%%%%%%%%%%%%%%%%%%%%%%%%%%%%%%%%%%%%%%
%
\subsection{Finite-Time Algorithms}
\label{ssec:IPFTC}
%
%%%%%%%%%%%%%%%%%%%%%%%%%%%%%%%%%%%%%%%%%%%%%%%%%%%%%%%%

The idea used in this subsection, to address the solution of
inverse problems, is a continuous time analog of
Subsection \ref{ssec:IPFT}.
From a sequential formulation of Bayesian inference we derive 
a filtering problem whose solution, at a particular time, 
gives the desired posterior.
Subsection \ref{sssec:IPFTFC} is devoted to the formulation,
Subsection \ref{sssec:smcC} to the use of Gaussian projected filters
and Subsection \ref{sssec:smcC2} to the use of ensemble Kalman methods.

%%%%%%%%%%%%%%%%%%%%%%%%%%%%%%%%%%%%%%%%%%%%%%%%
%
\subsubsection{Formulation}
\label{sssec:IPFTFC}
%
%%%%%%%%%%%%%%%%%%%%%%%%%%%%%%%%%%%%%%%%%%%%%%%%

Employing the reparametrization (\ref{eq:reparam}a) in \eqref{eq:ipsd} and taking
the continuum limit yields the SDE
\begin{subequations}
\label{eq:ipsdc}
\begin{empheq}[box=\widefbox]{align}
\dd u &= 0, \\
\dd z &= G(u)\dd t +\sqrt{\Gammas}\dd B,
\end{empheq}
\end{subequations}
with $B$ a unit Brownian motion in $\R^{d_z}.$
Given a specific realization of the observation process $z^\dagger(\cdot)$ we define $Z^\dagger(t)=\{\zd(s)\}_{0 \le s \le t},$ and consider the
filtering distribution for the random variable $u(t)|Z^\dagger(t).$ However
there is a twist on the standard filtering setting: we are interested in
the case where the data has constant derivative: $\dd\zd(t)/\dd t=\wwd.$ Since
the path $\zd$ has zero quadratic variation the probability distribution is found by setting $\zd(t)=t\wwd$ within the Stratonovich formulation of the nonlocal evolution equation for the density. Referring to
\eqref{eq:KSD_S2} we see that this yields the following
evolution for density $\rho(u,t)$ of $u(t)|Z^\dagger(t):$
\begin{align} \label{eq:KSD_S3}
    \partial_t \rho = \left\langle G- \E G, \wwd \right\rangle_\Gammas \rho  - \frac{1}{2} \left\{ \left| G \right|^2_\Gammas - \E \left| G \right|^2_\Gammas\right\} \rho.
    \end{align}
Here $\E$ denotes integration with respect to density $\rho(\cdot,t)$ 
so that the equation is nonlocal with respect to variable $u$ and
nonlinear with respect to density $\rho.$ This is the analog of the Kushner-Stratonovich equation for the filtering problem defined by \eqref{eq:ipsdc}, since the unconditioned variable $u$ has trivial dynamics and since we are studying the case where the data $\zd(t)=t\wwd$ has zero quadratic variation 
and is in fact differentiable. We may then show the following:
\begin{theorem}
\label{thm:need}
Consider the dynamical system \eqref{eq:ipsdc}, and
assume that $C_0 \succ 0$, $\Gammas \succ 0$,
$u_0 \sim \Ng(m_0,C_0)$, and that $u_0$ is independent of Brownian motion $B$. Let $\rho(\cdot, t)$ denote the probability density function associated with the random variable $u(t)|Z^\dagger(t)$ evolving according to \eqref{eq:ipsdc},
with data chosen as $z^\dagger(t)=t\wwd,$ for $t \in (0,1)$.
Then the density $\rho(\cdot,t)$ satisfies (\ref{eq:KSD_S3}), 
or equivalently for $\Phi$ given by (\ref{eq:phis}a),
\begin{equation} 
\label{eq:betterform}
\partial_t \rho  = -(\Phi  - \E \Phi)\rho.
\end{equation}
Furthermore this equation has solution given by the formulae
\begin{subequations}
\label{eq:mut}
\begin{align}
\rho(u,t) &= \frac{1}{{\normZ}(t)} \exp\bigl(-t\Phi(u)\bigr)\rho_0(u),\\
{\normZ}(t) &= \int_{\R^{d_u}} \exp\bigl(-t\Phi(u)\bigr)\rho_0(u)\dd u;
\end{align}
\end{subequations}
in particular $\rho(\cdot, 1)$ is equal
to $\ppi$, the density of the posterior distribution $\mu$.
$\Diamond$ \end{theorem}

\begin{proof}
Let $\rho_0$ denote the probability density function of the prior $\Ng(m_0,C_0).$
Recall that $\mu(t)$ has density given by 
equation \eqref{eq:KSD_S3} with initial condition $\rho|_{t=0}=\rho_0.$
Now note that, recalling definition (\ref{eq:phisc}a) of $\Phi$,
\begin{subequations} \label{eq:FRevolution}
\begin{align}
&\left\langle G- \E G, \wwd \right\rangle_\Gammas \rho  - \frac{1}{2} \left\{ \left| G \right|^2_\Gammas - \E \left| G \right|^2_\Gammas\right\}\rho\\
&\quad\quad\quad = \bigl(\langle \wwd, G \rangle_{\Gammas}-\frac12|G|_{\Gammas}^2\bigr)\rho-\E \bigl(\langle \wwd, G \rangle_{\Gammas}-\frac12|G|_{\Gammas}^2\bigr)\rho\\
&\quad\quad\quad = -(\Phi  - \E \Phi)\rho.
\end{align}
\end{subequations}
This establishes the equivalence of \eqref{eq:betterform} with \eqref{eq:KSD_S3} .

Note now that (\ref{eq:mut}b) gives
$$\frac{\dd{\normZ}}{\dd t}=-\E \Phi \,{\normZ},$$
where expectation is under $\rho.$
Hence it follows that differentiating (\ref{eq:mut}a) gives 
\eqref{eq:betterform}.
Since this is equivalent to \eqref{eq:KSD_S3} and since \eqref{eq:KSD_S3} characterizes
the law of $u(t)|Z^\dagger(t)$  when $\frac{\dd z^\dagger}{\dd t}=\wwd$, the result is proven.
\end{proof}

The theorem establishes that the evolution equation 
\eqref{eq:KSD_S3} can be rewritten in the form \eqref{eq:betterform},
from which a gradient structure is apparent. Indeed equation  \eqref{eq:betterform} 
can be written in the abstract form \eqref{eq:GGS} with
\begin{subequations}
\label{eq:abstractform3}
\begin{align}
\mathcal{E}(\rho) &:= \int \Phi\rho \dd u,\\
   \sM(\rho)^{-1} \psi &:= \rho \,\psi \in \Tp
\end{align}
\end{subequations}
for $\psi \in \Tp^\ast$. Note that the metric structure differs from what we have seen in \eqref{eq:abstractform2}. In particular we
no longer use the Wasserstein-2 metric: 
the metric defined by the choice (\ref{eq:abstractform3}b) of $\sM(\cdot)$ is known as the \emph{Fisher-Rao metric}. The Fisher-Rao metric requires a more careful consideration of the variational derivative of $\mathcal{E}(\rho)$ as provided by (\ref{eq:VD dual}). In particular, using
\begin{equation*}
    \psi = \frac{\delta \mE}{\delta \rho}_{\vert \Tp^\ast} = \Phi - \mathbb{E}[\Phi] \in \Tp^\ast
\end{equation*}
in (\ref{eq:abstractform3}b) implies that the integral of the right hand side of (\ref{eq:abstractform3}b) over $\R^{d_u}$ is zero so that it is, indeed, an element of the tangent space $\Tp$ given by \eqref{eq:citets}. 
%As for the Wasserstein-2 gradient flows the variational derivative is undefined upto a constant, and this constant can be chosen so that $\sM(\rho)$ is well-defined.

\begin{remark}
    \label{eq:statespaceFR}
The gradient flows defined by \eqref{eq:abstractform} and \eqref{eq:abstractform2} both arose
from considering the evolution equation for the probability density function associated
with an evolution equation for $u \in \R^{d_u},$ the state space, namely equations
\eqref{eq:GD1} and \eqref{eq:langevin} respectively. 
In contrast, while \eqref{eq:betterform} describes the evolution of the density $\rho(\cdot,t)$, we
did not derive it directly as the evolution equation for a random variable $u(t)$ in state space.
However we may seek to find such an evolution. To this end we postulate a mean field differential equation
$$
\frac{\dd u}{\dd t} = g(u,\rho).
$$
Here $\rho$ is the probability density function associated with the random variable $u$; since $u$ is governed by a deterministic ordinary differential equation the randomness in $u$ originates from
the initial condition $u(0).$ We now ask how to choose $g$ so that the evolution equation for the probability density of $u(t)$, started
from a random initialization, evolves according to \eqref{eq:betterform}. 
The evolution of this density will satisfy the associated nonlinear Liouville equation
\begin{equation} \label{eq:nonlinear Liouville}
    \partial_t u = -\nabla \cdot(\rho g(\cdot,\rho)).
\end{equation}
Equating (\ref{eq:betterform}) and (\ref{eq:nonlinear Liouville}), we obtain the condition
\begin{equation}\label{eq:conversion}
\nabla \cdot(\rho g(\cdot,\rho)) = 
\bigr(\Phi-\mathbb{E}(\Phi)\bigl) \rho
\end{equation}
on $g$. Whether or not this equation can be solved for $g$ depends on properties of $\Phi$, and hence $G$;
furthermore even if solvable, the solution may not be unique. We note that, writing $g(u,\rho)$ as
the gradient of a $\rho-$dependent potential $E$, so that $g(u,\rho)=\nabla_u E(u,\rho)$,
renders the identity \eqref{eq:conversion} as a nonlinear divergence form elliptic equation for $E.$
This elliptic equation is parameterized by probability density function $\rho$ and should be viewed as holding everywhere on $\R^{d_u}$.

We have sought a state space model for $u$ which is an ordinary differential equation, albeit of
mean field type. It is also possible to seek stochastic evolution equations,  such a birth-death
processes or mean field stochastic differential equations. $\blacksquare$
\end{remark}

%%%%%%%%%%%%%%%%%%%%%%%%%%%%%%%%%%%%%%%%%%%%%%%%%%%%%%%%%%%%%%
%
\subsubsection{Algorithms: Gaussian Projected Filter}
\label{sssec:smcC}
%
%%%%%%%%%%%%%%%%%%%%%%%%%%%%%%%%%%%%%%%%%%%%%%%%%%%%%%%%%%%%%

To further elucidate the structure of Gaussian projected filtering 
for the inverse problem, we study its continuous time formulation 
from Subsection \ref{ssec:GPFDc} when applied
to the specific state-observation model \eqref{eq:ipsdc}.
To this end first recall the discrete time model \eqref{eq:ipsd} which has continuous time
limit \eqref{eq:ipsdc},
Taking the limit $\Delta t \to 0$ in \eqref{eq:KF_analysis_add_IP_p}, the Gaussian
projected filter for \eqref{eq:ipsd},
we obtain the following evolution equations for mean and covariance:
\begin{subequations}
\label{eq:meanGPFcontIP}
\begin{empheq}[box=\widefbox]{align}
\frac{\dd m}{\dd t} &=  \Cov^{uG}\Gammas^{-1}\bigl(\wwd - \E G(u)\bigr),\\
\frac{\dd\Cov}{\dd t} &=  - \Cov^{uG}\Gammas^{-1}(\Cov^{uG})^\top,\\
\CuG &= \mathbb{E}\Bigl(\bigl(u-\mathbb{E}u\bigr)\otimes\bigl(G(u)-\mathbb{E}G(u)\bigr)\Bigr).
\end{empheq}
\end{subequations}
Here all expectations are computed under 
$u(t) \sim \Ng\bigl(m(t),C(t)\bigr)$. 
This is the continuous time Gaussian projected filter for the inverse problem \eqref{eq:ip}.

As in the discrete time case, the evolution for the mean promotes a Gaussian 
which is compatible with the data, through an innovation term which is weighted by covariance information. We 
now illustrate the filter by considering equations \eqref{eq:meanGPFcontIP}
in the setting of linear $G$, when they may be solved exactly.

\begin{example}
\label{ex:linear2}
Consider the setting of Example \ref{ex:linear}
in which $G(\cdot)=L\cdot $. The Gaussian projected filter \eqref{eq:meanGPFcontIP} becomes
\begin{subequations}
\label{eq:meanGPFcontIP2}
\begin{align}
\frac{\dd m}{\dd t} &=  \Cov L^\top\Gammas^{-1}(\wwd - Lm),\\
\frac{\dd\Cov}{\dd t} &=  - \Cov L^\top\Gammas^{-1}L \Cov.
\end{align}
\end{subequations}
Note that this coincides with the Kalman-Bucy filter \eqref{eq:KB_cov_evol} for the specific filtering problem defined by \eqref{eq:ipsdc} and with observed data $\frac{\dd z^\dagger}{\dd t}=\wwd.$

Now note that \eqref{eq:KSD_S3}, the Kushner-Stratonovich equation, is solved
by the Kalman-Bucy filter in the linear setting where $G(u)=Lu$. It follows
that the solution $\rho$ is given by the Gaussian $\Ng\bigl(m(t),C(t)\bigr)$ where $m(t), C(t)$ solve the Gaussian projected filter equations \eqref{eq:meanGPFcontIP2}. In particular the posterior measure $\mu$ is Gaussian and given by $\Ng\bigl(m(1),C(1)\bigr).$
To see this explicitly note that, from Theorem \ref{thm:need}, and in particular equation \eqref{eq:mut}, in the linear case
$G(\cdot)=L\cdot$, the solution of \eqref{eq:KSD_S3} is given by
\begin{equation} \label{eq:continuous Bayes}
\rho(u,t) \propto \exp\Bigl(-\frac{t}{2}|\wwd-Lu|_\Gammas^2 -\frac12|u-m_0|_{C_0}^2\Bigr).
\end{equation}
Completing the square shows that
this density corresponds to Gaussian $\Ng\bigl(m(t),C(t)\bigr)$ with
mean and covariance satisfying
\begin{subequations}
\label{eq:mCagain}
\begin{align}
C(t)^{-1}m(t) &=tL^\top\Gammas^{-1}\wwd+C_{0}^{-1}m_0,\\
C(t)^{-1}&=C_0^{-1}+tL^\top\Gammas^{-1}L.
\end{align}
\end{subequations}
Note that since $C_0 \succ 0$ it follows that
$C_0^{-1} \succ 0$ and hence (\ref{eq:mCagain}b) shows that, for 
all $t \ge 0$,
$C(t)^{-1} \succ 0$ and hence that $C(t) \succ 0$ for 
all $t \ge 0$; hence $C(t)$ is well-defined by (\ref{eq:mCagain}b)
and $m(t)$ is well-defined by (\ref{eq:mCagain}a).
It simply remains to show that $m(t)$ and $C(t)$ given by
these formulae solve \eqref{eq:meanGPFcontIP2}
when $m(0)=m_0$ and $C(0)=C_0.$ 

We thus turn our attention to the equations \eqref{eq:meanGPFcontIP2}.
Note that $C(t)$ solving (\ref{eq:meanGPFcontIP2}b) satisfies
$C(0)=C_0 \succ 0.$ Hence $C(0)^{-1} \succ 0.$ Thus, by continuity, $C(t)$ remains invertible for some positive interval of time $t \in [0,\tau)$ and,
on this interval, direct computation with 
(\ref{eq:meanGPFcontIP2}b) shows that
\begin{equation}
\label{eq:integrate}
\frac{\dd C^{-1}}{\dd t}=L^\top \Gammas^{-1} L.
\end{equation}
From this it follows by integration
that $C^{-1}(t) \succeq C_0^{-1} \succ 0$ for
all $t$ and hence that we may take $\tau=\infty.$ Furthermore, the
integration also shows that the solution of \eqref{eq:integrate},
solving (\ref{eq:meanGPFcontIP2}b),
delivers (\ref{eq:mCagain}b) as desired.

We then notice that, from (\ref{eq:meanGPFcontIP2}a),
\begin{align*}
C^{-1}\frac{\dd m}{\dd t}&=L^\top\Gammas^{-1}\wwd-L^\top\Gammas^{-1}Lm\\
&= L^T\Gammas^{-1}\wwd-\frac{\dd C^{-1}}{\dd t}m.
\end{align*}
It follows that
$$
\frac{\dd}{\dd t}\bigl(C^{-1}m)=L^\top\Gammas^{-1}\wwd
$$
and integration, together with use of the initial conditions, shows that (\ref{eq:meanGPFcontIP2}a) delivers the desired identity (\ref{eq:mCagain}a).

It is also useful to write \eqref{eq:mCagain} using an
explicit formula for $C(t)$ rather than the precision $C(t)^{-1}.$
To this end fix any $t>0$ and consider
the Gaussian random variable $(u,w)$ defined by choosing $u \sim \Ng(m_0,C_0)$
and $w|u=\Ng(Lu,t^{-1}\Gammas).$ Then the density $\rho$ given in
\eqref{eq:continuous Bayes} is solution of the Bayesian
inverse problem defined by the distribution of $u|w.$ 
The mean and covariance may be found from \eqref{eq:exactMF99} by
replacing $\Gammas$ by $t^{-1}\Gammas$ to yield
\begin{subequations}
\label{eq:exactMF}
\begin{align}
m(t) &=m_0 + C_0 L^\top (L C_0 L^\top + t^{-1} \Gammas)^{-1} (\wwd-Lm_0),\\
C(t) &= C_0 - C_0 L^\top (L C_0 L^\top + t^{-1} \Gammas)^{-1} L C_0.
\end{align}
\end{subequations}
We also observe that the expression (\ref{eq:exactMF}b) for $C(t)$  may
be derived from (\ref{eq:mCagain}b) by use of the Woodbury matrix identity. 
$\blacksquare$
\end{example}

\begin{remark}
We obtained \eqref{eq:meanGPFcontIP} as the time continuous limit of its discrete-time formulation \eqref{eq:KF_analysis_add_IP}. However the same evolution equations can be derived from the gradient flow \eqref{eq:betterform} through a sequence of approximations. This perspective is outlined in 
Appendix \ref{sec:AGF}.
$\blacksquare$
\end{remark}

%%%%%%%%%%%%%%%%%%%%%%%%%%%%%%%%%%%%%%%%%%%%%%%%%%%%%%%%%%%%%%
%
\subsubsection{Algorithms: Ensemble Kalman Filter}
\label{sssec:smcC2}
%
%%%%%%%%%%%%%%%%%%%%%%%%%%%%%%%%%%%%%%%%%%%%%%%%%%%%%%%%%%%%%

We now study the inverse problem using Kalman transport from Subsection \ref{sssec:kt_sdt}, taking the continuous time limit. We consider the specific state-observation model \eqref{eq:ipsdc}.
and recall the discrete time model \eqref{eq:ipsd} which has continuous time
limit \eqref{eq:ipsdc},
Taking the limit $\Delta t \to 0$ in \eqref{eq:ipKT}, the ensemble Kalman filter for \eqref{eq:ipsd},
we obtain the following evolution equations:
\begin{subequations}
\label{eq:ipKTc}
\begin{empheq}[box=\widefbox]{align}
\dd u &=  C^{uG}\Gammas^{-1} (\wwd\,\dd t-\dd\hz),\\
\dd\hz &= G(u)\dd t + \sqrt{\Gammas} \dd B,\\
\CuG &= \mathbb{E}\Bigl(\bigl(u-\mathbb{E}u\bigr)\otimes\bigl(G(u)-
\mathbb{E}G(u)\bigr)\Bigr).
\end{empheq}
\end{subequations}
Here $B \in \R^{d_w}$ is a unit Brownian motion and expectation
is under the law of $u$ itself. As in the discrete time case, the evolution for the state $u$ promotes a solution  which is compatible with the data, through an innovation term which is weighted by covariance information.

\begin{remark}
\label{rem:added0} 
To obtain the resulting continuous time formulation we may also start 
from the continuous time state estimation methodology from Subsection \ref{ssec:GPFDc} and apply it 
to the specific state-observation model \eqref{eq:ipsdc}. 
The SDE (\ref{eq:ipKTc}) may then be seen as a consequence of
\eqref{eq:KSA_mf} applied to this state-observation model. However, special care is required in deriving the equation this way since the observations $z^\dagger$ in Section \ref{sec:CT} were assumed to have non vanishing quadratic variation; in contrast in this section we have $\dd z^\dagger/\dd t = \wwd$, with $\wwd$ constant,
and hence zero quadratic variation.
$\blacksquare$
\end{remark}

To obtain further insight into the mean field dynamical system \eqref{eq:ipKTc} 
we once again consider the linear setting:

\begin{example}
\label{ex:linear5}
Consider the SDE \eqref{eq:ipKTc} in the setting where
$G(u)=Lu$ for matrix $L \in \R^{d_w \times d_u}$. 
Then $u(1) \sim \mu$ where $\mu$ is given in Example \ref{ex:linear}.
To see this recall that
the Gaussian projected filter is exact in the linear setting,
by Example \ref{ex:linear2}, and hence delivers the desired
posterior at time $t=1$, by Theorem \ref{thm:need}.
Thus it suffices to show that the mean and covariance
of $u$ from \eqref{eq:ipKTc} satisfy the Gaussian projected filter in the linear setting,
given by \eqref{eq:meanGPFcontIP2}.
We first note that
\begin{equation}
\label{eq:uwn}
\dd u =  CL^\top\Gammas^{-1} \bigl(\wwd\,\dd t-Lu\,\dd t-\sqrt{\Gammas}\dd B\bigr),
\end{equation}
where $C$ is the covariance of $u$.
By the It\^o formula, $m=\E u$ satisfies (\ref{eq:meanGPFcontIP2}a).
It follows that $e=u-m$ satisfies
$$\dd e=-CL^\top\Gammas^{-1}Le\,\dd t-CL^\top\Gammas^{-1/2}\dd B.$$
A second use of the It\^o formula shows that $C=\E (e \otimes e)$ satisfies
(\ref{eq:meanGPFcontIP2}b). The desired result is established.
$\blacksquare$
\end{example}

We can also derive continuous limits of deterministic transports.
Taking the $\dt \to 0$ limit in \eqref{eq:mfKBF0} results in the mean field ODE formulation
\begin{subequations}
\label{eq:mfKBF}
\begin{empheq}[box=\widefbox]{align}
\frac{\dd u}{\dd t} &=  C^{uG}\Gammas^{-1} \Bigl(\wwd- \frac{1}{2}\bigl(G(u)+\E G(u) \bigr)\Bigr),\\
\CuG &= \mathbb{E}\Bigl(\bigl(u-\mathbb{E}u\bigr)\otimes\bigl(G(u)-
\mathbb{E}G(u)\bigr)\Bigr).
\end{empheq}
\end{subequations}

Again the state evolution has a covariance weighted forcing term which promotes evolution towards the data. As
before we may study this formulation in the linear setting:
 
\begin{example}
\label{ex:linear6}
Consider the mean field model  \eqref{eq:mfKBF} in the setting where
$G(u)=Lu$ for matrix $L \in \R^{d_w \times d_u}$. 
Then $u(1) \sim \mu$ where the posterior distribution 
is given in Example \ref{ex:linear}.
To show this we note that Gaussian projected filter is exact in the linear setting,
by Example \ref{ex:linear2}, and hence delivers the desired
posterior at time $t=1$, by Theorem \ref{thm:need}.
Thus it suffices to show that the mean and covariance
of $u$ from \eqref{eq:mfKBF} satisfy the Gaussian projected filter in the linear setting; this is given by \eqref{eq:meanGPFcontIP2}.
We first note that the mean under \eqref{eq:mfKBF} satisfies
\begin{equation*} \label{eq:mfKBFm}
\frac{\dd m}{\dd t} =  CL^T\Gammas^{-1} (\wwd-Lm),
\end{equation*}
which is (\ref{eq:meanGPFcontIP2}a). Using this it also follows that $e=u-m$ satisfies
$$\frac{\dd e}{\dd t}=-\frac12 CL^T\Gammas^{-1}Le,$$
from which it follows that the variance satisfies (\ref{eq:meanGPFcontIP2}b).
$\blacksquare$
\end{example}

\begin{remark}
    \label{eq:statespaceFR2}
Note that the nonlinear Liouville equation associated with the
mean field model \eqref{eq:mfKBF} has the form
\begin{subequations} \label{eq:mfKBF_Liouville}
\begin{align}   
\partial_t \rho &= -\nabla \cdot (\rho g_{\rm KF}),\\
g_{\rm KF}(u,\rho) &= C^{uG}\Gammas^{-1} \Bigl(\wwd- \frac{1}{2}\bigl(G(u)+\E G(u) \bigr)\Bigr).
\end{align}
\end{subequations}
This evolution equation \emph{approximates} the evolution of the filtering distribution, except in the linear Gaussian setting when it is exact. On the other hand, as in Remark \ref{eq:statespaceFR}, 
we may seek a  mean field differential equation of the form
\begin{equation}
\label{eq:exactagain}
\frac{\dd u}{\dd t} = g(u,\rho)
\end{equation}
which \emph{exactly} replicates the filtering distribution in
general. To do this requires that we choose $g$ to solve \eqref{eq:conversion}:
\begin{equation*}
\nabla \cdot(\rho g(\cdot,\rho)) = 
\bigr(\Phi-\mathbb{E}(\Phi)\bigl) \rho.
\end{equation*}
In the linear Gaussian setting we can identify a solution of 
this equation by asking that \eqref{eq:exactagain}
replicates \eqref{eq:mfKBF} since we know the latter is exact in the linear
and Gaussian setting.

In order to derive this result, we note that (\ref{eq:mfKBF_Liouville}b) takes 
the form 
\begin{equation}
g_{\rm KF}(u,\rho) = C L^\top \Gammas^{-1} \bigl(\wwd- \frac{1}{2}L(u+m)\bigr)
\end{equation}
in the linear setting and $\rho$ is Gaussian with mean $m$ and covariance $C$.
Hence the right hand side of (\ref{eq:mfKBF_Liouville}a) can now be evaluated explicitly giving rise to
\begin{align*}
\nabla \cdot (\rho g_{\rm KF}) &=
- \rho (u-m)^\top C^{-1}
C L^\top \Gammas^{-1}\bigl(\wwd - \frac{1}{2}
L(u+m)\bigr)+ c_1 \rho\\
&=\frac{1}{2} |Lu - \wwd|_\Gammas^2 \rho + 
c_2\rho
\end{align*}
with normalization constants
$$
c_1 = -\frac{1}{2} \mathbb{E}\left(
(u-m)^\top L^\top \Gammas^{-1}L(u+m)\right)
$$
and
$$
c_2 = -\frac{1}{2}\mathbb{E}\left(|Lu - \wwd|_\Gammas^2\right).
$$
Hence we have shown that $g_{\rm KF}$ satisfies (\ref{eq:conversion}) for $\Phi(u) = \frac{1}{2}
|Lu-\wwd|_\Gammas^2$ in the linear Gaussian setting.
See the bibliography Subsection \ref{ssec:IPBIBC} for more details.
$\blacksquare$
\end{remark}

%%%%%%%%%%%%%%%%%%%%%%%%%%%%%%%%%%%%%%%%%%%%%%%%%%%%%%%%%%%
%
\subsection{Infinite-Time Algorithms}
\label{ssec:IPIFTC}
%
%%%%%%%%%%%%%%%%%%%%%%%%%%%%%%%%%%%%%%%%%%%%%%%%%%%%%%%%%%%

We now develop the ideas in Subsection \ref{ssec:IPIFT}
in the continuous time setting. In particular we study
algorithms posed on the infinite time-horizon to solve the optimization problem of minimizing
$\Phi_R$ given by \eqref{eq:phisc}, or to find the Bayesian posterior
distribution given by \eqref{eq:bayesa}. 
In Subsection \ref{sssec:IPOC} we consider this infinite time-horizon perspective
for the solution of optimization problems associated with the inverse problem
\eqref{eq:ip}. Subsection \ref{sssec:IPBC} considers the same perspective for Bayesian inversion.

%%%%%%%%%%%%%%%%%%%%%%%%%%%%%%%%%%%%%%%%%%%%%%%%
%
\subsubsection{Algorithms for Optimization Formulation}
\label{sssec:IPOC}
%
%%%%%%%%%%%%%%%%%%%%%%%%%%%%%%%%%%%%%%%%%%%%%%%%

This rather lengthy subsection is broken into
paragraphs concerning \emph{preconditioned gradient flow},
\emph{statistical linearization}, \emph{gradient descent and statistical
linearization}, \emph{algebraic convergence} and \emph{exponential convergence}.
The initial development on preconditioning enables us to introduce
\emph{affine invariance} and the discussion of statistical linearization
enables us to connect preconditioned gradient descent with ensemble Kalman
methods.  Then, as in Subsection \ref{sssec:IPO} where similar issues are 
discussed in discrete time, we initially discuss algorithms 
with algebraic convergence. We then introduce
generalizations which allow us to obtain exponential convergence.

\vspace{0.1in}

\paragraph{Preconditioned Gradient Flow}
We start by generalizing the standard gradient descent introduced in
Subsection \ref{ssec:IPOPCT}. Given objective function $\Psi: \R^{d_u} \to \R^+$ 
and given symmetric positive-definite
preconditioner $B \in \R^{d_u \times d_u}$ we introduce the equation
\begin{equation} 
\label{eq:GD122}
    \frac{\dd u}{\dd t} = -B\nabla \Psi(u).
\end{equation}
Note that, along solutions of \eqref{eq:GD122},
\begin{subequations}
\label{eq:gradientstructure22}
\begin{align}
    \frac{\dd}{\dd t}\Psi(u) 
    &= \Big\langle \nabla \Psi(u), \frac{\dd u}{\dd t} \Big\rangle\\
    &=  -\Bigl|\frac{\dd u}{\dd t}\Bigr|_{B}^2.
%    &= -\Big\langle B^{-1}\frac{du}{dt},\frac{du}{dt} \Big\rangle\\
%    &=-|B^{\frac12} \nabla \Psi(u)|^2 \le 0.
\end{align}
\end{subequations}
Equation (\ref{eq:GD122}) possesses a gradient flow structure in parameter space $\mathbb{R}^{d_u}$
with respect to a Euclidean metric
weighted by $B$; this weighted metric changes the underlying
geometry of the gradient flow, in comparison to the standard setting of
equation \eqref{eq:GD1}; but it once again leads to the non-increasing 
property of $\Psi(u)$ along trajectories.

Consider now the affine transformation $u \mapsto \tilde u$
given by
\begin{equation} \label{eq:at1}
    \tu = Au + b,
\end{equation}
where $A$ is an invertible matrix and $b$ a vector.
An important issue in all vector space optimization problems 
is the relative scaling of the components of the vector.
A highly desirable feature of an algorithm is that it be
insensitive to such scaling issues. This can be addressed
by looking at differences between: (i) the algorithm for $u$
rewritten in terms of $\tu$ given by \eqref{eq:at1}; (ii)
the same algorithm applied directly to variable $\tu$
optimizing $\tPsi(\tu)$, with the latter defined by
\begin{equation}
\label{eq:tpsi}
 \tPsi(\tu)=\Psi\Bigl(A^{-1}(\tu-b) \Bigr).   
\end{equation}

\begin{definition}
\label{d:GW}
When the two ways (i) and (ii) of using reparametrization \eqref{eq:at1}
lead to the same algorithm, 
for all choices of $A,b$, we say that the algorithm is
\emph{affine invariant}. 
$\blacksquare$ \end{definition}

\begin{remark}
\label{rem:remcite99_needed}
Affine invariant algorithms are highly
desirable as they are not sensitive to the scaling of
the variables. At the optimum, which is unknown \emph{a priori}, this scaling is
not known and hence cannot be used to improve algorithms. Hence algorithms which are
blind to such scaling are highly desirable.
\end{remark}

Applying the transformation \eqref{eq:at1} to (\ref{eq:GD122}) leads to
\begin{equation} \label{eq:GD2}
\frac{\dd\tilde u}{\dd t}=-AB \nabla \Psi\Bigl(A^{-1}(\tu-b) \Bigr),
\end{equation}
the descent approach underlying algorithm viewed as in (i). 
In contrast, applying the same gradient descent to $\tPsi$ given by \eqref{eq:tpsi}, 
leads to the descent approach underlying  algorithms viewed
as in (ii):
\begin{equation} \label{eq:GD202}
\frac{\dd\tilde u}{\dd t}=-BA^{-\top} \nabla \Psi\Bigl(A^{-1}(\tu-b) \Bigr).
\end{equation}
The two equations \eqref{eq:GD2}, \eqref{eq:GD202} only agree if
\begin{equation*}
AB=BA^{-\top}   
\end{equation*}
and such an identity cannot hold for all $A$, for a fixed $B$.
Thus the basic gradient descent \eqref{eq:GD122} is not affine invariant.
However it is a remarkable fact that, by generalizing \eqref{eq:GD122}
to allow for mean field dependence, we can achieve affine invariance.

To this end consider the mean field generalization of \eqref{eq:GD122}
\begin{equation} \label{eq:GD1MF}
    \frac{\dd u}{\dd t} = -B(\rho)\nabla \Psi(u),
\end{equation}
where $\rho(\cdot,t)$ is the probability density function
associated with the law of $u$, assuming that $u_0=u(0)$ is drawn at random
from  probability density function $\rho_0(\cdot).$ If we assume
that $B(\cdot)$ is positive-definite symmetric for all possible
input densities, then similar arguments to before show
that $\Psi(u)$ is non-increasing along trajectories of
\eqref{eq:GD1MF}. Furthermore similar arguments to before
show that the algorithm is affine invariant if, for all invertible
$A \in \R^{d_u \times d_u}$,
\begin{equation*}
AB(\rho)=B(\trho)A^{-\top},   
\end{equation*}
where $\trho$ is the density of $\tu$ related to $u$, with
density $\rho$, by \eqref{eq:at1}. This identity holds for all invertible
$A \in \R^{d_u \times d_u}$ if $B(\cdot)$ is chosen to be the covariance associated with its argument. We have thus discovered the affine invariant mean field gradient descent
\begin{empheq}[box=\widefbox]{equation}
\label{eq:gd}
\frac{\dd u}{\dd t}=-C\nabla \Psi(u),
\end{empheq}
where $C$ is the covariance operator under the law of $u(t)$
and $u(0)$ is chosen at random from probability measure
with density $\rho_0.$

Algorithms based on solving \eqref{eq:gd} are not themselves
ensemble Kalman methods, although we have drawn inspiration
from the power of mean field methods to motivate the approach.
In the next paragraph we introduce the idea of statistical
linearization, leading to a variety of ensemble Kalman methods
for optimization; and in the paragraph following it we use
this idea to approximate \eqref{eq:gd} by an ensemble Kalman 
version of gradient descent which obviates the need for
computing adjoints of the forward model $G(\cdot)$ in the setting where $\Psi(\cdot)=\Phi(\cdot)$ given by \eqref{eq:phisc}.

\vspace{0.1in}

\paragraph{Statistical Linearization}
\label{sssec:stat lin}

A basic building block in the Gaussian projected filter
and mean field ensemble Kalman models for inverse problems 
that we have presented in Subsections \ref{sssec:gpf_ip} and \ref{sssec:kt_sdt}, and their
continuous time-analogs in Subsections \ref{sssec:smcC} and \ref{sssec:smcC2}, is the object
\begin{equation}
\label{eq:cross}
\CuG = \mathbb{E}\Bigl(\bigl(u-\mathbb{E}u\bigr)\otimes\bigl(G(u)-\mathbb{E}G(u)\bigr)\Bigr), 
\end{equation}
here viewed as evolving in continuous time. Also of interest is the regularized
analog of $\CuG$ arising when $G_R$, as defined in (\ref{eq:Rip}), is used in place of $G$. 
In a methodology based on exact properties only for
first and second order statistics, it is natural that $\CuG$ should appear
when solving the inverse problem: the correlation between the 
parameter $u$, which we wish to estimate, and $G(u)$ which we observe, albeit 
polluted by additive noise. One way of understanding the role of $\CuG$ in algorithms
for inversion is through the idea of statistical linearization, providing
a link between ensemble methods and derivatives of the objective
function. The underlying principle is that the differences used in
ensemble methods, and covariances in particular,
act as a surrogate for derivatives.

The expectation defining \eqref{eq:cross}  is computed, for the algorithms we
consider, under the distribution of a Gaussian (for the
Gaussian projected filter) or a more general distribution
(for the mean field ensemble Kalman model). To get some insight into the connection
between ensemble methods and derivatives, we first consider
the setting where $u \sim \Ng\bigl(m,C\bigr).$
Note that such $u$ can be written as $u=m+\sqrt{C}\xi$ where
$\xi \sim \Ng(0,I).$ With this assumption on $u$, \eqref{eq:cross} may be reformulated as
\begin{equation}
\label{eq:cross2}
\CuG = \mathbb{E}\Bigl(\bigl(\sqrt{C}\xi\bigr)\otimes\bigl(G(m+\sqrt{C}\xi)-\mathbb{E}G(m+\sqrt{C}\xi)\bigr)\Bigr), \quad \xi \sim  \Ng(0,I).
\end{equation}
Using this we obtain the following connection 
between $\CuG$ and derivatives of $G:$

\begin{lemma} \label{lem:stats lin}
Assume that the second derivative of $G$ is small: there 
is $\epsilon \ll 1$ such that  
$$\sup_{u \in \R^{d_u}} |D^2G(u)[\zeta,\zeta]| \le \epsilon |\zeta|^2.$$
Then $\CuG$ given by \eqref{eq:cross2} satisfies
$$\CuG=C DG(m)^\top+{\mathcal O}(\epsilon).$$
Thus, when $C^{-1} \succ \lambda I$, for some 
$\lambda>0$ independent of $\epsilon$,
\begin{equation}
\label{eq:cross3}
DG(m)=(\CuG)^\top C^{-1}+{\mathcal O}(\epsilon).
\end{equation}
$\Diamond$ \end{lemma}

\begin{proof}
Note that
\begin{align*}
G(m+\sqrt{C}\xi)&=G(m)+DG(m)\sqrt{C}\xi+{\mathcal O}(\epsilon),\\
\E G(m+\sqrt{C}\xi)&=G(m)+{\mathcal O}(\epsilon).
\end{align*}
From \eqref{eq:cross2}
$$\CuG = \mathbb{E}\Bigl(\bigl(\sqrt{C}\xi\bigr)\otimes\bigl(DG(m)\sqrt{C}\xi+{\mathcal O}(\epsilon)\bigr)\Bigr), \quad \xi \sim  \Ng(0,I),$$
and the desired result follows.
\end{proof}

\begin{remark}
\label{rem:stein}
Another perspective on the preceding lemma is via Stein's identity. 
This states that
\begin{equation*}
    \CuG = C (\E DG)^\top
\end{equation*}
for (\ref{eq:cross2}) when expectation is computed under a Gaussian measure $\Ng(m,C)$. 
The identity can be verified via integration by parts.
Given the assumption stated in Lemma \ref{lem:stats lin}, 
$\E DG(u) = DG(m) + \mathcal{O}(\epsilon)$ 
and the approximation result (\ref{eq:cross3}) also follows. $\blacksquare$
\end{remark}

When $D^2G$ is indeed small it is reasonable to
use \eqref{eq:cross3} as the basis for an approximation to $DG(\cdot)$
at points in an ${\mathcal O}(1)$ ball around the mean. We now take
this idea further and consider random variable $u \in \R^{d_u}$
(not necessarily Gaussian) and compute $C$ and $\CuG$ as covariance 
of $u$ and cross-covariance of $u$ with $G(u)$ respectively. 
We refer to use of the approximation 
\begin{equation}
\label{eq:SLD}
DG(u) \approx (\CuG)^\top C^{-1}, \quad \text{for all}\, u \in \R^{d_u},
\end{equation}
as \emph{statistical linearization.}
The approximation can be invoked to replace the derivative within
any standard optimization or sampling algorithm to solve the inverse
problem \eqref{eq:ip}. Doing so results in a mean field algorithm; that algorithm
in turn can be approximated by particle methods. Statistical linearization allows
the conversion of standard single particle optimization and sampling
algorithms for inverse problems, with dependence on the derivative of
the forward model, into derivative-free interacting particle system
optimizers and samplers. An important application of this methodological
approach is in the development of ensemble Kalman approximations of
Gauss-Newton and Levenberg-Marquadt algorithms; pointers to the literature
will be given in the bibliography Subsection \ref{ssec:IPBIBC}.
We now turn to the use of statistical linearization in gradient
descent, perhaps the most basic setting in which it can be used
for optimization.

\vspace{0.1in}

\paragraph{Gradient Descent and Statistical Linearization}
We provide further insight into the statistical linearization approach from the previous subsection by applying it in the context of gradient descent. Recall $\Phi$ and $\Phi_R$ defined in \eqref{eq:phisc}. We consider the regularized least squares function so that $\Psi(\cdot)=\Phi_R(\cdot)$ in \eqref{eq:GD1}; but similar ideas may be developed in the unregularized setting where $\Psi(\cdot)=\Phi(\cdot).$ Recall that in the regularized setting
$$\Phi_R(u)=\Phi(u)+\frac12 |u-m_0|_{C_0}^2.$$
Note that
\begin{subequations}
\label{eq:combine}
\begin{align}
    \nabla \Phi(u) &= DG(u)^\top \Gammas^{-1}\bigl(G(u)-\wwd\big),\\
    \nabla \Phi_R(u) &= DG(u)^\top \Gammas^{-1}\bigl(G(u)-\wwd\big)-
    C_0^{-1}(m_0-u).
    \end{align}
\end{subequations}
Thus we obtain, from \eqref{eq:GD1},
\begin{equation*} \label{eq:GD12}
    \frac{\dd u}{\dd t} = -\nabla \Phi_R(u)= DG(u)^\top \Gammas^{-1}\bigl(\wwd-G(u)\big)+C_0^{-1}(m_0-u).
\end{equation*}
Similarly, the covariance preconditioned gradient 
flow \eqref{eq:gd} becomes
\begin{equation}
\label{eq:gds}
\frac{\dd u}{\dd t}= -C\nabla \Phi_R(u)=C DG(u)^\top \Gammas^{-1} \bigl(\wwd-G(u)\bigr)+CC_0^{-1}(m_0-u).
\end{equation}
We now approximate this equation using statistical linearization.

From \eqref{eq:SLD} we deduce the equivalent
(assuming $C$ is invertible) approximation
\begin{equation*}
CDG(u)^\top \approx \CuG, \quad \text{for all}\, u \in \R^{d_u}.
\end{equation*}
Combining this with \eqref{eq:combine} we obtain
\begin{equation} \label{eq:correlation approx}
C \nabla \Phi(u) \approx \CuG \Gammas^{-1}\bigl(G(u)-\wwd\bigr).
\end{equation}
Making this approximation in \eqref{eq:gds} 
gives the following ensemble Kalman
approximation of mean field gradient descent for $\Phi_R$: 
\begin{empheq}[box=\widefbox]{equation}
\label{eq:gdek}
\frac{\dd u}{\dd t}=\CuG \Gammas^{-1}\bigl(\wwd-G(u)\bigr)+CC_0^{-1}(m_0-u).
\end{empheq}

\begin{remark}
    \label{rem:affine}
It may be verified that the affine invariance of (\ref{eq:gd}) is preserved under the statistical linearization ansatz. To see this,
note that if $G(\cdot)=L\cdot$, then the covariance matrix $\CuG=\Cov L^\top$ so that
\begin{equation*}
    \CuG \Gammas^{-1}\bigl(\wwd-G(u)\bigr)=\Cov L^\top \Gammas^{-1}(\wwd-Lu).
\end{equation*}
From this it follows that, in this linear setting,
\begin{align*}
 \CuG \Gammas^{-1}\bigl(\wwd-G(u)\bigr)   &=-C \nabla \Phi(u),\\
\Phi(u)&=\frac12|Lu - \wwd|_\Gammas^2.
\end{align*}
Hence, \eqref{eq:ipKTc} is also 
affine invariant. Similar arguments also show that
\eqref{eq:mfKBF} and (\ref{eq:gdek}) are affine invariant. In fact, this property holds for all the ensemble Kalman approaches to inverse
problems developed in the previous (discrete time) and current (continuous time) section. $\blacksquare$
\end{remark}

\vspace{0.1in}
\paragraph{Algebraic Convergence}
We now observe that statistical linearization is exact for linear problems, and provide explicit calculations in this linear case. Recall $G_R, \Gammas_R$ defined by \eqref{eq:Rip} and assume that $\Gammas_R \succ 0$. 

\begin{example}
\label{ex:collapse}
Statistical linearization is exact in the linear setting: if $G(u)=Lu$ then
$DG(u) = (\CuG)^\top C^{-1}.$
Thus, in the setting of Example \ref{ex:linear}, the preconditioned gradient flow
\eqref{eq:gd} with $\Psi=\Phi_R$ reduces to \eqref{eq:gdek} with $G(\cdot)=L\cdot$:
\begin{equation}
\label{eq:pcgf}
\frac{\dd u}{\dd t}=CL^T \Gammas^{-1}\bigl(\wwd-Lu\bigr)+CC_0^{-1}(m_0-u).
\end{equation}
This equation leads to the following closed equations for evolution of the mean and covariance:
\begin{subequations}
\label{eq:meanGPFcontIP2023}
\begin{align}
\frac{\dd m}{\dd t} &=  \Cov L^\top\Gammas^{-1}(\wwd - Lm)+CC_0^{-1}(m_0-m),\\
\frac{\dd\Cov}{\dd t} &=  - 2\Cov L^\top\Gammas^{-1}L \Cov-CC_0^{-1}C=
- 2\Cov L_R^\top\Gammas_R^{-1}L_R \Cov.
\end{align}
\end{subequations}
%We will study properties of these equations later in this section.
It is readily verified that the precision  satisfies equation
\begin{equation}
\label{eq:precision}
\frac{\dd C^{-1}}{\dd t} =   2L_R^\top\Gammas_R^{-1}L_R
\end{equation}
and hence the precision grows linearly in time to infinity. 
As a consequence the covariance decays to zero at algebraic rate $\mathcal{O}(1/t).$
$\blacksquare$
\end{example}

\vspace{0.1in}
\paragraph{Exponential Convergence}
To obtain exponential convergence we must overcome covariance collapse; 
to this end we study the idea of covariance
inflation, introduced in discrete time in \eqref{eq:rescale_beta_dt},
in the continuous time setting. 
Again recall $G_R, \Gammas_R$ defined by \eqref{eq:Rip} and assume that 
$\Gammas_R \succ 0$. We then consider the continuous time limit of \eqref{eq:rescale_beta_dt}
to obtain
\begin{subequations}
\label{eq:rescale_beta_dt2}
\begin{empheq}[box=\widefbox]{align}
\dd u &= \sqrt{\beta \Sigma} \dd W,\\
\dd z &=  G_R(u)\dd t + \sqrt{\Gammas_R} \dd B,
\end{empheq}
\end{subequations}
where $W$ and $B$ are independent unit Brownian motions on $\R^{d_u}$ and $R^{d_w}$ 
respectively. To determine $\Sigmas$ set $\Zd(t)=\{\zd(s)\}_{s \in [0,T]}$  and
define $\Sigmas=C$, where $C$ is the covariance of random variable $u(t)|\Zd(t).$
When we apply the SDE \eqref{eq:rescale_beta_dt2}
to solve the inverse problem we take $\zd(s) := s \wwd_R$ for all $s \ge 0$.

\begin{remark}
As for the discrete-time model \eqref{eq:rescale_beta_dt}, equation 
\eqref{eq:rescale_beta_dt2} defines an unusual  form of mean field
model through dependence on the filtering distribution. As a
consequence, the filtering distribution is determined by a non-standard variant of the Kushner-Stratonovich equation which takes the form
\begin{align*}
\partial_t \mmu = \frac{1}{2}
\nabla \cdot \Bigl(\nabla \cdot \bigl(\mmu \CC(\rho)\bigl)\Bigr) 
 - \frac{1}{2} \left\{ \left| G_R \right|^2_{\Gammas_R} - \E \left| G_R \right|^2_{\Gammas_R} \right\} \mmu
+  
\left\langle G_R- \E G_R, \wwd_R \right\rangle_{\Gammas_R}\mmu.
\end{align*}
This may be derived from \eqref{eq:KSEadded} with $f \equiv 0$,
$\Sigmas=\CC(\rho)$ (the covariance under $\rho$), $h=G_R$
and $\dd\zd(t) := \wwd_R \dd t.$
Note that appearance of the covariance matrix $\CC(\rho)$
adds a further nonlocal nonlinearity which is not present in the density
evolution (\ref{eq:KSD_S3}) that arises without covariance inflation.
$\blacksquare$
\end{remark}

We now derive continuous time limits of the Gaussian projected filter 
\eqref{eq:KF_analysis_add_IP} and the ensemble Kalman filters
\eqref{eq:notsss1} and \eqref{eq:mfKBF01}, derived in Subsection \ref{ssec:IPIFT}
for solution of the inverse problem \eqref{eq:ip}.
Starting with the Gaussian projected filter and taking the limit $\dt \to 0$ we obtain
\begin{subequations}
\label{eq:KF_analysis_add_IP2}
\begin{empheq}[box=\widefbox]{align}
\frac{\dd m}{\dd t} &= \CuG_{R}\Gammas_R^{-1}\bigl(\wwd_R -\mathbb{E} G_R(u)\bigr),\\
\frac{\dd C}{\dd t} &=  \beta C- \CuG_{R}\Gammas_R^{-1}(\CuG_R)^\top,\\
\CuG_{R} &= \mathbb{E}\Bigl(\bigl({u}-\mathbb{E}{u}\bigr)\otimes\bigl(G_R({u})-
\mathbb{E}G_R({u})\bigr)\Bigr),
\end{empheq}
\end{subequations}
where all expectations are with respect to ${u}(t) \sim \Ng\bigl(m(t),C(t)\bigr)$.
Using the explicit form of $\Phi_R$, this set of equations may be shown to be equivalent to
\begin{subequations} \label{eq:cool}
    \begin{align}
        \frac{\dd m}{\dd t} &= \CuG \Gammas^{-1}\bigl(\wwd-\mathbb{E} G(u)\bigr)+CC_0^{-1}(m_0-m),\\
        \frac{\dd C}{\dd t} &=  \beta C- \CuG \Gammas^{-1} (\CuG)^\top - C C_0^{-1}C.
    \end{align}
\end{subequations}
As in discrete time we are using the approximation of the
covariance of the filtering distribution implied by the Gaussian
projected filter; this is since we do not have access to the
exact covariance.

\begin{example}
To highlight links with preconditioned gradient descent we now investigate these equations in the linear Gaussian setting. With $G(\cdot) = L\cdot$, we obtain from (\ref{eq:cool}) the following equations for the mean and covariance matrix evolution:
\begin{subequations} \label{eq:coolG}
    \begin{align}
        \frac{\dd m}{\dd t} &= C L^\top \Gammas^{-1}\bigl(\wwd-Lm\bigr)+CC_0^{-1}(m_0-m),\\
        \frac{\dd C}{\dd t} &=  \beta C- C L^\top \Gammas^{-1} L C - C C_0^{-1}C = \beta C - C L_R^\top \Gammas_R^{-1}
        L_R C.
    \end{align}
\end{subequations}
These evolution equations for mean and covariance are now compared to the corresponding equations derived in Example \ref{ex:collapse}, concerning
statistical linearization. Equations (\ref{eq:meanGPFcontIP2023}) in that example arise
from statistical linearization of preconditioned
gradient descent, and exactly recover preconditioned gradient descent
in the linear setting.  
We note that the evolution equations (\ref{eq:meanGPFcontIP2023}a) 
and (\ref{eq:coolG}a) for the mean agree, while the covariance
matrices $C$ defined by (\ref{eq:meanGPFcontIP2023}b) 
and (\ref{eq:coolG}b) undergo different evolutions.
$\blacksquare$
\end{example}

In the proposition which follows we now show that the different 
evolution equation for the covariance
(\ref{eq:coolG}b) leads to exponential convergence; this should be
contrasted with the algebraic convergence resulting 
from (\ref{eq:meanGPFcontIP2023}b).

\begin{proposition}
\label{lem:lin_c22}
Assume that $u_0$ is initialized at a Gaussian
$\Ng(m_0,C_0)$ and assume also that $C_0, \Gammas_R \succ 0$. Consider the setting where
$G_R(\cdot)=L_R \cdot$ for matrix $L_R \in \R^{(d_w+d_u) \times d_u}$. 
Now consider the filtering distribution $u(t)|Z^\dagger(t)$ defined by \eqref{eq:rescale_beta_dt2} for $\beta >0$, with data $Z^\dagger(t)$ defined by $z^\dagger(s)=s \wwd_R$, where $\wwd_R$ is defined in \eqref{eq:wwr}. Then the filtering distribution is Gaussian $\Ng(m(t),C_(t))$ for all $t \ge 0$. The mean and covariance converge at an exponential rate $\exp(-\beta t)$, as  $t \to \infty$, to the limits $m_{\infty}=\mop$ and
$C_{\infty}=\beta \Cp$, where $(\mop,\Cp)$ are the posterior mean \eqref{eq:pmean} and covariance \eqref{eq:pcov}. 
$\Diamond$ \end{proposition}

\begin{remark}
    It is a remarkable fact that the rate of convergence is independent of
    the properties of the limiting Gaussian posterior distribution, and in
    particular of the conditioning of the posterior covariance. This property is the continuous
    time analog of what we observed in Remark \ref{rem:remcite99}. The desirable
    universal convergence rate property is a result of the affine invariance of the Gaussian projected filter and ensemble Kalman methods studied in this subsection.
$\blacksquare$
\end{remark}

\begin{proof}[Proof of Proposition \ref{lem:lin_c22}]
We first note that $\beta >0$ implies that equation (\ref{eq:coolG}b) for the covariance has two equilibria: an unstable one at $C=0$ and a stable one satisfying
\begin{equation} \label{eq:equilibrium_beta}
C_\infty = \beta \left(L_R^\top \Gammas_R^{-1} L_R\right)^{-1}=\beta \Cp.
\end{equation}
The exponential convergence to $C_\infty$ for $\beta >0$ can be best seen by considering the evolution for the precision matrix
$C^{-1}$:
\begin{equation} \label{eq:precision_beta}
\frac{\dd C^{-1}}{\dd t} = -\beta C^{-1} + L_R^\top \Gammas_R^{-1}L_R.
\end{equation}
Direct calculation of the time-derivative of $C^{-1}(t)m(t)$ shows that
$$
\frac{\dd}{\dd t}(C^{-1}m) = -\beta C^{-1}m + L_R^\top \Gammas_R^{-1} \wwd_R.
$$
Thus $C^{-1}(t)m(t)$ converges exponentially fast to limit 
$\frac{1}{\beta}L_R^\top \Gammas_R^{-1} \wwd_R.$ Since $C(t)$ itself
converges exponentially fast to $\beta \Cp$ it follows that $m(t)$ converges
exponentially fast to posterior mean $\mop$ given by \eqref{eq:pmean}.
\end{proof}

\begin{remark}
For $\beta = 0$, we find algebraic convergence which we also found in equation (\ref{eq:precision}), from Example \ref{ex:collapse}, for the preconditioned gradient descent formulation. It should be noted that (\ref{eq:precision}) contains a pre-factor of two which is not present in 
(\ref{eq:precision_beta}) so, even when $\beta=0$, the evolution for mean
and covariance does not coincide with that arising from preconditioned gradient descent.
$\blacksquare$
\end{remark}

We now turn to ensemble Kalman methods, starting with the stochastic version.
We take continuous time limits in \eqref{eq:notsss1} to obtain
\begin{subequations}
\label{eq:notsss12}
\begin{empheq}[box=\widefbox]{align}
\dd u &= \sqrt{\beta C} \dd W + \CuG_{R}\Gammas_R^{-1}\bigl(\wwd_R \dd t - \dd\hz\bigr) ,\\
\dd\hz &= G_R(u)\dd t + \sqrt{\Gammas_R}\dd B,
\end{empheq}
\end{subequations}
where $\CuG_R$ is computed using (\ref{eq:KF_analysis_add_IP2}c), and $C$ is
the regular covariance, both under the law of $u.$
As in discrete time we are using the approximation of the
covariance of the filtering distribution implied by the Gaussian
projected filter; this is since we do not have access to the
exact covariance.
Here $W$ and $B$ are independent unit Brownian motions on $\R^{d_u}$ and $R^{d_w}$ 
respectively. The corresponding deterministic transport formulation is found by taking the continuous time limit in \eqref{eq:mfKBF01} to obtain
\begin{empheq}[box=\widefbox]{equation}
\label{eq:mfKBF012}
\dd u = \sqrt{\beta C} \dd W + \CuG_{R}\Gammas_R^{-1}\Bigl(\wwd_R - \frac12\bigl(G_R(u)-\mathbb{E} G_R(u)\bigr)\Bigr)\dd t,
\end{empheq}
where $W$ is again a unit Brownian motion  on $\R^{d_u}$ 
and where $\CuG_R$ is computed using (\ref{eq:KF_analysis_add_IP2}c), and $C$ is
the regular covariance, both under the law of $u.$

\begin{remark} 
\label{rem:det2}
We note that it is possible to replace \eqref{eq:rescale_beta_dt2} by the 
filtering problem associated with the model
\begin{subequations}
\label{eq:rescale_beta_dt22}
\begin{empheq}[box=\widefbox]{align}
\dd u &= \frac12 \beta(u-\mathbb{E}u)\dd t,\\
\dd z &=  G_R(u)\dd t + \sqrt{\Gammas_R} \dd B.
\end{empheq}
\end{subequations}
This results in the same Gaussian projected filter and ensemble Kalman
methods as before, in the linear Gaussian setting.
We note that using the filtering problem associated with
\eqref{eq:rescale_beta_dt22} leads to analogous ensemble Kalman methods to  (\ref{eq:mfKBF012}) or (\ref{eq:notsss12}); these can be obtained by directly replacing the term $\sqrt{\beta C} \dd W$ with $\frac12 \beta(u-\mathbb{E}u)\dd t$ in (\ref{eq:mfKBF012}) and (\ref{eq:notsss12}).
$\blacksquare$
\end{remark}

%%%%%%%%%%%%%%%%%%%%%%%%%%%%%%%%%%%%%%%%%%%%%%%%%%%%%%%%%%%
%
\subsubsection{Algorithms for Bayesian Formulation}
\label{sssec:IPBC}
%
%%%%%%%%%%%%%%%%%%%%%%%%%%%%%%%%%%%%%%%%%%%%%%%%%%%%%%%%

In Subsection \ref{sssec:IPB} we derived Gaussian projected filter
and ensemble Kalman methods that converge exponentially fast to
the exact posterior distribution of the inverse problem \eqref{eq:ip} in
the linear, Gaussian setting. This is achieved by choosing
$\beta=(1-\dt)^{-1}$ as stated in \eqref{eq:betat}. A similar argument
in the continuous time setting thus leads to the choice $\beta=1.$ We have the
following corollary of Proposition \ref{lem:lin_c22}:

\begin{corollary}
\label{cor:lin_c23}
Assume that $u_0$ is initialized at a Gaussian
$\Ng(m_0,C_0)$ and assume also that $C_0, \Gammas_R \succ 0$. Consider the setting where
$G_R(\cdot)=L_R \cdot$ for matrix $L_R \in \R^{(d_w+d_u) \times d_u}$. 
Now consider the filtering distribution $u(t)|Z^\dagger(t)$ defined by \eqref{eq:rescale_beta_dt2} for $\beta=1$, with data $Z^\dagger(t)$ defined by $z^\dagger(s)=s \wwd_R$, where $\wwd_R$ is defined in \eqref{eq:wwr}. Then the filtering distribution is Gaussian $\Ng(m(t),C(t))$ for all $t \ge 0$. The mean and covariance converge at an exponential rate, as  $t \to \infty$, to the limits $m_{\infty}=\mop$ and
$C_{\infty}=\Cp$, where $(\mop,\Cp)$ are the posterior mean \eqref{eq:pmean} and covariance \eqref{eq:pcov}. 
$\Diamond$ \end{corollary}

\paragraph{Preconditioned Gradient Flows}

While the above extension to Bayesian inference problems is straightforward and leads to an exact recovery of the posterior distribution in the linear Gaussian setting, the resulting methods deliver only approximations in the general nonlinear setting. This, of course, is a theme throughout this article.
In the remainder of this subsection we take a different approach to deriving algorithms for sampling which are exact in the linear Gaussian setting. We start by considering preconditioned gradient descent, and its Langevin analog; we then note that application of statistical linearization gives approximations of these evolutions which are  exact in the linear Gaussian setting, and hence also exhibit desirable convergence properties, in the linear Gaussian
setting.

First recall the preconditioned gradient descent
(\ref{eq:gd}) for $\Psi = \Phi_R$ resulting in
\begin{empheq}[box=\widefbox]{equation}
\label{eq:gd0}
\frac{\dd u}{\dd t}=-C\nabla \Phi_R(u).
\end{empheq} 
This methodology for optimization may be extended to a sampling methodology
by considering the preconditioned mean field Langevin SDE defined by
\begin{empheq}[box=\widefbox]{equation}
\label{eq:langevin0}
\dd u=-C\nabla \Phi_R(u)\dd t+\sqrt{2C} \dd W.
\end{empheq}

\begin{remark}
\label{rem:mer}
We note here that it may be verified that (\ref{eq:langevin0}) is invariant under affine transformations of type (\ref{eq:at1}).
Thus (\ref{eq:langevin0}) provides an attractive generalization of standard Langevin dynamics (\ref{eq:langevin}) for sampling from the posterior distribution, because of the
properties of affine invariant algorithms highlighted in Remark \ref{rem:remcite99_needed}. Although the desired posterior distribution is approached only in the limit $t\to \infty$, the fact that the convergence is exponential, with universal rate across all linear
Gaussian inverse problems, makes the approach potentially competitive. Theory concerning this equation is discussed in the bibliography Subsection \ref{ssec:IPBIBC}.

Use of the statistical linearization approximation \eqref{eq:SLD}, which we discuss next, converts both (\ref{eq:gd0}) and \eqref{eq:langevin0} into mean field ensemble Kalman methods which, through particle approximations, lead to implementable derivative-free methods.
$\blacksquare$
\end{remark}

\paragraph{Inexact Gradients}

It is possible to apply statistical linearization
\eqref{eq:correlation approx} to the preconditioned Langevin equation \eqref{eq:langevin0}, 
resulting in the evolution equation
\begin{empheq}[box=\widefbox]{equation}
\label{eq:langevin02}
\dd u=\CuG_R \Gammas_R^{-1}\bigl(\wwd_R-G_R(u)\bigr) \dd t+\sqrt{2C} \dd W.
\end{empheq}
Here $\CuG_R$ is computed using (\ref{eq:KF_analysis_add_IP2}c), and $C$ is
the regular covariance, both under the law of $u.$ Like \eqref{eq:notsss12}
and \eqref{eq:mfKBF012} this equation converges exponentially fast to the posterior
distribution in the linear Gaussian case. However the form of the mean field
stochastic differential equation is fundamentally different: here the Brownian
noise arises in state space $\R^{d_u}$, whereas in the ensemble Kalman methods
it appears in data space $\R^{d_y}.$ 
The  Kalman-Wasserstein gradient flow structure for the
Liouville and Fokker-Planck equations considered here
is not maintained under statistical linearization: they
are not of gradient descent type in $\mP_+$, except in the
case where $G(u)$ is linear.

\paragraph{Exact Gradients}

This paragraph concerns analysis of \eqref{eq:gd0} and \eqref{eq:langevin0} when the exact gradients of $\Phi$, and hence $\Phi_R$, are available: we consider the geometric properties of (\ref{eq:gd0}). We assume 
that $u(t)$ has smooth probability density $\rho(u,t)$ for
all $t\ge 0$ and recall that we denote the manifold of all smooth probability
density functions on $\R^{d_u}$ by $\mP_+$. Then $\rho$ satisfies the Liouville equation
\begin{equation} \label{eq:liouville}
    \partial_t \rho = \nabla \cdot (\rho C \nabla \Phi_R).
\end{equation}
Again the appearance of the covariance matrix $C = \CC(\rho)$ renders (\ref{eq:liouville}) a nonlinear and nonlocal partial differential equation on $\mP_+$. 
We will show that the evolution of $\rho$ on $\mP_+$ has gradient flow structure of the form given in \eqref{eq:GGS}.

In order to see this gradient structure we need to identify
the energy functional which is being minimized and then introduce a metric structure in which \eqref{eq:liouville} is a gradient flow. The relevant energy functional
 $\mE(\rho)$ on $\mP_+$ is simply
\begin{equation*}
    \mE(\rho) = \int_{\R^{d_u}} \Phi_R (u) \rho(u) \dd u,
\end{equation*}
namely the expected value of $\Psi(u).$ Note that the Fr\'echet derivative\footnote{Identified by writing  $\mE(\rho+\sigma)-
\mE(\rho)$ as a linear operator acting on $\sigma \in \Tp$ plus higher order terms.} of $\mE$ is given by
\begin{equation*}
\frac{\delta \mE}{\delta \rho}=\Phi_R.
\end{equation*}
Hence we can rewrite (\ref{eq:liouville}) as
\begin{equation} \label{eq:Gradient flow}
    \partial_t \rho = \nabla \cdot \Bigl(\rho \CC(\rho) \nabla \frac{\delta \mE}{\delta \rho}\Bigr).
\end{equation}

Recalling the tangent space $\Tp$ to $\mP_+$, at $\rho\in\mP_+$, 
the appropriate choice of Riemannian metric tensor 
$g_{\rho,\CC}\,:\, \Tp\times \Tp \to\R$ 
for \eqref{eq:liouville} is then
\begin{equation*}
    g_{\rho,\CC}(\sigma_1,\sigma_2):=\int_{\R^{d_u}} \left\langle\nabla \phhi_1\,,\,\CC(\rho)\nabla \phhi_2\right\rangle \,\rho\, \dd u,
\end{equation*}
where $\sigma_i=-\nabla\cdot\left(\rho\CC(\rho)\nabla \phhi_i\right)\in T_\rho\mathfrak{P}_+$ for $i=1,2$.\footnote{A precise mathematical treatment requires further assumptions on the considered set $\mP_+$ of probability functions ensuring that the required potentials $\phhi$ indeed exists for all $\sigma \in \Tp$. See the bibliography for references and discussion of open questions in this area.} We refer to the \emph{Kalman-Wasserstein metric} as the metric induced by this metric tensor, generalizing the Wasserstein-2 metric introduced in Subsection \ref{ssec:IPBPCT}.

The equation \eqref{eq:Gradient flow} takes gradient flow structure in $(\mathfrak{P}_+, g_{\rho,\mathcal{C}})$: 
\begin{equation} \label{eq:gradient structure 22}
\frac{\dd}{\dd t} \mE(\rho)=-g_{\rho,\CC}(\partial_t\rho,\partial_t\rho).
\end{equation}
To establish \eqref{eq:gradient structure 22} note that
\begin{align*}
    \frac{\dd}{\dd t} \mE(\rho)
&=\int_{\R^{d_u}} \frac{\delta \mE}{\delta \rho}\partial_t \rho\,\dd u\,\\
&=-\int_{\R^{d_u}} \rho \, \Bigl| \CC(\rho)^{\frac12} \nabla \frac{\delta \mE}{\delta \rho} \Bigr|^2\,\dd u\,\\
&=-\int_{\R^{d_u}} \left\langle\nabla \frac{\delta \mE}{\delta \rho}\,,\,\CC(\rho)\nabla \frac{\delta \mE}{\delta \rho}\right\rangle \,\rho\,\dd u\,\\
%&=-g_{\rho,\CC}(\partial_t\rho,\partial_t\rho),\\
&=-g_{\rho,\CC}(\sigma,\sigma)\,,
\end{align*}
where
$$\sigma=-\nabla\cdot\left(\rho\CC(\rho)\nabla \frac{\delta \mE}{\delta \rho}\right)=-\partial_t \rho.$$ 
Thus the gradient flow identity \eqref{eq:gradient structure 22}
is established.  It is interesting to compare the 
gradient structure (\ref{eq:gradient structure 22}) 
on $\mP_+$ to the gradient flow structure on $\R^{d_u}$ defined by 
\eqref{eq:GD122} with $B=\CC(\rho).$ The state space
gradient flow on $\R^{d_u}$ ensures decrease of $\Phi_R\bigl(u(t)\bigr)$
along trajectories whilst the probability space gradient flow on 
$\mP_+$ ensures decrease of the expected value of $\Phi_R\bigl(u(t)\bigr)$
across a distribution of trajectories found from random initialization of the state space problem. 

The preceding calculations demonstrates that \emph{any} evolution equation of 
type (\ref{eq:Gradient flow}) with appropriate energy functional $\mE$ induces a gradient flow on $\mP_+.$  In particular, using the energy functional
\begin{equation*} 
    \mE(\rho) = \int \left(\Phi_R+\ln\rho\right)\rho\,\dd u
\end{equation*}
defined in (\ref{eq:energy0}), we observe that
the associated evolution equation (\ref{eq:Gradient flow}) becomes
\begin{equation} \label{eq:EKS}
    \partial_t \rho = \nabla \cdot (\rho \CC(\rho) \nabla \Phi_R) + \nabla \cdot (\CC(\rho) \nabla \rho).
\end{equation}
This nonlinear and
nonlocal Fokker-Planck equation governs evolution of the probability
density funtion for the mean field SDE \eqref{eq:langevin0}. 
Now recall Remark \ref{rem:KLE} in which we note that $\mathrm{KL}[\rho \Vert \ppi]$ and $\mE(\rho)$
differ by a constant, and where $\ppi$ is the posterior density associated with posterior measure $\mu$ given by \eqref{eq:mud}. Thus the global minimizer of the gradient flow associated with (\ref{eq:energy0}) is attained at $\rho=\ppi$ and hence solves the Bayesian inverse problem.

Finally, it is also useful to see the nonlinear and nonlocal
Fokker-Planck equation \eqref{eq:EKS}
written in the abstract gradient form \eqref{eq:GGS}. In this case we
may choose
\begin{subequations}
\label{eq:abstractform2_2} 
\begin{align} 
\mathcal{E}(\rho) &\coloneqq 
    \mE(\rho) = \int \left(\Phi_R+\ln\rho\right)\rho\,\dd u,\\
\sM(\rho)^{-1} \psi &\coloneqq  -\nabla \cdot (\rho \CC(\rho) \nabla \psi) \in \Tp.
\end{align}
\end{subequations}

%%%%%%%%%%%%%%%%%%%%%%%%%%%%%%%%%%%%%%%%%%%%%%%%%%%%%%%%
%%%%%%%%%%%%%%%%%%%%%%%%%%%%%%%%%%%%%%%%%%%%%%%%%
%
\subsection{Bibliographical Notes}
\label{ssec:IPBIBC}
%
%%%%%%%%%%%%%%%%%%%%%%%%%%%%%%%%%%%%%%%%%%%%%%%%%

The notion of gradient flow plays a central role in this
paper. The subject is enormous and we cannot do justice to it
here. We point the reader to the text \citet{hirsch2013differential} for the study of gradient flows
in Euclidean space and to the text \citet{ambrosio2008gradient}
for gradient flows in metric spaces, including spaces
of probability measures. The paper \citet{chen2023sampling}
contains an overview of gradient flows for probability measures,
focused on applications to Bayesian inversion.

Continuous time limits of ensemble Kalman filters were first derived in \citet{bergemann2010localization} and \citet{bergemann2010mollified} and further explored in the context of continuous time transport in \citet{reich2011dynamical}. A connection between
the non-stochastic Kalman transport equations (\ref{eq:mfKBF}) and preconditioned gradient descent was first identified in \citet{bergemann2010mollified} for finite ensemble sizes and in \citet{reich2015probabilistic} for the mean field limit. The paper \citet{SR-PR21} adopts the time continuous setting. 
See also \citet{YBM14} for related formulations based on the feedback particle filter approach to continuous time filtering.

The papers \citet{schillings2017analysis} and
\citet{schillings2018convergence} studied the use of ensemble Kalman methods for optimization problems,
taking a continuous time limit, making a connection to preconditioned gradient descent and exploiting an invariant subspace property (see, for example, \citet{iglesias2013ensemble}[Theorem 2.1]) inherent in the basic form of the ensemble Kalman methodology.
This led to work on approximate sampling from the preconditioned Langevin equation in
\citet{garbuno2020interacting},
\citet{garbuno2020affine},
\citet{nusken2019note} and
\citet{zliu22}; in particular the papers \citet{nusken2019note,garbuno2020affine} demonstrated a finite ensemble size correction to the mean field limit introduced in \citet{garbuno2020interacting}.
The paper \citet{garbuno2020interacting} used the
non-standard Kalman-Wasserstein metric, first introduced in \citet{reich2015probabilistic}, to provide a framework to analyze the preconditioned Langevin equation. It remains open to fully develop the mathematical foundations of gradient flows using this metric. These papers demonstrate the role that the ensemble plays in preconditioning the dynamics. This makes a link to the important paper \citet{goodman2010ensemble} which introduced the concept of affine invariant ensemble samplers, an idea developed further in \citet{leimkuhler2018ensemble}.

The paper \citet{chada2020iterative} overviews the optimization perspective and provides a unifying framework, going beyond gradient descent-based methods. In particular, framing ensemble Kalman based optimization methods in terms of statistical linearization, as we do in Subsection \ref{sssec:stat lin}, originates in that paper.
The papers \citet{ReiWei21} and \citet{pavliotis2021derivative} show how to construct a derivative free Langevin sampler using localized ensembles; the use of localized ensembles to
train neural networks may be found in \citet{haber2018never}. An alternative interacting particle system approach to solving inverse problems and optimization tasks is the use of consensus based methods
\citep{tsianos2012consensus,
ha2021convergence,
fornasier2021consensus,
carrillo2018analytical,
carrillo2021consensus,
chen2020consensus,
pinnau2017consensus,
fornasier2020consensus}.
The papers \citet{ding2020ensemble}, \citet{ding2021ensemble,ding2021ensembleb}
undertake a systematic analysis of the link between interacting particle systems and mean field systems, mostly focused on the solution of inverse problems; however the methods developed are more widely applicable.

Recall that \eqref{eq:meanGPFcontIP}
is a derivative-free approach to approximately solving the Bayesian inverse problem by application of the Gaussian projected filter.  
In appendix Section \ref{sec:AGF} we show that these equations may
may also be derived from the Fisher-Rao gradient flow \eqref{eq:betterform}, deriving equations for the mean and covariance
evolution from it, and then by invoking a number of approximations. 
In appendix Section \ref{sec:AGF}, as a step in this derivation, we 
obtain the equations
\begin{subequations} \label{eq:closure_second_order}
\begin{align}
    \frac{\dd m}{\dd t} &= - \mathbb{E}\bigl(\Phi(u)(u-m)\bigr),\\
    \frac{\dd C}{\dd t} &= -\mathbb{E}\bigl(\Phi(u)(u-m)(u-m)^\top\bigr)+C\mathbb{E}\bigl(\Phi(u)\bigr),
\end{align}
\end{subequations}
which may be viewed
as a closed evolution when expectation is computed under the Gaussian defined by $(m,C).$ These closed equations also define a 
derivative-free approach to approximate the Bayesian inverse problem, different from  \eqref{eq:meanGPFcontIP}. It is natural to ask which of \eqref{eq:closure_second_order} and \eqref{eq:meanGPFcontIP} is preferable; computational
experiments underpinning the work in \citet{chen2023sampling} indicate that \eqref{eq:meanGPFcontIP} is more robust
in various settings.
Discussion of  the Fisher-Rao gradient flow projected into the manifold of Gaussians is contained in \citet{chen2023sampling}.

Theory showing that the Fokker-Planck equation \eqref{eq:FP} arises as a continuous time limit of the MCMC method \eqref{eq:mcmc} was initiated in \citet{gelman1997weak}, for the random walk Metropolis algorithm, and followed up for the Metropolis adjusted Langevin equation in \citet{roberts1998optimal}. See \citet{roberts2001optimal} for an overview of this field.
Most of the analysis is done at the level of weak convergence
of sample paths and hence works directly with \eqref{eq:langevin},
rather than with its density, which is governed by \eqref{eq:FP}.

\section{Conclusions and Open Problems}
\label{sec:C}

This paper presents a unifying perspective on the
derivation, interpretation and analysis of ensemble
Kalman methods through use of the ideas of mean 
field models, second order approximate transport
and particle approximation. Both state estimation
and parameter estimation (inverse problems)
are studied; similar ideas may be developed for joint parameter-state estimation problems but are not discussed in this paper. The ideas have been presented in discrete time and, through specific parametric scalings, continuous time limits have been identified. Our unifying approach constitutes a novel presentation of the subject, and creates a framework for the mathematical development
of the subject area.

Ensemble Kalman methods have been enormously impactful
in the geosciences, where they originated, and are starting
to be used in numerous other application domains.
However, if they are to realize their potential for
widespread adoption and application, many research
challenges remain. These challenges are both in
mathematical analysis and in the development of methodology.
One of the biggest challenges is the following: some theory,
and abundant numerical evidence, show that ensemble Kalman
methods perform well at state estimation and at parameter
estimation; however, there is very little theory, or empirical evidence, which identifies situations in which the statistical 
information in the ensemble constitutes valid approximate Bayesian inference. The recent papers \citet{carrillo2022ensemble} and \citet{calvello2024accuracy} make first steps
in this direction. In mathematical analysis a number of other substantial challenges
are presented by ensemble Kalman methods, which we list here.

\begin{itemize}

\item For state estimation, determine conditions, related to small noise scenarios,
under which the true state is
well-approximated by the mean or sample path of mean field
models based on second order transport; find sharp error estimates.

\item For state estimation, determine conditions under which the filtering distribution is well-approximated by mean field models based on second order transport; find sharp error estimates and appropriate metrics for
the analysis; furthermore, identify which problems satisfy these conditions.

\item For inverse problems, determine conditions under which the optimizer of a (to-be-identified) loss function is well-approximated by the mean or sample path of mean field models based on second order transport, in both the transport and iterative approaches to inversion; find sharp error estimates.

\item For inverse problems, determine conditions under which the Bayesian posterior is well-approximated by mean field models based on second order transport, in both the transport and iterative approaches to inversion;  find sharp error estimates and appropriate metrics for the analysis; furthermore, identify which inverse problems satisfy these conditions.

\item For all of the preceding four scenarios derive error bounds for
particle approximations of the mean field models; when low-rank
structure is present in covariances prove that ensemble Kalman methods
can correctly identify it, and exploit the low-rank structure in the
analysis of particle approximations.

\item In all of the particle methods arising above, compare the
cost/error trade-off with that arising for other methods,
to determine when ensemble Kalman methods are competitive.

\item All of the algorithms in this paper are studied in 
idealized scenarios, in the absence of
widely employed techniques such as covariance inflation and localization; 
developing analyses which account for 
covariance inflation and localization will be highly desirable.

\end{itemize}

On the methodology side there are also a number of significant
challenges, which we also list here.

\begin{itemize}

\item Given ability to compute an ensemble of evaluations of the combined state-observation dynamical system, determine the optimal (in terms of cost/error trade-off) way to combine the ensemble to either estimate the state given an observation sequence, or the filtering distribution.

\item Given ability to compute an ensemble of evaluations of the forward model, what is the optimal way to combine them to either
estimate the parameter given an observation, or the posterior distribution, for the corresponding inverse problem.

\item What role might be played by machine learning in addressing
the design of algorithms, and in particular in addressing the
preceding questions.

\item Develop an overarching interacting particle and mean field framework that subsumes ensemble Kalman and alternative particle filters as well as derivative-free and consensus based optimization methods and use the framework to create new methods.

\item Develop principles for the deployment of covariance inflation,
and  generalizations of localization, so that the resulting
methodology is widely applicable and does not need application-specific 
principles to be applied. Different ideas may be needed in the inverse problem setting.

\item Expand the preceding scenarios beyond the additive error models 
discussed in the survey, to include classification and other machine 
learning tasks.

\end{itemize}

\vspace{0.1in}
\noindent{\bf Acknowledgments} EC is grateful to the Kortschak Scholars Program within the CMS Department at Caltech for financial support. The work of SR is
supported by Deutsche Forschungsgemeinschaft (DFG) -- Project-ID 318763901 -- SFB1294. The work of AMS is supported by NSF award AGS1835860, and  by a Department of Defense Vannevar Bush Faculty
Fellowship, which also supports EC. In addition,  AMS and EC are supported by the SciAI Center, 
funded by the Office of Naval Research  (ONR), under Grant Number N00014-23-1-2729.

The authors are grateful to Dmitry Burov for helpful advice 
regarding the numerical experiments; they are also grateful to Arnaud Vadeboncoeur, Eviatar Bach, Ricardo Baptista 
and Daniel Sanz-Alonso for helpful discussions which improved this paper. The authors thank
Mark Asch for careful reading of the paper, and useful feedback.
Finally AMS is grateful for hospitality
offered at the University of Chicago, 
by Guillaume Bal, Peter McCullagh, Daniel Sanz-Alonso and Rebecca Willett, where he
delivered lectures based on a preliminary version of this paper in May 2022.
\vspace{0.1in}

%%%%%%%%%%%%%%%%%%%%%%%%%%%%%%%%%%%%%%%%%%%%%%%%%%
%%%%%%%%%%%%%%%%%%%%%%%%%%%%%%%%%%%%%%%%%%%%%%%%%%
%\clearpage 
\addcontentsline{toc}{section}{References}
\bibliography{bib-database}
%
%
%
%%%%%%%%%%%%%%%%%%%%%%%%%%%%%%%%%%%%%%%%%%%%%%%%%%

%%%%%%%%%%%%%%%%%%%%%%%%%%%%%%%%%%%%%%%%%%%%%%%%%%%%%%%%%%%%%%%%%%%%%%
%
%
%
%
%
%
%
%
%
%
%
%  Appendices
%
%
%
%
%
%
%
%
%
%
%
%%%%%%%%%%%%%%%%%%%%%%%%%%%%%%%%%%%%%%%%%%%%%%%%%%%%%%%%%%%%%%%%%%%%

%%%%%%%%%%%%%%%%%%%%%%%%%%%%%%%%%%%%%%%%%%%%%%%%%%%%%%%
%
\section{Appendix A (Pseudo-Code)}
\label{sec:AA}
%
%%%%%%%%%%%%%%%%%%%%%%%%%%%%%%%%%%%%%%%%%%%%%%%%%%%%%%%

In this appendix we provide pseudo-code describing several of the algorithms that we present
and deploy in this paper. Algorithms \ref{alg:3DVAR} and \ref{alg:EnKF}, 3DVAR and the ensemble Kalman filter (EnKF) respectively, are applied in the context of the problem of state estimation for discrete time dynamical systems presented in Subsections \ref{ssec:CT} and \ref{ssec:EKM}. The scheme 3DVAR is employed in Examples \ref{ex:3dvar}, \ref{ex:s3dvar}, \ref{ex:enkf}
and \ref{ex:3dvar_ms}. The ensemble Kalman filter is applied in Example \ref{ex:enkf}. Ensemble Kalman methods for inversion, as shown in Algorithms \ref{alg:EKTI}, \ref{alg:EKOI}, \ref{alg:EKI_post}, are presented in 
Subsections \ref{ssec:IPFT} and \ref{ssec:IPIFT},
and applied within Examples \ref{ex:EKI_1D} and \ref{ex:EKI}. We refer to
Algorithm \ref{alg:EKTI} as Ensemble Kalman Inversion (Transport),
as it arises from the approach to inversion described in Subsection \ref{sssec:kt_sdt};
we refer to Algorithm \ref{alg:EKOI} as Ensemble Kalman Inversion (Iteration to Infinity), as it arises from the approach to inversion described in Subsection \ref{ssec:IPIFT}; Algorithm \ref{alg:EKI_post} corresponds to an ensemble Kalman inversion method employing inflation by the covariance computed under the filtering distribution, as outlined in Subsection \ref{sssec:IPB}. We refer to Algorithm \ref{alg:EKI_post} as Ensemble Kalman Inversion (Inflated State).

\begin{algorithm}
\caption{3DVAR}\label{alg:3DVAR}
\begin{algorithmic}

\STATE \textbf{Input:} Initial $v_0 \in \R^{d_v}$ and fixed gain matrix $K \in \R^{d_v\times d_y}$.

\FOR{$n=0$ to $N-1$}

\STATE\textbf{Prediction:}
\[\hv_{n+1} = \Psi(v_n).
\]

\STATE\textbf{Analysis:}
\[v_{n+1} = \hv_{n+1}+K\Bigl(\yd_{n+1}-h\bigr(\hv_{n+1}\bigr)\Bigr).
\]

\ENDFOR

\STATE\textbf{Output:} Estimates $\{v_n\}_{n=0}^N$.

\end{algorithmic}
\end{algorithm}
%%%%%%%%%%%%%%%%%%%%%%%%%%%%%%%%%%%%%%%%%%%%%
\begin{algorithm}
\caption{EnKF}\label{alg:EnKF}
\begin{algorithmic}
\STATE \textbf{Input:} Ensemble size $J$, initial ensemble $\{v_0^{(j)}\}_{j=1}^J$.

\FOR{$n=0$ to $N-1$}

\STATE\textbf{Prediction:} for $j=1,\dots,J$ do
\begin{equation*}
    \hv_{n+1}^{(j)} = \Psi(v_n^{(j)}) + \xi_n^{(j)},\qquad \xi_n^{(j)} \sim \text{N}(0,\Sigma).
\end{equation*}

\STATE Compute \begin{align*}\pmean_{n+1} &= \frac{1}{J}\sum_{j=1}^J\hv_{n+1}^{(j)}, \quad \ho_{n+1} = \frac{1}{J}\sum_{j=1}^Jh(\hv_{n+1}^{(j)}),\\
\pCov^{vh}_{n+1} &= \frac{1}{J}\sum_{j=1}^J(\hv_{n+1}^{(j)} - \pmean_{n+1})\otimes (\hv_{n+1}^{(j)} - \ho_{n+1}),\\
\pCov^{hh}_{n+1} &= \frac{1}{J}\sum_{j=1}^J(\hv_{n+1}^{(j)} - \ho_{n+1})\otimes (\hv_{n+1}^{(j)} - \ho_{n+1}),\\
K_{n+1} &= \pCov_{n+1}^{vh}\left( \pCov_{n+1}^{hh} + \Gamma\right)^{-1}.
\end{align*}

\STATE\textbf{Analysis:} for $j=1,\dots,J$ do
\begin{align*}
\hy_{n+1}^{(j)} &= h(\hv_{n+1}^{(j)}) + \eta_n^{(j)}, \qquad \eta_n^{(j)} \sim \text{N}(0,\Gamma),\\
     v_{n+1}^{(j)} &= \hv_{n+1}^{(j)}+K_{n+1}\Bigl(y_{n+1}^{(j)}-h\bigr(\hv_{n+1}^{(j)}\bigr)\Bigr).
\end{align*}

\ENDFOR
\STATE\textbf{Output:} Ensembles $\{v_n^{(j)}\}_{j=1}^J$ for $n=0,\dots,N$.
\end{algorithmic}
\end{algorithm}
%%%%%%%%%%%%%%%%%%%%%%%%%%%%%%%%%%%%%%%%%%%%%
\begin{algorithm}
\caption{Ensemble Kalman Inversion (Transport to Finite Time)}\label{alg:EKTI}
\begin{algorithmic}
\STATE \textbf{Input:} Data $\wwd$, $N$ and $\dt$ such that $N\dt =1$, ensemble size $J$, initial ensemble $\{u_0^{(j)}\}_{j=1}^J$.

\FOR{$n=0$ to $N-1$}

\STATE\textbf{Prediction:} for $j=1,\dots,J$ do
\begin{equation*}
    \hu_{n+1}^{(j)} = u_{n}^{(j)}.
\end{equation*}

\STATE Compute \begin{align*}\pmean_{n+1} &= \frac{1}{J}\sum_{j=1}^J\hu_{n+1}^{(j)}, \quad \ho_{n+1} = \frac{1}{J}\sum_{j=1}^JG(\hu_{n+1}^{(j)}),\\
\pCov^{uG}_{n+1} &= \frac{1}{J}\sum_{j=1}^J\bigl(\hu_{n+1}^{(j)} - \pmean_{n+1}\bigr)\otimes \bigl(G(\hu_{n+1}^{(j)}) - \ho_{n+1}\bigr),\\
\pCov^{GG}_{n+1} &= \frac{1}{J}\sum_{j=1}^J\bigl(G(\hu_{n+1}^{(j)}) - \ho_{n+1}\bigr)\otimes \bigl(G(\hu_{n+1}^{(j)}) - \ho_{n+1}\bigr).\\
\end{align*}

\STATE\textbf{Analysis:} for $j=1,\dots,J$ do
\begin{align*}
\hw_{n+1}^{(j)} &= G\bigl(\hu_{n+1}^{(j)}\bigr) + \eta_n^{(j)}, \qquad \eta_n^{(j)} \sim \text{N}\Bigl(0,\frac{\Gammas}{\dt}\Bigr),\\
     u_{n+1}^{(j)} &= \hu_{n+1}^{(j)}+\dt\pCov_{n+1}^{uG}\left( \dt\pCov_{n+1}^{GG} + \Gammas\right)^{-1}\bigl(\wwd - \hw_{n+1}^{(j)}\bigr).
\end{align*}

\ENDFOR
\STATE\textbf{Output:} Ensemble $\{u_N^{(j)}\}_{j=1}^J$.
\end{algorithmic}
\end{algorithm}
%%%%%%%%%%%%%%%%%%%%%%%%%%%%%%%%%%%%%%%%%%%%%
\begin{algorithm}
\caption{Ensemble Kalman Inversion (Iteration to Infinity)}\label{alg:EKOI}
\begin{algorithmic}
\STATE \textbf{Input:} Data $\wwd$, $N_{\infty}$, $\dt$, ensemble size $J$, initial ensemble $\{u_0^{(j)}\}_{j=1}^J$.

\WHILE{$n < N_{\infty}$}

\STATE\textbf{Prediction:} for $j=1,\dots,J$ do
\begin{equation*}
    \hu_{n+1}^{(j)} = u_{n}^{(j)}.
\end{equation*}

\STATE Compute \begin{align*}\pmean_{n+1} &= \frac{1}{J}\sum_{j=1}^J\hu_{n+1}^{(j)}, \quad \ho_{n+1} = \frac{1}{J}\sum_{j=1}^JG(\hu_{n+1}^{(j)}),\\
\pCov^{uG}_{n+1} &= \frac{1}{J}\sum_{j=1}^J\bigl(\hu_{n+1}^{(j)} - \pmean_{n+1}\bigr)\otimes \bigl(G(\hu_{n+1}^{(j)}) - \ho_{n+1}\bigr),\\
\pCov^{GG}_{n+1} &= \frac{1}{J}\sum_{j=1}^J\bigl(G(\hu_{n+1}^{(j)}) - \ho_{n+1}\bigr)\otimes \bigl(G(\hu_{n+1}^{(j)}) - \ho_{n+1}\bigr).\\
\end{align*}

\STATE\textbf{Analysis:} for $j=1,\dots,J$ do
\begin{align*}
\hw_{n+1}^{(j)} &= G\bigl(\hu_{n+1}^{(j)}\bigr) + \eta_n^{(j)}, \qquad \eta_n^{(j)} \sim \text{N}\Bigl(0,\frac{\Gammas}{\dt}\Bigr),\\
     u_{n+1}^{(j)} &= \hu_{n+1}^{(j)}+\dt\pCov_{n+1}^{uG}\left( \dt\pCov_{n+1}^{GG} + \Gammas\right)^{-1}\bigl(\wwd - \hw_{n+1}^{(j)}\bigr).
\end{align*}

\ENDWHILE
\STATE\textbf{Output:} Ensemble $\{u_{N_{\infty}}^{(j)}\}_{j=1}^J$.
\end{algorithmic}
\end{algorithm}
%%%%%%%%%%%%%%%%%%%%%%%%%%%%%%%%%%%%%%%%%%%%%
%\begin{algorithm}
%\caption{Ensemble Kalman Inversion (Iteration to Infinity)}\label{alg:EKI_alpha}
%\begin{algorithmic}
%\STATE \textbf{Input:} Data $\wwd$, $N$, $\alpha \in (0,1)$, $r_0$, $\sigma',\gamma'$ ensemble size $J$, initial ensemble $\{u_0^{(j)}\}_{j=1}^J$.

%\FOR{$n=0$ to $N-1$}

%\STATE\textbf{Prediction:} for $j=1,\dots,J$ do
%\begin{align*}
%    \xi_n^{(j)} &\sim \text{N}(0,\sigma'\Sigmas), \quad \eta_n^{(j)}\sim \text{N}(0,\gamma'\Gammas) \\
%    \hu_{n+1}^{(j)} &= \alpha u_{n}^{(j)}+(1-\alpha)r_0+\xi_n^{(j)}\\
%    \hw_{n+1}^{(j)} &= w + \eta_n^{(j)}.
%\end{align*}

%\STATE Compute \begin{align*}\pmean_{n+1} &= \frac{1}{J}\sum_{j=1}^J\hu_{n+1}^{(j)}, \quad \ho_{n+1} = \frac{1}{J}\sum_{j=1}^JG(\hu_{n+1}^{(j)}),\\
%\pCov^{uG}_{n+1} &= \frac{1}{J}\sum_{j=1}^J\bigl(\hu_{n+1}^{(j)} - \pmean_{n+1}\bigr)\otimes \bigl(G(\hu_{n+1}^{(j)}) - \ho_{n+1}\bigr),\\
%\pCov^{GG}_{n+1} &= \frac{1}{J}\sum_{j=1}^J\bigl(G(\hu_{n+1}^{(j)}) - \ho_{n+1}\bigr)\otimes \bigl(G(\hu_{n+1}^{(j)}) - \ho_{n+1}\bigr).\\
%\end{align*}

%\STATE\textbf{Analysis:} for $j=1,\dots,J$ do
%\begin{equation*}
%     u_{n+1}^{(j)} = \hu_{n+1}^{(j)}+\pCov_{n+1}^{uG}\left(\pCov_{n+1}^{GG} + \gamma'\Gammas\right)^{-1}\bigl(\hw_{n+1}^{(j)} - G(\hu_{n+1}^{(j)})\bigr).
%\end{equation*}

%\ENDFOR
%\STATE\textbf{Output:} Ensembles $\{u_n^{(j)}\}_{j=1}^J$ for $n=0,\dots,N$.

%\end{algorithmic}
%\end{algorithm}
%%%%%%%%%%%%%%%%%%%%%%%%%%%%%%%%%%%%%%%%%%%%%
\begin{algorithm}
\caption{Ensemble Kalman Inversion (Inflated State)}
\label{alg:EKI_post}
\begin{algorithmic}
\STATE \textbf{Input:} Data $\wwd_R$, $N_{\infty}$, $\dt$, ensemble size $J$, initial ensemble $\{u_0^{(j)}\}_{j=1}^J$.

\WHILE{$n<N_{\infty}$}

\STATE\textbf{Prediction:} Compute
\begin{equation*}
    m_{n} = \frac1J\sum_{j=1}^Ju_n^{(j)}, \qquad
    C_n =\frac1J\sum_{j=1}^J \bigl( u_n^{(j)}-m_n\bigr)\otimes\bigl( u_n^{(j)}-m_n\bigr)
\end{equation*}

\qquad for $j=1,\dots,J$ do

\begin{equation*}
    \hu_{n+1}^{(j)} = u_{n}^{(j)} + \xij_n,  \qquad \xij_n \sim \text{N}\Bigl(0,\frac{\dt}{1-\dt}C_n\Bigr).
\end{equation*}

\vspace{1mm}
\STATE Compute \begin{align*}\pmean_{n+1} &= \frac{1}{J}\sum_{j=1}^J\hu_{n+1}^{(j)}, \qquad \ho_{n+1} = \frac{1}{J}\sum_{j=1}^JG_R(\hu_{n+1}^{(j)}),\\
\pCov^{uG}_{R,n+1} &= \frac{1}{J}\sum_{j=1}^J\bigl(\hu_{n+1}^{(j)} - \pmean_{n+1}\bigr)\otimes \bigl(G_R(\hu_{n+1}^{(j)}) - \ho_{n+1}\bigr),\\
\pCov^{GG}_{R,n+1} &= \frac{1}{J}\sum_{j=1}^J\bigl(G_R(\hu_{n+1}^{(j)}) - \ho_{n+1}\bigr)\otimes \bigl(G_R(\hu_{n+1}^{(j)}) - \ho_{n+1}\bigr).\\
\end{align*}

\STATE\textbf{Analysis:} for $j=1,\dots,J$ do
\[
\begin{aligned}
\hy_{n+1}^{(j)} &= \dt G_R\bigl(\hu^{(j)}_{n+1}\bigr) + \eta_n^{(j)}, \qquad \eta_n^{(j)} \sim \text{N}(0,\dt\Gammas_R),\\
     u_{n+1}^{(j)} &= \hu_{n+1}^{(j)}+\pCov_{R,n+1}^{uG}\left( \dt\pCov_{R,n+1}^{GG} + \Gammas_R\right)^{-1}\bigl(\dt\wwd_R -\hy^{(j)}_{n+1}\bigr).
\end{aligned}
\]
\ENDWHILE
\STATE\textbf{Output:} Ensemble $\{u_{N_{\infty}} ^{(j)}\}_{j=1}^J$.
\end{algorithmic}
\end{algorithm}

%%%%%%%%%%%%%%%%%%%%%%%%%%%%%%%%%%%%%%%%%%%%%
%
\section{Appendix B (Lorenz '96 Multiscale)}
\label{sec:AB}
%
%%%%%%%%%%%%%%%%%%%%%%%%%%%%%%%%%%%%%%%%%%%%%

To illustrate the problems of both state estimation and parameter
estimation we use, throughout this paper, variants on the Lorenz '96
model. In particular we use both the Lorenz '96 multiscale system, introduced in
subsection \ref{ssec:l96M}, and a singlescale closure derived from it in the scale-separated case, described in Subsections  \ref{ssec:l96S} and \ref{ssec:BSE}.
If we (i) generate data with the singlescale model,
and assimilate using the same  model, we are able to test algorithms in their basic (perfect model) form; on the other hand if we (ii) generate data with the multiscale model,
and assimilate using the singlescale model, this allows us to study 
the effect of model misspecification on data assimilation.
Subsection \ref{ssec:83} contains an example of the type (ii), whilst 
Examples \ref{ex:3dvar}, \ref{ex:s3dvar} and \ref{ex:enkf} are of type (i). 

%%%%%%%%%%%%%%%%%%%%%%%%%%%%%%%%%%%%%%%%%%%%%%%%%%%
%
\subsection{Lorenz '96 Multiscale Model}
\label{ssec:l96M}
%
%%%%%%%%%%%%%%%%%%%%%%%%%%%%%%%%%%%%%%%%%%%%%%%%%%%

Let $v \in C(\R^+,\R^L)$ and $w \in C(\R^+,\R^{L \times J})$.
We write down an ODE for $(v,w)$ in which each variable $v_\ell \in \R$ is coupled to a 
subgroup of fast variables $w_\ell=\{w_{\ell,j}\}_{j=1}^J \in \R^J$; further (discrete) boundary
condition couplings impose periodicity in $L$ and link
$\{w_{\ell,j}\}_{j=1}^J$ to $\{w_{\ell+1,j}\}_{j=1}^J.$ 
The particular form of the system of ODEs is as given in \citet{fatkullin2004computational}.
For  $\ell=1 \dots L$ and $j=1 \dots J$, the ODEs are
\begin{subequations}
  \label{eq:l96ms}
\begin{align}
\dot{v}_\ell &= f_\ell(v) + h_v \overline{w}_\ell, \quad \overline{w}_\ell = \frac{1}{J} \sum_{j=1}^J w_{\ell,j},\\
\dot{w}_{\ell,j} &= \frac{1}{\eps} r_{j}(v_\ell, w_\ell), 
\end{align}
\end{subequations}
where
\begin{subequations}
  \label{eq:l96ms_define}
\begin{align}
f_\ell(v) &:= -v_{\ell-1}( v_{\ell-2} - v_{\ell+1}) - v_{\ell} + F, \\
r_{j}(v_\ell, w_\ell) &:= -w_{\ell,j+1}(w_{\ell,j+2} - w_{\ell,j-1}) - w_{\ell,j}  +  h_w v_{\ell},
\end{align}
\end{subequations}
and we impose the boundary conditions
\begin{equation}
\label{eq:bc}
v_{\ell + L} = v_\ell, \quad w_{\ell+L,j} = w_{\ell,j}, \quad w_{\ell,j+J} = w_{\ell+1,j}.
\end{equation}
Here $\eps > 0$ is a scale-separation parameter, $h_v, h_w \in \R$ 
govern the couplings between the fast and slow systems, 
and $F > 0$ provides a constant forcing. 

%%%%%%%%%%%%%%%%%%%%%%%%%%%%%%%%%%%%%%%%%%%%%%%%%%%%
%
\subsection{Lorenz '96 Singlescale Model}
\label{ssec:l96S}
%
%%%%%%%%%%%%%%%%%%%%%%%%%%%%%%%%%%%%%%%%%%%%%%%%%%%%%

Let $v=\{v_\ell\}_{\ell=1}^L, w=\{w_\ell\}_{\ell=1}^L$ and $\overline{w}=\{\overline{w}_\ell\}_{\ell=1}^L.$ Then we may write \eqref{eq:l96ms} in the form
\begin{subequations}
  \label{eq:l96ms2}
\begin{align}
\dot{v} &= F(v) + h_v \overline{w},\\
\dot{w} &= \frac{1}{\eps} R(v, w), 
\end{align}
\end{subequations}
for suitable definitions of $F,R.$ If $\eps \ll 1$ then the dynamics for the $w$ governed 
(\ref{eq:l96ms2}b) evolve on a much faster timescale than the
dynamics for the $v$ governed by (\ref{eq:l96ms2}a).
Thus it is a reasonable approximation to think of $v$ as
frozen in (\ref{eq:l96ms}b). 
If we assume that the dynamics of $w$ with
$v$ frozen are ergodic with invariant measure $\mu^v(dw)$ (a measure in $w$,
parameterized by $v$) then the averaging principle 
\citep{vanden2003fast,abdulle2012heterogeneous,pavliotis2008multiscale}
suggests that  we may make the approximation
$$\overline{w} \approx M(v):=\int \overline{w}\,\mu^v(\dd w).$$
If we also invoke the approximation $M_\ell(v) \approx m(v_\ell)$,
which is shown to be valid for large $J$ in \citet{fatkullin2004computational}, 
then we arrive at the singlescale Lorenz '96 model \eqref{eq:l96}. This program of analysis for the Lorenz '96 model, and studies of the validity of the resulting singlescale approximation, was established in the paper \citet{fatkullin2004computational}.
The function $m(\cdot)$ is not given explicitly, but may be fit to data in
various different ways, as explained in \citet{fatkullin2004computational}.
Figure \ref{fig:multiscale_m} shows such an $m$, fit using Gaussian process regression 
methodology as detailed in Subsection 4.3 of \citet{burov2021kernel}.

\subsection{Example}
\label{ssec:83}

The following example relates to the discussion in Remark \ref{rem:what}.

\begin{example}
\label{ex:3dvar_ms}

Throughout this example we set $L = 9, \ J = 8,\  h_v = -0.8,\ h_w = 1,\ F = 10$ and $\eps=2^{-7}$. Note that the equation
(\ref{eq:l96ms}a) has the form of the singlescale 
Lorenz model \eqref{eq:l96} for $v$ with a coupling to the fast variables $w$ governed by (\ref{eq:l96ms}b) replacing the function $m(\cdot).$
In Figure \ref{fig:true_multiscale} we display the dynamics of a slow variable and one associated fast variable of the Lorenz '96 multiscale model \eqref{eq:l96ms}; at the parameter values chosen the system is chaotic.

We consider observations $\{\yd_n\}_{n \in \Z^+}$ arising from the model
\begin{align*}
\label{eq:multiscale_experiment_dynamics}
(\vd_{n+1},\wwd_{n+1}) &= \Psi_{\rm{mult}}(\vd_n,\wwd_n), \\
\yd_{n+1} &= h(\vd_{n+1}),
\end{align*}
where $\Psi_{\rm{mult}}$ is the solution operator to the multiscale model \eqref{eq:l96ms}--\eqref{eq:bc} over the observation time interval $\tau$. 
We then take the data from the multiscale model and assimilate it into
the singlescale model \eqref{eq:l96}, recalling the specific
function $m$ shown in Figure \ref{fig:multiscale_m}, and discussion in the preceding subsection
concerning elimination of the fast variables in favour of a simple closure, using
the averaging principle. 

We now discuss data assimilation using this multiscale data. As in Example \ref{ex:3dvar}, we assume that the observation function is linear: $h(v)=Hv$ for matrix $H:\R^9 \to \R^6$ defined by \eqref{eq:L96H}. As before we choose the gain $K:\R^6 \to \R^9$ to be defined by \eqref{eq:L96K} and employ the 3DVAR algorithm \eqref{eq:3DVARL96} with $\Psi$ the solution operator over time-interval $\tau$ for the singlescale model. Thus model misspecification is present, because data $\Yd$
comes from the multiscale model.

\begin{figure}[h!]
\centering
\includegraphics[width=\linewidth]{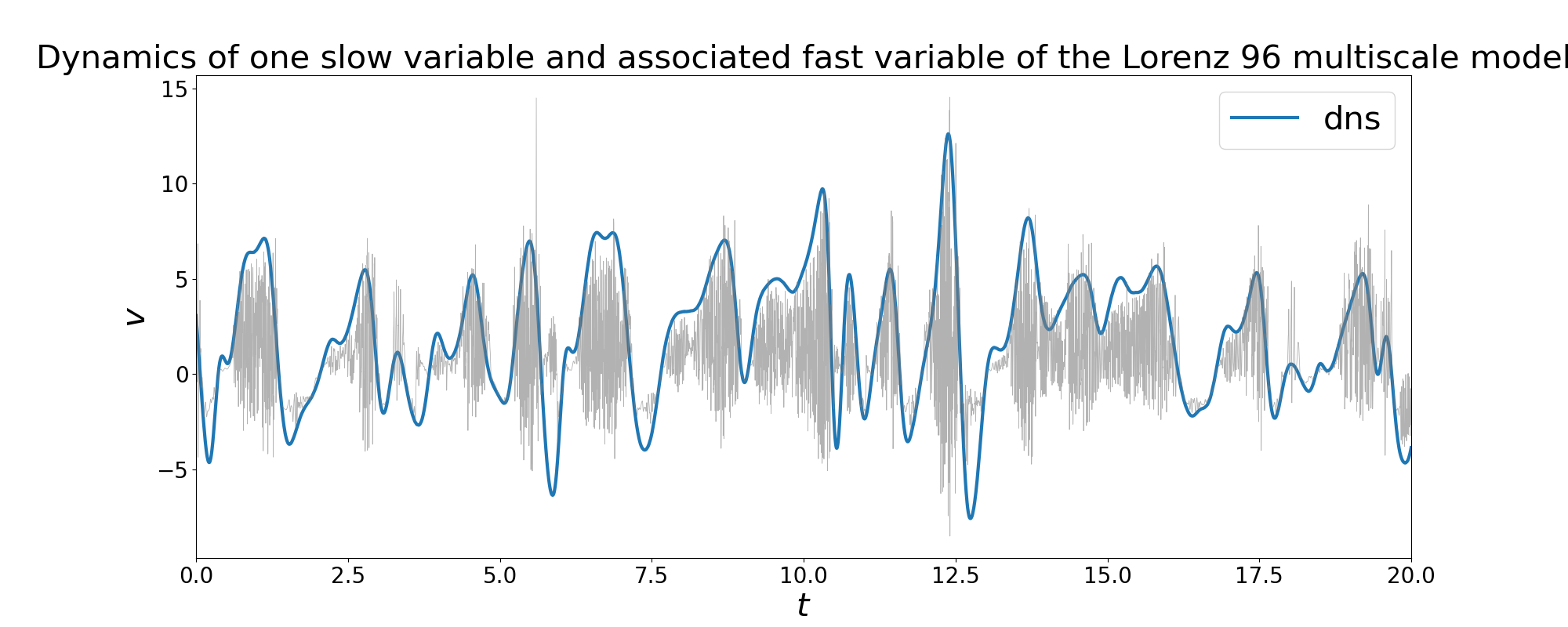}
\caption{Dynamics of a slow variable and an associated fast variable.
Here ``dns'', in blue, labels the slow variable computed by direct 
numerical simulation; in grey is the related fast variable.}
\label{fig:true_multiscale}
\end{figure}

As in Example \ref{ex:3dvar} we display the results of 3DVAR on the unobserved component $v_3.$ 
Figure \ref{fig:multi_dt} shows the behavior of the algorithm for $\tau =10^{-3}.$ Despite the fact that the data is generated from the
multiscale model, whilst assimilation is conducted using the singlescale
model, 3DVAR produces an accurate estimate of the true state with, after
synchronization, the only discernible errors being slight under and overshoots.
Figure \ref{fig:multi_tau1} shows the effect of varying $\tau$, the time between observations. As in Example \ref{ex:3dvar} the assimilation is significantly worse
when $\tau$ is larger.
$\blacksquare$\end{example}

\begin{figure}[h!]
\centering
\begin{subfigure}{\textwidth}
  \centering
  \includegraphics[width=1\linewidth]{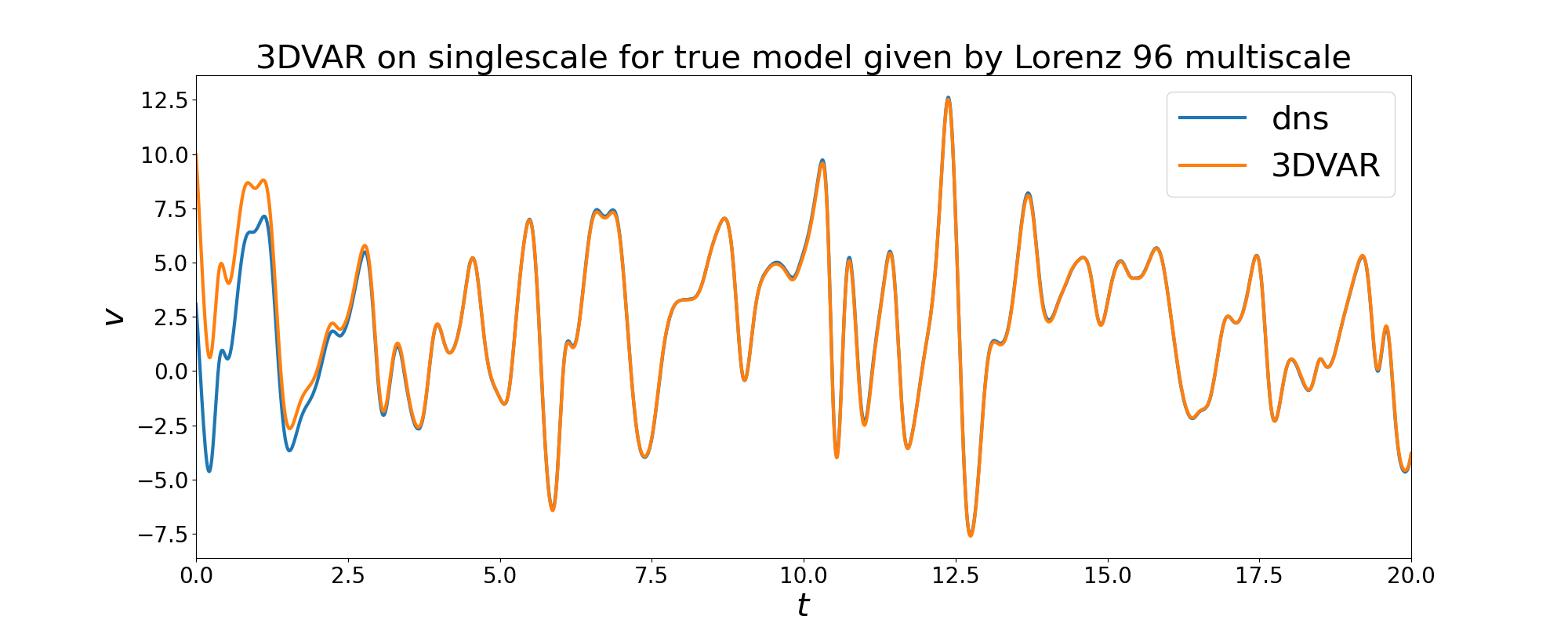}
  \caption{}
  \label{fig:multi_dt}
\end{subfigure}
\begin{subfigure}{\textwidth}
  \centering
  \includegraphics[width=1\linewidth]{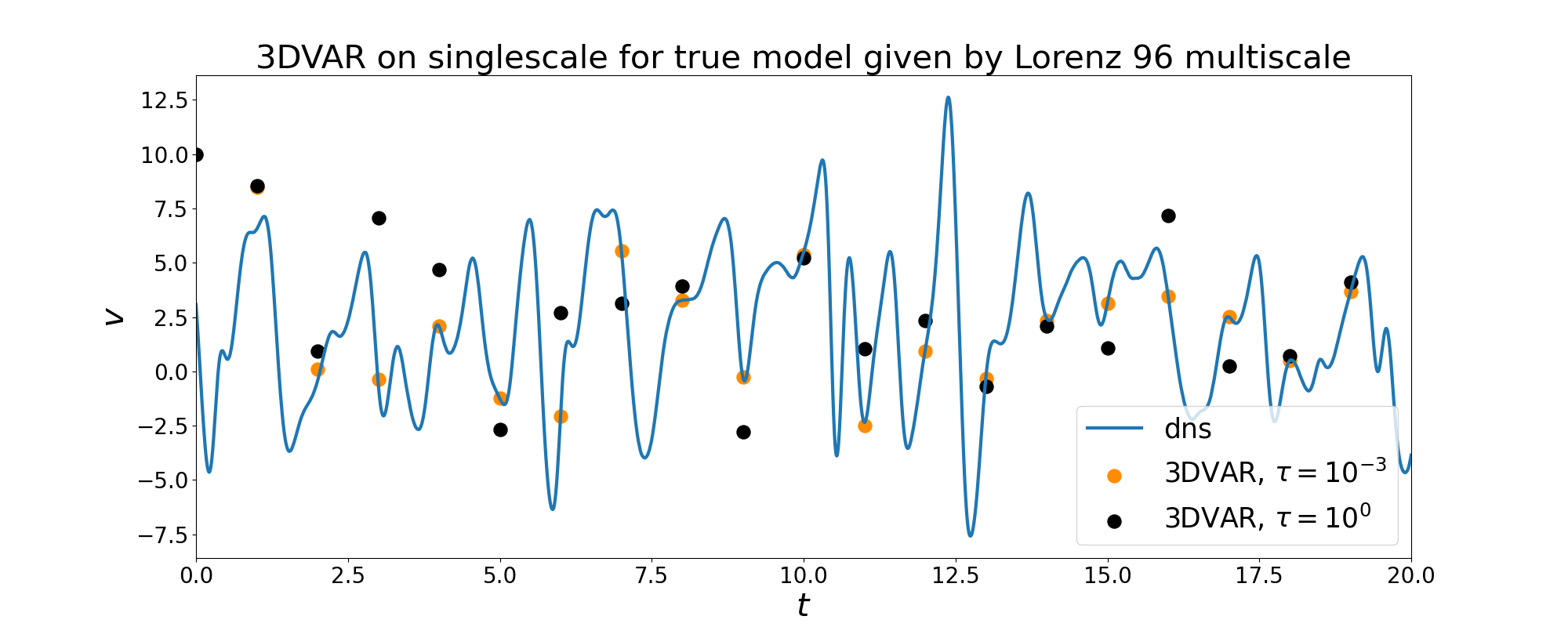}
  \caption{}
  \label{fig:multi_tau1}
\end{subfigure}
\caption{In (a) the estimates of $v_3$ in time produced by 3DVAR using $\tau =10^{-3}$ are displayed against the true dynamics. In (b) the estimates at each unit time obtained using 3DVAR with assimilation at $\tau = 10^{0}$ and $\tau = 10^{-3}$ are shown. Again the acronym ``dns'' refers to direct numerical simulation. 3DVAR successfully synchronizes with the direct numerical simulation at the smaller value of $\tau$ but fails to do so as well when $\tau$ is larger. It is noteworthy that the synchronization takes place here in the context of model misspecification: the data is generated by the multiscale model, but 3DVAR is applied using the singlescale model.}
\label{fig:multi_tau}
\end{figure}

%%%%%%%%%%%%%%%%%%%%%%%%%%%%%%%%%%%%%%%%%%%%%%%%%%%%%%%%%%%
%
\section{Appendix C (Mean Field Maps)}
\label{appendix:C}
%
%%%%%%%%%%%%%%%%%%%%%%%%%%%%%%%%%%%%%%%%%%%%%%%%%%%%%%%%%%%%

In Subsection \ref{ssec:mfmsd} we discuss the existence, form and properties of mean field maps
which carry out the program of approximate transport described in Subsection \ref{sssec:ats}; these
maps require access to simulated data and are hence referred as stochastic second order transport maps.
In Subsection \ref{ssec:mfmnsd} we discuss the existence, form and properties of mean field maps
which carry out the program of approximate transport described in Subsection \ref{sssec:atd}; these
maps do not require access to simulated noisy data and are hence referred to
as deterministic second order transport maps. The two approaches provide fundamental underpinnings of ensemble Kalman
methods as we develop them in this paper, but were not adopted in its historical development. Subsection \ref{ssec:TM_MVA} is devoted to the minimum variance approximation, a way of deriving the Kalman transport map \eqref{eq:sd2nn_add}, which is part  of the historical development of the subject, but which does not play a central role in our presentation and analysis of the subject.

%%%%%%%%%%%%%%%%%%%%%%%%%%%%%%%%%%%%%%%%%%%%%%%%%%%%%%%%%%%
%
\subsection{Mean Field Maps -- Simulated Data}
\label{ssec:mfmsd}
%
%%%%%%%%%%%%%%%%%%%%%%%%%%%%%%%%%%%%%%%%%%%%%%%%%%%%%%%%%%%

We note that the maps of interest in this subsection take $\law(\hv_{n+1},\hy_{n+1})$
into $\law(v_{n+1})$. Since the analysis is independent of any specific value of discrete time $n$, we drop the subscript $n+1$ throughout the analysis; this streamlines the notation. Similar considerations apply in Subsection \ref{ssec:mfmnsd} and in Subsection \ref{ssec:TM_MVA}.

Let $(\hv,\hy) \sim \nu$, where 
\begin{equation*}
\label{eq:nu}
    \op G\nu=\Ng\left(
    \begin{bmatrix}
    \pmean\\
    \ho
    \end{bmatrix}, 
    \begin{bmatrix}
   \pCov & \pCov^{vy}\\
    (\pCov^{vy})^\top & \pCov^{yy},
    \end{bmatrix}
    \right).
\end{equation*}
Define
\begin{subequations}
\label{eq:KF_analysisTM}
\begin{align}
        \mean &= \pmean + \pCov^{vy} (\pCov^{yy})^{-1} \bigl(\yd - \ho \bigr),\\
         \Cov &= \pCov - \pCov^{vy}(\pCov^{yy})^{-1} (\pCov^{vy})^\top.
    \end{align}
\end{subequations}
Note that these are the mean and covariance of the Gaussian random variable
$\hv$ conditioned on $\hy=\yd$ on the assumption that $(\hv,\hy)$ is distributed
according to the Gaussian measure $\op G\nu;$ see \eqref{eq:KF_analysis}.
The quantities in the above identities are as defined in Subsection \ref{sssec:IN},
with the caveat that the subscripts $n+1$ have been dropped for clarity of exposition.

Our goal is to identify maps of the form
\begin{equation}
\label{eq:themap}
v=A\hv+B\hy+a    
\end{equation}
so that, if $(\hv,\hy) \sim \nu$, then $v$ has mean $\mean$ and covariance $\Cov$
given by \eqref{eq:KF_analysisTM}. 
We make the following assumptions on the covariance under $\op G\nu$ and on the matrices $A,B$ and vector $a:$

\begin{assumptions}
\label{a:inv}
The covariance under $\op G\nu$ is invertible.
The matrices $A,B$ and vector $a$ may depend on $\yd$ and measure
$\nu$ but not on the random variables $(\hv,\hy)$; thus 
\eqref{eq:themap} takes the explicit form
\[
v=A(\nu,\yd)\hv+B(\nu,\yd)\hy+a(\nu,\yd). \qquad \blacksquare
\]
\end{assumptions}
\vspace{0.1in}

Recall the discussion of pushforward of measures in the introduction to
Subsection \ref{ssec:MFM}.
Under Assumptions \ref{a:inv} 
pushforward under the map \eqref{eq:themap}, when chosen to match the desired first
and second order statistics, defines a nonlinear
map on the space of measures, and in particular on $\nu$ itself.
In what follows all expectations are computed under $\nu.$ 
Since the covariance under $\op G\nu$ is strictly positive-definite so are the
marginal covariances $\hC$ and $\hCyy$ (see
\citet{stu}[Lemma 6.21]). Thus we may define the conditional
mean and covariance by \eqref{eq:KF_analysisTM},
as well as Kalman gain $K$, and conditional covariance $\tC$, given by
\begin{align}
\label{eq:theccovs}
K&=\hCvy(\hCyy)^{-1},\\
\tC&=\hCyy-(\hCvy)^\top(\hC)^{-1}\hCvy;
\end{align}
we note that the conditional covariances $C$ and $\tC$
are also strictly positive-definite
(see \citet{stu}[Lemma 6.21]). From equations (\ref{eq:KF_analysisTM}a) and
\eqref{eq:themap} the following is immediate:

\begin{lemma}
\label{lem:mean}
Let Assumptions \ref{a:inv} hold and let $(\hv,\hy) \sim \nu.$ If $v$ given by \eqref{eq:themap} has mean given by (\ref{eq:KF_analysisTM}a) then
$$a=(I-A)\bbE \hv+K\yd-(B+K)\bbE \hy.$$
$\Diamond$ \end{lemma}

As a consequence it follows that
\begin{equation}
\label{eq:newv}
v=\bbE \hv+A(\hv-\bbE \hv)+B(\hy-\bbE \hy)+K(\yd-\bbE \hy).
\end{equation}
and that
\begin{equation*}
\label{eq:newv2}
\bbE\Bigl((v-\bbE v)\otimes(v-\bbE v)\Bigr)=A\hC A^\top+B\hCyy B^\top+
A\hCvy B^\top+B(\hCvy)^\top A^\top.
\end{equation*}
Thus, to match the covariance of the conditioned random
variable, we obtain
\begin{equation}
\label{eq:newv3}
C=A\hC A^\top+B\hCyy B^\top+
A\hCvy B^\top+B(\hCvy)^\top A^\top.
\end{equation}

\begin{theorem}
\label{t:thet}
Let Assumptions \ref{a:inv} hold and let $(\hv,\hy) \sim \nu.$ If 
$a$ is given by Lemma \ref{lem:mean}
then $v$ defined by \eqref{eq:themap} has covariance  (\ref{eq:KF_analysisTM}b) if and only if real-valued matrices $A$ and $B$ are related by the identity
\begin{equation}
\label{eq:id}
F\hC^{-1}F^\top=C-B\tC B^\top,
\end{equation}
where
\begin{equation}
\label{eq:F}
F=A\hC+B(\hCvy)^\top.
\end{equation}
$\Diamond$ \end{theorem}

\begin{proof}
We complete the square on the right hand side of \eqref{eq:newv3} to obtain
$$(A\hC+B(\hCvy)^\top)\hC^{-1}(A\hC+B(\hCvy)^\top)^\top=C'.$$
where
$$C'=C-B\tC B^\top.$$
Rearranging gives the desired result.
\end{proof}

Define
$$\cB=\{B \in \bbR^{d_v \times d_y}: C' \succ 0\},$$
noting that this set is non-empty, and contains an open (and hence uncountable) set of $B$,
since $C \succ 0$. For $B \in \cB$ consider the eigenvalue problem
\begin{equation*}
C'\varphi^{(i)}=(s^{(i)})^2 \varphi^{(i)},\quad
\langle \varphi^{(i)}, \varphi^{(j)} \rangle=\delta_{ij}.
\end{equation*}
Note that this has $d_v$ real solutions, up to sign changes in
the eigenvectors and assuming the $s^{(i)}$ to be non-negative.
We now seek to express $F$ in terms of $B \in \cB.$ Writing
the SVD  $F\hC^{-\frac12}=U\Xi V^\top$, where $U,V \in \bbR^{d_v \times d_v}$
are orthogonal matrices and $\Xi \in \bbR^{d_v \times d_v}$ is a
diagonal matrix, we see from \eqref{eq:id} that
$$U\Xi^2U^\top=C'$$
so that $U$ has columns given by the $\{\varphi^{(i)}\}_{i=1}^{d_v}$
and corresponding diagonal entries of
$\Xi$, $\pm s^{(i)}$. We define
\begin{equation}
\label{eq:US}
U=\bigl(\varphi^{(1)}, \cdots, \varphi^{(d_v)}\bigr),\quad \Xi={\rm diag}\bigl(\pm s^{(1)},\cdots,\pm s^{(d_v)}\bigr).
\end{equation}

\begin{corollary}
\label{cor:deen}
For every $B \in \cB$ the choices of $A$ such that
the pair $(A,B)$ satisfies the criterion of Theorem \ref{t:thet}
are defined as follows. For $U, \Xi$ as given in \eqref{eq:US}, set
$$F=U\Xi V^\top \hC^{\frac12},$$
where $V$
is an arbitrary orthogonal matrix in $\bbR^{d_v \times d_v}.$
Then
$$A=\Bigl(F-B(\hCvy)^\top\Bigr)\hC^{-1}.$$
$\Diamond$ \end{corollary}

\begin{example}
\label{ex:twoe}
Among the uncountably many possible solutions for pairs $(A,B)$
we highlight two. The choice $B=0$ is interesting because it does not require
the data variable $\hy$ in the definition of $v$. The choice $A=I$ is interesting
because it does not require action of an operator acting on $\hv.$ 

The first, with $B=0$, allows the choice
$F=C^{\frac12}\hC^{\frac12}$ and then $A=C^{\frac12}\hC^{-\frac12}.$
Thus the map \eqref{eq:newv} becomes
\begin{align}
\label{eq:newv22}
v&=\bbE \hv+C^{\frac12}\hC^{-\frac12}(\hv-\bbE \hv)+\hCvy(\hCyy)^{-1}(\yd-\bbE \hy)\\
&=m+C^{\frac12}\hC^{-\frac12}(\hv-\bbE \hv).
\end{align}

The second comes from setting $B=-K$ which
leads to the possible choice $F=C$ and $A=I$ under which
the map \eqref{eq:newv} becomes
\begin{equation*}
\label{eq:newv11}
v=\hv+\hCvy(\hCyy)^{-1}(\yd-\hy).
\end{equation*}
We refer to this as the {\em Kalman transport} solution.
$\blacksquare$
\end{example}

Given the plethora of solutions for matrices $(A,B)$, all of which
effect the desired measure transport from $\nu$ into the
conditional, it is natural to ask how to choose a specific pair
$(A,B)$. One possibility is to use optimal transport. To this end
we define, for positive definite $W \in \bbR^{d_v \times d_v}$,
\begin{equation*}
\label{eq:II}
I_{W}(A,B)=\frac12 \bbE\langle (v-\hv), W(v-\hv) \rangle.
\end{equation*}

\begin{theorem}
\label{t:OTT}
Let $v$ be given by \eqref{eq:newv}. Consider the problem of 
finding minimizers of $I_{W}(A,B)$ over pairs $(A,B)$ satisfying 
\eqref{eq:id}, \eqref{eq:F}; we refer to the resulting map
evaluated at such an $(A,B)$ as an {\em optimal transport in the $W-$weighted Euclidean distance.}
For any positive definite $W$, such minimizers satisfy $B=0.$
In particular the Kalman transport solution is not an optimal transport solution.
$\Diamond$ \end{theorem}

\begin{proof}
We formulate the optimization problem over the pair $(F,B)$ since,
because $\hC$ is invertible, there is a bijection between $(A,B)$
and $(F,B).$ Let $:$ denote the Frobenius inner-product. Then, under constraint \eqref{eq:F},
\begin{subequations}
\label{eq:JJ}
\begin{align}
I_{W}(A,B)&=J_{W}(F)+{\rm const},\\ 
J_{W}(F)&=-F:W,
\end{align}
\end{subequations}
where ${\rm const}$ denotes a matrix independent of $A,B,F$ (with exact value changing from
instance to instance). To see this note that
$$v-\hv=-(\hv-\bbE \hv)+A(\hv-\bbE \hv)+B(\hy-\bbE \hy)+K(\yd-\bbE \hy).$$
Now, using \eqref{eq:newv3},
$$\bbE \Bigl((v-\hv)\otimes(v-\hv)\Bigr)=\hC-\hC A^\top-A\hC-\hCvy B^\top-B(\hCvy)^\top+C+{\rm const.}$$
Noting that
$$I_{W}(A,B)=\frac12 \bbE \Bigl((v-\hv)\otimes(v-\hv)\Bigr):W,$$
and that $D:W=D^\top:W$ for all $D$ since $W$ is symmetric, identity \eqref{eq:JJ} follows
from \eqref{eq:F}. It then also follows
that minimization of $I_{W}(A,B)$ subject to the constraints given by 
\eqref{eq:id}, \eqref{eq:F}
is equivalent to minimization of $J_{W}(F)$ subject to the
constraint \eqref{eq:id}.

To effect this latter minimization we introduce the Lagrange
multiplier $L \in \bbR^{d_v \times d_v}$, symmetric because the
constraint is symmetric, and define 
\begin{equation*}
\label{eq:JJt}
\widetilde{J}_{W}(F,B,L)=-F:W+L:\bigl(F\hC^{-1}F^\top+B\tC B^\top-C\bigr).
\end{equation*}
Differentiating with respect to $F,B$ and $L$ respectively gives
\begin{subequations}
\label{eq:lm}
\begin{align}
-W+2LF\hC^{-1}&=0,\\
2LB\tC &=0,\\
F\hC^{-1}F^\top+B\tC B^\top&=C.
\end{align}
\end{subequations}
Since $F$, $W$, and $\hC$ in (\ref{eq:lm}a) are all invertible, the Lagrange multiplier $L$ is necessarily invertible. Furthermore, since $\tC$ is invertible because $\Sigma$ is, it follows from (\ref{eq:lm}b) that $B=0$ as required.
\end{proof}

\begin{example}
\label{ex:cite}
Define symmetric matrix $P$, and from it symmetric $A$, by
$$P=\Bigl(C^{\frac12}\hC C^{\frac12}\Bigr)^{-\frac12},\quad A=C^{\frac12}P C^{\frac12}.$$

We notice that if $B=0$, as required for an optimal transport solution, then with this
choice of  the pair $(A,B)$, equation (\ref{eq:lm}c)
has as a solution $F=A\hC$ and (\ref{eq:lm}b) is satisfied.
Furthermore (\ref{eq:lm}a) delivers the Lagrange multiplier $L$.  Note that the solution is
independent of the specific choice of positive-definite $W$.
$\blacksquare$
\end{example}

%%%%%%%%%%%%%%%%%%%%%%%%%%%%%%%%%%%%%%%%%%%%%%%%%%%%%%
%
\subsection{Mean Field Maps -- No Simulated Data}
\label{ssec:mfmnsd}
%
%%%%%%%%%%%%%%%%%%%%%%%%%%%%%%%%%%%%%%%%%%%%%%%%%%%%%%
Let $\hv \sim \hmu$ and assume that
$\hy=h(\hv)+\eta$ where $\eta \sim \Ng(0,\Gamma)$ is independent of $\hv.$ 
Implicitly we have defined the joint distribution $\nu$ of $(\hv,\hy).$
We note that then, expressed in terms
of $\hh:= h(\hv)$ and quantities defined in \eqref{eq:nu}, 
\begin{align*}
\label{eq:nu2}
\pCov^{vh}:=&\bbE (\hv-\bbE \hv) \otimes (\hh-\bbE \hh) = \pCov^{vy},\\
\pCov^{hh}:=&\bbE (\hh-\bbE \hh) \otimes (\hh-\bbE \hh) = \pCov^{yy}-\Gamma.
\end{align*}
Thus we may rewrite \eqref{eq:KF_analysisTM} as
\begin{subequations}
\label{eq:KF_analysisTM2}
\begin{align}
        \mean &= \pmean + \pCov^{vh} (\pCov^{hh}+\Gamma)^{-1} \bigl(\yd - \bbE \hh \bigr),\\
         \Cov &= \pCov - \pCov^{vh}(\pCov^{hh}+\Gamma)^{-1} (\pCov^{vh})^\top,
    \end{align}
\end{subequations}
In so doing we have eliminated reference to $\hy$
and our goal becomes the idenitification of maps of the form
\begin{equation}
\label{eq:themap2}
v=R\hv+S\hh+r    
\end{equation}
so that, if $\hv \sim \hmu$, then $v$ has mean $\mean$ and covariance $\Cov$
given by \eqref{eq:KF_analysisTM2}. 
We make the following assumptions on the covariance under $\op G\nu$ and on
the matrices $R, S$ and vector $r:$

\begin{assumptions}
\label{a:inv2}
The covariance under $\op G\nu$ is invertible.
The matrices $R,S$ and vector $r$ may depend on $\yd$ and measure
$\hmu$ but not on the random variable $(\hv,\hh)$; thus \eqref{eq:themap2} takes the explicit form
\[
v=R(\nu,\yd)\hv+S(\nu,\yd)\hh+r(\nu,\yd). \qquad \blacksquare
\]
\end{assumptions}
\vspace{0.1in}

With these assumptions the pushforward under map \eqref{eq:themap}, when constrained to match the desired first and second order statistics, defines a nonlinear
map on the space of measures, and in particular on measure $\hmu$.
In what follows all expectations are computed under $\hmu.$ 
We note that matrix $\hC$ is invertible and that $K$ in \eqref{eq:theccovs} may be rewritten as 
\begin{equation*}
\label{eq:theccovs2}
K=\hCvh(\hChh+\Gamma)^{-1}.\\
\end{equation*}
From equations (\ref{eq:KF_analysisTM2}a) and
\eqref{eq:themap2} the following is immediate:

\begin{lemma}
\label{lem:mean2}
Let Assumptions \ref{a:inv2} hold and let $\hv \sim \hmu.$ If $v$ given by \eqref{eq:themap2} 
has mean given by (\ref{eq:KF_analysisTM2}a) then
$$r=(I-R)\bbE \hv+K\yd-(S+K)\bbE \hh.$$
$\Diamond$ \end{lemma}

As a consequence it follows that
\begin{equation}
\label{eq:newv2i}
v=\bbE \hv+R(\hv-\bbE \hv)+S(\hh-\bbE \hh)+K(\yd-\bbE \hh).
\end{equation}
and that
\begin{equation}
\label{eq:newv22i}
\bbE\Bigl((v-\bbE v)\otimes(v-\bbE v)\Bigr)=R\hC R^\top+S\hChh S^\top+
R\hCvh S^\top+S(\hCvh)^\top R.
\end{equation}
Thus, to match the covariance of the conditioned random
variable, we obtain
\begin{equation}
\label{eq:newv32}
C=R\hC R^\top+S\hChh S^\top+
R\hCvh S^\top+S(\hCvh)^\top R.
\end{equation}
Define
$$\ttC=\hChh-(\hCvh)^\top(\hC)^{-1}\hCvh.$$

\begin{theorem}
\label{t:thet1}
Let Assumptions \ref{a:inv2} hold and let $\hv \sim \hmu.$ If 
$s$ is given by Lemma \ref{lem:mean2}
then $v$ defined by \eqref{eq:themap2} has covariance  (\ref{eq:KF_analysisTM2}b) if and only if real-valued matrices $R$ and $S$ are related by the identity
\begin{equation*}
\label{eq:id2}
E\hC^{-1}E^\top=C-S\ttC S^\top,
\end{equation*}
where
\begin{equation*}
\label{eq:F2}
E=R\hC+S(\hCvh)^\top.
\end{equation*}
$\Diamond$ \end{theorem}

\begin{proof}
We complete the square on the right hand side of \eqref{eq:newv32} to obtain
$$(R\hC+S(\hCvh)^\top)\hC^{-1}(R\hC+S(\hCvh)^\top)^\top+S\ttC S^\top=C.$$
Rearranging gives the desired result.
\end{proof}

Define
$$\ccB=\{S \in \bbR^{d_v \times d_y}: C-S\ttC S^\top \succ 0\}$$
and, for $S \in \ccB$ consider the eigenvalue problem
\begin{equation*}
(C-S\ttC S^\top)\varpsi^{(i)}=(o^{(i)})^2 \varpsi^{(i)},\quad
\langle \varpsi^{(i)}, \varpsi^{(j)} \rangle=\delta_{ij},
\end{equation*}
noting that this has $d_v$ real solutions, upto sign changes in
the eigenvectors and assuming the $o^{(i)}$ to be non-negative.
We now seek to express $E$ in terms of $S \in \ccB.$ Writing
the SVD  $E\hC^{-\frac12}=W\ssS Z^\top$, where $W,Z \in \bbR^{d_v \times d_v}$
are orthogonal matrices and $\ssS \in \bbR^{d_v \times d_v}$ is a
diagonal matrix, we see from \eqref{eq:id} that
$$W\ssS^2W^\top=C-S\ttC S^\top$$
so that $W$ has columns given by the $\{\varpsi^{(i)}\}_{i=1}^{d_v}$
and corresponding diagonal entries of
$\ssS$, $\pm o^{(i)}$. We define
\begin{equation}
    \label{eq:divine}
W=\bigl(\varpsi^{(1)}, \cdots, \varpsi^{(d_v)}\bigr),\quad \ssS={\rm diag}\bigl(\pm o^{(1)}, \dots, \pm o^{(d_v)}\bigr).
\end{equation}

\begin{corollary}
For every $S \in \ccB$ the choices of $R$ such that
the pair $(R,S)$ satisfies the criterion of Theorem \ref{t:thet1}
are defined as follows. For $W, \ssS$ as given in \eqref{eq:divine}, set
$$E=W\ssS Z^\top \hC^{\frac12},$$
where $Z$
is an arbitrary orthogonal matrix in $\bbR^{d_v \times d_v}.$
Then
$$R=\Bigl(E-S(\hCvh)^\top\Bigr)\hC^{-1}.$$
$\Diamond$ \end{corollary}

\begin{example}
\label{ex:twoee}
As in Example \ref{ex:twoe} we highlight two examples, here corresponding to
$S=0$ and to $R=I.$ The first, with $S=0$, allows the choice
$E=C^{\frac12}\hC^{\frac12}$ and hence $R=C^{\frac12}\hC^{-\frac12}.$
We thus obtain \eqref{eq:newv22} again.

The second comes from setting $R=I.$ 
This leads from \eqref{eq:newv22i} to the following equation for $S:$
$$S\hChh S^\top+\hCvh S^\top+S(\hCvh)^\top=-\hCvh(\hChh+\Gamma)^{-1}(\hCvh)^\top.$$
We seek a solution for $S$ in the form
$$S=-\hCvh Y^{-1}$$
and determine $Y$. 
We obtain the equation
$$Y^{-1}\hChh Y^{-T} -Y^{-T} -Y^{-1} =-(\hChh+\Gamma)^{-1}.$$
Thus, premultiplying by $Y$ and post multiplying by $Y^\top$, we obtain
$$Y(\hChh+\Gamma)^{-1} Y^\top -Y-Y^\top+\hChh=0$$
which factorizes to give
$$\Bigl(Y(\hChh+\Gamma)^{-1}-I\Bigr)(\hChh+\Gamma)\Bigl(Y(\hChh+\Gamma)^{-1}-I\Bigr)^\top=\Gamma.$$
Thus, taking the symmetric square root, we have
$$\Bigl(Y(\hChh+\Gamma)^{-1}-I\Bigr)(\hChh+\Gamma)^{\frac12}=\Gamma^{\frac12}.$$
Hence
$$\Bigl(Y(\hChh+\Gamma)^{-1}-I\Bigr)=\Gamma^{\frac12}(\hChh+\Gamma)^{-\frac12}.$$
Rearranging gives
$$Y(\hChh+\Gamma)^{-1}=\Gamma^{\frac12}(\hChh+\Gamma)^{-\frac12}+(\hChh+\Gamma)^{\frac12}(\hChh+\Gamma)^{-\frac12}.$$
Thus
$$Y=\Bigl(\Gamma^{\frac12}+(\hChh+\Gamma)^{\frac12}\Bigr)(\hChh+\Gamma)^{\frac12}.$$
We obtain $S=-\tK$ where
$$\tK= \pCov^{vh}
    \Bigl( (\pCov^{hh}+\Gamma) + \Gamma^{1/2} (\pCov^{hh}+\Gamma)^{1/2}
    \Bigr)^{-1}.$$
The map \eqref{eq:newv2i} becomes  
\begin{align*}
\label{eq:newv2ii}
v&=\hv-\tK(\hh-\bbE \hh)+\hCvh(\hChh+\Gamma)^{-1}(\yd-\bbE \hh)\\
&=\hv-\tK(\hh-\bbE \hh)+K(\yd-\bbE \hh)\\
    &=m+(\hv-\pmean)-\tK(\hh-\bbE \hh).
\end{align*}
$\blacksquare$
\end{example}

%%%%%%%%%%%%%%%%%%%%%%%%%%%%%%%%%%%%%%%%%%%%%%%%%%%%%%%%
%
\subsection{Minimum Variance Approximation}
\label{ssec:TM_MVA}
%
%%%%%%%%%%%%%%%%%%%%%%%%%%%%%%%%%%%%%%%%%%%%%%%%%%%%%%

In the two preceding subsections we have identified an uncountable set of transport maps that
match the second order statistics of true transport. Among all these,
the \emph{Kalman transport} \eqref{eq:sd2nn_add} has a particular appeal because it is
constructed around the \emph{Kalman gain} familiar from filtering in the
linear Gaussian setting. In this subsection we show how the
principle of minimizing the variance within
a class of \emph{linear} estimators of the state, 
given observation, leads to this choice of transport map.  
We believe that the second order transport approach, which we highlight in the
main text, provides a more fundamental viewpoint on the mean field models which
underpin ensemble Kalman methods; however the minimum variance perspective is 
widely adopted in the statistics and geophysics communities (see Subsection
\ref{ssec:BSE}) and hence
has an important place in the subject of ensemble Kalman methods.

Recall that the Kalman transport approach leads back to the 
map (\ref{eq:sd2n}c), motivated by control theoretic considerations, and identifies 
a specific choice of Kalman gain $K_n$, a choice which depends on the law of
the predicted state and data. In the Gaussian case the Kalman transport map
\eqref{eq:sd2nn_add} exactly generates the desired transport, 
a fact which we discuss in Example \ref{ex:mfk}.
The general form of approximate stochastic second order transport
maps which we study using the second order transport approach in the
main text is \eqref{eq:sd2nn-seek}. Our goal here is to motivate a specific choice of
$\tT$ in (\ref{eq:sd2nn-seek}c), in particular to derive (\ref{eq:sd2nn_add}c). 
In this subsection we achieve this by first defining,
and identifying, the \emph{best linear unbiased estimator}, BLUE for short. From this we 
will derive Kalman transport.

\begin{lemma}
\label{l:t1} Assume that $\Gamma \succ 0.$ Let all expectations be computed under  $\law(\hv_{n+1},\hy_{n+1})$ and
consider $\mb_{n+1}$ in the form
\begin{equation*}\label{eq:BLUE}
\mb_{n+1}=a'+B\hy_{n+1}.
\end{equation*}
Define $\hC_{n+1}, \hCvy_{n+1}$ and $\hCyy_{n+1}$ 
by \eqref{eq:KF_pred_mean}, \eqref{eq:KF_pred_mean2} and define
$$\sI(a',B) := \E|\hv_{n+1}-\mb_{n+1}|^2.$$
Then $\mb_{n+1}$ is said to be the BLUE of $\hv_{n+1}$ given $\hy_{n+1}$,
if vector $a'$ and matrix $B$ are independent of $(\hv_{n+1},\hy_{n+1})$, but
may depend on $\law(\hv_{n+1},\hy_{n+1})$, and are chosen to minimize
$\sI(a',B).$ Then
\begin{equation}
\label{eq:mean0}
\mb_{n+1}=\E\hv_{n+1}+\hCvy_{n+1}(\hCyy_{n+1})^{-1}(\hy_{n+1}-\E \hy_{n+1}).
\end{equation}
Furthermore, the estimator (\ref{eq:mean0}) is unbiased, that is,
$$
\mathbb{E}(\hv_{n+1}-\mb_{n+1}) = 0,
$$
and its covariance with respect to $\law(\hv_{n+1},\hy_{n+1})$,
$$\Cblue_{n+1}:=\E\bigl((\hv_{n+1}-\mb_{n+1})(\hv_{n+1}-\mb_{n+1})^\top\bigr),$$
is given by
\begin{align}
\label{eq:cov0}
\Cblue_{n+1} &= \pCov_{n+1} - \pCov_{n+1}^{vy}(\pCov_{n+1}^{yy})^{-1} ({\pCov_{n+1}^{vy}})^\top.
\end{align}
$\Diamond$ \end{lemma}

\begin{proof}
Noting that we may, without
loss of generality, reparametrize the proposed form
of $\mb_{n+1}$ as
$$\mb_{n+1}=\E\hv_{n+1}+a+B(\hy_{n+1}-\E \hy_{n+1}),$$
we see that the desired objective to be minimized
over vector-matrix pair $(a,B)$ is
\begin{align*}
\sJ(a,B)&:=\frac12\E|\hv_{n+1}-\E\hv_{n+1}-a-B(\hy_{n+1}-\E \hy_{n+1})|^2\\
&=\frac12\E|\hv_{n+1}-\E\hv_{n+1}|^2+\frac12 |a|^2\\
&\quad\quad\quad+\frac12\E|B(\hy_{n+1}-\E\hy_{n+1})|^2
-\E\langle \hv_{n+1}-\E\hv_{n+1}, B(\hy_{n+1}-\E\hy_{n+1})\rangle\\
&=\frac12\E|\hv_{n+1}-\E\hv_{n+1}|^2+\frac12 |a|^2
+\frac12(BB^\top):\pCov_{n+1}^{yy}- B:\pCov_{n+1}^{vy}.
\end{align*}
Clearly the minimizer with respect to $a$ is achieved
by setting $a=0$.
Differentiating with respect to $B$, noting that
$\pCov_{n+1}^{yy}$ is symmetric, shows that $B=K_n$
given by \eqref{eq:Kalman_gain}. That the resulting pair is indeed a minimizer, and not another
critical point, follows from the fact that $\pCov_{n+1}^{yy}$ is positive-definite; this is implied by the assumption $\Gamma \succ 0.$ Thus we obtain the  estimator \eqref{eq:mean0}.
A straightforward calculation shows that the estimator is unbiased and that its covariance is given by \eqref{eq:cov0}.
\end{proof}

We now connect the BLUE to an approximate (second order) transport map.
To this end, note that \eqref{eq:mean0} may be viewed as mapping
$\hy_{n+1}$ into $\mb_{n+1}=\mb(\hy_{n+1}).$
With this notation we define $m_{n+1}^\dagger$ by
\begin{subequations}
\label{eq:dda1}
\begin{align}
m_{n+1}^\dagger&=\mb(\yd_{n+1})\\ &=\E\hv_{n+1}+\hCvy_{n+1}(\hCyy_{n+1})^{-1}(\yd_{n+1}-\E \hy_{n+1}).
\end{align}
\end{subequations}
We may now make the following connection between BLUE and Kalman transport:
\begin{theorem}
Let $(\hv_{n+1},\hy_{n+1})$ be distributed according to measure $\nu_{n+1}$. Then
the transport map (\ref{eq:sd2nn_add}c) has the properties:
\begin{align*}
   \E v_{n+1}&=m_{n+1}^\dagger\\ 
    \E\bigl((v_{n+1}-m_{n+1}^\dagger)(v_{n+1}-m_{n+1}^\dagger)^\top\bigr)&=\Cblue_{n+1},
\end{align*}
where $m_{n+1}^\dagger$ and $\Cblue_{n+1}$ are given by \eqref{eq:dda1} and \eqref{eq:cov0} respectively.
$\Diamond$ \end{theorem}

\begin{proof}
Note that the transport map (\ref{eq:sd2nn_add}c) can now be reformulated as
\begin{equation*} \label{eq:EnKF_BLUE}
    v_{n+1} = \hv_{n+1} + (m_{n+1}^\dagger-\mb_{n+1}).
\end{equation*}
Thus $$v_{n+1}=m_{n+1}^\dagger + \hv_{n+1} - \mb_{n+1} .$$
The desired properties of the mean and covariance follow from
Lemma \ref{l:t1}.
\end{proof}

\begin{remark}

We now have two derivations of the Kalman gain, the approximate
transport derivation from Subsection \ref{sssec:ats}, and the minimum variance derivation given here. We include a third, dimensional, argument that motivates its form. Let {\tt state} denote the physical units associated with the state variable and {\tt data} those associated with the observation. Then the physical units of the gain matrix $K$ should equal {\tt state}/{\tt data}. We note that $\hCyy$ has units of {\tt data} squared whilst the units of $\hCvy$ are the product of {\tt state} and {\tt data}. If we then constrain the gain to be determined by covariance matrices involving the state and the data it is natural to choose it to be formed as $\hCvy (\hCyy)^{-1}$ where $\hCyy$ is an estimate of covariance in the data, and $\hCvy$ an estimate of covariance between state and data. Making a choice of this type leads to the right units for the gain, since the innovation has units {\tt data}, and relies only on use of first and second order statistics. This dimensional argument motivates the form \eqref{eq:Kalman_gain} as derived in both Subsections \ref{ssec:TM_MVA} and \ref{ssec:GPFD}. $\blacksquare$
\end{remark}

%%%%%%%%%%%%%%%%%%%%%%%%%%%%%%%%%%%%%%%%%%%%%%%%%%%%
%
\section{Appendix D (Stochastic Calculus Considerations)}
\label{sec:AC}
%
%%%%%%%%%%%%%%%%%%%%%%%%%%%%%%%%%%%%%%%%%%%%%%%%%%%

In Subsection \ref{subsec:AuxiliaryLemmas} we demonstrate how to change between
It\^o and Stratonovich integration in the Kushner-Stratonovich Equation.
Subsection \ref{subsec:AK} contains statement and proof of a lemma needed
in proof of Theorem \ref{thm:MF_KS_equal}. In Subsection \ref{subsec:diamond}
we study the diamond form of stochastic integral discussed in Remark \ref{rem:diamond}.

%%%%%%%%%%%%%%%%%%%%%%%%%%%%%%%%%%%%%%%%%%%%%%%%%%%%%
%
\subsection{Derivation of the Kushner-Stratonovich Equation}
\label{subsec:AuxiliaryLemmas}
%
%%%%%%%%%%%%%%%%%%%%%%%%%%%%%%%%%%%%%%%%%%%%%%%%%%%%%%

In Section \ref{sec:CT}, where we  derived continuous time limits from discrete models,
we used the two small parameters $\delta$ and $\dt$. The first characterized the data frequency and the second a time-increment. In this appendix we make the
choice $\delta = \dt$ and study the limit $\Delta t\to 0$ to derive continuum models. Studying this limit enables us to convert between different modes of stochastic 
integration in an explicit fashion; this may be helpful to some readers as it avoids the
need to invoke abstract results on covariation. In particular we now restate, and then prove from
first principles, Lemma \ref{lem:ISC}:

\begin{lemma}
\label{lem:ISadd}
Assume that $\Gammas \succ 0$ and that $\zd$ is given by \eqref{eq:data_ct}. The It\^o and Stratonovich interpretations of the stochastic forcing term in (\ref{eq:KSD_S2}) are related through
\begin{subequations}
\label{eq:addback}
\begin{align}
\dd r &= \left\langle \hs- \E \hs, \circ \dd\zd \right\rangle_\Gammas  r - \frac{1}{2} \left\{ \left| \hs \right|^2_\Gammas - \E \left| \hs \right|^2_\Gammas\right\} r \dd t\\
    &= \left\langle \hs - \E \hs, \dd\zd - \E \hs \dd t\right\rangle_\Gammas r.
\end{align}
\end{subequations}
$\Diamond$ \end{lemma}

\begin{proof}
Consider (\ref{eq:addback}a) and the term 
$r\langle \hs- \E \hs, \circ \dd\zd \rangle_\Gammas$ in
particular. Define the corresponding It\^o and Stratonovich integrals 
$$I:= \int_0^T r\langle \hs- \E \hs,\dd\zd \rangle_\Gammas, \quad S:=\int_0^T r\langle \hs- \E \hs, \circ \dd\zd \rangle_\Gammas.$$ 
We aim to write $S$ as the sum of $I$ and an added correction. Consider the space $$\mathcal{X} := \{{\varrho}\in L^1(\R^{d_v};\R^+): \|\varrho\|_{L^1} =1\}$$ 
of probability density functions. We first choose an increasing sequence $(t_j)_{j=1,\dots,N}$ with $\dt := t_{j+1}-t_j$ such that $N\dt = T$. We next define 
\begin{equation*}
    \Delta r_{t_j} := r_{t_{j+1}} - r_{t_j} \quad \Delta \hs_{t_j} := \hs_{t_{j+1}} - \hs_{t_j} \quad \Delta \zd_{t_j} :=\zd_{t_{j+1}}-\zd_{t_j} .
\end{equation*}
Now recall the driving evolution equation (\ref{eq:data_ct}b) for $\zd$, namely
\begin{equation}
\label{eq:recalled}
    \dd\zd = \hs(\vd)\dd t + \sqrt{\Gammas}\dd B^\dagger.
\end{equation}
From this, recalling the properties of Brownian motion $B^\dagger$, we deduce the discretization 
\begin{equation}
\label{delta_z_proof}
    \Delta\zd_{t_j} = \hs(\vd_{t_j})\Delta t+ \sqrt{\Gammas\Delta t}\xi_{t_j}+\mathcal{O}(\dt^{3/2}),
\end{equation}
where $\xi_{t_j}$ are independent mean-zero normal random variables with variance $I_{d_y}$ for each $j$. The fact that this
discretization is accurate up to terms of $\mathcal{O}(\dt^{3/2})$
follows from It\^o-Taylor expansion \citep{kloeden1991numerical}.
Recalling the definition of the Stratonovich stochastic integral as a limit,
we consider the following finite sum approximation of $S$:
\begin{equation}
   S_{\dt} := \sum_{j=1}^{N} \frac{r_{t_{j+1}}+r_{t_j}}{2}\left\langle \hs - \int \frac{r_{t_{j+1}}+r_{t_j}}{2} \hs \dd v, \Delta\zd_{t_j}\right\rangle_\Gammas;
\label{bigsum}
\end{equation}
this converges in the $L^2_{\mathbb{P}}\left(\Omega; {C}([0,T];\mathcal{X})\right)$ sense to $S$ as $\dt \rightarrow 0$ (and thus as $N\rightarrow \infty$). Expanding now the sum in \eqref{bigsum}, we obtain
\begin{subequations}
\label{expansion}
\begin{align}
    S_{\dt} 
    &= \sum_{j=1}^{N} \left(r_{t_j} + \frac{\Delta r_{t_j}}{2} \right)\left\langle \hs - \int \left(r_{t_j} + \frac{\Delta r_{t_j}}{2} \right)\hs \dd v, \Delta\zd_{t_j}\right\rangle_\Gammas\\
    &= \sum_{j=1}^{N}r_{t_j}\left\langle \hs -\int r_{t_j}\hs \dd v,\Delta \zd_{t_j}\right\rangle_\Gammas +\frac{1}{2}\sum_{j=1}^{N}\Delta r_{t_j}\left\langle \hs - \int r_{t_j}\hs \dd v, \Delta \zd_{t_j} \right\rangle_\Gammas\\
    &-\frac{1}{2}\sum_{j=1}^{N}r_{t_j}\left\langle \int \Delta r_{t_j}\hs \dd v, \Delta \zd_{t_j} \right\rangle_\Gammas - \frac{1}{4}\sum_{j=1}^{N} \Delta r_{t_j}\left\langle \int \Delta r_{t_j}\hs \dd v, \Delta \zd_{t_j} \right\rangle_\Gammas.
\end{align}
\end{subequations}
Now, we note that by  discretizing  (\ref{eq:addback}a) or (\ref{eq:addback}b) we can write
$$\Delta r_{t_j} = r_{t_j} \left\langle \hs - \int r_{t_j}\hs \dd v,\Delta \zd_{t_j}\right\rangle_\Gammas + \mathcal{O}(\dt).$$ 
By substituting this expression in the expanded sum in (\ref{expansion}) we notice that the last term in (\ref{expansion}c) is of order $\mathcal{O}(N\dt^{3/2})$.  We can thus write \eqref{expansion} as   
\begin{equation}\label{shorthand}
    S_{\dt} = I_{\dt} + J_{1,\dt} + J_{2,\dt} + \mathcal{O}(N\dt^{3/2}),
\end{equation}
where we have defined 
\begin{align*}
    I_{\dt}&:=\sum_{j=1}^{N}r_{t_j}\left\langle \hs -\int r_{t_j}\hs \dd v,\Delta \zd_{t_j}\right\rangle_\Gammas,\\ 
    J_{1,\dt} &:= \frac{1}{2}\sum_{j=1}^{N}r_{t_j}\left|\left\langle \hs - \int r_{t_j}\hs \dd v, \Delta \zd_{t_j} \right\rangle_\Gammas\right|^2,\\
    J_{2,\dt}&:= - \frac{1}{2}\sum_{j=1}^{N}r_{t_j}\left\langle\int r_{t_j} \left\langle \hs - \int r_{t_j}\hs \dd v,\Delta \zd_{t_j}\right\rangle_\Gammas \hs \dd v, \Delta \zd_{t_j} \right\rangle_\Gammas.
\end{align*}
%\begin{subequations}
%\begin{align}\label{hugesum}
%    S_{\dt} &= \sum_{j=1}^{N} \frac{r_{t_{j+1}}+r_{t_j}}{2}\langle h - \int \frac{r_{t_{j+1}}+r_{t_j}}{2} hdv, \Delta\zd_{t_j}\rangle_\Gamma\\
%    &=\sum_{j=1}^{N}r_{t_j}\langle h -\int r_{t_j}hdv,\Delta \zd_{t_j}\rangle_\Gamma +\frac{1}{2}\sum_{j=1}^{N}r_{t_j}\langle h - \int r_{t_j}hdv, \Delta \zd_{t_j} \rangle_\Gamma^2\\
%    &+\sum_{j=1}^{N}r_{t_j}\langle h - \frac{1}{2}\int r_{t_j} \langle h_{t_j} - \int r_{t_j}hdv,\Delta \zd_{t_j}\rangle_\Gamma hdv, \Delta \zd_{t_j} \rangle_\Gamma\\
%    &+ \frac{1}{2}\sum_{j=1}^{N} r_{t_j} \langle h - \int r_{t_j}hdv\Delta \zd_{t_j}\rangle_\Gamma\langle h - \frac{1}{2}\int r_{t_j} \langle h - \int r_{t_j}hdv,\Delta \zd_{t_j}\rangle_\Gamma hdv, \Delta \zd_{t_j} \rangle_\Gamma.
%\end{align}
%\end{subequations
The first term in the sum \eqref{shorthand}, $I_{\dt}$, converges in the $L^2_{\mathbb{P}}$ sense to $I$ as $\dt \rightarrow 0$. Considering now the term $J_{1,\dt}$ and using the discretization \eqref{delta_z_proof} we have that
\begin{align*}
J_{1,\dt}& = \frac{1}{2}\sum_{j=1}^{N}r_{t_j}\left|\left\langle \hs - \int r_{t_j}\hs \dd v, \Delta \zd_{t_j} \right\rangle_\Gammas\right|^2\\
&= \frac{1}{2}\sum_{j=1}^{N}r_{t_j}\dt\left( \hs - \int r_{t_j}\hs \dd v\right)^\top\Gammas^{-1/2}\xi_{t_j}\xi_{t_j}^\top\Gammas^{-1/2}\left( \hs - \int r_{t_j}\hs \dd v\right)\\&\qquad + \sum_{j=1}^{N} \mathcal{O}(\dt^{3/2}).
\end{align*}
Since $\sum_{j=1}^N \mathcal{O}(\dt^{3/2}) =\mathcal{O}(\dt^{1/2})$, taking the $\dt \rightarrow 0$ limit and using the independence of the $\xi_{t_j}$ for each $j$, we see that
\begin{equation*}
    J_{1,\dt} \rightarrow \; \frac{1}{2}\int_0^T\left[r\left( \hs - \mathbb{E}\hs\right)^\top\Gammas^{-1/2}\mathbb{E}\xi\xi^\top\Gammas^{-1/2}\left( \hs - \mathbb{E}\hs\right) \right]\dd t= \int_0^\top\frac{r}{2}\left|\hs-\mathbb{E}\hs\right|^2_\Gammas \dd t,
\end{equation*}
in the $L^2_{\mathbb{P}}$ sense. Similarly for $J_{2,\dt}$ we have that,
\begin{align*}
J_{2,\dt}&= - \frac{1}{2}\sum_{j=1}^{N}r_{t_j}\left\langle \int r_{t_j} \left\langle \hs - \int r_{t_j}\hs \dd v,\Delta \zd_{t_j}\right\rangle_\Gammas \hs \dd v, \Delta \zd_{t_j} \right\rangle_\Gammas\\
&=- \frac{1}{2}\sum_{j=1}^{N}r_{t_j}\int r_{t_j}\left\langle  \left\langle \hs - \int r_{t_j}\hs \dd v,\Delta \zd_{t_j}\right\rangle_\Gammas \hs, \Delta \zd_{t_j} \right\rangle_\Gammas \dd v\\
&=- \frac{1}{2}\sum_{j=1}^{N}r_{t_j}\dt \int r_{t_j} \hs^\top\Gammas^{-1/2}\xi_{t_j}\xi_{t_j}^\top\Gammas^{-1/2}\left( \hs - \int r_{t_j}\hs \dd v\right)\dd v \\ &\qquad+ \sum_{j=1}^N \mathcal{O}(\dt^{3/2}).
\end{align*}
Since $\sum_{j=1}^N \mathcal{O}(\dt^{3/2}) =\mathcal{O}(\dt^{1/2})$, taking the $\dt \rightarrow 0$ limit and using the independence of the $\xi_{t_j}$ for each $j$, we see that, in the $L^2_{\mathbb{P}}$ sense,
\begin{align*}
    J_{2,\dt} \rightarrow &\; -\frac{1}{2}\int_0^Tr\int r \hs^\top\Gammas^{-1/2}\mathbb{E}\left(\xi\xi^\top\right)\Gammas^{-1/2}\left( \hs - \mathbb{E}\hs\right) \dd v\dd t \\&= - \int_0^T\frac{r}{2}\mathbb{E}\left|\hs-\mathbb{E}\hs\right|^2_\Gammas \dd t.
\end{align*}
Finally, by taking the $\dt \rightarrow 0$ limit on both sides of \eqref{shorthand}, we conclude
\begin{equation*}
    S = I + \int_0^T\frac{r}{2}\left|\hs-\mathbb{E}\hs\right|^2_\Gammas \dd t - \int_0^T\frac{r}{2}\mathbb{E}\left|\hs-\mathbb{E}\hs\right|^2_\Gammas \dd t,
\end{equation*}
in the $L^2_{\mathbb{P}}$ sense. Hence by using the above Stratonovich-to-\^Ito correction in (\ref{eq:addback}a), we obtain
\begin{align*}
    \dd r &= r\left\langle \hs- \E \hs, \circ \dd\zd \right\rangle_\Gammas  - \frac{r}{2} \left\{ \left| \hs \right|^2_\Gammas - \E \left| \hs \right|^2_\Gammas\right\} \dd t\\
    &=r\left\langle \hs- \E \hs, \dd\zd \right\rangle_\Gammas + \frac{r}{2}\left|\hs-\mathbb{E}\hs\right|^2_\Gammas \dd t - \frac{r}{2}\mathbb{E}\left|\hs-\mathbb{E}\hs\right|^2_\Gammas \dd t - \frac{r}{2} \left\{ \left| \hs \right|^2_\Gammas - \E \left| \hs \right|^2_\Gammas\right\} \dd t\\
    &=r \left\langle \hs - \E \hs, \dd\zd - \E \hs \dd t\right\rangle_\Gammas.
\end{align*}
\end{proof}

%%%%%%%%%%%%%%%%%%%%%%%%%%%%%%%%%%%%%%%%%%%%%%%%%%%%%
%
\subsection{Lemma for Proof of Theorem \ref{thm:MF_KS_equal}}
\label{subsec:AK}
%
%%%%%%%%%%%%%%%%%%%%%%%%%%%%%%%%%%%%%%%%%%%%%%%%%%%%%%

We now  establish the following lemma, the conclusions of which are used in proof of 
Theorem \ref{thm:MF_KS_equal}.

\begin{lemma}
\label{lm::FP_analysis}
The probability density for the solution $v$ of  \eqref{eq:KS_mfn},
with respect to randomness induced by the law of $\hz$ and
the initial condition $v(0)$, with $\zd$ a fixed data sample path,
is given by \eqref{eq:FPE_mf}.
$\Diamond$ \end{lemma}

\begin{proof}
Recall that in equation \eqref{eq:KS_mfn} the
evolution of $\zd$ is given by \eqref{eq:recalled}, 
with $\Bd$ and $B$ independent draws from unit
Brownian motion in $\R^{d_y}.$ Thus
\begin{align*} \label{eq:KS_mfna}
    \dd v &= a(v;\mmu)\dd t + K(v;\mmu)\Bigl(\bigl(\hs(\vd)-\hs(v)\bigr)\dd t+
    \sqrt{\Gammas}\dd\Bd-\sqrt{\Gammas}\dd B\Bigr).
\end{align*}
We wish to find the probability density function $\rho(v,t)$
for $v$, with respect to randomness induced by the law of $B$ and
the initial condition $v(0)$, but with $\vd,\Bd$ fixed signal
and observational noise sample paths. Let $\phi:\R^{d_v} \to \R$
be smooth and note that the It\^o formula shows that
$$\dd\phi\bigl(v(t)\bigr)=\bigl(\cL \phi\bigr)\bigl(v(t)\bigr)\dd t
+\big \langle K\bigl(v(t);\rho(\cdot,t)\bigr)\sqrt{\Gammas}(\dd\Bd-\dd B),\nabla \phi \bigl(v(t)\bigr)\big\rangle,$$
where\footnote{Here we denote by $\cL$ the infinitesimal generator of the diffusion process.}
$$\cL \psi(v)=\Big \langle\Bigl(a(v;\rho)+K(v;\rho)\bigl(\hs(\vd)-\hs(v)\bigr)\Bigr),\nabla \psi\Big\rangle+K(v;\rho)\Gammas K(v;\rho)^\top: \nabla \nabla \psi,$$
and the integrals are to be interpreted in the It\^o sense.
Note the factor $1$ in the second order term, arising because
of the independent quadratic variation contributions 
from both $B$ and $\Bd.$
Taking expectation $\E$ with respect to $B$ and initial condition
$v(0)$, with $\Bd$ fixed, for $a$ and $K$ functions of $v$ and $\rho$, yields,
using again that $\zd$ is given by \eqref{eq:data_ct},
\begin{align*}
\dd\E\phi(v) 
&=\E\Big\langle\Bigl(a+K\bigl(\hs(\vd)-\hs\bigr)\Bigr),\nabla \phi\Bigr\rangle \dd t+\E\langle K\sqrt{\Gammas}\dd\Bd, \nabla \phi \rangle \rangle+\E K\Gamma K^\top: \nabla \nabla \phi \dd t\\
&=\E\big\langle\bigl(a-K\hs\bigr),\nabla \phi\bigr\rangle \dd t+\E\langle K\dd\zd, \nabla \phi \rangle \rangle+\E K\Gammas K^\top: \nabla \nabla \phi \dd t.
\end{align*}
Noting that the expectation operation corresponds to
multiplication by $\rho(v,t)$ and integration over 
$v \in \R^{d_v},$
integrating by parts shows that $\rho$ satisfies
the desired equation \eqref{eq:FPE_mf}, in a weak sense.
\end{proof}

%%%%%%%%%%%%%%%%%%%%%%%%%%%%%%%%%%%%%%%%%%%
%
\subsection{Diamond Integration}
\label{subsec:diamond}
%
%%%%%%%%%%%%%%%%%%%%%%%%%%%%%%%%%%%%%%%%%

In this section we use notation analogous to that established in
the proof of Lemma \ref{lem:ISadd}.
Consider the equation \eqref{eq:KS_mf3}, repeated here for convenience:
\begin{subequations}
\label{eq:KS_mf3R}
\begin{empheq}[box=\fbox]{align}
\dd v & = f(v)\dd t + \sqrt{\Sigmas} \dd W +\nabla \cdot (K\Gammas K^\top)\dd t - K\Gammas \nabla \cdot K^\top \dd t + K (\dd\zd - \dd\hz),\\
\dd\hz &= \hs(v)\dd t + \sqrt{\Gammas} \dd B.
\end{empheq}
\end{subequations}
Our focus here is on the contribution $K (\dd\zd - \dd\hz)$,
recalling that $\zd$ is governed by \eqref{eq:recalled}.
We first choose an increasing sequence $(t_j)_{j=0,\dots,N-1}$ with $\dt := t_{j+1}-t_j$ such that $N\dt = T$. 
The It\^o integral interpretation of this contribution, on time interval $(0,T)$,
is as the $L^2_{\mathbb{P}}$ limit of
\begin{equation*}
\label{eq:ito}
     I_{\dt}:=\sum_{j=0}^{N-1}K(v_{t_j}; \rho_{t_j})
     (\Delta \zd_{t_j}-\Delta \hz_{t_j}).
\end{equation*}
We define the Stratonovich integral interpretation as 
the $L^2_{\mathbb{P}}$ limit of
\begin{equation*}
\label{eq:strat}
    S_{\dt}:= \sum_{j=0}^{N-1}K\left(\frac{v_{t_{j+1}}+v_{t_j}}{2}; \frac{\rho_{t_{j+1}}+\rho_{t_j}}{2}\right)
     (\Delta \zd_{t_j}-\Delta \hz_{t_j});
\end{equation*}
as is standard, we use $\circ$ to denote
Stratonovich stochastic integration. Finally we define the 
diamond integral interpretation as the $L^2_{\mathbb{P}}$ limit of
\begin{equation*}
\label{eq:diamond}
    R_{\dt}:= \sum_{j=0}^{N-1}K\left(\frac{v_{t_{j+1}}+v_{t_j}}{2}; \rho_{t_j}\right)
     (\Delta \zd_{t_j}-\Delta \hz_{t_j}).
\end{equation*}
Note that this is akin to a Stratonovich integral, but only with respect to
variation of $K$ with respect to $v$, not $\rho.$
We use $\diamond$ to denote
this unusual form of stochastic integration.

Throughout this subsection we inter-convert between these three forms of stochastic integration. To shorten the presentation we will sometimes
use the $=$ symbol when in fact we mean equality up to an additive constant
which disappears in the $\Delta t \to 0$ limit.
Note that the evolution equation  \eqref{eq:KS_mf3R} for $v$ is driven by 
$\zd$ and $\hz$, whereas the Kushner-Stratonovich equation \eqref{eq:KSE} 
for $\rho$ driven only by $\zd.$ This difference will have implications
for the calculations that follow, in which we compute inter-conversions
between the different stochastic integrals.
The first result highlights an interesting interpretation of the contribution 
$$a=\nabla \cdot (K\Gammas K^\top) - K\Gammas \nabla \cdot K^\top$$
to the drift in \eqref{eq:KS_mf3R}, namely that it is
simply the It\^o-to-diamond correction, giving a compact re-interpretation
of the mean field model:

\begin{lemma}
Let $\rho$ solve the Kushner-Stratonovich equation \eqref{eq:KSE} and let $\zd$
be given by \eqref{eq:data_ct}.
The It\^o interpretation of equation \eqref{eq:KS_mf3} and its interpretation with respect to the $\diamond$ form of stochastic integration are related through the following equivalence. The system
\begin{align*}
%\label{eq:KS_mf3_A}
\dd v & = f(v)\dd t + \sqrt{\Sigmas} \dd W +\nabla \cdot (K\Gammas K^\top)\dd t - K\Gammas \nabla \cdot K^\top \dd t + K (\dd\zd - \dd\hz),\\
\dd\hz &= \hs(v)\dd t + \sqrt{\Gammas} \dd B,
\end{align*}
where $K=K(v,\mmu)$, is equivalent to
\begin{align*} %\label{eq:diamond-conversion}
    \dd v &=f(v)\dd t + \sqrt{\Sigmas} \dd W + K \diamond (\dd\zd - \dd\hz),\\
    \dd\hz &= \hs(v)\dd t + \sqrt{\Gammas} \dd B.
\end{align*}
$\Diamond$ \end{lemma}

\begin{proof}
We first recall the evolution equations for $\zd$ and $\hz$:
\begin{align*}
\label{eq:zzs}
    \dd\zd &= \hs(\vd)\dd t + \sqrt{\Gammas} \dd\Bd,\\
    \dd\hz &= \hs(v)\dd t + \sqrt{\Gammas} \dd B.
\end{align*}
Noting that $v$ is driven by $(\zd-\hz)$ and that we are inter-converting
between It\^o and diamond integration, so that the noisy driving of the
$\rho$ equation does not play a role, we see that it suffices to
consider the  $L^2_{\mathbb{P}}$ limits of the two quantities
\begin{align*}
     I_{\dt}&=\sum_{j=0}^{N-1}K(v_{t_j}; \rho_{t_j})\sqrt{2\Gammas\dt}\xi_{t_j},\\
R_{\dt}&= \sum_{j=0}^{N-1}K\left(\frac{v_{t_{j+1}}+v_{t_j}}{2}; \rho_{t_j}\right)\sqrt{2\Gammas\dt}\xi_{t_j},
\end{align*}
where the $\xi_{t_j}$ are i.i.d. draws from a unit Gaussian.
By adding and subtracting the It\^o contribution, we obtain
\begin{align*}
    &R_{\dt}= \sum_{j=0}^{N-1}K\left(\frac{v_{t_{j+1}}+v_{t_j}}{2}; \rho_{t_j}\right)\sqrt{2\Gammas\dt}\xi_{t_j}\\
    &= \sum_{j=0}^{N-1}K(v_{t_j}; \rho_{t_j})\sqrt{2\Gammas\dt}\xi_{t_j} + \sum_{j=0}^{N-1} \left(K\left(\frac{v_{t_{j+1}}+v_{t_j}}{2}; \rho_{t_j}\right) - K(v_{t_j}; \rho_{t_j}) \right)\sqrt{2\Gammas\dt}\xi_{t_j}\\
    &= I_{\dt} + \sum_{j=0}^{N-1} \left( \frac{1}{2}D_v K(v_{t_j};\mmu_{t_j})(v_{t_{j+1}}-v_{t_j})+\mathcal{O}\left(|v_{t_{j+1}}-v_{t_j}|^2\right)\right)\sqrt{2\Gammas\dt}\xi_{t_j}, 
\end{align*}
where the last line follows from a first order Taylor expansion. Using a discretization of the evolution for $v$, and neglecting the terms that do not contribute to the quadratic variation when computing the  $L^2_{\mathbb{P}}$ limit, we substitute for $v_{t_{j+1}}-v_{t_j}$ to obtain
\begin{equation}
\label{eq:usefulforquadvariation}
    I_{\dt} + \sum_{j=0}^{N-1}\frac{1}{2}D_v K(v_{t_j};\mmu_{t_j})\left(K(v_{t_j}; \rho_{t_j})\sqrt{2\Gammas\dt}\xi_{t_j}\right)\sqrt{2\Gammas\dt}\xi_{t_j}.
\end{equation}
We now consider the correction term resulting from \eqref{eq:usefulforquadvariation}, which is determined by expectation of
the summand with respect to the random increments $\xi$, scaled
by $\Delta t^{-1}.$ Dropping the $v$ and $\mmu$ dependence of $K$ and the $t_j$ dependence for notational convenience, its $k$'th component is given by\footnote{Using Einstein summation convention, as described for example in
\citet{gonzalez2008first}, so that index repeated twice is summed over.}
\begin{align*}
    \left[\mathbb{E}\left(D_{v}K\right) \left(K\Gammas^{1/2}\xi\right)\Gammas^{1/2}\xi\right]_k&=\left[\mathbb{E} \left(\partial_{v_i}K\right)\left[K\Gammas^{1/2}\xi\right]_i\Gammas^{1/2}\xi\right]_k\\
    &=\left[\mathbb{E} \left(\partial_{v_i}K\right)\left(K_{il}(\Gammas^{1/2})_{lj}\xi_j\right)\Gammas^{1/2}\xi\right]_k\\
    &=\mathbb{E}\left(K_{il}(\Gammas^{1/2})_{lj}\xi_j\right) \left(\partial_{v_i}K\right)_{kn}(\Gammas^{1/2})_{nm}\xi_m\;\;\;\;\\
    &=\left(\partial_{v_i}K_{kn}\right)\Gammas_{ln}K_{il}.
\end{align*}
But,
\begin{align*}
    \left(\partial_{v_i}K_{kn}\right)\Gammas_{ln}K_{il} &=\partial_{v_i}\left(K_{kn}\Gammas_{ln}K_{il}\right) - K_{kn}\Gammas_{ln}\left(\partial_{v_i}K_{il}\right)\\
    &= \left[\nabla\cdot\left(K\Gammas K^\top \right) \right]_k - \left[K\Gammas\nabla\cdot\left(K^\top \right) \right]_k,
\end{align*}
Recalling that the only contributions in $R_{\dt}$ which do not vanish under the $\dt \rightarrow 0$ limit are the ones given by \eqref{eq:usefulforquadvariation}, taking the $\dt\rightarrow 0$ limit of $R_{\dt}$ yields
\begin{equation*}
\label{eq:afterlimit}
    R= I + \int_0^T\left( \nabla\cdot\left(K\Gammas K^\top \right) - K\Gammas\nabla\cdot\left(K^\top \right)\right)\dd t,
\end{equation*}
where the convergence is in the $L^2_{\mathbb{P}}\left(\Omega; {C}([0,T];\R^{d_v})\right)$ sense; this concludes the proof.
\end{proof}

The preceding lemma relates diamond integration to It\^o integration; the following lemma
relates it to Stratonovich integration.
\begin{lemma}
Let $\rho$ solve the Kushner-Stratonovich equation \eqref{eq:KSE} and let $\zd$
be given by \eqref{eq:data_ct}.
The interpretation with respect to the $\diamond$ form of stochastic integration of equation \eqref{eq:KS_mf3} and its Stratonovich interpretation are related through the following equivalence. The system
\begin{align*}
%\label{eq:KS_mf3_A1}
\dd v &=f(v)\dd t + \sqrt{\Sigmas} \dd W + K \diamond (\dd\zd - \dd\hz),\\
    \dd\hz &= \hs(v)\dd t + \sqrt{\Gammas} \dd B,
\end{align*}
where $K=K(v,\mmu)$, is equivalent to
\begin{align*} %\label{eq:diamond-conversion1}
    \dd v &=f(v)\dd t + b\dd t + \sqrt{\Sigmas} \dd W + K \circ (\dd\zd - \dd\hz),\\
    \dd\hz &= \hs(v)\dd t + \sqrt{\Gammas} \dd B,
\end{align*}
where the term $b=b(v,\mmu)$ satisfies
\begin{equation*}
    b(v,\mmu) = \frac{\mmu}{2}D_\rho K(\hs -\mathbb{E}\hs).
\end{equation*}
$\Diamond$ \end{lemma}
\begin{proof}
To make the inter-conversion between diamond and Stratonovich integration
we need only consider the quadratic variation contribution induced by
$\zd$ since the equation for $\rho$ is driven only by $\zd$ and not by $\hz.$
Recalling \eqref{eq:recalled} and its discrete form \eqref{delta_z_proof} we see that
it suffices to compare the $L^2_{\mathbb{P}}$ limits of
\begin{align*}
     R_{\dt}&= \sum_{j=0}^{N-1}K\left(\frac{v_{t_{j+1}}+v_{t_j}}{2}; \rho_{t_j}\right)\sqrt{\Gammas\dt}\xi^\dag_{t_j},\\
S_{\dt}&= \sum_{j=0}^{N-1}K\left(\frac{v_{t_{j+1}}+v_{t_j}}{2}; \frac{\mmu_{t_{j+1}}+\mmu_{t_j}}{2}\right)\sqrt{\Gammas\dt}\xi^\dag_{t_j};
\end{align*}
Here $\xi^\dag_{t_j}$ are i.i.d.~draws from a unit Gaussian.
By adding and subtracting the diamond contribution, we obtain
\begin{align*}
    S_{\dt}&= \sum_{j=0}^{N-1}K\left(\frac{v_{t_{j+1}}+v_{t_j}}{2}; \frac{\mmu_{t_{j+1}}+\mmu_{t_j}}{2}\right)\sqrt{\Gammas\dt}\xi^\dag_{t_j}\\
    &= R_{\dt} + \sum_{j=0}^{N-1} \left(K\left(\frac{v_{t_{j+1}}+v_{t_j}}{2}; \frac{\mmu_{t_{j+1}}+\mmu_{t_j}}{2}\right) - K\left(\frac{v_{t_{j+1}}+v_{t_j}}{2}; \rho_{t_j}\right) \right)\sqrt{\Gammas\dt}\xi^\dag_{t_j}.
    \end{align*}
Thus
\begin{align*}
    &S_{\dt}-R_{\dt}=\\
    &\sum_{j=0}^{N-1} \left( \frac{1}{2}D_\mmu K\left(\frac{v_{t_{j+1}}+v_{t_j}}{2};\mmu_{t_j}\right)(\mmu_{t_{j+1}}-\mmu_{t_j})+\mathcal{O}\left(|\mmu_{t_{j+1}}-\mmu_{t_j}|^2\right)\right)\sqrt{\Gammas\dt}\xi^\dag_{t_j}, 
\end{align*}
where the last line follows from a first order Taylor expansion in $\rho$. In the following
we will use a discretization of the Kushner-Stratonovich equation \eqref{eq:KSE} given in Theorem \ref{thm:KS_equation},  and repeated here:
\begin{equation}\label{eq:KSER}
\dd\mmu = -\nabla \cdot (\mmu f) \dd t + \frac{1}{2}
\nabla \cdot (\nabla \cdot (\mmu \Sigmas))  \dd t +  \left\langle \hs - \E \hs, \dd\zd - \E \hs \dd t\right\rangle_\Gammas \mmu.
\end{equation}
Substituting an increment for $\rho$, in time, and summarizing the 
terms that do not contribute to the quadratic variation in 
the $L^2_{\mathbb{P}}$ limit as $\mathcal{O}(\dt)$, we obtain
\begin{align*}
    &S_{\dt} -R_{\dt}=\\
    &\sum_{j=0}^{N-1}\left(\frac{1}{2}D_\mmu K\left(\frac{v_{t_{j+1}}+v_{t_j}}{2};\mmu_{t_j}\right)\langle \mmu_{t_j}\Gammas^{-1}(\hs_{t_j}-\mathbb{E}\hs_{t_j}),\sqrt{\Gammas\dt}\xi^\dag_{t_j} \rangle+\mathcal{O}(\dt)\right)\sqrt{\Gammas\dt}\xi^\dag_{t_j}.
\end{align*}
We now perform a first order Taylor expansion in $v$ of $D_\mmu K$. Disregarding $\mathcal{O}(\dt^{3/2})$ terms which will vanish in the $\dt \rightarrow 0$ limit, we obtain
\begin{equation}
\label{eq:Sdt_sum}
    S_{\dt} = R_{\dt} + \sum_{j=0}^{N-1}\frac{1}{2}D_\mmu K(v_{t_j};\mmu_{t_j})\langle \mmu_{t_j}\Gammas^{-1}(\hs_{t_j}-\mathbb{E}\hs_{t_j}),\sqrt{\Gammas\dt}\xi^\dag_{t_j} \rangle\sqrt{\Gammas\dt}\xi^\dag_{t_j}.
\end{equation}
We now consider the correction term resulting from \eqref{eq:Sdt_sum}, which is determined by expectation of the summand with respect to the random increments $\xi^\dag$, scaled by $\Delta t^{-1}.$ Indeed, dropping the $v$ and $\mmu$ dependence of $K$, for each $j=0,\dots,N-1$, its $l$'th component is given by
\footnote{Again using Einstein summation convention, but not with respect to index $j$ which simply denotes a fixed time.}
\begin{align*}
    \mathbb{E}\left[\frac{1}{2}D_\mmu\right. & K(v_{t_j};\mmu_{t_j})\langle \mmu_{t_j}\Gammas^{-1}(\hs_{t_j}-\mathbb{E}\hs_{t_j}),\sqrt{\Gammas}\xi^\dag_{t_j} \rangle\sqrt{\Gammas}\xi^\dag_{t_j}\bigg]_l =\\
    &= \frac{1}{2}\mmu_{t_j}  \mathbb{E}\left( \left[\Gammas^{-1/2}(\hs_{t_j}-\mathbb{E}\hs_{t_j}) \right]_k\left[\xi^\dag_{t_j}\right]_k \left[D_\mmu K\right]_{lm}\left[\Gammas^{1/2}\right]_{mn}\left[\xi^\dag_{t_j}\right]_n\right)\\
    &= \frac{1}{2}\mmu_{t_j} \left[D_\mmu K\right]_{lm}[\hs_{t_j}-\mathbb{E}\hs_{t_j}]_m.
\end{align*}
Thus taking the $\dt\rightarrow 0$ limit of $S_{\dt}$ yields
\begin{equation*}
\label{eq:afterlimit_second}
    S= R + \int_0^T \frac{1}{2}\mmu D_\mmu K(v;\mmu)(\hs - \mathbb{E}\hs)\dd t,
\end{equation*}
where the convergence is in the $L^2_{\mathbb{P}}\left(\Omega; {C}([0,T];\R^{d_v})\right)$ sense; this concludes the proof.
\end{proof}
%%%%%%%%%%%%%%%
%%%%%%%%%%%%%%%
%%%%%%%%%%%%%%%

%%%%%%%%%%%%%%%%%%%%%%%%%%%%%%%%%%%%%%%%%%%%%%%%%%%%
%
\section{Appendix E (Flows In The Gaussian Manifold)}
\label{sec:AGF}
%
%%%%%%%%%%%%%%%%%%%%%%%%%%%%%%%%%%%%%%%%%%%%%%%%%%%

Recall that we obtained \eqref{eq:meanGPFcontIP} as the time continuous limit of its discrete-time formulation \eqref{eq:KF_analysis_add_IP}. Here we demonstrate that the same evolution equations can be derived from the gradient flow \eqref{eq:betterform} through a sequence of approximations. 
To this end we first define, for any function $g$ defined on state space $\R^{d_u}$,
$$C^{g \Phi}=\mathbb{E}\bigl(g(u) \Phi(u)\bigr)-\mathbb{E}\bigl(g(u)\bigr) \mathbb{E}\bigl(\Phi(u)\bigr).$$
Now consider random variable $u$ with probability
density function $\rho$ evolving according to \eqref{eq:betterform}. Then
$$
\frac{\dd}{\dd t}\bigl(\mathbb{E}g(u)\bigr) = -C^{g \Phi}.
$$
By making linear and quadratic choices for $g$ this identity leads to the  equations \eqref{eq:closure_second_order} for the mean $m$ and the covariance matrix $C$ under $\rho$, repeated here for convenience:
\begin{align*}
    \frac{\dd m}{\dd t} &= - \mathbb{E}\bigl(\Phi(u)(u-m)\bigr),\\
    \frac{\dd C}{\dd t} &= -\mathbb{E}\bigl(\Phi(u)(u-m)(u-m)^\top\bigr)+C\mathbb{E}\bigl(\Phi(u)\bigr).
\end{align*}
These equations do not define, in general, a closed evolution for the pair $(m,C)$ because, in general,
$\rho$ is not Gaussian. In order to find closed evolution equations, we take the expectation
on the right-hand side not with respect to $\rho$ but with respect to the Gaussian 
$\Ng(m(t),C(t))$. The resulting closed evolution equation for the pair 
$\bigl(m(t),C(t)\bigr)$ is different from the continuous time Gaussian projected filter \eqref{eq:meanGPFcontIP}.

We now show that the two closed evolution equations \eqref{eq:closure_second_order} and \eqref{eq:meanGPFcontIP} can be connected through a sequence of steps involving an integration
by parts, which is exact under conditions regarding the tail behaviour of $\Phi(u)$, followed by a number of approximations. 

The integration by parts\footnote{Related to Stein's identity} step
in \eqref{eq:closure_second_order} results in
\begin{align*}
    \frac{\dd m}{\dd t} &= -C \,\mathbb{E}\bigl(\nabla \Phi(u)\bigr),\\
    \frac{\dd C}{\dd t} &= -C\,\mathbb{E}\bigl(
    D^2 \Phi(u)\bigr) \,C.
\end{align*}
Now  we use the explicit expression \eqref{eq:phisc} for $\Phi(u)$ in terms of $G(u)$ to derive various approximations.
First, using this expression, we utilise the Gauss--Newton approximation of the Hessian 
$D^2 \Phi(u)$ to obtain
$$
D^2 \Phi(u) \approx DG(u)^\top \Gammas^{-1} DG(u).
$$
We also replace $\nabla \Phi(u)$ by its explicit expression, resulting from \eqref{eq:phisc}, to obtain the modified evolution equations
\begin{align*}
    \frac{\dd m}{\dd t} &= -C \,\mathbb{E}\bigl(
    DG(u)^\top \Gammas^{-1}(G(u)-w^\dagger)
    \bigr),\\
    \frac{\dd C}{\dd t} &= -C\,\mathbb{E}
    \bigl(DG(u)^\top \Gammas^{-1} DG(u)\bigr)  \,C.
\end{align*}
Secondly these equations are further modified by replacing expectations of product terms by
products of expectations to yield
\begin{align*}
    \frac{\dd m}{\dd t} &= -C \,\mathbb{E}
    \bigl(DG(u)\bigr)^\top \Gammas^{-1} \mathbb{E}\left(G(u)-w^\dagger\right),\\
    \frac{\dd C}{\dd t} &= -C\,\mathbb{E}
    \bigl(DG(u)\bigr)^\top\Gammas^{-1} \mathbb{E}\bigl(DG(u)\bigr) \,C.
\end{align*}
The final step consists in eliminating the Jacobian $DG(u)$ using
$$
C\,\mathbb{E}(DG(u))^\top = C^{uG},
$$
which is obtained by yet another integration by parts under the Gaussian $\Ng(m(t),C(t))$. Thus we
have recovered the Gaussian projected filter \eqref{eq:meanGPFcontIP}.

\label{lastpage}
\end{document}